\documentclass{amsart}

\usepackage{amsmath,amsthm,amscd,amsfonts,wasysym,amssymb,epic,eepic,bbm,tikz-cd,mathtools,multicol,enumitem,mathrsfs}
\usepackage[pagebackref,colorlinks=true,linkcolor=blue,citecolor=blue]{hyperref}
\usepackage[noabbrev]{cleveref}
\usepackage{MnSymbol}

\makeatletter
\newcounter{dummy}
\newcommand\myitem[1][]{\item[#1]\refstepcounter{dummy}\def\@currentlabel{#1}}
\makeatother

\allowdisplaybreaks

\newtheorem{thm}{Theorem}[section]
\newtheorem{cor}[thm]{Corollary}

\newtheorem{lem}[thm]{Lemma}
\newtheorem{prop}[thm]{Proposition}

\theoremstyle{definition}
\newtheorem{definition}[thm]{Definition}

\newtheoremstyle{boldremark}
  {6pt}   
  {6pt}   
  {\normalfont} 
  {}      
  {\bfseries}   
  {.}     
  { }     
  {}      

\theoremstyle{boldremark}
\newtheorem{rem}{Remark}[subsection]

\newcounter{remarkscounter}

\numberwithin{equation}{subsection}

\newcommand{\quash}[1]{}

\theoremstyle{definition}
\newtheorem{defn}[thm]{Definition}

\renewcommand{\bar}{\overline}

\def\D{{\mathcal D}}

\def\G{{\mathbb G}}

\def\g{{\mathfrak g}}

\def\Hom{{\mathrm{Hom}}}

\def\Aut{{\mathrm{Aut}\,}}
\def\End{{\mathrm{End}}}

\def\Id{{\mathrm{Id}}}
\def\gr{{\mathrm{gr}}}
\def\ker{{\mathrm{ker}\,}}

\def\coker{{\mathrm{coker}}}
\def\im{{\mathrm{im}\,}}
\def\ord{{\mathrm{ord}}}
\def\tr{{\mathrm{tr}}}

\def\Spec{{\mathrm{Spec}}}
\def\Sym{{\mathrm{Sym}}}

\def\dim{{\mathrm{dim}}}
\def\Ad{{\mathrm{Ad}}}
\def\ad{{\mathrm{ad}}}

\def\GL{{\mathrm{GL}}}

\def\SL{{\mathrm{SL}}}

\def\sl{{\mathfrak{sl}}}
\def\O{{\mathrm{O}}}
\def\SO{{\mathrm{SO}}}
\def\so{{\mathfrak{so}}}
\def\Sp{{\mathrm{Sp}}}

\def\u{{\mathfrak{u}}}

\def\Span{{\mathrm{Span}}}

\def\Mat{{\mathrm{Mat}}}

\def\R{{\mathbb{R}}}
\def\C{{\mathbb{C}}}

\def\A{{\mathbb{A}}}
\def\Z{{\mathbb{Z}}}
\def\N{{\mathbb{N}}}

\def\P{{\mathbb{P}}}

\def\Gr{{Gr}}

\def\Res{{\mathrm{Res}}}

\def\Ind{{\mathrm{Ind}}}

\def\res{{\mathrm{res}}}

\def\rk{{\mathrm{rk}}}

\def\op{{\mathrm{op}}}

\begin{document}

\title[The Geometric Schr\"odinger Model]{Geometrization of the Schr\"odinger Model for the Minimal Representation of an Even Orthogonal Group: The de Rham Setting}

\author{Aaron Slipper}
\address{Department of Mathematics\\
Duke University\\
Durham, NC 27708}
\email{aaron.slipper@duke.edu}

\begin{abstract}
We construct and compare three $\mathcal D$-module models for the minimal representation of the conformal group of an even-dimensional quadratic space. Let $V$ be a quadratic space over a field $\kappa$ of characteristic $0$, $C\subset V^*$ be the isotropic cone, $\Delta\in D_V$ be the associated Laplace--Beltrami operator, $G$ be the conformal group of $V$, and $D_C$ be the algebra of Grothendieck differential operators on $C$. We prove that the category of finitely generated $D_C$-modules is equivalent both to a Kazhdan--Laumon glued category attached to the smooth locus $C^o$ and to a category of ``harmonic'' twisted $\mathcal D$-modules on the projective conformal compactification $G/P \supset V$. The gluing is governed by the quadric Fourier transform, while the harmonic model is built from a distinguished $G$-equivariant sheaf $\mathcal H$ on $G/P$ extending the local quotient $D_V/D_V\Delta$. We prove a new geometric interpretation of higher symmetries of the Laplacian as global sections of $\mathcal H$, and use this connection to give a geometric proof of the theorem of Levasseur, Smith, and Stafford that $D_C$ is Noetherian despite the singularity of $C$. We also study, via a descent procedure we call ``$F$-moment descent,'' the algebraic geometry of the closure of the minimal nilpotent orbit of $G$, which is the quasiclassical analogue of the minimal representation. Finally, we analyze the filtered structure of $D_V/D_V\Delta$ as a right $D_C$-module, identifying its associated graded layers through a flat degeneration of an affine flag multicone whose special fiber is the Rees space of a natural ideal in $\kappa[\overline{\mathbb O}_{\min}]$.
\end{abstract}

\maketitle

\setcounter{tocdepth}{2}
\tableofcontents

\section{Introduction}

The primary object of this paper is to provide three different models for the categorification of the minimal representation of the even orthogonal group via $\mathcal{D}$-modules, and to prove their equivalence.\footnote{A followup paper will consider an $\ell$-adic geometrization of this representation, along the lines of the geometrization of the Weil representation due to P. Deligne, S. Gurevich, and R. Hadani \cite{DeligneKazhdanFourier, GurevichHadani2007}.} In particular, let $V$ be a vector space of dimension $n = 2k$ over a field $\kappa$ of characteristic 0, with nondegenerate quadratic form $Q$; let $C \subset V^*$ be the isotropic locus of $Q^*$, and $C^o$ its smooth locus. We will show the equivalence of three different categories:

\vspace{3mm}

i) the category of modules over the algebra $D_C$ of (global) Grothendieck differential operators on $C$;

\vspace{3mm}

ii) The Kazhdan-Laumon glued category of $\mathcal{D}$-modules on the smooth locus of the cone $C^o$ (see Section \ref{KazhLaumGlueSect}); and

\vspace{3mm}

iii) The category of ``harmonic" $\mathcal{D}$-modules on the conformal compactification of $V$.

\vspace{3mm}

We will also discuss the quasiclassical analogue of the Schr\"odinger Model. We note that because $C$ is singular, it is something of a miracle that $D_C$ is well-behaved; as a byproduct of the proofs of the above equivalences, this paper will also provide a new, geometric proof that $D_C$ is Noetherian.

Let $G$ be the conformal group of $V$; it is an even orthogonal group containing $\O(Q)$ as a Levi factor. Unlike $\O(Q)$, the group $G$ does \textit{not} act geometrically on $C$. However, it manifestly acts on the conformal compactification of $V$, and so on the category iii). As we will explain, this categorical action on the space of harmonic $\mathcal{D}$-modules can be viewed as a geometrization of the minimal representation of $G$. The corresponding $G$-action on $D_C$-modules then provides a geometrization of the ``Schr\"odinger model" of this minimal representation.

In the present introductory section, we will first provide some background on the Schr\"odinger model of the minimal representation of the even orthogonal group, stressing its similarity to the metaplectic Weil representation. We will then offer an overview of T. Kobayashi's $F$-method, which is the specific approach that we categorify. Finally, we provide a more detailed outline of the present paper.

\subsection{\texorpdfstring{The Minimal Representation of $\O(p+1,q+1)$ and the Metaplectic Weil Representation}{The Minimal Representation of O(p+1,q+1) and the Metaplectic Weil Representation}} 
Let $C$ denote the isotropic cone of a quadratic form of signature $(p,q)$ in $\R^n= \R^{p+q}$: 

\begin{align}\label{Cquadric}
Q(x_1, \ldots, x_n) :&=  x_1^2 + \cdots + x_p^2-x_{p+1}^2 -\cdots -x_n^2, \\
C :&= \{(x_1, \ldots, x_n) \in \R^n : Q(x) = 0\}.
\end{align}

\noindent Now and for the entirety of this paper, we will assume that $n$ is even and at least 4.

Consider the Hilbert space $L^2(C)$, taken with respect to a natural measure on $C$ coming from the embedding into $\R^n$. Clearly  $\text{O}(p,q)$ acts on $L^2(C)$, as $\text{O}(p,q)$ acts geometrically on $C$. It is somewhat more surprising that this representation extends to an action of $\text{O}(p+1, q+1)$ on $L^2(C)$, where $\O(p,q)$ lives in $\O(p+1, q+1)$ as a Levi factor of a maximal parabolic. This action does \textit{not} arise as an action on the underlying geometry of $C$. It turns out that $L^2(C)$ is a model -- the so-called Schr\"odinger Model -- for the minimal representation of $\text{O}(p+1, q+1)$ \cite{Kobayashi:Mano}. That is to say, $L^2(C)$ is the representation with minimal Gelfand-Kirillov dimension \cite{GanSavin}.

As pointed out by Kazhdan \cite{GurK:Cone}, the Schr\"odinger model of the minimal representation of $\O(p+1,q+1)$ has a striking resemblance to the metaplectic Weil representation. We now briefly recall the Weil representation. Let $\mathcal{V}$ be a symplectic vector space, with $\mathcal{L}$ a Lagrangian subspace. $\GL(\mathcal{L})$ manifestly acts on $L^2(\mathcal{L})$ (unitarized appropriately), since $\GL(\mathcal{L})$ acts on $\mathcal{L}$. This ``geometric" action extends to a representation of $\textrm{Mp}_{2n}(\R)$ on $L^2(\mathcal{L})$, where $\textrm{Mp}_{2n}(\R)$ is a double-cover of $\Sp_{2n}(\R)$  -- the Weil Representation. $\GL_n$ lives as a Levi factor of a Siegel parabolic of $\textrm{Mp}_{2n}$. 

It is standard to present the Weil representation's action on $L^2(\mathcal{L})$ via generators of $\textrm{Mp}_{2n}$. Under this presentation, the Levi $GL_n \subseteq \textrm{Mp}_{2n}$ acts via the geometric action on $\mathcal{L}$; the unipotent radical of the Siegel parabolic acts via pointwise multiplication by certain Gaussians. (These are often called ``modulations" \cite{getz2025modulationgroups}.) Finally, the Weyl element that conjugates the Siegel parabolic to its opposite acts via the standard linear Fourier transform on $L^2(\mathcal{L})$ \cite{Prasad:Weil}.

A parallel story exists for the minimal representation of $O(p+1,q+1)$. One has $\O(p,q) \subseteq P \subseteq \O(p+1,q+1)$, where $P$ is a maximal parabolic with Levi $\G_m \times \O(p,q)$. The unipotent radical of $P$ acts on functions in $L^2(C)$ via pointwise multiplication by unitary characters of $\R^n$ (``modulations"). And, finally, the analogous Weyl element in $\O(p+1,q+1)$, which conjugates $P$ to $P^{\op}$, corresponds to an involutive unitary operator on $L^2(C)$, which is the natural analogue of the Fourier transform for the quadric cone $C$. Explicitly, one may write this transform as a convolution by a kind of Bessel function \cite{Kobayashi:Mano}, just as the classical Fourier transform can be written as a convolution with an exponential.

The Schr\"odinger model of the minimal representation of the even orthogonal group has been extended to arbitrary local fields by N. Gurevich and D. Kazhdan \cite{GurK:Cone}. Moreover, the Bessel Fourier transform given by the action of the Weyl element is perhaps the simplest example of a ``nonlinear" Fourier transform of the kind expected (e.g., in Braverman-Kazhdan theory and relative Langlands duality \cite{BK-lifting, Ngo:Hankel, benzvi2024relativelanglandsduality}) to exist in great generality. Remarkably, J. Getz has provided a full Poisson Summation formula on quadric cones using this transform (with a complete geometric description of ``boundary terms" and no restriction on test functions) \cite{Getz:Quadric}. Such Poisson summation formulae are expected exist on function spaces associated to a large class of algebraic varieties (e.g., reductive monoids, spherical varieties \cite{SakellaridisSph, Ngo:Hankel, SakellaridisICM2022}) and, in the mold of Tate's thesis \cite{Tate1967FourierAnalysis}, are expected to yield many instances of Langlands functoriality. The analogue of the metaplectic/minimal representation is expected to arise in this more general context too: geometric symmetries of the variety, along with modulations, and these (conjectural) Fourier transforms give rise to so-called ``modulation groups" \cite{getz2025modulationgroups}.

The function space of a quadric cone is, after the case of vector spaces, one of the simplest and best-understood examples of the above paradigm. It is worth noting, however, that the quadric cone's Fourier transform is already a great deal more subtle than the (very classical) Fourier transform on vector spaces: the Bessel kernel is singular \cite{Kobayashi:Mano}, and there are nontrivial boundary terms in the Poisson Summation Formula \cite{Getz:Quadric}. In particular, these appear to be related to the (algebro-geometric) singularity of the quadric cone at the origin.

Finally, we note that in the \textit{global} setting (that is to say, on the moduli stack of principal $\SO_{2n}$-torsors over a global curve, which corresponds to the geometrization of the Adelic minimal representation), V. Lafforgue and S. Lysenko have constructed a (de-Rham) geometrization of the minimal representation \cite{LafforgueLysenko2013Geometrizing} in the spirit of the Geometric Langlands Program. (Lysenko constructed a global geometrization of the metaplectic Weil representation in \cite{Lysenko2006MetaplecticThetaSheaves}.) In this paper, we work in the more classical setting over a field $\kappa$ of characteristic 0; we do not consider stacky or derived phenomena. 

\subsection{\texorpdfstring{The $F$-Method}{The F Method}} T. Kobayashi introduced the idea of the $F$-method ($F$ for ``Fourier" \cite{KobayashiPevzner2016Fmethod}), which we will consider from a purely algebraic standpoint. Presently, we shall offer a brief review of the $F$-method (in the Archimedean setting) for quadric cones.

The basic idea is that while $G := \O(p+1, q+1)$ does not act regularly on $\R^n = \R^{p,q}$, it \textit{does} act rationally via conformal transformations of $\R^n$ with respect to the quadratic form $Q$.\footnote{Physicists will recognize this phenomenon: $O(2, 4)$ is the conformal group of Minkowski space-time $\R^{1,3}$.}  We note, in particular, that the Weyl element $w_0$ (which corresponds to the quadratic Fourier transform) acts on $\R^n$ via a conformal inversion: for $v \in \R^n$,
\begin{equation}\label{inversion}
w_0(v) = -\frac{v}{Q(v)}.
\end{equation}

Given a function on $C \subset \R^n$, we may view it as a distribution on $\R^n$ supported on $C$ and consider its (inverse) Fourier transform on all of $\R^n$. Because the Fourier transform of the resulting distribution on $\R^n$ is supported on $C$, the resulting function is harmonic, that is to say, annihilated by the Laplacian.\footnote{In Minkowski spacetime $\R^{1,3}$, functions/distributions whose Fourier transforms are supported on the light cone are called ``massless waves."} Thus the usual linear Fourier transform gives us an isomorphism of function spaces $L^2(C) \to \mathcal{C}_{\textrm{harmonic}}(\R^n)$.\footnote{We are intentionally not getting too bogged down in the details of the functional analysis here.} But, up to a conformal factor (which may be absorbed into the action of $G$), conformal transformations preserve harmonic functions.\footnote{Physicists again may recognize this phenomenon in the conformal invariance of electro-magnetism.}

More explicitly, a distribution on $\R^n$ which arises as a linear Fourier transform of a distribution supported on $C$ is annihilated by the Laplacian,
\begin{equation}\label{Laplace}
\Delta:=\frac{\partial^2}{\partial x_1 ^2}+ \cdots + \frac{\partial^2}{\partial x_p ^2}- \frac{\partial^2}{\partial x_{p+1} ^2} -\cdots-\frac{\partial^2}{\partial x_{n} ^2}.
\end{equation}

Let $\varphi \in \O(p+1, q+1)$, viewed as a conformal diffeomorphism of $\R^{p,q}$. There is a function 
\begin{equation}\label{OGConfFactr}
    \Omega_{\varphi} : \R^{p,q} \to \R,
\end{equation} 

\noindent the so-called conformal factor, defined by 
\begin{equation}
    \varphi^*g_0 = \Omega_\varphi^2g_0 
\end{equation}

\noindent where $g_0$ is the flat metric on $\R^n$ associated with the quadratic form $Q$. (Such a scalar function $\Omega_\varphi$ exists by the definition of conformality.) Then
\begin{equation}
\Delta(f) = 0 \implies \Delta \left( \Omega_\varphi(x)^{\frac{n-2}{2}} (\varphi^*f)(x)\right) = 0.
\end{equation}

\noindent Defining the action 
\begin{equation}\label{conformalaction}
\varphi:\{ x\mapsto f(x)\} \mapsto \{x \mapsto \Omega_\varphi(x)^{\frac{n-2}{2}} (\varphi^*f)(x)\},
\end{equation}

\noindent we see that $\O(p+1,q+1)$ acts\footnote{\label{Kelvinfootnote}These ``twisted" conformal transforms may be thought of as generalizations of the classical Kelvin transform \cite{AxlerBourdonRamey2001KelvinTransform, Thomson1845Kelvin}. In particular, our Weyl-Fourier element $w_0$ (see \ref{WeylFourier} below) corresponds to a standard inversion and so is precisely the classical Kelvin transform.} on the space of harmonic distributions
on $\R^n$, 
and so, 
by transfer of structure, on distributions on 
$C$ 
\cite{KobayashiOrsted2003MinimalOpqI}.
This is precisely the Schr\"odinger model of the minimal representation of the orthogonal group.

While conformal transformations come from differential geometry and are not necessarily algebraic, one may recast the above into a more algebro-geometric mold. In particular, these conformal transformations may be understood as coming from a bona fide action of $G$ on a flag space $G/P$, which compactifies $\R^{p,q}$.  Moreover, as we will see, the theory of $\mathcal{D}$-modules offers a purely algebro-geometric interpretation of the conformal factors; they may be interpreted as sections of a certain line bundle over $G/P$, while the harmonic functions themselves will be interpreted as sections of the solutions sheaf of a corresponding \textit{twisted} $\mathcal{D}$-module. 

Hence part of this paper may be viewed as offering a purely algebraic interpretation of the theory of conformal densities and the Kelvin Transform (see Footnote \ref{Kelvinfootnote}).

\subsection{Outline of This Paper} We now outline the structure and main results of this paper.

\vspace{3mm}

Section \ref{GeometryofConfPact} fixes the geometric notation used throughout the paper. 

\vspace{3mm}

Section \ref{orthgroupsubsect} introduces the split quadratic space $(V,Q)$, the orthogonal group $H=\O(V,Q)$, and the conformal group $G=\O(V^+,Q^+)$.\footnote{This is also known as the ``modulation group" of the cone under the standard embedding $C \subset V^*$ \cite{getz2025modulationgroups}.} We single out a parabolic subgroup $P\subset G$, the stabilizer of an isotropic line. 

\vspace{5mm}

Section \ref{FlagGeomSection} studies the partial flag variety $G/P$, also known as the projective conformal compactification of $V$.\footnote{Physicists will recognize this as Penrose's conformal compactification of spacetime \cite{penrose1963asymptotic, penrose1964conformal, penrose2011republication}.} Explicitly, it is the smooth quadric hypersurface $\{[a:v:b]:Q(v)+ab=0\}\subset\P^{n+1}$. We identify the vector space $V\cong U^{\op}P/P$ with the big Bruhat cell in $G/P$; the left action of the Weyl element $w_0$ on this chart is given by the rational conformal inversion $v\mapsto -v/Q(v)$. We will always let $C \subset V^*$ denote the isotropic cone in the \textit{dual} of $V$; this paper is primarily dedicated to studying the surprising geometric properties of this variety.
 
\vspace{3mm}

Section \ref{CotangentandDiffOpSect} studies the affinization of $T^*C^o$ and the algebra $D_C$ of Grothendieck differential operators on $C$.

\vspace{3mm}

In Section \ref{cotangentsubsect}, we prove that the locally closed embedding $C^o\subset V^*$ satisfies what we call $F$\textit{-moment descent}: after applying the quasiclassical Fourier transform, the moment map for the rational conformal action of $G$ on $T^*V$ descends from $C^o\times V$ to $T^*C^o$. This gives an explicit presentation of $\kappa[T^*C^o]$ by invariant functions and identifies the affinization of $T^*C^o$ with the closure of the minimal nilpotent orbit in $\g^*$.\footnote{It is known in folklore that the affinization of the cotangent bundle of the smooth locus of a quadric cone is the closure of a minimal nilpotent orbit of an orthogonal group; this result appeared explicitly in \cite{getz2025modulationgroups}.}

\vspace{3mm}

\noindent \textbf{Proposition~\ref{CotBundDescr}}. \textit{The affinization of $T^*C^o$ is isomorphic to $\overline{\mathbb O}_{\min}$, and the matrix coefficients of the descended $F$-moment map generate $\kappa[T^*C^o]$.}

\vspace{3mm}

\noindent We then discuss an algebro-geometric presentation of this orbit closure, first as the scheme of square-zero rank $\le 2$ orthogonal linear operators on $V^+$, and then, with the reduced scheme structure, as the subvariety of $\g$ cut out by the square-zero equations, the orthogonality equations, and the $4\times4$ Pfaffians.

\vspace{3mm}

Section \ref{DiffOpSect} outlines the quantization of $F$-moment descent. The final payoff from Section \ref{CotangentandDiffOpSect} is the construction of a map $\rho:\g\to D_C$, obtained by composing the linear Weyl-algebra Fourier involution with the infinitesimal conformal action of $G$ on $V$ and, after the appropriate normalization, descending to the cone. This is the algebraic avatar of Kobayashi's $F$-method. 

\vspace{3mm}

Section \ref{TwistedGActSection} shows how to properly twist the action of $G$ on $\kappa(V) = \kappa(G/P)$ so that the quantized $F$-moment map will ultimately descend to $C$. This provides the purely algebraic avatar of conformal transformations from differential geometry. Moreover, we explain how to interpret this twisted action as arising from the action of $G$ on the line bundle $\mathcal{L} \cong \mathcal{O}(1-k)$. In particular, the action of $G$ on local sections of $\mathcal{L}$, trivialized by a fixed local section of $\mathcal{L}$ over the big Bruhat cell $V \subset G/P$, gives precisely the extra conformal factors. The action of $G$ on $\kappa(V)$ can be transferred to a rational action of $G$ on local differential operators via conjugation. Section \ref{TwistedGActSection} concludes by calculating the infinitesimalization of this action explicitly. 

\vspace{3mm}

Section \ref{InteractionLaplaceSection} shows that, under the correctly twisted action from Section \ref{TwistedGActSection}, $G$ simply scales the Laplace-Beltrami operator $\Delta$ by a nonzero cocycle. Section \ref{InteractionLaplaceSection} also provides an explicit formula for the twisted action of $w_0$ on functions; in the local coordinates of $V$, this is known classically as the Kelvin transform $K$ \cite{AxlerBourdonRamey2001HarmonicFunctionTheory}. We use the transform $K$ throughout the rest of the paper. 

\vspace{3mm}

Section \ref{rhoambSect} calculates the composition of the infinitesimalized action of $G$ on $D_V$ with the linear Fourier transform $\tau$. This does \textit{not} yet lead to a map $\mathcal{U}(\g) \to D_C$.

\vspace{3mm}

Section \ref{AmbientCorrectionSection} corrects the map $\rho_{\textrm{amb}}$ appearing in Section \ref{rhoambSect}. This is the $\mathcal{D}$-module analogue of passing from distributions to functions. The result is a Lie algebra homomorphism $\rho:\g\to D_C$, which extends to a map $\mathcal{U}(\g)\to D_C$. This map was first discovered by Goncharov \cite{Goncharov1976Weil}; it was later proved to be surjective in \cite{LevasseurSmithStafford1989Joseph}. Later in the paper (Section \ref{H1VanishingNoetherianSection}) we provide an independent proof of surjectivity.

\vspace{3mm}

Finally, Section \ref{QuadricExplicitSect} offers a detailed study of the quadric Fourier transform $\mathcal F$, the automorphism of $D_C$ induced by the Weyl element $w_0$. We note that many of the calculations here appear, with a different normalization of the quadric $Q$, in \cite{Kobayashi:Mano}.

\vspace{3mm}

Section \ref{KazhLaumGlueSect} proves the Kazhdan--Laumon model of the category of $D_C$-modules. The quadric Fourier transform $\mathcal F$ supplies the gluing operator relating two copies of the category of $\mathcal D$-modules on $C^o$. This section is very much modeled on \cite{BezrukavnikovBravermanPositselskii2002Gluing}. 

\vspace{3mm}

Section \ref{TwoGradings} decomposes $\kappa[C]$ and $D_C$ as representations of $H = \O(Q)$ and the scaling action of $\G_m$. This is in direct analogy with the decomposition of $\mathcal{O}(G/U)$ and $D(G/U)$ via highest weights under the left action of $G$ and via characters under the right action of $T$.

\vspace{3mm}

Section \ref{ShapovalovDetSect} provides the main calculation behind the proof of the gluing theorem: the quadric analogue of the Shapovalov determinant \cite{Shapovalov1974Structure}. For $d\geq 0$, let $P^d:=\kappa[C]_d$ be the space of degree-$d$ functions on the cone. The quadratic form gives a nondegenerate $H$-invariant pairing $\langle\ ,\ \rangle_d$ on $P^d$.

\vspace{3mm}

\noindent \textbf{Proposition~\ref{Shapovalovdet}}. \textit{Let $\{f_i\}$ be a basis of $P^d$, and let $\{g_i\}$ be the dual basis with respect to $\langle\ ,\ \rangle_d$. Then the lowering-then-raising quadric Shapovalov operator is $\sum_i f_i\mathcal F(g_i)=\prod_{j=1}^{d}(E-j+1)\prod_{j=1}^{d}(E+k-j-1)$, while the opposite, raising-then-lowering operator is $\sum_i\mathcal F(g_i)f_i=\prod_{j=1}^{d}(E+2k+j-3)\prod_{j=1}^{d}(E+k+j-1)$.}

\vspace{3mm}

\noindent We note that this explicit presentation of the quadric Shapovalov determinant is used again in the proof of Theorem~\ref{DeltaGradedIdealTheorem} (Section~\ref{FiltrationSubSect}) and in the proof of Theorem~\ref{SecondEquiv} (Section~\ref{ProofofEquivSubSect}). 

\vspace{3mm}

Section \ref{GluingProofSect} proves the gluing theorem:

\vspace{3mm}

\noindent \textbf{Theorem~\ref{thm:glued-equiv}}. \textit{The category obtained by the Kazhdan--Laumon gluing \cite{KazhdanLaumon1988Gluing, Polishchuk2001GluingBasicAffine} of two copies of the category of coherent $\mathcal D$-modules on $C^o$ along the quadric Fourier transform $\mathcal F$ is equivalent to the category of finitely generated modules over the algebra $D_C$ of Grothendieck differential operators on $C$.}

\vspace{3mm}

\noindent This is the quadric-cone analogue of the gluing theorem of Bezrukavnikov, Braverman, and Positselskii for basic affine space \cite{BezrukavnikovBravermanPositselskii2002Gluing}, and it is a remarkably direct consequence of the fact that the two Fourier-dual Shapovalov operators share no common roots.\footnote{We note that independently this result was recently generalized to all Braverman-Kazhdan spaces associated to a maximal parabolic by C.-H. Hsu \cite{HsuWeylAlgebrasBK}.}

\vspace{3mm}

Section \ref{DModHarmSect} introduces the category of ``harmonic" $\mathcal{D}$-modules on $G/P$, and its basic object, the harmonic sheaf $\mathcal{H}$. 

\vspace{3mm}

Section \ref{IntroHarmSheavesSection} introduces the category $\mathscr{H}$ in analogy with T. Kobayashi's $F$-method, and states the main equivalence between $\mathscr{H}$ and $D_C\textbf{-mod}^{\textrm{fg}}$. 

\vspace{3mm}

Section \ref{IntroHarmsheaf} constructs the harmonic ideal $\mathcal{J}_\Delta$ and the harmonic sheaf $\mathcal{H}:= \mathcal{D}_\mathcal{L}/\mathcal{J}_\Delta$ on $G/P$. The sheaves $\mathcal{J}_\Delta$ and $\mathcal{H}$ are global geometric replacements for the left harmonic ideal $D_V\Delta \subset D_V$ and the harmonic quotient $D_V/D_V\Delta$, respectively. 

\vspace{3mm}

We next turn to global sections of $\mathcal{H}$. Sections \ref{SymmetryImpliesGlobalSect} and \ref{GlobalImpliesSymmetrySect} prove the following theorem: 

\vspace{3mm}

\noindent \textbf{Theorem~\ref{GlobalSectionsTheorem}}. \textit{Restriction of $\Gamma(\mathcal{H})$ to the big Bruhat cell identifies $\Gamma(G/P,\mathcal H)$ with the normalizer quotient $N(D_V\Delta)/D_V\Delta$, and hence with $D_C$. Equivalently, global sections of $\mathcal H$ are the higher symmetries of the Laplacian.}

\vspace{3mm}

\noindent This gives a geometric interpretation of higher symmetries of the Laplacian in the sense of Eastwood and Michel \cite{Eastwood2005HigherSymmetries, Michel2014HigherSymmetries}. 

\vspace{3mm}

More specifically, Section \ref{SymmetryImpliesGlobalSect} proves that every symmetry $\xi \bmod D_V\Delta \in N(D_V\Delta)/D_V\Delta$ extends globally as a section of $\mathcal H$. The key local input is a principal-symbol argument: one first proves injectivity for the Hamiltonian derivation $\{Q^*,-\}$ on the localized quotient $\left(\kappa[V\times V^*]/(Q^*)\right)[1/Q]$, and then lifts this, using the order filtration, to the corresponding injectivity statement for $\ad_\Delta$ on the pole quotient $D_{V_Q}/(D_V+D_{V_Q}\Delta)$. The snake lemma then shows that the Kelvin conjugate $K\xi K$ has a regular representative on the standard chart; equivalently, that $\xi$ has a compatible representative on the opposite Bruhat chart. These representatives glue over $X_{\textrm{big}}=V\cup w_0V$, and a filtered Hartogs argument extends the resulting section uniquely across the codimension-$2$ complement $G/P\setminus X_{\textrm{big}}=\P(C_V)$.

\vspace{3mm}

Section \ref{GlobalImpliesSymmetrySect} proves the converse of Section \ref{SymmetryImpliesGlobalSect}: a global section of $\mathcal{H}$ is a symmetry of the Laplacian. This argument is quite delicate. The section begins by developing the Fischer decomposition for $\kappa[V]$, together with the corresponding right Fischer decomposition for differential operators. These tools give precise control over pole orders after applying the Kelvin transform. The main argument goes as follows: a global section $\xi$ is regular on both $V$ and $w_0V$; hence both $\xi$ and its Kelvin conjugate $K\xi K \in D_{V_{w_0}}/D_{V_{w_0}}\Delta$ naturally lie in $D_V/D_V\Delta \subset D_{V_Q}/D_{V_Q\Delta}$. Writing $\Delta\xi=\delta_\xi\Delta+R$ by the right Fischer decomposition, with $R\in\kappa[V]\otimes_\kappa Z_\partial$, it remains to prove that the Fischer remainder $R$ vanishes. Now, Kelvin-regularity, Fischer decomposition, and the identity $K\Delta K= Q^2\Delta$ imply that $R(Kh)$ lies in $K(Q^2\kappa[V])$ for every harmonic polynomial $h\in Z_\Delta$. The final step is a contrapositive pole-order argument. If $R\neq 0$, then we may choose a suitable nonzero null linear form $\ell$ (depending on $R$). Since $\ell$ is null, all powers $\ell^N$ are harmonic. For all sufficiently large $N$, we find that $R(K(\ell^N))$ has a pole too large to lie in $K(Q^2\kappa[V])$. Hence $R=0$, so $\Delta\xi=\delta_\xi\Delta$, and $\xi\bmod D_V\Delta$ is a symmetry of the Laplacian.

\vspace{3mm}

Section \ref{SingSupptRichardsonSection} describes the geometry of the moment map of $G/P$, the associated Richardson variety, and the singular support of harmonic sheaves. The main results can be summarized as follows:

\vspace{3mm}

\noindent \textbf{Proposition}. \textit{The moment map $\mu: T^*G/P \to \g^*$ is birational onto its image, the Richardson variety corresponding to the closure of the nilpotent orbit associated to the partition $(3, 1^{2k-1})$. The closure of the \textnormal{minimal} nilpotent orbit, $\overline{\mathbb{O}}_{\min}$, corresponding to the partition $(2, 1^{2k})$, is a proper closed subset of this Richardson variety, with preimage $G \times^P C \subset G \times^P \mathfrak{u} \cong T^*G/P$. The harmonic sheaf $\mathcal{H}$ has singular support exactly equal to $G \times^P C$, while a general object of $\mathscr{H}$ has singular support contained in a conical closed subvariety of $G \times^P C$.}

\vspace{3mm}

\noindent Section \ref{SingSupptRichardsonSection} also shows that the infinitesimal action map $\mathcal{U}(\g) \to \Gamma(\mathcal{D}_\mathcal{L})$ is surjective onto the global sections of $\mathcal{L}$-twisted differential operators.

\vspace{3mm}

Section \ref{H1VanishingNoetherianSection} proves the cohomological vanishing needed to lift global sections of $\mathcal H$ to $\Gamma(\mathcal{D}_\mathcal{L})$:

\vspace{3mm}

\noindent \textbf{Theorem~\ref{HOneJDeltaVanishingTheorem}}. \textit{$H^1(G/P,\mathcal J_\Delta)=0$.}

\vspace{3mm}

\noindent The proof is carried out by an explicit \v{C}ech-cocycle computation on the open set $V \cup w_0V$: we decompose a local section of the harmonic ideal on $V_Q = V \cap w_0V$ into a difference of local sections of the harmonic ideal over $V$ and over $w_0V$. This decomposition is algorithmic; it is accomplished by induction on the so-called ``positive excess" of an operator in $D_V[1/Q]\cdot\Delta$. We then use a depth/codimension-$2$ argument to extend the vanishing from this open set to all of $G/P$.

\vspace{3mm}

\noindent Together with the generation of global twisted differential operators by the infinitesimal $G$-action, this gives a geometric proof of the theorem of Levasseur, Smith, and Stafford on the Joseph ideal for even orthogonal groups \cite{LevasseurSmithStafford1989Joseph}:

\vspace{3mm}

\noindent \textbf{Theorem~\ref{rhoSurjective}}. \textit{The map $\rho:\mathcal U(\g)\to D_C$ is surjective. In particular, $D_C$ is finitely generated and Noetherian.}

\vspace{3mm}

\noindent The Noetherianness of $D_C$ is notable because $C$ is singular; for contrast, the algebra of differential operators on the cubic cone is not finitely generated \cite{BernsteinGelfandGelfand1972CubicCone}.

\vspace{3mm}

Section \ref{CompatibilitiesSect} uses Theorem~\ref{GlobalSectionsTheorem} to construct a $(\mathcal{D}_{\mathcal{L}}, D_C)$-bimodule structure on $\mathcal{H}$. It also checks that the adjoint $G$-action on $\mathcal{U}(\g)$ descends under the map $\rho: \mathcal{U}(\g) \to D_C$ to the manifest geometric action of $G$ on $\Gamma(\mathcal{H})$, once we identify $\Gamma(\mathcal{H})$ and $D_C$ via Theorem~\ref{GlobalSectionsTheorem}. This gives a geometric interpretation for the mysterious $G$-action on $D_C$.

\vspace{3mm}

Section \ref{FiltrationSubSect} studies the filtered structure of the local harmonic quotient $M:=D_V/D_V\Delta$. There are two natural filtrations on $M$: the $\Delta$-filtration $F_i=\ker(\Delta^{i+1}:M\to M)$, and the Kelvin filtration, defined by measuring the pole order of a local section $m \in M = \mathcal{H}(V)$ along the boundary divisor $G/P \setminus V$. In degree $0$, these two filtrations agree: this is precisely Theorem~\ref{GlobalSectionsTheorem}. A class in $D_V/D_V\Delta$ extends to a global section of $\mathcal H$ exactly when it lies in $F_0$; equivalently, when it is a higher symmetry of the Laplacian. It is quite surprising that the filtrations agree in general:

\vspace{3mm}

\noindent \textbf{Theorem~\ref{FiltrationTheorem}}. \textit{For every $i\geq -1$, the $\Delta$-filtration and the Kelvin filtration on $D_V/D_V\Delta$ agree.}

\vspace{3mm}

\noindent The $\Delta$-filtration $F_i$ filters $M$ by right $D_C$-submodules. Thus the maps $\Delta^i:F_i/F_{i-1}\hookrightarrow F_0\cong D_C$ produce right ideals $I^{[i]}\subset D_C$. To compute them, we pass to the Bernstein associated graded. There $\gr_{\mathsf B}M$ is $R=\kappa[V\times C]$, and $\ad_\Delta$ becomes the locally nilpotent derivation $\delta=\{Q^*,-\}$, whose invariant ring is $S=R^\delta\cong\kappa[\overline{\mathbb O}_{\min}]$. The filtration $\mathscr F_iR=\ker(\delta^{i+1})$ of the commutative algebra $R$ has a geometric description in terms of the affine flag multicone $\widehat{\operatorname{OFl}}(1,2;V^+)$. The fiber $u_+=1$ is $V\times C$, while the special fiber $u_+=0$ is the affine Rees space of the ideal $\mathfrak a:=((\nu^\sharp)_1,\ldots,(\nu^\sharp)_{2k},\alpha)\subset S=\kappa[\overline{\mathbb O}_{\min}]$, where $\alpha=\langle\nu,v\rangle$. Under the identification $\overline{\mathbb O}_{\min}\cong\overline{T^*C^o}$, this is the reduced ideal of the fiber over $0\in C$; equivalently, the radical ideal generated by the quasiclassical Fourier-dual coordinate functions of the zero section.

\vspace{3mm}

\noindent \textbf{Theorem~\ref{MulticoneReesDeformationTheorem}}. \textit{The morphism $\widehat{\operatorname{OFl}}(1,2;V^+)\to\A^1$ defined by $u_+$ is flat. Its nonzero fibers are isomorphic to $V\times C$, and its special fiber is $\Spec\bigl(\bigoplus_{i\geq0}\mathfrak a^it^i\bigr)$.}

\vspace{3mm}

\noindent This deformation identifies the associated graded of the classical $\delta$-filtration with the Rees algebra of $\mathfrak a$. Quantizing this calculation gives the associated graded layers of the $\Delta$-filtration, and hence the ideals $I^{[i]}$.

\vspace{3mm}

\noindent \textbf{Theorem~\ref{DeltaGradedIdealTheorem}}. \textit{If $I^{[i]}$ is the image of $\Delta^i:F_i/F_{i-1}\to F_0\cong D_C$, then $\gr_{\mathsf B}I^{[i]}=\mathfrak a^i$, and $I^{[i]}=P^iD_C+c_i(E)D_C$, where $c_i(E)=\prod_{r=1}^{i}(E+k-r-1)$.}

\vspace{3mm}

Finally, Section \ref{ProofofEquivSubSect} proves the equivalence between $D_C$-modules and harmonic twisted $\mathcal D$-modules on $G/P$. The proof shows that the global sections functor $\Gamma$, while not conservative, is exact on the category of coherent $\mathcal{D}_\mathcal{L}$-modules. The argument is inspired by Bezrukavnikov's proof of Beilinson-Bernstein localization via differential operators on $G/U$ \cite{BezrukavnikovLocalizationNotes}. In particular, it makes critical use of the Shapovalov determinant formula in the spirit of \cite{BezrukavnikovBravermanPositselskii2002Gluing}. We note that while this argument uses ideas related to Beilinson--Bernstein localization, it is \textit{not} a direct application of Beilinson--Bernstein localization: the twist $\mathcal L\cong\mathcal O(1-k)$ lies outside the usual Beilinson--Bernstein range. 

\vspace{3mm}

\noindent \textbf{Theorem~\ref{SecondEquiv}}. \textit{Let $\mathscr H$ be the category of coherent $\mathcal D_{\mathcal L}$-modules admitting a finite $(\mathcal H,D_C)$-presentation. Then the harmonic transform $N\mapsto\mathcal H\otimes_{D_C}N$, resulting from the $(\mathcal{D}_{\mathcal{L}}, D_C)$-bimodule structure of $\mathcal{H}$, gives an equivalence between $D_C\textbf{-mod}^{\,\mathrm{fg}}$ and $\mathscr H$, with an inverse functor given by global sections.}

\vspace{3mm}

An appendix (Section \ref{examplesSect}) considers explicit examples of both the quadric Fourier transform and the harmonic transform of various $D_C$-modules, including the structure sheaf $\kappa[C]$, the apex Dirac mass, and a nonzero Dirac mass. A final section discusses the special case of $n = 6$ (and $k=3$), where the $\Z/2\Z$-symmetry of the Fourier transform extends to a full $S_3$-action on $D_C$. This is simultaneously an incarnation of the Gelfand-Graev action for $\SL_3/U$ and of the triality symmetry of $\so(8)$. Our quadric Fourier transform in this case corresponds to the longest element of the Weyl group in $S_3$.

\vspace{3mm}

In summary, this paper shows that the category of finitely generated $D_C$-modules, viewed as a de Rham categorification of the Schr\"odinger model $L^2(C)$, admits two complementary geometric descriptions. It is a Kazhdan--Laumon glued category, where the gluing is governed by the quadric Fourier transform, and it is also equivalent to a finite-presentation category of harmonic twisted $\mathcal D$-modules on the smooth projective quadric $G/P$. These two models explain, respectively, the Fourier transform and the conformal $G$-action on the Schr\"odinger model, neither of which comes from an ordinary geometric action of $G$ on the singular cone $C$.

\subsection{Acknowledgments} I would like to thank Jayce Getz and my advisor Ng\^o Bao Ch\^au for their generous mentorship, and Laurens Gunnarsen for his support, encouragement, and help throughout this project, especially with the physics-related ideas. This project was inspired in part by work with Marie-H\'el\`ene Tome under Duke's Research Scholars program. I would also like to thank Victor Ginzburg for first teaching me about the mathematics of $\mathcal{D}$-modules, and for several helpful comments on the manuscript; my co-advisor Kazuya Kato for his encouragement; Tom Gannon, Nikolay Grantcharov, Chun-Hsien Hsu, Takeshi Kobayashi, Zhilin Luo, Calder Morton-Ferguson, Semon Rezchikov, Minh-T\^am Trinh, Tong Zhou, and Jialiang Zou for helpful discussions related to this project. I would also like to thank Scott Larson for helping to correct an error in an earlier version of this work. This work was partly supported by Duke University’s Number Theory RTG under DMS-2231514.

I also acknowledge extensive use of ChatGPT (OpenAI) in the preparation of this manuscript.\footnote{Here we are following a developing convention for disclosing the use of AI tools in mathematical writing; compare the AI tool disclosure in \cite{TaoLocalBernstein}.} The tool was used as an interactive research assistant for exposition, LaTeX editing, notation checks, bibliography formatting, and for designing, running, and interpreting symbolic checks of selected computational identities. It also helped identify places where earlier draft formulations of the paper, including the definition of the harmonic category and line-bundle conventions, needed to be made consistent with the final version. All mathematical arguments, final statements, citations, and uses of AI-generated suggestions were reviewed and edited by the author, who is solely responsible for the content of the paper.

\section{Geometry of the Conformal Compactification}\label{GeometryofConfPact}

We begin by discussing in detail the relevant geometric structures.

\subsection{Presentation of the Orthogonal Group}\label{orthgroupsubsect} Throughout, we shall work over a field $\kappa$ of characteristic 0.\footnote{Some of the results here hold for odd characteristic as well. However, the ring of Grothendieck differential operators on affine space is quite a bit larger than the usual Weyl algebra in characteristic $p$; thus many of the results of Sections \ref{DiffOpSect}, \ref{KazhLaumGlueSect}, and \ref{DModHarmSect} would need to be altered to account for, e.g., the presence of divided powers. This would be an interesting direction for further investigation.} Let $V = V_1 \oplus V_2$ denote a quadratic vector space, where $V_1 \cong V_2 \cong \A^k$, and $V \cong \A^{2k} =\A^n$. Thus $n = 2k$ is even; we assume $k \ge 2$ and $n \ge 4$. 

For $(x_1, \ldots, x_k; y_1, \ldots, y_k) \in V$, we define the quadratic form
\begin{equation}\label{QDef}
    Q(x_1, \ldots, x_k; y_1, \ldots, y_k) := x_1y_k + \cdots + x_k y_1.
\end{equation}

\noindent Let $B(v,w) := Q(v+w) - Q(v) - Q(w)$ denote the associated bilinear form: 
\begin{equation}
B\left(\left(x_1, \ldots, x_k; y_1, \ldots, y_k\right),\left(u_1, \ldots, u_k; v_1, \ldots, v_k\right)\right) = x_1 v_k + \cdots + x_kv_1 + y_1u_k + \cdots + y_k u_1.
\end{equation}
Much of this paper is independent of the isomorphism class of $Q$ (so long as it is nondegenerate); we will try to point out where we assume that $Q$ is given explicitly by \ref{QDef} (and in particular, is in the trivial Grothendieck-Witt class). We let $H$ denote the orthogonal group of $V$ with respect to the quadratic form $Q$.

Throughout this paper, we will attempt to carefully distinguish between $V$ and $V^*$, though naturally both are identified by the bilinear form $B$. Under the isomorphism $V \simeq V^*$ induced by $B$, we will write $v^\flat \in V^*$ for the covector corresponding to $v \in V$, and $\nu^{\sharp} \in V$ for the vector corresponding to $\nu \in V^*$.

Next we let $V^{+}:=\{(x_0, x_1, \ldots, x_k; y_1, y_2, \ldots, y_{k+1})\}\cong \A^{n+2}$, under the quadratic form 
\begin{equation}
Q^+(x_0, x_1, \ldots, x_k; y_1, y_2, \ldots, y_k, y_{k+1}) = x_0y_{k+1} + x_1 y_k + \cdots + x_ky_1,
\end{equation}

\noindent and we let $G$ denote the corresponding orthogonal group. Inspired by physicists' notation, we will let 
\begin{equation}\label{eplusminus}
    e_+ := (1, 0, \ldots, 0, 0) \in V^+  \qquad e_- := (0, 0, \ldots, 0, 1) \in V^+.
\end{equation}
Importantly, although we will make an effort to present results which are agnostic with respect to the isomorphism class of $Q$, it is essential that $Q^+ = Q + e_+^\flat e_{-}^{\flat}$. In particular, $Q^+$ must be in the same Grothendieck-Witt class as $Q$.

We observe that $H \subset G$; explicitly, $H \times \G_m$ is a Levi factor of a parabolic $P$. In matrices, we let
\begin{equation}\label{defJ}
J\;:=\;J_k \;:=\;
\left(\begin{smallmatrix}
0 & 0 & \cdots & 0 & 1\\
0 & 0 & \cdots & 1 & 0\\
\vdots & \vdots & \ddots & \vdots & \vdots\\
0 & 1 & \cdots & 0 & 0\\
1 & 0 & \cdots & 0 & 0
\end{smallmatrix}\right) \in M_k; \,\,\,\,\,\,\, J_V \;:=\;
\left(\begin{smallmatrix}
0   & J_k\\
J_k & 0
\end{smallmatrix}\right) \in M_{n};
\end{equation}

\noindent with $H \;:=\; \{\, h\in \GL(V) \mid h^{\mathsf T} J_V h = J_V \,\}$. Next we let
\begin{equation}\label{defJ+}
J_+ \;:=\;
\left(\begin{smallmatrix}
0 & 0   & 1\\
0 & J_V & 0\\
1 & 0   & 0
\end{smallmatrix}\right) \in M_{n+2},
\end{equation}

\noindent with $
G \;:=\; \{\, g\in \GL_{2k+2} \mid g^{\mathsf T} J_+ g = J_+ \,\}$. We have:
\begin{equation}
\iota : H \hookrightarrow G,\qquad
h \longmapsto
\left(\begin{smallmatrix}
1 & 0 & 0\\
0 & h & 0\\
0 & 0 & 1
\end{smallmatrix}\right).
\end{equation}

\noindent Next we define $P$ and $P^{\op}$ and the Levi $L = \G_m\times H$.
\begin{equation}
P \;=\;
\left\{
\left(\begin{smallmatrix}
t & -\,t\,v^{\mathsf T}J_V & -\,t\,Q(v)\\
0 & h & h v\\
0 & 0 & t^{-1}
\end{smallmatrix}\right)
\ \middle|\ t\in\G_m,\ h\in H,\ v\in V
\right\},
\end{equation}
\begin{equation}
P^{\mathrm{op}} \;=\;
\left\{
\left(\begin{smallmatrix}
t & 0 & 0\\
h v & h & 0\\
-\,t^{-1}Q(v) & -\,t^{-1}v^{\mathsf T}J_V & t^{-1}
\end{smallmatrix}\right)
\ \middle|\ t\in\G_m,\ h\in H,\ v\in V
\right\};
\end{equation}

\noindent with 
\begin{equation}\label{unipotents}
U \;=\;
\left\{
u_v :=\left(\begin{smallmatrix}
1 & -\,v^{\mathsf T}J_V & -\,Q(v)\\
0 & I & v\\
0 & 0 & 1
\end{smallmatrix}\right)
\ \middle|\ v\in V
\right\},
\qquad
U^{\mathrm{op}} \;=\;
\left\{
u_v^{\op}:=\left(\begin{smallmatrix}
1 & 0 & 0\\
v & I & 0\\
-\,Q(v) & -\,v^{\mathsf T}J_V & 1
\end{smallmatrix}\right)
\ \middle|\ v\in V
\right\};
\end{equation}

\noindent and 
\begin{equation}\label{LLevi}
L \;=\;
\left\{
\left(\begin{smallmatrix}
t & 0 & 0\\
0 & h & 0\\
0 & 0 & t^{-1}
\end{smallmatrix}\right)
\ \middle|\ t\in\G_m,\ h\in H
\right\}.
\end{equation}

\noindent We repeat that we shall refer to the elements in \ref{unipotents} as $u_v$ and $u_v^{\op}$ respectively.

Almost everything we write below will hold for any quadratic form $Q$ with matrix $J_V$; importantly, however, $Q^+$ (resp. $J_+$) will always mean the direct sum of $(V,Q)$ with a hyperbolic space (resp., $J_+$ will always be defined as in \ref{defJ+}). We will try to make it clear when we actually utilize the assumption that $J_V$ is given by \ref{defJ}, such as in certain explicit coordinate-based calculations for differential operators.

\subsection{The Geometry of the Flag Variety}\label{FlagGeomSection} The flag variety $G/P$ is given by the quadric hypersurface in $\P^{n+1}$ corresponding to the null-cone of $Q^{+}$ in $V^+$. It is the moduli space of isotropic lines in $V^+$. In homogeneous coordinates:
\begin{equation}
G/P \cong \{[a: v: b] : Q(v) + ab = 0\}.
\end{equation}

\noindent We observe that $U^{\op} \cong V$ lives in $G/P$ via $U^{\op}P/P$; in coordinates, it is the dense affine open subset
\begin{equation}
\{[1: v: -Q(v)]: v \in V\}.
\end{equation}

\noindent This is the ``big Bruhat cell"; its complement consists of the union of the affine quadric ``cone at infinity"
\begin{equation}
C_{\infty}:=\{[0: v : 1] : Q(v) = 0\},
\end{equation}

\noindent along with the (projective) boundary of $C_{\infty}$,
\begin{equation}
\P(C_V) := \{[0: v : 0] : Q(v) = 0\}.
\end{equation}

We may transfer $Q$ to $U^{\op}$; $Q$ is the rational function on $G/P$ given by
\begin{equation}
Q ([a:v:b]) = Q(v)/a^2 = -b/a.
\end{equation}

\noindent On $V \cong U^{\op}$, this is simply $Q([1: v: -Q(v)]) = Q(v)$. We let 
\begin{equation}\label{defc*}
C_V := \{[1: v: 0]: Q(v) = 0\} \subset G/P.
\end{equation}

\noindent The bilinear form $B$ permits us to define a quadratic form $Q^*$ on $V^*$, and we will let 
\begin{equation}
C := C_{V^*}
\end{equation}

\noindent denote the isotropic cone of $Q^*$ in $V^*$.

The ``deep" boundary 
\begin{equation}\label{deepboundary}
    \P(C_V)
\end{equation} 

\noindent is simultaneously the boundary of the isotropic cone $C_V \subset V$ and of the infinite null-cone $C_{\infty}$; their closures intersect transversely in $\P(C_V)$.

\begin{rem}
    From now on, we will make the following identifications: 

    \begin{align}
        V \cong U^{\op} \cong U^{\op}&P/P \subset G/P;\\
        V^* &\cong U.
    \end{align}

\noindent We shall also view the cone $C$ and $C^o = C \setminus\{0\}$ as living in $V^* \cong U$, while we shall use $C_V$ to denote the locally closed subvariety $C_V \subset U^{\op}P/P \subset G/P$ of \ref{defc*}. Viewing ``the" cone $C$ as living in $U$ (and so $V^*$) rather than $U^{\op}$ is more mathematically parsimonious, and will clarify matters in later sections.
\end{rem}

In physics, $G/P$ is the Penrose conformal compactification \cite{penrose1963asymptotic, penrose1964conformal, penrose2011republication}, and the point 
\begin{equation}\label{SpiDef}
[0:0:1] \in C_\infty,
\end{equation} 

\noindent the vertex of $C_\infty$ (which is singular in $C_\infty$, though smooth in $G/P$), is often called ``spatial infinity" or ``Spi" \cite{ashtekarhansen1978unifiedI}.

In $G$, there is an element
\begin{equation}\label{WeylFourier}
w_0 = \left(\begin{smallmatrix}
    0 & 0 & 1\\
    0 & \Id & 0\\
    1 & 0 & 0
\end{smallmatrix}\right),
\end{equation}

\noindent which is an involution and satisfies $w_0Pw_0^{-1} = P^{\op}$. Geometrically, the action of $w_0$ on $G/P$ is given by 
\begin{equation}
w_0: [a: v: b] \mapsto [b: v: a].
\end{equation}

\noindent Thus $w_0$ exchanges $C_V$ and $C_\infty$, fixing $\P(C_V)$; we also see from this that the rational action of $w_0$ on $V$ is indeed the conformal inversion of \eqref{inversion}.

The variety $C^{o}_V = C^{\textrm{sm}}_V = C_V \setminus \{0\}$ is an $H$-orbit of 
\begin{equation}\label{p0def}
v_0 := [1:1: 0 : \cdots :0: 0 : \cdots :0:0 : 0] \in V \subset G/P;
\end{equation}

\noindent its closure in $U^{\op}$ is the singular quadric cone $C_V$, and its closure in $G/P$ is $\overline{C}_V = C_V \cup \P(C_V)$. We note that this is in fact the intersection of the quadric hypersurface $G/P \subset \P^{n+1}$ with the hyperplane $\{[a: v: 0]\}\subset\P^{n+1}$. The $\G_m$ factor of the Levi $L$ acts on $C_V^o$ as well, via scaling. It is perhaps somewhat surprising that the entirety of $P$ acts on $\overline{C}_V$ with dense open orbit $\overline{C}_V\setminus \{[1:0:0]\}$; in particular, the unipotent radical $U$ acts via:
\begin{equation}
u_v=\left(\begin{smallmatrix}
1 & -\,v^{\mathsf T}J_V & -\,Q(v)\\
0 & I & v\\
0 & 0 & 1
\end{smallmatrix}\right) \;:\; [1: x: 0] \mapsto [1 - B(v, x): x : 0].
\end{equation}

\noindent The scaling $\G_m$ action fixes the origin, spatial infinity, and the deep boundary $\P(C_V)$.

\begin{rem}
As an algebraic group, $G$ is disconnected, with two connected components distinguished by the determinant; equivalently, its identity component is $\SO(Q^+)$. The Weyl element $w_0$ of \eqref{WeylFourier} lies in the non-identity connected component, since $\det(w_0)=-1$.

Nevertheless, the flag variety $G/P$ is connected. Indeed, $G/P$ is the smooth quadric hypersurface in $\P^{n+1}$ cut out by $Q^+$, and hence is irreducible. Consequently, if we let $V_{w_0}:=w_0V\subset G/P$, then $V$ and $V_{w_0}$ are both dense open subsets of the connected variety $G/P$, and so $V\cap V_{w_0}\neq\emptyset$.

In fact, one checks from the formula \eqref{inversion} that
\begin{equation}
V_Q := V\cap V_{w_0}=V\setminus \{Q=0\}.
\end{equation}

\noindent This observation will be used repeatedly below.
\end{rem}

\section{The Cotangent Bundle and Differential Operators}\label{CotangentandDiffOpSect}

\subsection{\texorpdfstring{The Cotangent Bundle of the Quadric Cone and $F$-Moment Descent}{The Cotangent Bundle of the Quadric Cone and F-Moment Descent}}\label{cotangentsubsect}

We now discuss the cotangent bundle of $C^o$. It is known in folklore that the cotangent bundle of a quadric cone is a minimal nilpotent orbit of an orthogonal group; an explicit presentation of this fact is found in \cite{getz2025modulationgroups}, Proposition 7.2. We will review this in detail because our treatment of it here will highlight a guiding principle which we shall call ``$F$-moment descent.''

Recall that, for $v\in V$, we write $v^\flat:=B(v,-)\in V^*$, and, for $\nu\in V^*$, we write $\nu^\sharp\in V$ for the unique vector satisfying $(\nu^\sharp)^\flat=\nu$. Thus $Q^*(\nu)=Q(\nu^\sharp)$. We observe that $C^o$ is smooth and is naturally a locally closed subvariety of $U\cong V^*$. We have a natural exact sequence:
\begin{equation}
\begin{tikzcd}
0 \arrow[r] & TC^o \arrow[r] & C^o\times V^* \arrow[r] & \mathcal{L} \arrow[r] & 0
\end{tikzcd}
\end{equation}

\noindent where $TC^o$ denotes the tangent bundle, $C^o\times V^*$ denotes the trivial vector bundle with fiber $V^*$ over $C^o$, and $\mathcal{L}$ denotes the rank-$1$ normal bundle. Dualizing this exact sequence gives
\begin{equation}\label{cotexactseq}
\begin{tikzcd}
0 \arrow[r] & \mathcal{L}^{\vee} \arrow[r] & C^o \times V \arrow[r, "\pi"] & T^*C^o \arrow[r] & 0
\end{tikzcd}
\end{equation}

\noindent where the fiber of $\mathcal{L}^{\vee}$ over $\nu\in C^o$ is the line $\kappa\nu^\sharp\subset V$. This presents $T^*C^o$ as a quotient and allows us to describe its global functions via pullback under $\pi$.

\begin{prop}\label{invfunctioniso}
The ring of functions $\kappa[T^*C^o]$ is the invariant subring of $\kappa[C^o\times V]$ under the $\G_a$-action
\begin{equation}\label{Gasymm}
\G_a\times(C^o\times V)\longrightarrow C^o\times V,
\qquad
(t,(\nu,v))\longmapsto(\nu,v+t\nu^\sharp).
\end{equation}
\end{prop}

\begin{proof}
Any global function on $T^*C^o$ pulls back, by \eqref{cotexactseq}, to a function on $C^o\times V$ invariant under translation by the line subbundle $\mathcal{L}^{\vee}$. Conversely, the fibers of $\pi:C^o\times V\to T^*C^o$ are precisely the $\G_a$-orbits in \eqref{Gasymm}, so every invariant regular function descends to $T^*C^o$.
\end{proof}

Now, as we have discussed, $G$ acts rationally on $V$ via its natural action on $G/P$. However, the infinitesimal action of this rational action gives a \textit{regular} map:

\begin{align}\label{actiondiff}
\mathfrak{g}\times V &\longrightarrow V\\
(\xi,v) &\longmapsto X_\xi(v),\nonumber
\end{align}

\noindent where, identifying $T_vV\cong V$, we have $X_\xi(v)=\left.\frac{d}{dt}\right|_{t=0}\exp(t\xi)(v)$; more algebraically, $X_\xi$ comes from taking the left-invariant vector field on $G$ corresponding to $\xi\in\g$, pushing it forward along the map $d\pi:TG\to T(G/P)$, and then restricting to the big Bruhat cell. Notice that for any fixed $v\in V$, the map $\mathfrak{g}\times\{v\}\to V$ is linear; this map is \textit{not} linear in the $v$-factor.

For ease, we write \eqref{actiondiff} in coordinates. Firstly,
\begin{equation}
\mathfrak g
=
\mathfrak{o}(Q^+)
=
\Bigl\{
\xi\in\Mat_{2k+2}(\kappa)\ \bigm|\ \xi^{\mathsf T}J_+ + J_+\xi = 0
\Bigr\}.
\end{equation}

\noindent Thus, if
\begin{equation}\label{genericg}
\xi
:=
\left(\begin{smallmatrix}
\alpha & -\lambda^{\mathsf T}J_V & 0\\
\mu    & X                       & \lambda\\
0      & -\mu^{\mathsf T}J_V     & -\alpha
\end{smallmatrix}\right),
\qquad
\alpha\in\kappa;\ \mu,\lambda\in V;\ X\in\mathfrak{o}(Q),
\end{equation}

\noindent then
\begin{equation}\label{infactionformula}
X_\xi(v)
=
\mu+Xv-\alpha v+B(\lambda,v)v-Q(v)\lambda.
\end{equation}

We may also dualize \eqref{actiondiff} and \eqref{infactionformula}, giving the moment map
\begin{equation}\label{momentmap}
\Phi:V\times V^*\longrightarrow\g^*.
\end{equation}

\noindent It is defined by $\Phi(v,\nu)(\xi)=\langle\nu,X_\xi(v)\rangle$, or, in coordinates,
\begin{equation}\label{momentmapcoord}
\begin{aligned}
\Phi(v,\nu)(\alpha,X,\mu,\lambda)
={}&\langle\nu,\mu\rangle+\langle\nu,Xv\rangle-\alpha\langle\nu,v\rangle\\
&+B(\lambda,v)\langle\nu,v\rangle-Q(v)\langle\nu,\lambda\rangle.
\end{aligned}
\end{equation}

\noindent Equivalently, each pairing with $\nu$ may be written using $B$ and $\nu^\sharp$. We point out once more that, despite the symmetric appearance of \eqref{momentmapcoord}, the map $\Phi$ is only linear on the fibers $\{v\}\times V^*$; the maps $V\times\{\nu\}\to\mathfrak{g}^*$ are not usually linear.

We now introduce the critical geometric relationship between the cone $C$, the vector space $V$, and the group $G$.

\begin{definition}\label{momentmapdesc}
Consider a locally closed smooth embedding $C^o\subset V^*$, where $V$ is a vector space and $V^*$ its dual, and let $C$ denote the closure of $C^o$ in $V^*$. Suppose that we are given a rational\footnote{In particular, the action need neither be linear nor even regular.} action of a group $G$ on $V$, and consider the associated moment map \eqref{momentmap}. Define the composite
\begin{equation}
\Phi^F:C^o\times V\hookrightarrow V^*\times V\xrightarrow{\tau}V\times V^*\xrightarrow{\Phi}\mathfrak{g}^*,
\end{equation}
\noindent where $\tau(\nu,v)=(v,-\nu)$ is the quasiclassical Fourier transform. Suppose we have the following commutative diagram:
\begin{equation}\label{momentmapdescent}
\begin{tikzcd}
C^o \times V \arrow[d, "\pi"'] \arrow[r, "\Phi^F"] & \mathfrak{g}^* \\
T^*C^o \arrow[ru, "\overline{\mu}"']               &
\end{tikzcd}
\end{equation}

\noindent where $\pi$ is the map from \eqref{cotexactseq}. Then we say that $C$ satisfies $F$-\textbf{moment descent} with respect to the action of $G$.
\end{definition}

\begin{prop}\label{quadricmomentdescent}
Let $C^o$ be the smooth locus of the quadric cone of $Q^*$ in $V^*$, and let $G$ be the orthogonal group of $(V^+,Q^+)$. Consider the rational action of $G$ on $V$ given by identifying $V\cong U^{\op}$ with the big Bruhat cell of $G/P$ and restricting the regular action of $G$ on $G/P$ to $U^{\op}P/P$. Then $C$ satisfies $F$-moment descent.
\end{prop}

\begin{proof}
By Proposition~\ref{invfunctioniso}, it suffices to verify that the map $\Phi^F$ of Definition~\ref{momentmapdesc} is invariant under the $\G_a$-action \eqref{Gasymm}; that is, under $(\nu,v)\mapsto(\nu,v+t\nu^\sharp)$. Since $\tau(\nu,v)=(v,-\nu)$, we have $\Phi^F(\nu,v)=\Phi(v,-\nu)=-\Phi(v,\nu)$, because $\Phi$ is linear in its second variable. Thus it suffices to verify that \eqref{momentmapcoord} is invariant under
\begin{equation}\label{cotangentivarianceproperty}
(v,\nu)\longmapsto(v+t\nu^\sharp,\nu)=:(v',\nu)
\end{equation}
\noindent when $Q^*(\nu)=0$. Indeed,
\begin{align}
\Phi(v',\nu)-\Phi(v,\nu)
={}&B(\nu^\sharp,X(v'-v))-\alpha B(\nu^\sharp,v'-v)\nonumber\\
&+B(\lambda,v')B(\nu^\sharp,v')-B(\lambda,v)B(\nu^\sharp,v)\nonumber\\
&-(Q(v')-Q(v))B(\nu^\sharp,\lambda).
\end{align}

\noindent Since $v'-v=t\nu^\sharp$, the first term is $tB(\nu^\sharp,X\nu^\sharp)=0$ because $X$ is $B$-skew, while the second is $-t\alpha B(\nu^\sharp,\nu^\sharp)=0$ because $Q^*(\nu)=Q(\nu^\sharp)=0$. Moreover, $B(\nu^\sharp,v')=B(\nu^\sharp,v)$ and $Q(v')=Q(v)+tB(v,\nu^\sharp)$, so
\begin{align}
&B(\lambda,v')B(\nu^\sharp,v')-Q(v')B(\nu^\sharp,\lambda)\nonumber\\
&\qquad=\bigl(B(\lambda,v)+tB(\lambda,\nu^\sharp)\bigr)B(\nu^\sharp,v)
-\bigl(Q(v)+tB(v,\nu^\sharp)\bigr)B(\nu^\sharp,\lambda)\nonumber\\
&\qquad=B(\lambda,v)B(\nu^\sharp,v)-Q(v)B(\nu^\sharp,\lambda),
\end{align}
\noindent since the $t$-terms cancel by the symmetry of $B$. We conclude that $\Phi(v',\nu)=\Phi(v,\nu)$, and hence that $\Phi^F$ descends along \eqref{Gasymm}.
\end{proof}

In fact, we may use the map $\overline{\mu}:T^*C^o\to\mathfrak{g}^*$ to give a presentation of the invariant functions generating $\kappa[T^*C^o]$. Let us return once more to the moment map of the natural regular $G$-action on $G/P$. Recall that $G$ has a Hamiltonian action on
\begin{equation}
T^*(G/P)\cong G\times^P(\mathfrak{g}/\mathfrak{p})^*,
\end{equation}
\noindent with moment map $\mu:T^*(G/P)\to\g^*$ given by $\mu([g,\xi])=\Ad^*(g)(\xi)$. We identify $(\mathfrak{g}/\mathfrak{p})^*$ with $\mathfrak{p}^{\perp}\subset\g^*$, and $\mathfrak{p}^{\perp}$ with $\mathfrak{u}$, via the Killing form. Restricting to the big Bruhat cell $V\cong U^{\op}\cong U^{\op}P/P\subset G/P$, we obtain
\begin{equation}\label{unpackmoment}
\mu:V\times V^*\longrightarrow\mathfrak{g}^*,
\qquad
(v,\nu)\longmapsto u_v^{\op}n_\nu^+u_v^{{\op}^{-1}},
\end{equation}
\noindent where
\begin{equation}
n_\nu^+
:=
\left(\begin{smallmatrix}
0 & -(\nu^\sharp)^{\mathsf T}J_V & 0\\
0 & 0 & \nu^\sharp\\
0 & 0 & 0
\end{smallmatrix}\right),
\qquad
n_\nu^-
:=
\left(\begin{smallmatrix}
0 & 0 & 0\\
\nu^\sharp & 0 & 0\\
0 & -(\nu^\sharp)^{\mathsf T}J_V & 0
\end{smallmatrix}\right).
\end{equation}

\noindent Here $u_v^{\op}$ is defined in \eqref{unipotents}, and we identify $\mathfrak{g}^*$ with $\mathfrak{g}$ via the Killing form. Restricting \eqref{unpackmoment} to $\nu\in C^o$, so that $Q^*(\nu)=Q(\nu^\sharp)=0$, and expanding in coordinates, we obtain
\begin{equation}\label{invariantfunccot}
\left(\begin{smallmatrix}
\alpha & -(\nu^\sharp)^{\mathsf T}J_V & 0\\
\mu & v\wedge\nu^\sharp & \nu^\sharp\\
0 & -\mu^{\mathsf T}J_V & -\alpha
\end{smallmatrix}\right),
\end{equation}
\noindent where
\begin{equation}\label{dotprodinv}
\alpha:=\langle\nu,v\rangle=B(\nu^\sharp,v)=v^{\mathsf T}J_V\nu^\sharp,
\end{equation}
\begin{equation}\label{dualconeinv}
\mu:=\alpha v-Q(v)\nu^\sharp,
\end{equation}
\noindent and the middle block $v\wedge\nu^\sharp\in\mathfrak{o}(Q)$ is the rank-$2$ skew endomorphism
\begin{equation}\label{wedges}
(v\wedge\nu^\sharp)(z):=B(v,z)\nu^\sharp-B(\nu^\sharp,z)v.
\end{equation}
\noindent Equivalently,
\begin{equation}
v\wedge\nu^\sharp
=
\nu^\sharp v^{\mathsf T}J_V-v(\nu^\sharp)^{\mathsf T}J_V
=
\nu^\sharp v^\flat-v\nu.
\end{equation}

\noindent Here and below, juxtaposition of a vector and a covector denotes the corresponding rank-$1$ endomorphism.

It follows from Proposition~\ref{quadricmomentdescent} that each of \eqref{dotprodinv}, \eqref{dualconeinv}, and \eqref{wedges} is invariant under the $\G_a$-action \eqref{Gasymm}. This can also be verified directly. Thus they all descend to functions on $T^*C^o$.

\begin{prop}\label{CotBundDescr}
The affinization of $T^*C^o$ is isomorphic to the closure of the minimal nilpotent orbit $\overline{\mathbb{O}}_{\min}$ of $\mathfrak{g}^*$. The matrix coefficients of \eqref{invariantfunccot} generate the global regular functions on $T^*C^o$.
\end{prop}

\begin{proof}
(Cf.\ \cite{getz2025modulationgroups}.) We first note that $\overline{\mu}:T^*C^o\to\mathfrak{g}^*$ has image contained in a nilpotent orbit: every point in the image is conjugate to some $n_\nu^+$ with $\nu\in C^o$. Moreover, $n_\nu^+$ is a square-zero matrix of rank $2$, so all Jordan blocks have size at most $2$, and there are exactly two Jordan blocks of size $2$. Hence every matrix in the image of $\overline{\mu}$ has Jordan type $(2,2,1^{n-2})$, corresponding to the minimal nilpotent orbit $\mathbb{O}_{\min}$ in $\mathfrak{g}^*$ \cite{Jantzen:Nilpotent}.

A straightforward calculation shows that $\overline{\mu}$ is injective. Indeed, a point of $T^*C^o$ may be represented by a pair $(\nu,[v])$, with $\nu\in C^o$ and $[v]\in V/\kappa\nu^\sharp$. From \eqref{invariantfunccot}, one recovers $\nu^\sharp$ from the rightmost column, hence also $\nu=(\nu^\sharp)^\flat$, and then recovers $[v]$ from the middle block $v\wedge\nu^\sharp$. In fact, if $(v-v')\wedge\nu^\sharp=0$, then $v-v'\in\kappa\nu^\sharp$, since for any $z\in(\nu^\sharp)^\perp$ one has
\begin{equation}
0=((v-v')\wedge\nu^\sharp)(z)=B(v-v',z)\nu^\sharp,
\end{equation}
\noindent so $B(v-v',z)=0$ for all $z\in(\nu^\sharp)^\perp$, whence $v-v'\in((\nu^\sharp)^\perp)^\perp=\kappa\nu^\sharp$.

We next claim that $\overline{\mu}$ is \'{e}tale. Since $\dim T^*C^o=2(2k-1)=4k-2=\dim\mathbb{O}_{\min}$, it suffices to show that the differential is injective. Let $F:C^o\times V\to\mathfrak{g}^*$ denote the lift of $\overline{\mu}$ defined by \eqref{invariantfunccot}, so that $F$ is invariant under the $\G_a$-action \eqref{Gasymm}. Writing $\alpha=\langle\nu,v\rangle$ and $\mu=\alpha v-Q(v)\nu^\sharp$, a tangent vector at $(\nu,v)\in C^o\times V$ is a pair $(\dot\nu,\dot v)$ with $\dot\nu\in T_\nu C^o$, equivalently $B(\nu^\sharp,\dot\nu^\sharp)=0$. Differentiating \eqref{invariantfunccot} gives
\begin{equation}
dF_{(\nu,v)}(\dot\nu,\dot v)
=
\left(\begin{smallmatrix}
\dot\alpha & -(\dot\nu^\sharp)^{\mathsf T}J_V & 0\\
\dot\mu & \dot v\wedge\nu^\sharp+v\wedge\dot\nu^\sharp & \dot\nu^\sharp\\
0 & -\dot\mu^{\mathsf T}J_V & -\dot\alpha
\end{smallmatrix}\right),
\end{equation}
\noindent where
\begin{equation}
\dot\alpha=B(\dot\nu^\sharp,v)+B(\nu^\sharp,\dot v),
\qquad
\dot\mu=\dot\alpha v+\alpha\dot v-B(v,\dot v)\nu^\sharp-Q(v)\dot\nu^\sharp.
\end{equation}

If $dF_{(\nu,v)}(\dot\nu,\dot v)=0$, then the rightmost column shows that $\dot\nu^\sharp=0$, and hence $\dot\nu=0$. The middle block then gives $\dot v\wedge\nu^\sharp=0$, so $\dot v\in\kappa\nu^\sharp$ by the same argument as above. Thus $(\dot\nu,\dot v)$ is tangent to the $\G_a$-orbit through $(\nu,v)$, and the induced differential
\begin{equation}
d\overline{\mu}_{(\nu,[v])}:T_{(\nu,[v])}(T^*C^o)\longrightarrow T_{\overline{\mu}(\nu,[v])}\mathbb{O}_{\min}
\end{equation}
\noindent is injective. Since the source and target have the same dimension, $\overline{\mu}$ is \'{e}tale. In particular, it identifies $T^*C^o$ with an open subset of $\mathbb{O}_{\min}$.

The variety $T^*C^o$ is Noetherian and regular, hence normal, while the minimal nilpotent orbit closure is affine and normal; see \cite{Jantzen:Nilpotent}, Section 8.6.

Finally, following \cite{Tome}, we examine the complement of the image of $\overline{\mu}$ in $\mathbb{O}_{\min}$ and show that it has codimension $2$. The image of $\overline{\mu}$ is conical, so it suffices to compute the codimension after projectivizing. The variety $\P(\mathbb{O}_{\min})$ is the space of isotropic $2$-planes in $V^+=\langle e_+\rangle\oplus V\oplus\langle e_-\rangle$. The image of $\overline{\mu}$ corresponds precisely to the isotropic $2$-planes not contained in $V\oplus\langle e_-\rangle$ (following the notation of \ref{eplusminus}).

Let $Z$ denote the complementary locus of isotropic $2$-planes contained in $V\oplus\langle e_-\rangle$. Since $Q^+(0,v,b)=Q(v)$, a generic such plane projects onto an isotropic $2$-plane $L\subset V$. To recover the original plane, it suffices to specify a map $\lambda\in\Hom(L,\langle e_-\rangle)$, giving the plane $\{(0,v,\lambda(v)e_-):v\in L\}$. Thus $\dim Z=[2(2k-2)-3]+2=4k-5$, where the bracketed term is the dimension of the isotropic Grassmannian of $2$-planes in $V$, while $\dim\P(\mathbb{O}_{\min})=2(2k)-3=4k-3$. We conclude that $Z$ has codimension $2$ in $\P(\mathbb{O}_{\min})$, and hence that the complement of $T^*C^o$ in $\overline{\mathbb{O}}_{\min}$ has codimension $2$.

We now know that $\overline{\mu}:T^*C^o\to\overline{\mathbb{O}}_{\min}$ identifies $T^*C^o$ with an open subset whose complement has codimension $2$. Hartogs's lemma therefore shows that $\overline{\mathbb{O}}_{\min}$ is the affinization of $T^*C^o$. Since the matrix entries of \eqref{invariantfunccot} generate the coordinate ring of $\mathfrak{g}^*$, their restrictions generate the coordinate ring of $\overline{\mathbb{O}}_{\min}$.
\end{proof}

\begin{prop}\label{CotBundlePresent}
Let $X\in\mathfrak{o}(Q)$, let $\mu\in V$ and $\nu\in V^*$, let $\alpha\in\kappa$, and form the block matrix
\begin{equation}
M
:=
\left(\begin{smallmatrix}
\alpha & -\nu & 0\\
\mu & X & \nu^\sharp\\
0 & -\mu^\flat & -\alpha
\end{smallmatrix}\right).
\end{equation}

\noindent Then $\kappa[T^*C^o]\cong\kappa[\overline{\mathbb{O}}_{\min}]$ is the coordinate ring of the reduced closed subvariety of such matrices satisfying
\begin{equation}
M^2=0
\qquad\text{and}\qquad
\rk(M)\leq 2.
\end{equation}

Equivalently, if $S:=\kappa[\alpha,\mu,\nu,X]$ and $\mathfrak{J}\subset S$ is the ideal generated by the entries of $M^2$ together with the $3\times3$ minors of $M$, then $\kappa[T^*C^o]\cong S/\sqrt{\mathfrak{J}}$.

The condition $M^2=0$ is equivalent to the relations
\begin{equation}\label{wmurelat}
Q^*(\nu)=0,
\qquad
Q(\mu)=0,
\qquad
\langle\nu,\mu\rangle=\alpha^2,
\end{equation}
\begin{equation}\label{XwandXmurelat}
X\nu^\sharp=\alpha\nu^\sharp,
\qquad
X\mu=-\alpha\mu,
\end{equation}
\noindent and
\begin{equation}\label{X^2Relat}
X^2=\nu^\sharp\mu^\flat+\mu\nu.
\end{equation}

\noindent Among the consequences of the condition $\rk(M)\leq2$ are
\begin{equation}\label{outerrankrelat}
\alpha X=\nu^\sharp\mu^\flat-\mu\nu
\end{equation}
\noindent and the ``Pl\"ucker relations''
\begin{equation}
X\wedge X=0.
\end{equation}

Explicitly, for the form $Q$ of \eqref{QDef}, write $\bar{i}:=2k+1-i$. Then the condition $X=(X_{ij})\in\mathfrak{o}(Q)$ is $X_{ij}=-X_{\bar{j}\bar{i}}$ for all $1\leq i,j\leq2k$, while the Pl\"ucker relations on $X$ are
\begin{equation}\label{plcukerrelat}
X_{i\bar{j}}X_{\ell\bar{m}}-X_{i\bar{\ell}}X_{j\bar{m}}+X_{i\bar{m}}X_{j\bar{\ell}}=0
\end{equation}
\noindent for all $1\leq i<j<\ell<m\leq2k$. These displayed equations are useful consequences of the rank condition, but the ideal $\mathfrak{J}$ retains the full collection of $3\times3$ minors of $M$.
\end{prop}

\begin{proof}
By the classification of nilpotent orbits in the orthogonal Lie algebra (see \cite{Jantzen:Nilpotent}), the minimal nonzero nilpotent orbit of $G$ corresponds to the partition $(2,2,1^{n-2})$. Equivalently, its nonzero elements are precisely the nilpotent matrices whose Jordan blocks all have size at most $2$ and whose rank is $2$. Thus
\begin{equation}
\overline{\mathbb{O}}_{\min}
=
\{M\in\mathfrak{so}(V^+,Q^+):M^2=0,\ \rk(M)\leq2\}.
\end{equation}

\noindent Indeed, if $M^2=0$, then every Jordan block has size at most $2$; if, in addition, $\rk(M)=2$, then there are exactly two Jordan blocks of size $2$, so the Jordan type is $(2,2,1^{n-2})$. Conversely, an element of Jordan type $(2,2,1^{n-2})$ is square-zero and has rank $2$.

It remains to expand the conditions $M^2=0$ and $\rk(M)\leq2$ in block coordinates. Direct multiplication gives
\begin{equation}
M^2
=
\left(\begin{smallmatrix}
\alpha^2-\langle\nu,\mu\rangle & -\alpha\nu-\nu X & -\langle\nu,\nu^\sharp\rangle\\
\alpha\mu+X\mu & X^2-\nu^\sharp\mu^\flat-\mu\nu & X\nu^\sharp-\alpha\nu^\sharp\\
-B(\mu,\mu) & \alpha\mu^\flat-\mu^\flat X & \alpha^2-\langle\nu,\mu\rangle
\end{smallmatrix}\right).
\end{equation}

\noindent Since $\langle\nu,\nu^\sharp\rangle=B(\nu^\sharp,\nu^\sharp)=2Q^*(\nu)$ and $B(\mu,\mu)=2Q(\mu)$, the diagonal and corner entries give \eqref{wmurelat}; the $(2,3)$ and $(2,1)$ entries give \eqref{XwandXmurelat}; and the $(2,2)$ entry gives \eqref{X^2Relat}. Conversely, these relations force every block of $M^2$ to vanish, using the $B$-skewness of $X$ for the $(1,2)$ and $(3,2)$ blocks.

The rank condition is equivalent to the vanishing of all $3\times3$ minors of $M$. Among these, the minors involving the first and last columns yield \eqref{outerrankrelat}, while the minors supported on the middle block $X$ yield the Pl\"ucker relations \eqref{plcukerrelat}. In general, however, these two displayed families do not replace the remaining mixed $3\times3$ minors. Retaining the full collection of minors, the reduced square-zero, rank-$\leq2$ locus is exactly $\overline{\mathbb{O}}_{\min}$. Its coordinate ring is therefore $S/\sqrt{\mathfrak{J}}\cong\kappa[T^*C^o]$.
\end{proof}

The preceding proposition gives a set-theoretically transparent presentation, but leaves the reduced defining ideal hidden inside $\sqrt{\mathfrak{J}}$. Replacing the cubic determinantal equations by quadratic Pfaffian equations gives the prime defining ideal directly.

\begin{prop}\label{CotBundlePrimeIdeal}
Let $S$ and $M$ be as in Proposition~\ref{CotBundlePresent}, and set $A:=J_+M$. Let $\mathfrak{K}\subset S$ be the ideal generated by the entries of $M^2$ and the $4\times4$ Pfaffians of the alternating matrix $A$. Then $\mathfrak{K}$ is prime, and
\begin{equation}
\kappa[T^*C^o]
\cong
\kappa[\overline{\mathbb{O}}_{\min}]
\cong
S/\mathfrak{K}.
\end{equation}
\noindent Moreover,
\begin{equation}
\sqrt{\mathfrak{J}}=\mathfrak{K}.
\end{equation}
\end{prop}

\begin{proof}
Since $M\in\mathfrak{o}(Q^+)$, we have $M^{\mathsf T}J_++J_+M=0$, and hence $A^{\mathsf T}=-A$. Moreover, multiplication by $J_+$ does not change rank, so $\rk(A)=\rk(M)$. The $4\times4$ Pfaffians of $A$ are therefore the Pl\"ucker quadrics cutting out the rank-$\leq2$ locus of alternating matrices, namely the affine cone over $\Gr(2,V^+)$. We also have $AJ_+^{-1}A=J_+M^2$, so the entries of $M^2$ are equivalent to the quadratic equations asserting that the corresponding $2$-plane is isotropic.

It follows that the projectivization of the zero locus of $\mathfrak{K}$ is $\operatorname{OGr}(2,V^+)$. By \cite[Theorems~2.1 and~2.9]{ElMaazouzMandelshtamPositiveOGr}, the Pl\"ucker quadrics together with the quadratic isotropy equations generate its prime ideal.\footnote{The proof of \cite[Theorem~2.9]{ElMaazouzMandelshtamPositiveOGr} invokes Kostant's theorem that the homogeneous ideal of a projective highest-weight orbit is generated by its quadratic part; see \cite[Theorem~3.1]{BurstallKostantPlucker}. As explained in \cite[Remark~3.2]{BurstallKostantPlucker}, this theorem is due to Kostant and appeared, with attribution, in \cite{LancasterTowberFlagAlgebrasI}; its set-theoretic part was independently proved in \cite{LichtensteinHighestWeightQuadrics}. In the present case, $\operatorname{OGr}(2,V^+)\subset\P(\bigwedge^2V^+)$ is the projective highest-weight orbit, so it suffices to identify its full degree-$2$ ideal.} Since $AJ_+^{-1}A=J_+M^2$ and $\dim V^+=2k+2>4$, their result applies to $\mathfrak{K}$. See also \cite[Corollary~2.11 and Remark~2.12]{FriedmanRosanaSturmfelsDistance} for the equivalent square-zero matrix formulation.

The cited result is stated over $\C$, but the split orthogonal Grassmannian and all of the displayed equations are defined over $\mathbb Q$. Equality of the homogeneous ideals therefore descends by faithful flatness and then base-changes to $\kappa$. Since $\operatorname{OGr}(2,V^+)$ is geometrically integral for $\dim V^+>4$, the ideal $\mathfrak{K}$ is prime over $\kappa$.

The affine cone over $\operatorname{OGr}(2,V^+)$ is $\overline{\mathbb{O}}_{\min}$, so Proposition~\ref{CotBundDescr} gives $\kappa[T^*C^o]\cong S/\mathfrak{K}$. Finally, the $3\times3$ minors of $M$ and the $4\times4$ Pfaffians of $A$ both impose the condition $\rk(M)\leq2$ set-theoretically. Thus $\mathfrak{J}$ and $\mathfrak{K}$ have the same zero locus. Since $\mathfrak{K}$ is prime, it follows that $\sqrt{\mathfrak{J}}=\mathfrak{K}$.
\end{proof}

\begin{rem}\label{ThreeMinorIdealNonreduced}
The radical in Proposition~\ref{CotBundlePresent} cannot in general be omitted. Indeed, suppose that $k=3$, so that $\dim V^+=8$, and choose a hyperbolic basis in which
\begin{equation}
J_+
=
\left(\begin{smallmatrix}
0&I_4\\
I_4&0
\end{smallmatrix}\right).
\end{equation}
\noindent Set
\begin{equation}
\Theta
:=
\left(\begin{smallmatrix}
0&1&0&0\\
-1&0&0&0\\
0&0&0&1\\
0&0&-1&0
\end{smallmatrix}\right),
\qquad
M_0
:=
\left(\begin{smallmatrix}
0&\Theta\\
0&0
\end{smallmatrix}\right).
\end{equation}
\noindent Since $\Theta^{\mathsf T}=-\Theta$, one has $M_0\in\mathfrak{o}(Q^+)$. Moreover, $M_0^2=0$ and $\rk(M_0)=4$. If $A_0:=J_+M_0$, then the lower-right $4\times4$ block of $A_0$ is $\Theta$, whose Pfaffian is $1$.

Let $p$ be the corresponding $4\times4$ Pfaffian of $J_+M$. Since $\mathfrak{J}$ is homogeneous and its determinantal generators have degree $3$, the degree-$2$ part of $\mathfrak{J}$ is spanned by the entries of $M^2$. If $p\in\mathfrak{J}$, it would therefore vanish at every square-zero matrix. This is impossible, since $M_0^2=0$ while $p(M_0)=1$. Thus $p\notin\mathfrak{J}$.

On the other hand, every point of the zero locus of $\mathfrak{J}$ has rank at most $2$, so all $4\times4$ Pfaffians of $J_+M$ vanish there. It follows, after base change to an algebraic closure and faithful-flat descent, that $p\in\sqrt{\mathfrak{J}}$. Hence $\mathfrak{J}$ is not radical.
\end{rem}

\begin{rem}
We note that the action of $G$ on the ring of global functions on $T^*C^o$ may be thought of as the quasiclassical analogue of the minimal representation. Likewise, the minimal representation of $G(F)$ on $L^2(C(F))$ may be thought of as a geometric quantization of the minimal nilpotent orbit, whose affinization is $\overline{\mathbb{O}}_{\min}=\Spec\kappa[T^*C^o]$.

It is also interesting to note that two common sources of symplectic manifolds are cotangent bundles and coadjoint orbits; Proposition~\ref{CotBundDescr} shows that $\mathbb{O}_{\min}$ is, up to affinization, an example of both simultaneously.
\end{rem}

\subsection{\texorpdfstring{Quantization of $F$-moment Descent: Overview}{Quantization of F-moment Descent: Overview}}\label{DiffOpSect}

We remind the reader that we assume $\textrm{char}(\kappa)=0$; 
not much is lost by assuming $\kappa=\C$. We have just seen that $G$ acts on the ring of global regular functions $\kappa[T^*C^o]$. 
This picture admits a quantization in terms of differential operators on the cone, yielding an action of $G$ on $D_C$, the ring of Grothendieck differential operators on $C$. 
This construction originates in work of Goncharov \cite{Goncharov1976Weil, LevasseurSmithStafford1989Joseph, Kobayashi:Mano}; our goal in this section is to motivate it from the perspective of quantizing $F$-moment descent (cf.\ Definition~\ref{momentmapdesc}).

More precisely, the quantization of $F$-moment descent proceeds in three steps. First, we consider a twisted rational action of $G$ on the big Bruhat cell $V\cong U^{\op}$. This twisting will later be given a geometric interpretation in terms of the action of a line bundle on $G/P$. Differentiating this action yields a Lie algebra homomorphism
\begin{equation}\label{phideftwist}
    \phi:\mathfrak g\to D_V.
\end{equation}

\noindent These operators preserve the principal left ideal $D_V\Delta$ in $D_V$ under the adjoint action, where $\Delta$ is the Laplace-Beltrami operator associated with $Q$.

Second, we apply the linear Fourier transform to obtain an auxiliary ambient action,
\begin{equation}
    \rho_{\mathrm{amb}}:=\tau\circ\phi:\mathfrak g\to D_{V^*},
\end{equation}

\noindent where
\begin{equation}\label{linFourtrans}
    \tau:D_V\to D_{V^*},
\end{equation}

\noindent is the canonical linear Fourier transform, given by sending $\lambda\in V^*\subset \kappa[V]$ to the constant vector field $\partial_\lambda\in D_{V^*}$,\footnote{That is, for $f\in \kappa[V^*]$ and $\xi\in V^*$, one has
\begin{equation}
    (\partial_\lambda f)(\xi)=\lim_{t\to 0}\frac{f(\xi+t\lambda)-f(\xi)}{t}.
\end{equation}} and sending $\partial_v\in D_V$ to $-v\in V\cong (V^*)^*\subset \kappa[V^*]$.\footnote{This is slightly different from the Fourier transform appearing in \cite{Goncharov1976Weil}, which involves $i=\sqrt{-1}$.} The resulting ambient operators preserve the principal left ideal $D_{V^*}Q^*$ in $D_{V^*}$ under the adjoint action. However, they do not preserve the ideal $(Q^*)$ of the quadric cone $C\subset V^*$ in $\kappa[V^*]$.

Our third step is to correct this. We pass from the ambient Fourier-side action to the local cohomology module
\begin{equation}
    H^1_{(Q^*)}(\kappa[V^*])=\kappa[V^*][(Q^*)^{-1}]/\kappa[V^*].
\end{equation}

\noindent Writing
\begin{equation}
    \delta_C:=\bigl[(Q^*)^{-1}\bigr]\in H^1_{(Q^*)}(\kappa[V^*]),
\end{equation}

\noindent multiplication by $\delta_C$ defines an isomorphism
\begin{equation}
    T:\kappa[C]=\kappa[V^*]/(Q^*)\xrightarrow{\sim}\kappa[V^*]\delta_C,
    \qquad
    \overline{f}\longmapsto f\delta_C.
\end{equation}

\noindent The ambient operators preserve the subspace $\kappa[V^*]\delta_C$, and we define the actual cone operators by transport of structure across $T$. In this way, the desired action on $\kappa[C]$ is obtained not by restricting the ambient operators directly to the hypersurface $Q^*=0$, but rather by passing through the distribution $\delta_C=[(Q^*)^{-1}]$ and then transporting back to the cone.\footnote{This is akin to the $L^2$-setting of the $F$-method, where the linear Fourier transform of a distribution annihilated by the Laplace operator yields a distribution, rather than a function, supported on the cone.}

This action of $\mathfrak g$ on $\kappa[V^*]\delta_C$ yields a Lie algebra homomorphism
\begin{equation}
    \rho:\mathfrak g\to D_C,
\end{equation}

\noindent and hence, by universality, an algebra homomorphism
\begin{equation}
    \rho:\mathcal{U}(\mathfrak g)\to D_C.
\end{equation}

\noindent It will turn out that this map is surjective; this is a special case of the main result of \cite{LevasseurSmithStafford1989Joseph}.\footnote{The kernel of this map is the Joseph ideal.} We will offer an independent proof later in the exposition (Corollary~\ref{rhoSurjective}). Moreover, the adjoint action of $G$ on $\mathcal{U}(\mathfrak g)$ descends to an action on $D_C$; this is the quantization of the action of $G$ on $\kappa[T^*C^o]$.

We note that the action of $G$ on $D_C$ does not arise from a geometric action of $G$ on $C$, but rather from a Fourier-transformed ambient action together with a natural passage to the singular hypersurface.

\subsection{\texorpdfstring{The Map $\phi$ and the Twisted $G$-action on $\kappa(V)$}{The Map phi and the Twisted G-action on kappa(V)}}\label{TwistedGActSection}

The first step in quantizing $F$-moment descent is to compute the map $\phi:\mathfrak{g}\to D_V$ of \eqref{phideftwist}. Already here one encounters an important subtlety not present in the quasiclassical case. In particular, $\phi$ is \textit{not} obtained simply by mapping $\xi\in \mathfrak g$ to the corresponding infinitesimal vector field on $G/P$ and restricting to the big Bruhat cell. Rather, it is twisted so as to account for the conformal factors.

We define $\phi$ as follows. Let $y\in V\cong U^{\op}\subset G$; we stress that we are thinking of $y$ as a point of $G$. For generic $(g,y)$, the element $gy$ lies in the big cell $U^{\op}P\subset G$. Thus we may uniquely write
\begin{equation}\label{ptwistdef}
    gy = y' p(g,y),
\end{equation}
\noindent with $y'\in U^{\op}$ and $p(g,y)\in P$. In this way we obtain rational maps
\begin{equation}\label{rationalGacttype}
    G\times V\dashrightarrow V,
    \qquad
    (g,y)\mapsto y',
\end{equation}
\begin{equation}\label{pcocycledeftype}
    p:G\times V\dashrightarrow P, 
    \qquad
    (g,y)\mapsto p(g,y).
\end{equation}

\noindent The first of these corresponds to the rational action of $G$ on the big cell $V$: when the factorization \eqref{ptwistdef} is defined, we write $gy:=y'\in U^{\op}\cong V\subset G/P$.

In the notation of \eqref{unipotents}, we have
\begin{equation}
    g u_v^{\op}=u_{gv}^{\op}p(g,v).
\end{equation}
which, again, provides a definition for $gv \in V$ \eqref{rationalGacttype} and $p(g,v)\in P$ \eqref{pcocycledeftype}.\footnote{When we write $gv$ below, we shall be referring to the rational action \eqref{rationalGacttype}; if we write $gu$ or $gu_v^{\op}$, we will mean the product in $G$.} Moreover:  
\begin{equation}
    g'(g u_v^{\op})
    =
    g'(u_{gv}^{\op}p(g,v))
    =
    u_{g'gv}^{\op}p(g',gv)\,p(g,v),
\end{equation}
while
\begin{equation}
    (g'g)u_v^{\op}=u_{g'gv}^{\op}p(g'g,v).
\end{equation}
By uniqueness of the factorization in $U^{\op}P$, we deduce the cocycle relation
\begin{equation}\label{basiccocycle}
    p(g'g,v)=p(g',gv)\,p(g,v).
\end{equation}
Equivalently,
\begin{equation}\label{altcocycle}
    p(g'g,v)\,p(g,v)^{-1}=p(g',gv).
\end{equation}

Now let $\chi_0:P\to L\to \G_m$ be the character of $P$ obtained by pulling back the basic character $\chi$ of $L$:
\begin{equation}
    \chi:\G_m\times \O(Q)\cong L\to \G_m,
    \qquad
    (\alpha,h)\mapsto \alpha.
\end{equation}
Let
\begin{equation}
    \eta:=\chi_0^{-(k-1)}.
\end{equation}
We then define a rational action of $G$ on the function field $\kappa(V)$ by
\begin{equation}\label{twistedaction}
    (g f)(v)=\eta(p(g^{-1},v))\,f(g^{-1}v).
\end{equation}

\noindent With the convention \eqref{ptwistdef}, \eqref{twistedaction} defines a left action of $G$. Indeed,
\begin{align}
    (h(gf))(v)
    &=
    \eta(p(h^{-1},v))\,(gf)(h^{-1}v)
    \nonumber\\
    &=
    \eta(p(h^{-1},v))\,\eta(p(g^{-1},h^{-1}v))\,f(g^{-1}h^{-1}v)
    \nonumber\\
    &=
    \eta\bigl(p(h^{-1},v)\,p(g^{-1},h^{-1}v)\bigr)\,f(g^{-1}h^{-1}v)
    \nonumber\\
    &=
    \eta(p(g^{-1}h^{-1},v))\,f(g^{-1}h^{-1}v)
    \qquad\text{(by \eqref{basiccocycle})}
    \nonumber\\
    &=
    ((hg)f)(v).
\end{align}

This twisting by $\eta=\chi_0^{-(k-1)}$ has a natural geometric interpretation. Namely, let
\begin{equation}\label{twistbundledef}
    \mathcal{L}:=G\times^{P,-\eta}\A^1\cong \mathcal{O}(1-k)\cong K^{1/2}\otimes \mathcal{O}(1),
\end{equation}
\noindent where $\mathcal{O}(1)$ is the hyperplane bundle for the embedding $G/P\subset \P^{n+1}$, where $K^{1/2}\cong \mathcal{O}(-k)$ is the square root of the canonical bundle $K=\det(T^*G/P)\cong \mathcal{O}(-n)$, and where $G\times^{P,-\eta}\A^1$ denotes the quotient of $G\times \A^1$ by the relation
\begin{equation}
    (gp,x)\sim (g,\eta(p)^{-1}x).
\end{equation}

We now explain precisely how the action of $G$ on the equivariant bundle \eqref{twistbundledef} gives rise to \eqref{twistedaction}. Over the big Bruhat cell $V\cong U^{\op}\cong U^{\op}P/P$, we may choose a section of the $P$-bundle $G\to G/P$, namely $v\mapsto u_v^{\op}\in G$ in the notation of \eqref{unipotents}. Over $V$ we then obtain a trivializing section $\sigma_0:V\to \mathcal{L}$ given by
\begin{equation}\label{sigm0def}
    \sigma_0(v)=[u_v^{\op},1]\in G\times^{P,\eta^{-1}}\A^1.
\end{equation}
More generally, a local section of $\mathcal{L}$ over $V$ may be written as
\begin{equation}
    \sigma(v)=[u_v^{\op},f(v)]\in G\times^{P,\eta^{-1}}\A^1,
\end{equation}
\noindent for $v\in V$ and $f$ a rational function on $V$.

\begin{prop}\label{linetwistingProp}
    The action of $G$ on $\mathcal{L}$, expressed in the trivialization over the big Bruhat cell $V\cong U^{\op}P/P$ determined by the section $\sigma_0$, induces the action
    \[
        (g f)(v)=\eta(p(g^{-1},v))\,f(g^{-1}v).
    \]
\end{prop}

\begin{proof}
Recall that the action of $g_0\in G(\kappa)$ on a section $v\mapsto [g_v,f(v)]\in G\times^{P,\eta^{-1}}\A^1$ is given by
\begin{equation}
    v\mapsto [g_0g_{g_0^{-1}v},f(g_0^{-1}v)].
\end{equation}
Applying this to the section $f\sigma_0$, we compute
\begin{align}
    g\cdot (f\sigma_0) 
    &=
    \left\{
    v\mapsto [g u_{g^{-1}v}^{\op},f(g^{-1}v)]
    \right\}
    \\
    &=
    \left\{
    v\mapsto [u_v^{\op}p(g,g^{-1}v),f(g^{-1}v)]
    \right\}
    \nonumber\\
    &=
    \left\{
    v\mapsto [u_v^{\op}p(g^{-1},v)^{-1},f(g^{-1}v)]
    \right\}
    \qquad\text{(by \eqref{altcocycle}, and the fact that $p(e,v) = e$)}
    \nonumber\\
    &=
    \left\{
    v\mapsto [u_v^{\op},\eta(p(g^{-1},v))\,f(g^{-1}v)]
    \right\}
    \nonumber\\
    &=
    \left\{
    v\mapsto \eta(p(g^{-1},v))\,f(g^{-1}v)\cdot \sigma_0(v)
    \right\}.
    \nonumber
\end{align}
This is precisely \eqref{twistedaction}.
\end{proof}

We now return to the question of infinitesimalizing the twisted action \eqref{twistedaction}.

\begin{prop}\label{ConfActionProp}
Let
\begin{equation}
    \xi=
    \left(\begin{smallmatrix}
        \alpha & -\lambda^{\mathsf T}J_V & 0\\
        \mu    & X                       & \lambda\\
        0      & -\mu^{\mathsf T}J_V     & -\alpha
    \end{smallmatrix}\right)
\end{equation}

\noindent be as in \eqref{genericg}. Then
\begin{equation}
    \phi(\xi)
    =
    -\partial_{\mu}
    \;-\;
    \langle Xv,\nabla\rangle
    \;+\;
    \alpha\bigl(E+(k-1)\bigr)
    \;-\;
    (\lambda^{\mathsf T}J_Vv)\,\bigl(E+(k-1)\bigr)
    \;+\;
    Q(v)\,\partial_{\lambda}.
\end{equation}

\noindent Here
\begin{equation}
    \nabla
    =
    \bigl(
        \partial_{x_1},
        \dots,
        \partial_{x_k},
        \partial_{y_1},
        \dots,
        \partial_{y_k}
    \bigr)^{\mathsf T},
\end{equation}
\begin{equation}\label{EulerDef}
    E
    =
    \sum_{i=1}^k x_i\partial_{x_i}
    +
    \sum_{i=1}^k y_i\partial_{y_i},
\end{equation}

\noindent where $E$ is the Euler operator and $\langle -,-\rangle$ denotes the usual evaluation map.
\end{prop}

\begin{proof}
We compute the infinitesimal action associated to the twisted $G$-action \eqref{twistedaction}
\begin{equation}
    (g f)(v)=\chi_0(p(g^{-1},v))^{-(k-1)}\,f(g^{-1}v),
\end{equation}

\noindent where
\begin{equation}
    g^{-1}u_v^{\op}=u_{g^{-1}v}^{\op}p(g^{-1},v)
\end{equation}

\noindent with $g^{-1}v\in V$ and $p(g^{-1},v)\in P$.

For ease of exposition, we use exponentials. Let $\xi\in\mathfrak g$ and set $g(t):=\exp(t\xi)$. We define
\begin{equation}
    \phi(\xi)f
    :=
    \left.\frac{d}{dt}\right|_{t=0}(g(t)f).
\end{equation}

\noindent Since $p(1,v)=1$ and $g(t)^{-1}v=v$ at $t=0$, differentiating gives
\begin{equation}\label{eq:phi-general}
    \phi(\xi)f(v)
    =
    \left.\frac{d}{dt}\right|_{t=0} f(g(t)^{-1}v)
    \;-\;
    (k-1)
    \left.\frac{d}{dt}\right|_{t=0}
    \log\chi_0\bigl(p(g(t)^{-1},v)\bigr)\cdot f(v).
\end{equation}

\noindent We now treat the $\mu$-, $X$-, $\alpha$-, and $\lambda$-components separately. Throughout we use the explicit matrix model for $G$ and the embedding $V\simeq U^{\op}\subset G$ via $v\mapsto u_v^{\op}$ as in \eqref{unipotents}. For each component we compute $g(t)^{-1}u_v^{\op}$, refactor it as $u_{g(t)^{-1}v}^{\op}p(g(t)^{-1},v)$, and then apply \eqref{eq:phi-general}.

\medskip

\noindent\textbf{The $\mu$-component.}
Take
\begin{equation}
    \xi_\mu:=
    \left(\begin{smallmatrix}
        0 & 0 & 0\\
        \mu & 0 & 0\\
        0 & -\mu^{\mathsf T}J_V & 0
    \end{smallmatrix}\right).
\end{equation}

\noindent Then $g(t)=\exp(t\xi_\mu)=u_{t\mu}^{\op}$. Since $U^{\op}$ is abelian, we have $g(t)^{-1}u_v^{\op}=u_{-t\mu}^{\op}u_v^{\op}=u_{v-t\mu}^{\op}$. Thus $p(g(t)^{-1},v)=1$, so $\chi_0(p(g(t)^{-1},v))\equiv 1$, and hence the twisting term in \eqref{eq:phi-general} vanishes. Therefore
\begin{equation}
    \phi(\xi_\mu)f(v)
    =
    \left.\frac{d}{dt}\right|_{t=0} f(v-t\mu)
    =
    -\partial_\mu f(v).
\end{equation}

\medskip

\noindent\textbf{The $X$-component.}
Take
\begin{equation}
    \xi_X:=
    \left(\begin{smallmatrix}
        0 & 0 & 0\\
        0 & X & 0\\
        0 & 0 & 0
    \end{smallmatrix}\right),
    \qquad
    X\in\mathfrak{so}(V,Q).
\end{equation}

\noindent Then $g(t)=\exp(t\xi_X)=
    \left(\begin{smallmatrix}
        1 & 0 & 0\\
        0 & e^{tX} & 0\\
        0 & 0 & 1
    \end{smallmatrix}\right)$. A direct multiplication gives $g(t)^{-1}u_v^{\op}
    =
    u_{e^{-tX}v}^{\op}
    \left(\begin{smallmatrix}
        1 & 0 & 0\\
        0 & e^{-tX} & 0\\
        0 & 0 & 1
    \end{smallmatrix}\right)$, so $g(t)^{-1}v=e^{-tX}v$ and $\chi_0(p(g(t)^{-1},v))=1$. Thus
\begin{equation}
    \phi(\xi_X)f(v)
    =
    \left.\frac{d}{dt}\right|_{t=0} f(e^{-tX}v)
    =
    -\langle Xv,\nabla\rangle f(v).
\end{equation}

\medskip

\noindent\textbf{The $\alpha$-component.}
Take
\begin{equation}
    \xi_\alpha:=
    \left(\begin{smallmatrix}
        \alpha & 0 & 0\\
        0 & 0 & 0\\
        0 & 0 & -\alpha
    \end{smallmatrix}\right).
\end{equation}

\noindent Then $g(t)=\exp(t\xi_\alpha)=
    \left(\begin{smallmatrix}
        e^{t\alpha} & 0 & 0\\
        0 & I & 0\\
        0 & 0 & e^{-t\alpha}
    \end{smallmatrix}\right)$. A direct computation gives $g(t)^{-1}u_v^{\op}
    =
    u_{e^{t\alpha}v}^{\op}
    \left(\begin{smallmatrix}
        e^{-t\alpha} & 0 & 0\\
        0 & I & 0\\
        0 & 0 & e^{t\alpha}
    \end{smallmatrix}\right)$, so $g(t)^{-1}v=e^{t\alpha}v$ and $\chi_0(p(g(t)^{-1},v))=e^{-t\alpha}$. Substituting into \eqref{eq:phi-general} yields
\begin{equation}
    \phi(\xi_\alpha)f(v)
    =
    \left.\frac{d}{dt}\right|_{t=0} f(e^{t\alpha}v)
    \;-\;
    (k-1)\left.\frac{d}{dt}\right|_{t=0}\log(e^{-t\alpha})\cdot f(v)
    =
    \alpha\,Ef(v)+\alpha(k-1)f(v).
\end{equation}

\noindent That is,
\begin{equation}
    \phi(\xi_\alpha)=\alpha\bigl(E+(k-1)\bigr).
\end{equation}

\medskip

\noindent\textbf{The $\lambda$-component.}
Take
\begin{equation}
    \xi_\lambda:=
    \left(\begin{smallmatrix}
        0 & -\lambda^{\mathsf T}J_V & 0\\
        0 & 0 & \lambda\\
        0 & 0 & 0
    \end{smallmatrix}\right).
\end{equation}

\noindent Then $g(t)=\exp(t\xi_\lambda)=u_{t\lambda}$, so $g(t)^{-1}=u_{-t\lambda}$. Multiplying out $g(t)^{-1}u_v^{\op}$ and refactoring to first order gives
\begin{equation}\label{eq:lambda-factor}
    g(t)^{-1}u_v^{\op}
    =
    u_{v'(t)}^{\op}
    \cdot
    \left(\begin{smallmatrix}
        1+t\,\lambda^{\mathsf T}J_V v & * & *\\
        0 & I & *\\
        0 & 0 & 1-t\,\lambda^{\mathsf T}J_V v
    \end{smallmatrix}\right)
    +O(t^2),
\end{equation}

\noindent where $v'(t)=v+t\,Q(v)\lambda-t\,(\lambda^{\mathsf T}J_V v)\,v$, and where the starred entries are irrelevant for $\chi_0$. In particular, $\chi_0(p(g(t)^{-1},v))
    =
    1+t\,\lambda^{\mathsf T}J_V v+O(t^2)$. Thus
\begin{equation}
    -(k-1)\left.\frac{d}{dt}\right|_{t=0}\log\chi_0(p(g(t)^{-1},v))
    =
    -(k-1)\,(\lambda^{\mathsf T}J_V v).
\end{equation}

Using \eqref{eq:phi-general} and \eqref{eq:lambda-factor}, we conclude
\begin{align}
    \phi(\xi_\lambda)f(v)
    &=
    \left.\frac{d}{dt}\right|_{t=0} f(v'(t))
    \;-\;
    (k-1)\,(\lambda^{\mathsf T}J_V v)\,f(v)
    \nonumber\\
    &=
    Q(v)\,\partial_\lambda f(v)
    \;-\;
    (\lambda^{\mathsf T}J_V v)\,Ef(v)
    \;-\;
    (k-1)\,(\lambda^{\mathsf T}J_V v)\,f(v);
\end{align}

\noindent that is,
\begin{equation}
    \phi(\xi_\lambda)
    =
    Q(v)\,\partial_\lambda
    \;-\;
    (\lambda^{\mathsf T}J_V v)\,\bigl(E+(k-1)\bigr).
\end{equation}

\medskip

\noindent Combining these terms and using linearity of $\phi$ in $\xi$ gives the result.
\end{proof}

\subsection{Interaction With the Laplacian}\label{InteractionLaplaceSection}

Next we examine how the operators $\phi(\g)$ in $D_V$ interact with the Laplace-Beltrami operator.

\begin{prop}\label{DeltaPres}
Let $\Delta$ be the Laplace-Beltrami operator associated with $Q$,
\begin{equation}\label{LapDef}
    \Delta
    =
    \frac{\partial^2}{\partial x_1\partial y_k}
    +
    \cdots
    +
    \frac{\partial^2}{\partial x_k\partial y_1}.
\end{equation}

\noindent For all $\xi=(\alpha,\mu,X,\lambda)\in\g$ written as in \eqref{genericg}, we have
\begin{equation}\label{eq:comm-phi-Delta}
    \bigl[\phi(\alpha,\mu,X,\lambda),\Delta\bigr]
    =
    2\bigl(B(\lambda,v)-\alpha\bigr)\Delta,
\end{equation}

\noindent where $B(\lambda,v)=\lambda^{\mathsf T}J_Vv$. In particular,
\begin{equation}
    \ad_{\phi(\xi)}(D_V\Delta)\subset D_V\Delta.
\end{equation}
\end{prop}

\begin{proof}
By Proposition~\ref{ConfActionProp},
\begin{equation}
    \phi(\alpha,\mu,X,\lambda)
    =
    -\partial_{\mu}
    -
    \langle Xv,\nabla\rangle
    +
    \alpha\bigl(E+(k-1)\bigr)
    -
    B(\lambda,v)\bigl(E+(k-1)\bigr)
    +
    Q(v)\partial_{\lambda}.
\end{equation}

\noindent Since the commutator is linear in $\xi$, we treat the four summands separately.

The $\mu$-term commutes with $\Delta$, since $\partial_\mu$ is a constant-coefficient vector field. Likewise,
\begin{equation}
    [E,\Delta]=-2\Delta,
\end{equation}

\noindent so the $\alpha$-term contributes $\bigl[\alpha(E+(k-1)),\Delta\bigr]
    =
    -2\alpha\,\Delta$. For the $X$-term, the operator $\langle Xv,\nabla\rangle$ is the infinitesimal generator of the linear action of $\O(V,Q)$ on $V$, while $\Delta$ is the Laplace operator attached to the $\O(V,Q)$-invariant quadratic form $Q$. Hence $[\langle Xv,\nabla\rangle,\Delta]=0$.

It remains to compute the $\lambda$-term. Set
\begin{equation}
    B_\lambda(v):=B(\lambda,v).
\end{equation}

\noindent A direct product-rule calculation gives
\begin{equation}\label{eq:Delta-B-short}
    [\Delta,B_\lambda]=\partial_\lambda,
    \qquad
    [B_\lambda,\Delta]=-\partial_\lambda,
\end{equation}

\noindent and similarly
\begin{equation}\label{eq:Delta-Q-short}
    [\Delta,Q]=E+k,
    \qquad
    [Q,\Delta]=-(E+k).
\end{equation}

\noindent Therefore $\bigl[-B_\lambda(E+(k-1)),\Delta\bigr]
    =
    -B_\lambda[E,\Delta]-[B_\lambda,\Delta](E+(k-1))
    \nonumber
    =
    2B_\lambda\Delta+\partial_\lambda(E+(k-1))$, and $\bigl[Q\partial_\lambda,\Delta\bigr]
    =
    [Q,\Delta]\partial_\lambda
    =
    -(E+k)\partial_\lambda$. Since $[\partial_\lambda,E]=\partial_\lambda$, we have $\partial_\lambda(E+(k-1))=(E+k)\partial_\lambda$, and thus the $\partial_\lambda$-terms cancel. Hence the $\lambda$-term contributes $\bigl[-B(\lambda,v)(E+(k-1))+Q(v)\partial_\lambda,\Delta\bigr]
    =
    2B(\lambda,v)\Delta$. Summing the four contributions gives
\begin{equation}
    [\phi(\alpha,\mu,X,\lambda),\Delta]
    =
    -2\alpha\Delta+2B(\lambda,v)\Delta,
\end{equation}

\noindent which is \eqref{eq:comm-phi-Delta}.

For the final statement, if $A\in D_V$, then $[\phi(\xi),A\Delta]
    =
    [\phi(\xi),A]\Delta+A[\phi(\xi),\Delta]$. Since $[\phi(\xi),A]\in D_V$ and $[\phi(\xi),\Delta]$ is a left multiple of $\Delta$ by \eqref{eq:comm-phi-Delta}, both terms lie in $D_V\Delta$. Thus
\begin{equation}
    \ad_{\phi(\xi)}(D_V\Delta)\subset D_V\Delta.
\end{equation}
\end{proof}

Let $g\in G(\kappa)$. Define the induced action of $G$ on differential operators on $\kappa(V)$ by transport of structure:
\begin{equation}\label{eq:action-on-operators}
    (g\cdot \xi)(f):=g\Bigl(\xi(g^{-1}f)\Bigr),
    \qquad
    f\in \kappa(V).
\end{equation}
We now record the very useful calculation of the action of $w_0$ on both $\Delta$ and on functions.

\begin{prop}\label{w0onDelta}
We write $K(f)$ for the twisted $w_0$-action on a function $f \in \kappa(V)$. That is:
\begin{equation}\label{Kelvindef}
    (Kf)(v):=(w_0f)(v)=(-Q(v))^{-(k-1)}\,f\left(-\frac{v}{Q(v)}\right).
\end{equation}
\noindent This is the \textbf{Kelvin Transform}.

The action of $w_0$ on $\Delta$, as defined in \eqref{eq:action-on-operators}, is given by
\begin{equation}\label{w_0actsonDelta}
    w_0\cdot \Delta=Q^2\Delta.
\end{equation}
 Since $w_0^2=1$, we have $K^2=\Id$, and therefore $K\Delta K^{-1}=K\Delta K=w_0\cdot \Delta=Q^2\Delta$.
\end{prop}

\begin{proof}
On the big cell $V=\{[1:v:-Q(v)]\}$, the Weyl element $w_0$ sends $[1:v:-Q(v)]$ to $[-Q(v):v:1]\sim [1:-v/Q(v):-1/Q(v)]$. Thus the induced rational map on $V$ is $v\mapsto -v/Q(v)$. Comparison with the factorization $w_0^{-1}u_v^{\op}=u_{v'}^{\op}p(w_0^{-1},v)$ shows that $\chi_0(p(w_0^{-1},v))=-Q(v)$. Since $w_0^{-1}=w_0$, this is equivalently $\chi_0(p(w_0,v))=-Q(v)$, and therefore $(w_0f)(v)=(-Q(v))^{-(k-1)}f(-v/Q(v))$, which is \eqref{Kelvindef}. A direct chain-rule computation gives $\Delta(w_0f)=(-Q(v))^{-2}w_0(\Delta f)$. Applying $w_0$ once more, and using $Q(-v/Q(v))=Q(v)^{-1}$, we obtain $(w_0\cdot \Delta)(f)=w_0(\Delta(w_0f))=w_0((-Q(v))^{-2}w_0(\Delta f))=Q(v)^2\Delta f$. Thus $w_0\cdot \Delta=Q^2\Delta$.
\end{proof}

We now generalize Proposition~\ref{w0onDelta} to general $g \in G(\kappa)$.

\begin{prop}\label{prop:group-action-on-Delta}
Recall the definition of $p(g,v)$ in \eqref{ptwistdef}
and the action of $G$ on functions given by \eqref{twistedaction}. We have
\begin{equation}\label{eq:Delta-semi-invariant}
    g\cdot \Delta=\chi_0\bigl(p(g^{-1},v)\bigr)^2\,\Delta.
\end{equation}
\noindent Equivalently, for all $f\in \kappa(V)$,
\begin{equation}\label{eq:Delta-intertwining}
    \Delta(gf)=\chi_0\bigl(p(g^{-1},v)\bigr)^{-2}\,g(\Delta f).
\end{equation}
\noindent Infinitesimally, \eqref{eq:Delta-semi-invariant} recovers the commutator formula \eqref{eq:comm-phi-Delta}.
\end{prop}

\begin{proof}
Let $S:=\{\,g\in G(\kappa):g\cdot \Delta=\chi_0(p(g^{-1},v))^2\,\Delta\,\}$. We first note that $S$ is a subgroup. For any rational function $c(v)$ and any differential operator $D$, one has
\begin{equation}\label{eq:scalar-transport}
    g\cdot (c(v)D)=c(g^{-1}\cdot v)\,(g\cdot D).
\end{equation}
\noindent Indeed, applying both sides to $f$ gives $\bigl(g\cdot (c(v)D)\bigr)(f)=g(c(v)D(g^{-1}f))=c(g^{-1}\cdot v)g(D(g^{-1}f))=c(g^{-1}\cdot v)(g\cdot D)(f)$.

Now let $g_1,g_2\in S$. Using \eqref{eq:scalar-transport}, we compute $((g_1g_2)\cdot \Delta)
=
\chi_0(p(g_2^{-1},g_1^{-1}\cdot v))^2
\chi_0(p(g_1^{-1},v))^2\Delta.$ By the cocycle relation, $p((g_1g_2)^{-1},v)=p(g_2^{-1}g_1^{-1},v)=p(g_2^{-1},g_1^{-1}\cdot v)p(g_1^{-1},v)$, hence 
$((g_1g_2)\cdot \Delta) =\chi_0(p((g_1g_2)^{-1},v))^2 \Delta.$
Thus $g_1g_2\in S$.

Now we check that $S$ contains the standard generators of $G$. If $g=u_\mu^{\op}\in U^{\op}$, then $g^{-1}u_v^{\op}=u_{v-\mu}^{\op}$, so $p(g^{-1},v)=1$ and $(gf)(v)=f(v-\mu)$. Since $\Delta$ is translation-invariant, we get $g\cdot \Delta=\Delta=\chi_0(p(g^{-1},v))^2\Delta$. Hence $U^{\op}\subset S$.

Next, let $g=\left(\begin{smallmatrix} a & 0 & 0\\ 0 & h & 0\\ 0 & 0 & a^{-1} \end{smallmatrix}\right)\in L^\circ$, where $a\in\G_m$ and $h\in SO(V,Q)$. Then $g^{-1}u_v^{\op}=u_{ah^{-1}v}^{\op}g^{-1}$, so $p(g^{-1},v)=g^{-1}$ and $\chi_0(p(g^{-1},v))=a^{-1}$. Therefore $(gf)(v)=a^{k-1}f(ah^{-1}v)$. Since $h$ preserves $\Delta$ and scaling by $a$ multiplies $\Delta$ by $a^{-2}$, it follows that $g\cdot \Delta=a^{-2}\Delta=\chi_0(p(g^{-1},v))^2\Delta$. Thus $L^\circ\subset S$.

Finally, Proposition~\ref{w0onDelta} shows that $w_0\cdot \Delta=Q^2\Delta=\chi_0(p(w_0^{-1},v))^2\Delta$, since $w_0^{-1}=w_0$. Since $G$ is generated by $U^{\op}$, $L^\circ$, and $w_0$, we conclude that $S=G(\kappa)$, proving \eqref{eq:Delta-semi-invariant}. The equivalence with \eqref{eq:Delta-intertwining} is immediate from \eqref{eq:action-on-operators}, since $(g\cdot \Delta)(gf)=g(\Delta f)$.

For the infinitesimal statement, let $g(t)=\exp(t\xi)$. Differentiating \eqref{eq:Delta-semi-invariant} at $t=0$ gives $[\phi(\xi),\Delta]=2\left.\frac{d}{dt}\right|_{t=0}\log \chi_0(p(g(t)^{-1},v))\,\Delta$. By the component calculations in the proof of Proposition~\ref{ConfActionProp}, the derivative on the right is $B(\lambda,v)-\alpha$. Hence $[\phi(\xi),\Delta]=2(B(\lambda,v)-\alpha)\Delta$, recovering \eqref{eq:comm-phi-Delta}.
\end{proof}

\begin{rem}
    We can now see that $\chi_0(p(g^{-1},v))^{-1}$ is precisely the algebraic avatar of the conformal factor $\Omega_g$ discussed in \eqref{OGConfFactr}.
\end{rem}

\subsection{\texorpdfstring{The Map $\rho_{\mathrm{amb}}$}{The Map rho-amb}}\label{rhoambSect}

We now turn to calculating $\rho_{\mathrm{amb}}(\g)$. By definition,
\begin{equation}
    \rho_{\mathrm{amb}}=\tau\circ\phi,
\end{equation}

\noindent so it remains to apply the linear Fourier transform $\tau$ to the explicit formula for $\phi$ obtained above.

\begin{prop}\label{rhofromula}
We have
\begin{equation}\label{Goncharovops}
    \rho_{\mathrm{amb}}
    \left(\begin{smallmatrix}
        \alpha & -\lambda^{\mathsf T}J_V & 0\\
        \mu    & X                       & \lambda\\
        0      & -\mu^{\mathsf T}J_V     & -\alpha
    \end{smallmatrix}\right)
    =
    \mu
    +
    \bigl\langle X^{*}\nu,\nabla_\nu\bigr\rangle
    -
    \alpha(E+k+1)
    +
    (E+k+1)\partial_{\lambda^\flat}
    -
    \lambda\Delta.
\end{equation}

\noindent Here $\nu\in V^*$, we regard $\mu\in V$ as a linear function on $V^*$, $\lambda^\flat=B(\lambda,*)\in V^*$, $E$ is the Euler operator on $V^*$, and $\Delta$ is the Laplace-Beltrami operator of $Q^*$ on $V^*$.
\end{prop}

\begin{proof}
Apply $\tau$ to the formula for $\phi(\xi)$ in Proposition~\ref{ConfActionProp}. Recall that $\tau$ is the algebra automorphism determined by
\begin{equation}
    \tau(v_i)=\partial_{\nu_i},
    \qquad
    \tau(\partial_{v_i})=-\nu_i.
\end{equation}

\noindent Here the operators on the right-hand side are viewed as elements of $D_{V^*}$ (cf.\ \eqref{linFourtrans}).

The $\mu$-term is immediate: $\tau(-\partial_\mu)=\mu$. 

For the $X$-term, writing $-\langle Xv,\nabla\rangle
    =
    -\sum_{i,j}X_{ij}v_j\partial_{v_i}$, we obtain
$\tau\bigl(-\langle Xv,\nabla\rangle\bigr)
    =
    -\sum_{i,j}X_{ij}\,\tau(v_j)\tau(\partial_{v_i})
    =
    \sum_{i,j}X_{ij}\,\partial_{\nu_j}\nu_i
    =\sum_{i,j}X_{ij}\bigl(\nu_i\partial_{\nu_j}+\delta_{ij}\bigr)
    =
    \langle X^*\nu,\nabla_\nu\rangle+\operatorname{tr}(X)$. Since $X\in\mathfrak{so}(V,Q)$, we have $\operatorname{tr}(X)=0$, and therefore $\tau\bigl(-\langle Xv,\nabla\rangle\bigr)
    =
    \langle X^*\nu,\nabla_\nu\rangle$.

For the $\alpha$-term, $\tau(E)=-E-2k$, so $\tau\bigl(\alpha(E+(k-1))\bigr)
    =
    \alpha(-E-k-1)
    =
    -\alpha(E+k+1)$.

It remains to transform the $\lambda$-term $-B(\lambda,v)(E+(k-1))+Q(v)\partial_\lambda$.
Since $\tau(B(\lambda,v))=\partial_{\lambda^\flat}$, we have $\tau\bigl(-B(\lambda,v)(E+(k-1))\bigr)
    =
    -\partial_{\lambda^\flat}(-E-k-1)
    =
    \partial_{\lambda^\flat}(E+k+1)$. Using $[\partial_{\lambda^\flat},E]=\partial_{\lambda^\flat}$, this becomes $\partial_{\lambda^\flat}(E+k+1)
    =
    (E+k+2)\partial_{\lambda^\flat}$. Next, $\tau\bigl(Q(v)\partial_\lambda\bigr)
    =
    \tau(Q(v))\,\tau(\partial_\lambda)
    =
    \Delta(-\lambda)
    =
    -\Delta\lambda$. Since $[\Delta,\lambda]=\partial_{\lambda^\flat}$, equivalently
$\Delta\lambda=\lambda\Delta+\partial_{\lambda^\flat}$, this becomes $\tau\bigl(Q(v)\partial_\lambda\bigr)
    =
    -\lambda\Delta-\partial_{\lambda^\flat}$. Adding the two contributions, the total $\lambda$-term is $(E+k+2)\partial_{\lambda^\flat}-\lambda\Delta-\partial_{\lambda^\flat}
    =
    (E+k+1)\partial_{\lambda^\flat}-\lambda\Delta$.

Combining all terms gives \eqref{Goncharovops}.
\end{proof}

\begin{prop}\label{goncharovresult}
    For every $\xi\in \mathfrak g$, one has
\begin{equation}
    \ad_{\rho_{\mathrm{amb}}(\xi)}(D_{V^*}Q^*)\subset D_{V^*}Q^*.
\end{equation}

\noindent Equivalently, $\rho_{\mathrm{amb}}(\mathfrak g)$ lies in the normalizer of the principal left ideal $D_{V^*}Q^*\subset D_{V^*}$.
\end{prop}
\begin{proof}
Since $\rho_{\mathrm{amb}}=\tau\circ \phi$ and $\tau(\Delta)=Q^*$, this is immediate from Proposition~\ref{DeltaPres}. Indeed, applying the algebra isomorphism $\tau:D_V\to D_{V^*}$ to the inclusion
\begin{equation}
    \ad_{\phi(\xi)}(D_V\Delta)\subset D_V\Delta
\end{equation}
\noindent yields
\begin{equation}
    \ad_{\tau(\phi(\xi))}\bigl(\tau(D_V\Delta)\bigr)\subset \tau(D_V\Delta).
\end{equation}
\noindent Since $\tau(\phi(\xi))=\rho_{\mathrm{amb}}(\xi)$ and $\tau(D_V\Delta)=D_{V^*}\tau(\Delta)=D_{V^*}Q^*$, we obtain the claim.
\end{proof}

\begin{rem}\label{ambnottangentrem}
    Proposition~\ref{goncharovresult} does not imply that the operators $\rho_{\mathrm{amb}}(\mathfrak g)$ are tangent to the cone $C\subset V^*$ in the usual sense. Indeed, although they normalize the principal left ideal $D_{V^*}Q^*$, they do not in general preserve the function ideal $(Q^*)\subset \kappa[V^*]$.

    Concretely, if in \eqref{Goncharovops} we specialize $\lambda$ to the coordinate function $x_i$, then $\lambda^\flat=y_{k+1-i}$ and the $\lambda$-term becomes $(E+k+1)\partial_{y_{k+1-i}}-x_i\Delta$. Similarly, if we specialize $\lambda$ to the coordinate function $y_i$, then $\lambda^\flat=x_{k+1-i}$ and the $\lambda$-term becomes $(E+k+1)\partial_{x_{k+1-i}}-y_i\Delta$. For these operators one finds $\Bigl((E+k+1)\partial_{y_{k+1-i}}-x_i\Delta\Bigr)(Q^*)=2x_i$, and similarly $\Bigl((E+k+1)\partial_{x_{k+1-i}}-y_i\Delta\Bigr)(Q^*)=2y_i$. Thus the ambient Fourier-side operators are not yet bona fide differential operators on the cone. The correction will be carried out in the next subsection.
\end{rem}

\subsection{\texorpdfstring{Constructing $\rho$: the $\delta_C$-Distribution and Passage to the Cone}{Constructing rho: the delta-C Distribution and Passage to the Cone}}\label{AmbientCorrectionSection}

Let
\begin{equation}
R:=\kappa[V^*],
\qquad
I:=(Q^*)\subset R.
\end{equation}
\noindent We write
\begin{equation}
\mathbb I(I):=\{P\in D_{V^*}:P(I)\subset I\},
\qquad
\mathbb A(I):=\{P\in D_{V^*}:P(R)\subset I\}.
\end{equation}
\noindent Then $\mathbb I(I)$ is the idealizer of $I$ in $D_{V^*}$, and $\mathbb A(I)$ is a two-sided ideal of $\mathbb I(I)$. Since $C\subset V^*$ is the hypersurface cut out by $Q^*$, we have
\begin{equation}\label{DCNormalizerQuotient}
D_C
=
D_{R/I\mid \kappa}
\cong
\mathbb I(I)/\mathbb A(I).
\end{equation}
\noindent \noindent This standard idealizer description is recorded, for example, in \cite[Theorem~15.5.13]{McConnellRobson2001}; here we apply it to the hypersurface $C=\Spec(R/I)$. Since $R$ is a polynomial algebra, $\mathbb A(I)=I D_{V^*}=(Q^*)D_{V^*}$.

We now carry out the construction via the local cohomology class $\delta_C$. Set
\begin{equation}
M:=H^1_{(Q^*)}(R)=R[(Q^*)^{-1}]/R,
\qquad
\delta_C:=\bigl[(Q^*)^{-1}\bigr]\in M.
\end{equation}

\begin{prop}\label{deltaQtransport}
    Multiplication by $\delta_C$ induces an $R$-module isomorphism
\begin{equation}
    T:R/(Q^*)\xrightarrow{\sim}R\delta_C,
    \qquad
    \overline{f}\longmapsto f\delta_C.
\end{equation}
\noindent Let $\xi\in\g$, and write $D_\xi:=\rho_{\mathrm{amb}}(\xi)$. Then $D_\xi$ preserves the subspace $R\delta_C\subset M$. More precisely, if
\begin{equation}
    A_\xi:=2(\partial_{\lambda^\flat}-\alpha),
\end{equation}
\noindent then for every $f\in R$ one has
\begin{equation}\label{deltaTransportFormula}
    D_\xi(f\delta_C)=\bigl((D_\xi-A_\xi)f\bigr)\delta_C.
\end{equation}
\noindent Equivalently, in the localization $R[(Q^*)^{-1}]$ one has
\begin{equation}\label{localizedConjFormula}
    Q^*D_\xi(Q^*)^{-1}=D_\xi-A_\xi.
\end{equation}
\end{prop}

\begin{proof}
The map $T$ is clearly $R$-linear and surjective. Its kernel consists of those $f\in R$ such that $f\delta_C=0$, equivalently $f/Q^*\in R$, i.e.\ $f\in(Q^*)$. Thus $T$ is an isomorphism.

By Proposition~\ref{DeltaPres}, after applying the Fourier transform $\tau$ and using $\tau(\Delta)=Q^*$, we obtain
\begin{equation}
    [D_\xi,Q^*]=A_\xi Q^*.
\end{equation}
\noindent Since $Q^*$ is invertible in $R[(Q^*)^{-1}]$, this implies
\begin{equation}
    D_\xi(Q^*)^{-1}=(Q^*)^{-1}(D_\xi-A_\xi),
\end{equation}
\noindent which is equivalent to \eqref{localizedConjFormula}. Applying this identity to $f\in R$ gives
\begin{equation}
    D_\xi\left(\frac{f}{Q^*}\right)
    =
    \frac{(D_\xi-A_\xi)f}{Q^*}.
\end{equation}
\noindent Passing to $M=R[(Q^*)^{-1}]/R$, we obtain \eqref{deltaTransportFormula}. In particular, $R\delta_C$ is $D_\xi$-stable.
\end{proof}

\begin{prop}\label{rhoConstructionProp}
    Define
\begin{equation}
    \rho(\xi):=T^{-1}\rho_{\mathrm{amb}}(\xi)T\in \End_\kappa(\kappa[C]).
\end{equation}
\noindent Then the image lies in $D_C$, and $\rho:\g\to D_C$ is a Lie algebra homomorphism.\footnote{Later, we will also use $\rho$ to refer to the associated map $\mathcal{U}(\g) \to D_C$.} Moreover, $\rho(\xi)$ is represented by the ambient differential operator
\begin{equation}\label{CorrectedAmbientFormula}
    \widetilde{\rho}(\xi):=Q^*\,\rho_{\mathrm{amb}}(\xi)\,(Q^*)^{-1}
    =
    \rho_{\mathrm{amb}}(\xi)-2(\partial_{\lambda^\flat}-\alpha),
\end{equation}
\noindent which preserves the ideal $(Q^*)\subset R$. Explicitly,
\begin{equation}\label{rhoConeFormula}
    \widetilde{\rho}
    \left(\begin{smallmatrix}
        \alpha & -\lambda^{\mathsf T}J_V & 0\\
        \mu    & X                       & \lambda\\
        0      & -\mu^{\mathsf T}J_V     & -\alpha
    \end{smallmatrix}\right)
    =
    \mu
    +
    \bigl\langle X^*\nu,\nabla_\nu\bigr\rangle
    -
    \alpha(E+k-1)
    +
    (E+k-1)\partial_{\lambda^\flat}
    -
    \lambda\Delta.
\end{equation}
\noindent Thus $\widetilde{\rho}(\xi)\in \mathbb I(I)$, and $\rho(\xi)\in D_C$ is the image of $\widetilde{\rho}(\xi)$ under the identification \eqref{DCNormalizerQuotient}. Equivalently, the image of $\g$ in $D_C$ is represented by the operators \eqref{rhoConeFormula}, taken modulo $(Q^*)D_{V^*}$.
\end{prop}

\begin{proof}
By Proposition~\ref{deltaQtransport}, the transported operator $T^{-1}D_\xi T$ is induced by $D_\xi-A_\xi$, which is \eqref{CorrectedAmbientFormula}. Substituting the explicit formula of Proposition~\ref{rhofromula} for $\rho_{\mathrm{amb}}(\xi)$ gives \eqref{rhoConeFormula}.

To see that $\widetilde{\rho}(\xi)$ preserves $(Q^*)$, rewrite the commutator identity as
\begin{equation}
    (D_\xi-A_\xi)\circ Q^*=Q^*\circ D_\xi.
\end{equation}
\noindent Therefore, for every $f\in R$,
\begin{equation}
    \widetilde{\rho}(\xi)(Q^*f)
    =
    (D_\xi-A_\xi)(Q^*f)
    =
    Q^*\,D_\xi(f)\in(Q^*).
\end{equation}
\noindent Thus $\widetilde{\rho}(\xi)\in \mathbb I(I)$, and so defines an element of $D_C$ via \eqref{DCNormalizerQuotient}. Finally, since $\rho$ is obtained from $\rho_{\mathrm{amb}}$ by transport of structure across the isomorphism $T$, it preserves commutators. Hence $\rho:\g\to D_C$ is a Lie algebra homomorphism.
\end{proof}

\begin{rem}
    The passage through $\delta_C$ is exactly what produces the shift
\begin{equation}
    k+1 \rightsquigarrow k-1
\end{equation}
\noindent between the ambient formula \eqref{Goncharovops} and the cone formula \eqref{rhoConeFormula}.
\end{rem}

\begin{rem}
    Proposition~\ref{rhoConstructionProp} is essentially the main result of \cite{Goncharov1976Weil} in the case of an even orthogonal group (see, in particular, the example at the end of loc.\ cit.). Here we have  provided a very direct, independent (if somewhat calculation-heavy) proof, which makes all the actions explicit and geometrically motivated. This explicitness will be rewarded later.
\end{rem}

\begin{defn}\label{hattaudef}
For later use we define the corrected Fourier-to-cone map
\begin{equation}\label{defwidehattau}
\widehat{\tau}:N(D_V\Delta)/D_V\Delta \longrightarrow D_C
\end{equation}
\noindent as follows. Here $N(D_V\Delta)$ denotes the normalizer of the principal left ideal $D_V\Delta$. For $[\xi]\in N(D_V\Delta)/D_V\Delta$, let $\tau(\xi)\in D_{V^*}$ denote the Fourier transform of any representative $\xi\in D_V$. Then set
\begin{equation}
\widehat{\tau}([\xi])
:=
\bigl[\,Q^*\,\tau(\xi)\,(Q^*)^{-1}\,\bigr]\in D_C,
\end{equation}
\noindent where the bracket on the right denotes the image in the quotient \eqref{DCNormalizerQuotient}. Equivalently, $\widehat{\tau}$ is the composition of the linear Fourier transform $\tau$ with the correction of Proposition~\ref{rhoConstructionProp}.
\end{defn}

\begin{rem}\label{tauisom}
The map $\widehat{\tau}$ is well-defined. Indeed, if $\xi$ is replaced by $\xi+a\Delta$, then $\tau(\xi)$ is replaced by $\tau(\xi)+\tau(a)Q^*$, and hence
\begin{equation}
Q^*\,\tau(\xi+a\Delta)\,(Q^*)^{-1}
\equiv
Q^*\,\tau(\xi)\,(Q^*)^{-1}
\pmod{(Q^*)D_{V^*}}.
\end{equation}
\noindent Moreover, $\widehat{\tau}$ is an isomorphism. The Fourier transform identifies $N(D_V\Delta)/D_V\Delta$ with the normalizer quotient $N(D_{V^*}Q^*)/D_{V^*}Q^*$ for the principal left ideal $D_{V^*}Q^*$. The correction $A\mapsto Q^*A(Q^*)^{-1}$ identifies this quotient with the idealizer quotient \eqref{DCNormalizerQuotient}; explicitly, the condition $A\in N(D_{V^*}Q^*)$ says that $Q^*A\in D_{V^*}Q^*$, so $Q^*A(Q^*)^{-1}$ is again an ambient differential operator, and it preserves the function ideal $(Q^*)$. Changing $A$ by an element of $D_{V^*}Q^*$ changes $Q^*A(Q^*)^{-1}$ by an element of $(Q^*)D_{V^*}$. Thus $\widehat{\tau}$ is the desired isomorphism.
\end{rem}

To connect the construction of $\rho$ with the quasiclassical picture, we now compute the principal symbols of the operators $\rho(\xi)$ on $C^o$. As in \eqref{invariantfunccot}, write a point of $T^*C^o$ as $(\nu,[v])$, where $\nu\in C^o\subset V^*$ and $v\in V$ represents a class in $V/\kappa\nu^\sharp$. On the ambient cotangent bundle $T^*V^*=V^*\times V$, the principal symbol is determined by $\sigma(f)(\nu,v)=f(\nu)$ for $f\in\kappa[V^*]$ and $\sigma(\partial_\eta)(\nu,v)=\langle\eta,v\rangle$ for $\eta\in V^*$. Although these ambient expressions may depend on the representative $v$, the principal symbols of the operators \eqref{rhoConeFormula} are invariant under $v\mapsto v+t\nu^\sharp$ and therefore descend to functions on $T^*C^o$. In particular,
\begin{equation}
\sigma(E)=\langle\nu,v\rangle,
\qquad
\sigma(\partial_{\lambda^\flat})=B(\lambda,v),
\qquad
\sigma(\Delta)=Q(v),
\qquad
\sigma(\lambda)=\langle\nu,\lambda\rangle.
\end{equation}

\noindent We next record the normalization for the middle block. For $X\in\mathfrak{o}(Q)$, the $X$-term in \eqref{rhoConeFormula} has principal symbol
$\langle X^*\nu,v\rangle=\langle\nu,Xv\rangle=B(\nu^\sharp,Xv)$. On the other hand, by \eqref{wedges},
\begin{equation*}
\begin{aligned}
\tr\bigl((v\wedge\nu^\sharp)X\bigr)
&=
\tr\bigl((\nu^\sharp v^\flat-v\nu)X\bigr)\\
&=
B(v,X\nu^\sharp)-B(\nu^\sharp,Xv)\\
&=
-2B(\nu^\sharp,Xv),
\end{aligned}
\end{equation*}
\noindent since $X$ is $B$-skew. Hence
$B(\nu^\sharp,Xv)=-\frac{1}{2}\tr\bigl((v\wedge\nu^\sharp)X\bigr)$. Equivalently, under the identification of the contragredient action with $-X$ induced by $B$, this equals $\frac{1}{2}\tr\bigl((v\wedge\nu^\sharp)X^*\bigr)$.

We temporarily write the scalar parameter in $\g$ as $a$, and the parameter in its $\mu$-slot as $\mu_0$, to distinguish them from the invariants in \eqref{invariantfunccot}. The four basic cases of Proposition~\ref{rhoConstructionProp} then have principal symbols
\begin{equation}
\begin{aligned}
\sigma\bigl(\rho(a,0,0,0)\bigr)
&=
-a\langle\nu,v\rangle,
&
\sigma\bigl(\rho(0,\mu_0,0,0)\bigr)
&=
\langle\nu,\mu_0\rangle,
\\
\sigma\bigl(\rho(0,0,X,0)\bigr)
&=
-\frac{1}{2}\tr\bigl((v\wedge\nu^\sharp)X\bigr),
&
\sigma\bigl(\rho(0,0,0,\lambda)\bigr)
&=
\langle\nu,v\rangle B(\lambda,v)-Q(v)\langle\nu,\lambda\rangle.
\end{aligned}
\end{equation}

\noindent Set $\alpha:=\langle\nu,v\rangle$ and $\mu_{\nu,v}:=\alpha v-Q(v)\nu^\sharp$, so that $\mu_{\nu,v}$ is the vector denoted by $\mu$ in \eqref{dualconeinv}. The final symbol may then be rewritten as
\begin{equation}
\sigma\bigl(\rho(0,0,0,\lambda)\bigr)
=
B(\mu_{\nu,v},\lambda).
\end{equation}

\noindent Thus the four families of operators recover the four components of the invariant matrix \eqref{invariantfunccot}: the scalar $\alpha=\langle\nu,v\rangle$, the vector $\nu^\sharp$ (equivalently, the covector $\nu$), the middle block $v\wedge\nu^\sharp$, and the vector $\mu_{\nu,v}=\alpha v-Q(v)\nu^\sharp$, with the scalar sign and trace-pairing normalization displayed above. These are precisely the matrix coefficients which generate $\kappa[T^*C^o]$ by Proposition~\ref{CotBundDescr}. Equivalently, $\rho$ quantizes the $F$-moment map: over $C^o$, the principal symbol of $\rho(\xi)$ is the descended invariant function corresponding to $\xi$.

See \cite{getz2025modulationgroups} for a conjectural generalization of this phenomenon to so-called ``modulation groups.''

\begin{prop}\label{Envelopmap}
    The Lie algebra homomorphism
\begin{equation}
    \rho:\mathfrak{g}\to D_C
\end{equation}

\noindent extends uniquely to an algebra homomorphism
\begin{equation}\label{rhodef}
    \rho:\mathcal{U}(\mathfrak{g})\to D_C.
\end{equation}

\noindent This map is surjective, and its kernel is the Joseph ideal $J\subset \mathcal{U}(\mathfrak{g})$.
\end{prop}

\begin{proof}
This is the even orthogonal case of the main result of Levasseur, Smith, and Stafford \cite{LevasseurSmithStafford1989Joseph}. We will later give an independent, geometric proof after we develop some theory of the Harmonic Sheaf; see Theorem \ref{rhoSurjective}.
\end{proof}

As an immediate consequence, we obtain an induced action of $G$ on $D_C$.

\begin{prop}\label{G-actsdef}
    The adjoint action of $G$ on $\mathfrak g$ extends to algebra automorphisms
\begin{equation}
    \Ad_g:\mathcal{U}(\mathfrak g)\xrightarrow{\sim}\mathcal{U}(\mathfrak g).
\end{equation}

\noindent Since the Joseph ideal $J\subset \mathcal{U}(\mathfrak g)$ is $G$-stable, this action descends to the quotient $\mathcal{U}(\mathfrak g)/J$. Via Proposition~\ref{Envelopmap}, we therefore obtain algebra automorphisms
\begin{equation}
    \alpha_g:D_C\cong \mathcal{U}(\mathfrak g)/J \xrightarrow{\sim} \mathcal{U}(\mathfrak g)/J \cong D_C,
    \qquad
    \alpha_g(\overline{x})=\overline{\Ad_g(x)}.
\end{equation}
\end{prop}

\begin{proof}
This follows immediately from Proposition~\ref{Envelopmap} and the $\Ad(G)$-stability of the Joseph ideal. We will give an independent proof in Proposition~\ref{IndActionDef} below, where we will also give a geometric interpretation of this action.
\end{proof}

\begin{rem}\label{twoactsremk}
    It is important to distinguish between two different $G$-actions on differential operators that have appeared in the discussion.

    First, the twisted action of $G$ on rational functions $\kappa(V)$ in \eqref{twistedaction} induces, by conjugation, the corresponding action of $G$ on differential operators on $\kappa(V)$ as in \eqref{eq:action-on-operators}. This action has been constructed here directly and explicitly, independently of \cite{LevasseurSmithStafford1989Joseph}. Moreover, it ultimately comes from the regular action of $G$ on the projective variety $G/P$: by Proposition~\ref{linetwistingProp}, the twisted action on $\kappa(V)$ is the translation action on local sections of the line bundle of conformal densities $\mathcal{L}$, and hence the induced action on differential operators is naturally interpreted as an action on the sheaf $\mathcal{D}_{\mathcal L}$ of twisted differential operators on $G/P$.

    Second, the action of $G$ on $D_C$ obtained in Proposition~\ref{G-actsdef} is of a different nature. Its very definition requires the use of Proposition~\ref{Envelopmap}; in particular, the surjectivity of the map $\rho:\mathcal U(\mathfrak g)\to D_C$. For clarity, one might call the first action the sheaf-theoretic action and the second the Joseph action. The former acts rationally on differential operators on the big cell $V$, or more precisely on the sheaf $\mathcal{D}_{\mathcal L}$ on $G/P$, whereas the latter acts on the global ring $D_C$ of differential operators on the quadric cone.

    The relationship between these two actions will be made explicit in Proposition~\ref{IndActionDef}. In fact, we will later use this connection to offer a geometric proof of Proposition~\ref{Envelopmap} independent of \cite{LevasseurSmithStafford1989Joseph}. 

    Until Section~\ref{DModHarmSect}, we will freely use Propositions~\ref{Envelopmap} and \ref{G-actsdef}. In particular, we will use them in defining the Fourier transform on $D_C$ below: since $\mathcal F$ will first be specified on generators, we need Proposition~\ref{Envelopmap} to know that these operators do indeed generate $D_C$. 
\end{rem}

\subsection{The Quadric Fourier Transform on Differential Operators}\label{QuadricExplicitSect}

We now specialize the Joseph action to the Weyl element $w_0$. We denote the resulting automorphism $\alpha_{w_0}$ of $D_C$ by $\mathcal F$, and call it the quadric Fourier transform.

For ease of reference, we record its action on the standard generators of $D_C$. Let $\{x_i,y_i\}_{i=1,\dots,k}$ be the coordinate functions on $V^*$, with $Q^*$ as in \eqref{QDef}. We define
\begin{equation}
    \mathfrak X_i:=\mathcal F(x_i),
    \qquad
    \mathfrak Y_i:=\mathcal F(y_i).
\end{equation}

\begin{prop}\label{QuadFourierProp}
The quadric Fourier transform acts by
\begin{equation}
\begin{aligned}
    \mathcal F(x_i)=\mathfrak X_i
    &:=
    (E+k-1)\frac{\partial}{\partial y_{k+1-i}}-x_i\Delta,
    &
    \mathcal F(y_i)=\mathfrak Y_i
    &:=
    (E+k-1)\frac{\partial}{\partial x_{k+1-i}}-y_i\Delta.
\end{aligned}
\end{equation}
\begin{equation}\label{Etrans}
    \mathcal F(E+k-1)=-(E+k-1),
    \qquad
    \text{equivalently}\qquad
    \mathcal F(E)=-E-2k+2.
\end{equation}
Moreover,
\begin{equation}
    \mathcal F(D_{ij})=D_{ij},
    \qquad
    \mathcal F(B_{ij})=B_{ij},
    \qquad
    \mathcal F(C_{ij})=C_{ij},
\end{equation}
where
\begin{equation}
\begin{aligned}
    D_{ij}
    &:=
    x_j\frac{\partial}{\partial x_i}
    -
    y_{k+1-i}\frac{\partial}{\partial y_{k+1-j}},
    &&1\le i,j\le k,
    \\
    B_{ij}
    &:=
    y_{k+1-j}\frac{\partial}{\partial x_i}
    -
    y_{k+1-i}\frac{\partial}{\partial x_j},
    &&1\le i<j\le k,
    \\
    C_{ij}
    &:=
    x_j\frac{\partial}{\partial y_{k+1-i}}
    -
    x_i\frac{\partial}{\partial y_{k+1-j}},
    &&1\le i<j\le k.
\end{aligned}
\end{equation}
Finally, $\mathcal F(\mathfrak X_i)=x_i$ and $\mathcal F(\mathfrak Y_i)=y_i$. In particular, the operators $\mathfrak X_i$ and $\mathfrak Y_j$ commute, and
\begin{equation}\label{fundidentop}
    \sum_{i=1}^k \mathfrak X_i\mathfrak Y_{k+1-i}=0
\end{equation}
in $D_C$.
\end{prop}

\begin{proof}
By Proposition~\ref{G-actsdef}, we have $\mathcal F=\alpha_{w_0}$, so
\begin{equation}
    \mathcal F(\rho(\xi))=\rho(\Ad_{w_0}\xi)
\end{equation}

\noindent for all $\xi\in\g$. Since $w_0$ exchanges the $\mu$- and $\lambda$-coordinates and sends $\alpha$ to $-\alpha$, while fixing the Levi factor $H$, the formulas follow immediately from \eqref{rhoConeFormula}. From the $\mu$ and $\lambda$ coordinates we find that
\begin{equation}
    \mathcal F(x_i)
    =
    (E+k-1)\frac{\partial}{\partial y_{k+1-i}}-x_i\Delta,
    \qquad
    \mathcal F(y_i)
    =
    (E+k-1)\frac{\partial}{\partial x_{k+1-i}}-y_i\Delta.
\end{equation}

\noindent Likewise, since $\rho(1,0,0,0)=-(E+k-1)$, we obtain \eqref{Etrans}. The operators $D_{ij},B_{ij},C_{ij}$ arise from the component of the Levi subalgebra $\textrm{Lie}(H)=\mathfrak{o}(Q)\subset \g$, which is fixed by $\Ad_{w_0}$, so they are fixed by $\mathcal F$.

Since $w_0^2=1$, we have $\mathcal F^2=\Id$, and hence
\begin{equation}
    \mathcal F(\mathfrak X_i)=x_i,
    \qquad
    \mathcal F(\mathfrak Y_i)=y_i.
\end{equation}

\noindent Because $\mathcal F$ is an algebra automorphism, the commutativity of the $x_i$ and $y_j$ implies that the $\mathfrak X_i$ and $\mathfrak Y_j$ commute. Applying $\mathcal F$ to the relation
\begin{equation}
    \sum_{i=1}^k x_i y_{k+1-i}=Q^*=0
\end{equation}

\noindent in $D_C$ yields \eqref{fundidentop}.
\end{proof}

\begin{rem}
    Kobayashi and Mano prove the commutativity of the operators $\mathfrak X_i,\mathfrak Y_j$ and the identity \eqref{fundidentop} directly. In our setting, both follow immediately from the fact that $\mathcal F$ is an algebra automorphism of $D_C$.

    It is also worth emphasizing that the Euler operator is not fixed by $\mathcal F$. Rather, by \eqref{Etrans} one has
\begin{equation}
    \mathcal F(E+k-1)=-(E+k-1),
\end{equation}

\noindent and hence
\begin{equation}
    \mathcal F(E)=-E-2k+2=-E-n+2.
\end{equation}

\noindent This shift is another manifestation of the twisting already encountered in Proposition~\ref{linetwistingProp}. Equivalently, it reflects the conformal weighting built into the construction of $\rho$.
\end{rem}

\begin{rem}\label{FundamentalOperatorsLocalVectorFieldsRem}
The operators $\mathfrak X_i$ and $\mathfrak Y_i$ are of order $2$ when written as ambient operators on $V^*$. This reflects the singularity of the cone. Indeed, Nakai's conjecture predicts that, for varieties in characteristic $0$, if the ring of differential operators is generated by functions and vector fields (i.e. differential operators of order $\le 1$) then the variety is smooth. Since $C$ is singular we should not expect vector fields and functions to generate the algebra of all differential operators; and, indeed, the operators $\mathfrak{X}_i$ and $\mathfrak{Y}_i$  cannot be so decomposed (\cite{Kobayashi:Mano}, Remark 2.4.9). 

However, the converse of Nakai's conjecture -- that a smooth variety has Grothendieck differential operators generated by functions and derivations -- is classical (see, for example, \cite[Corollary~15.5.6]{McConnellRobson2001}). Since the only singular point of $C$ is the origin, we discover the intriguing fact that, after restricting to any open subset of $C$ which avoids the origin, the fundamental operators $\mathfrak X_i$ and $\mathfrak Y_i$ must become decomposable into regular differential operators tangent to $C$. We will now make this explicit.

Recall that $C\subset V^*$ is cut out by
\begin{equation}
Q^*=\sum_{a=1}^k x_a y_{k+1-a},
\qquad
\Delta=\sum_{a=1}^k
\frac{\partial^2}{\partial x_a\partial y_{k+1-a}},
\end{equation}
\noindent where $x_i,y_i$ are the coordinate functions on $V^*$. Fix $i$, and work on the chart $C_{x_i}:=C\cap D(x_i)$. On this chart the equation $Q^*=0$ eliminates the coordinate $y_{k+1-i}$. For $a\neq i$, set
\begin{equation}
\xi_a^{(i)}
:=
\frac{\partial}{\partial x_a}
-
\frac{y_{k+1-a}}{x_i}
\frac{\partial}{\partial y_{k+1-i}},
\qquad
\eta_a^{(i)}
:=
\frac{\partial}{\partial y_{k+1-a}}
-
\frac{x_a}{x_i}
\frac{\partial}{\partial y_{k+1-i}}.
\end{equation}
\noindent These ambient vector fields preserve the ideal $(Q^*)$, since $\xi_a^{(i)}(Q^*)=\eta_a^{(i)}(Q^*)=0$, and therefore descend to vector fields on $C_{x_i}$. Then, as operators on $C_{x_i}$, one has
\begin{equation}\label{LocalVectorFieldDecompXi}
\mathfrak X_i
=
-x_i\sum_{a\neq i}\xi_a^{(i)}\eta_a^{(i)}.
\end{equation}
\noindent To see this, represent a function on $C_{x_i}$ by an expression independent of the eliminated coordinate $y_{k+1-i}$. On such representatives, the term $(E+k-1)\partial/\partial y_{k+1-i}$ vanishes, and the $a=i$ summand of $\Delta$ also vanishes. Hence
\begin{equation}
\mathfrak X_i
=
-x_i\sum_{a\neq i}
\frac{\partial^2}{\partial x_a\partial y_{k+1-a}}.
\end{equation}
\noindent But on these eliminated-coordinate representatives, $\xi_a^{(i)}$ acts as $\partial/\partial x_a$ and $\eta_a^{(i)}$ acts as $\partial/\partial y_{k+1-a}$, giving \eqref{LocalVectorFieldDecompXi}.

Similarly, on the chart $C_{y_i}:=C\cap D(y_i)$, the equation $Q^*=0$ eliminates the coordinate $x_{k+1-i}$. For $a\neq k+1-i$, set
\begin{equation}
\widetilde{\xi}_a^{(i)}
:=
\frac{\partial}{\partial x_a}
-
\frac{y_{k+1-a}}{y_i}
\frac{\partial}{\partial x_{k+1-i}},
\qquad
\widetilde{\eta}_a^{(i)}
:=
\frac{\partial}{\partial y_{k+1-a}}
-
\frac{x_a}{y_i}
\frac{\partial}{\partial x_{k+1-i}}.
\end{equation}
\noindent Again these vector fields preserve $(Q^*)$, and hence descend to vector fields on $C_{y_i}$. As operators on $C_{y_i}$, one has
\begin{equation}\label{LocalVectorFieldDecompYi}
\mathfrak Y_i
=
-y_i\sum_{a\neq k+1-i}
\widetilde{\xi}_a^{(i)}\widetilde{\eta}_a^{(i)}.
\end{equation}

More generally, if $f\in\kappa[C]$ vanishes at the origin, then the principal open $C_f$ is smooth. Since $f$ becomes invertible on $C_f$, the ideal $(x_1,\ldots,x_k,y_1,\ldots,y_k)$ becomes the unit ideal in $\mathcal O(C_f)$. Thus the opens $C_f\cap D(x_i)$ and $C_f\cap D(y_i)$ cover $C_f$, and the decompositions above show explicitly, on this cover, how the fundamental operators are generated by tangent vector fields.
\end{rem}

\section{\texorpdfstring{Kazhdan-Laumon Gluing of $\mathcal{D}$-Modules on the Quadric Cone}{Kazhdan-Laumon Gluing of D-Modules on the Quadric Cone}}\label{KazhLaumGlueSect}

In \cite{KazhdanLaumon1988Gluing}, Kazhdan and Laumon construct certain ``glued" categories of perverse $\ell$-adic sheaves on basic affine space $G/U$, where the gluing data is indexed by the Weyl group, and simple reflections are realized by Fourier--Deligne transforms. These glued categories have a rich structure and relate to, for example, Braverman--Kazhdan theory and Deligne--Lusztig theory \cite{Polishchuk2001GluingBasicAffine, MortonFerguson2025KazhdanLaumonCategoryO}. In general, the Kazhdan--Laumon gluing construction on basic affine space $G/U$ (for $G$ reductive and $U$ a maximal unipotent) produces categories which are \emph{not} equivalent to the category of perverse sheaves on any space.\footnote{The basic exception is $\SL_2/U\cong \A^2\setminus\{0\}$, in which case the Kazhdan--Laumon glued category is equivalent to $\textrm{Perv}(\A^2)$.} By contrast, in the $\mathcal{D}$-module setting, \cite{BezrukavnikovBravermanPositselskii2002Gluing} show that the corresponding glued category is equivalent to the category of modules over the Grothendieck ring of differential operators on the affinization $\overline{G/U}$.\footnote{That is, one considers the Grothendieck ring of differential operators on the algebra of regular functions $\kappa[G/U]$. It is a theorem of Grosshans \cite{Grosshans1983InvariantsUnipotentRadicals} that $G/U$ is strongly quasi-affine: $\kappa[G/U]$ is finitely generated, and the natural map $G/U\to \overline{G/U}:=\Spec \kappa[G/U]$ is an open embedding. In particular, $\kappa[G/U]\cong \kappa[\overline{G/U}]$.}

As discussed above, there is a natural analogue of Fourier transform on both $L^2$-spaces and on differential operators on the quadric cone, namely the action of the Weyl element $w_0\in G:=\O(Q^+)$ (see \eqref{WeylFourier} and Proposition \ref{QuadFourierProp}). Of course, this transformation does not act on the cone $C$ itself. Rather, in the larger orthogonal group $G$, it conjugates $P$ to $P^{\op}$ \cite{Kobayashi:Mano}, while commuting with the smaller group $H:=\O(Q)$, which \emph{does} act on $C$.

In this section we will show, following \cite{BezrukavnikovBravermanPositselskii2002Gluing}, that the category obtained by gluing $\mathcal{D}$-modules on two copies of $C^o$ along this non-linear Fourier transform is equivalent to the category of modules over the algebra $D_C$.

A note on terminology: throughout what follows, when we speak of $\mathcal{D}$-modules on a variety, we will always mean \textit{coherent} $\mathcal{D}$-modules, that is, $\mathcal{D}$-modules which are locally finitely generated over $\mathcal{D}$ \cite{HTT2008}.

\subsection{\texorpdfstring{Two Gradings}{Two Gradings}}\label{TwoGradings}

We recall that the ``smaller" orthogonal group $H:=\O(Q)$ acts on $C^o$, and hence on
\begin{equation}
    \kappa[C^o]=\kappa[C]=\Sym(V)/(Q^*).
\end{equation}

\noindent Accordingly, $\kappa[C]$ decomposes into irreducible representations of $H$. Choosing a Borel subgroup $B\subset G$ with $B\subset P$ determines a Borel subgroup of $H$. If $\mu$ is a dominant weight of $H$, we write $\mathcal O(\mu)$ (resp.\ $D_C(\mu)$) for the $V_\mu$-isotypic component of $\mathcal O_C$ (resp.\ $D_C$), where $V_\mu$ denotes the irreducible $H$-representation of highest weight $\mu$. We call this the \textit{weight grading}.

We also have the scaling action of $\G_m$ on $C$, induced from the natural scaling action on $V^*$. Its infinitesimal generator is the Euler operator $E$, and this gives rise to a second grading on $\mathcal O_C$ and $D_C$. Namely, we write $\mathcal O^d$ for the space of functions $f$ satisfying
\begin{equation}
    f(tx)=t^d f(x)
\end{equation}

\noindent for $t\in \G_m$, and $D_C^d$ for the space of operators $\xi\in D_C$ such that
\begin{equation}
    [E,\xi]=d\,\xi.
\end{equation}

\noindent We call this the \textit{degree grading}. Finally, we write $\mathcal O(\mu)^d$ and $D_C(\mu)^d$ for the corresponding intersections of the weight and degree gradings.

\begin{prop}
    Let $\omega_1$ denote the first fundamental weight of $\O(Q)$, i.e.\ the weight of the standard representation. Then
\begin{equation}\label{weightisdeg}
    \mathcal{O}_C(d\omega_1)=\mathcal{O}_C^d\cong \Sym^d(V)\big/ Q^*\Sym^{d-2}(V),
\end{equation}

\noindent and

\begin{equation}\label{weightalgdec}
    \kappa[C]=\mathcal{O}_C=\bigoplus_{d\ge 0}\mathcal{O}_C(d\omega_1)^d,
\end{equation}

\noindent with each $\mathcal{O}_C(d\omega_1)$ appearing with multiplicity $1$. Moreover, the Fourier transform $\mathcal F$ on $D_C$ preserves the weight grading and negates the degree grading:
\begin{equation}\label{Fourdeg}
    \xi\in D_C^d
    \iff
    \mathcal F(\xi)\in D_C^{-d}.
\end{equation}
\end{prop}

\begin{proof}
The identities \eqref{weightisdeg} and \eqref{weightalgdec} are the standard decomposition of the coordinate ring of the quadric cone into spherical harmonics (cf. Lemma \ref{FischerDecompositionLemma} below).

Since $w_0$ commutes with $H$, the Fourier transform $\mathcal F=\alpha_{w_0}$ commutes with the $H$-action on $D_C$, and hence preserves the weight grading. For the degree grading, observe that $[E,\xi]=d\,\xi$ if and only if $\mathcal{F}([E,\xi]) = [-E-n+2,\mathcal{F}(\xi)]=-d\,\mathcal{F}(\xi)$. Hence $\xi\in D_C^d\iff \mathcal F(\xi)\in D_C^{-d}$.
\end{proof}

Next, we have the following (cf.\ \cite{BezrukavnikovBravermanPositselskii2002Gluing}, Lemma 3.14):

\begin{prop}\label{degwt0}
    There is an isomorphism of algebras
\begin{equation}
    D_C(0)^0 \cong \kappa[E].
\end{equation}

\noindent In other words, the $H$-invariant degree $0$ differential operators on $C$ are precisely the polynomials in the Euler operator $E$.
\end{prop}

\begin{proof}
The inclusion $\kappa[E]\subset D_C(0)^0$ is clear. Conversely, let $\xi\in D_C(0)^0$. Since $\xi$ has degree $0$, it preserves each graded piece $\mathcal O_C^d$, and since $\xi$ is $H$-invariant while $\mathcal O_C^d=\mathcal O_C(d\omega_1)$ is an irreducible $H$-module, Schur's lemma implies that $\xi$ acts on $\mathcal O_C^d$ by a scalar $c_d\in \kappa$.

We claim that $d\mapsto c_d$ is a polynomial function. Let $T$ denote the forward difference operator, defined by
\begin{equation}
    (Tc)(n):=c_{n+1}-c_n.
\end{equation}

\noindent If $f\in \mathcal O_C^1$ and $g\in \mathcal O_C^e$, then
\begin{equation}
    [f,\xi](g)=f(c_eg)-c_{e+1}fg=-(Tc)(e)\,fg.
\end{equation}

\noindent Iterating, if $f_1,\dots,f_r\in \mathcal O_C^1$, then
\begin{equation}
    [f_1,[f_2,\dots,[f_r,\xi]\dots]](g)=(-1)^r(T^rc)(e)\,f_1\cdots f_rg.
\end{equation}

\noindent But $\xi$ has finite order, so for $r>\ord(\xi)$ the left-hand side vanishes identically. Hence $T^rc=0$ for all sufficiently large $r$. Since $\mathrm{char}(\kappa)=0$, it follows that $c_d$ is given by a polynomial $p(d)\in \kappa[d]$.

Now $E$ acts on $\mathcal O_C^d$ by multiplication by $d$, so $p(E)$ acts on $\mathcal O_C^d$ by the scalar $p(d)=c_d$. Thus $\xi$ and $p(E)$ agree on every graded piece $\mathcal O_C^d$, hence on all of $\kappa[C]=\bigoplus_{d\ge 0}\mathcal O_C^d$. Therefore $\xi=p(E)$.
\end{proof}

\subsection{The Quadric Shapovalov Determinant}\label{ShapovalovDetSect}

The central calculation of \cite{BezrukavnikovBravermanPositselskii2002Gluing} is a product formula for a Shapovalov determinant \cite{Shapovalov1974Structure}. We now establish the analogous formula in the differential operator algebra of the quadric cone.

We first observe that the irreducible $H$-representation
\begin{equation}
    P^d:=\mathcal{O}_C(d\omega_1)^d
\end{equation}

\noindent of degree-$d$ polynomial functions on the cone $C$ is self-dual, via the identification $V\cong V^*$ induced by the quadratic form $Q$. Using the Fourier transform $\mathcal{F}$, we therefore obtain a bilinear map
\begin{equation}
    m:P^d\otimes P^d\to D_C,
    \qquad
    f\otimes g\mapsto f\cdot \mathcal{F}(g).
\end{equation}

\noindent Since $P^d$ is irreducible and self-dual, the space $(P^d\otimes P^d)^H$ is one-dimensional. Let $C_d\in P^d\otimes P^d$ be a nonzero $H$-invariant vector, unique up to scalar. Concretely, $C_d$ is spanned by $B(x,y)^d$, where $B$ is the $H$-invariant bilinear form on $V$. We define
\begin{equation}\label{Bddef}
    \mathfrak{B}_d:=m(C_d).
\end{equation}

\noindent Since $\mathcal{F}$ preserves weights and negates degrees, it follows that $\mathfrak{B}_d\in D_C(0)^0$. Hence, by Proposition~\ref{degwt0}, the operator $\mathfrak{B}_d$ is a polynomial in the Euler operator $E$.

\begin{prop}\label{Shapovalovdet}
    \textbf{The Quadric Shapovalov Determinant}. We have
\begin{equation}\label{ShapForm}
    \mathfrak{B}_d
    =
    \prod_{j=1}^{d} (E-j+1)\;\cdot\;\prod_{j=1}^{d} (E+k-j-1).
\end{equation}
\end{prop}

\begin{proof}
We know from Proposition \ref{degwt0} and its proof that $\mathfrak{B}_d$ is a polynomial in $E$, and acts on the degree $r$ space $P^r$ by a scalar $c_r$; moreover, $\mathfrak{B}_d=p(E)$ if and only if $p(r)=c_r$ for all $r\in \N$. Thus it suffices to determine the scalar $c_r$, and for this it is enough to evaluate $\mathfrak{B}_d$ on a single nonzero vector in $P^r$, say $x_1^r$.

Expanding $B((x,y),(u,v))^d$ gives
\begin{equation}
    B((x,y),(u,v))^d
    =
    \sum_{\substack{\alpha_1,\ldots,\alpha_k\ge 0\\ \beta_1,\ldots,\beta_k\ge 0\\
    \alpha_1+\cdots+\alpha_k+\beta_1+\cdots+\beta_k=d}}
    \frac{d!}{\alpha_1!\cdots \alpha_k!\,\beta_1!\cdots \beta_k!}
    \prod_{i=1}^k (x_i\,v_{k+1-i})^{\alpha_i}(y_i\,u_{k+1-i})^{\beta_i}.
\end{equation}

\noindent Therefore
\begin{equation}\label{Shapexp}
    \mathfrak{B}_d
    =
    \sum_{\substack{\alpha_1,\ldots,\alpha_k\ge 0\\ \beta_1,\ldots,\beta_k\ge 0\\
    \alpha_1+\cdots+\alpha_k+\beta_1+\cdots+\beta_k=d}}
    \frac{d!}{\alpha_1!\cdots \alpha_k!\,\beta_1!\cdots \beta_k!}
    \prod_{i=1}^k x_i^{\alpha_i}y_i^{\beta_i}\,
    \mathfrak{Y}_{k+1-i}^{\alpha_i}\mathfrak{X}_{k+1-i}^{\beta_i}.
\end{equation}

\noindent By Proposition \ref{QuadFourierProp}, we have
\begin{equation}
    \mathfrak{X}_i=\mathcal{F}(x_i)=(E+k-1)\frac{\partial}{\partial y_{k+1-i}}-x_i\Delta,
    \qquad
    \mathfrak{Y}_i=\mathcal{F}(y_i)=(E+k-1)\frac{\partial}{\partial x_{k+1-i}}-y_i\Delta.
\end{equation}

Now $\mathfrak{X}_i(x_1^r)=0$ for every $i$, since $\frac{\partial}{\partial y_{k+1-i}}(x_1^r)=0$ and $\Delta(x_1^r)=0$. Likewise, $\mathfrak{Y}_i(x_1^r)=0$ unless $i=k$, since only $\frac{\partial}{\partial x_1}$ acts nontrivially on $x_1^r$. It follows that the only summand of \eqref{Shapexp} which does not annihilate $x_1^r$ is the one with $\alpha_1=d$ and all other $\alpha_i,\beta_i$ equal to $0$. Hence
\begin{equation}
    \mathfrak{B}_d(x_1^r)=x_1^d\,\mathfrak{Y}_k^d(x_1^r).
\end{equation}

\noindent It therefore remains to compute $\mathfrak{Y}_k^d(x_1^r)$. We have
\begin{equation}
    \mathfrak{Y}_k(x_1^r)
    =
    \left((E+k-1)\frac{\partial}{\partial x_1}-y_k\Delta\right)(x_1^r)
    =
    (E+k-1)(r x_1^{r-1}),
\end{equation}

\noindent since $\Delta(x_1^r)=0$. Thus
\begin{equation}
    \mathfrak{Y}_k(x_1^r)
    =
    r\bigl((r-1)+(k-1)\bigr)x_1^{r-1}
    =
    r(r+k-2)x_1^{r-1}.
\end{equation}

\noindent Iterating gives $\mathfrak{Y}_k^d(x_1^r)
    =
    \prod_{j=1}^d (r-j+1)\prod_{j=1}^d (r+k-j-1)\,x_1^{r-d}$ and hence
\begin{equation}
    \mathfrak{B}_d(x_1^r)
    =
    \prod_{j=1}^d (r-j+1)\prod_{j=1}^d (r+k-j-1)\,x_1^r.
\end{equation}

\noindent Replacing $r$ by $E$, we obtain \eqref{ShapForm}.
\end{proof}

\begin{rem}
    Our proof is quite a bit more direct than the analogous argument in \cite{BezrukavnikovBravermanPositselskii2002Gluing}, Proposition 3.17. This is because the degree grading here is substantially simpler.
\end{rem}

\subsection{Kazhdan-Laumon Gluing}\label{GluingProofSect}

We are now in a position to prove that the glued category of coherent $\mathcal{D}$-modules on $C^o$ is equivalent to the category of finitely generated $D_C$-modules. We begin by reviewing the Kazhdan-Laumon gluing construction \cite{KazhdanLaumon1988Gluing, Polishchuk2001GluingBasicAffine}. We write $D_C$-\textbf{mod} for the category of finitely generated $D_C$-modules, and from this point onward we assume that all $\mathcal{D}$-modules, and all categories of $\mathcal{D}$-modules, are coherent.

Observe that for any left $D_C$-module $M$, we may define a new left $D_C$-module $\mathcal{F}(M)$ by keeping the same underlying $\kappa$-module and twisting the action by the algebra involution $\mathcal{F}:D_C\to D_C$; explicitly,
\begin{equation}\label{quadricFourdef}
    \xi\cdot_{\mathcal{F}(M)} m:=\mathcal{F}(\xi)\cdot_M m,
    \qquad
    \xi\in D_C,\ m\in M.
\end{equation}

\noindent Equivalently, one may write
\begin{equation}
    \mathcal{F}(M)\cong {}_{\mathcal{F}}D_C\otimes_{D_C} M,
\end{equation}

\noindent where ${}_{\mathcal{F}}D_C$ denotes the $(D_C,D_C)$-bimodule whose underlying right $D_C$-module is the regular one, and whose left action is twisted by $\mathcal{F}$.

Since $\mathcal{F}^2=\Id_{D_C}$, this defines an involutive autoequivalence of $D_C\textbf{-mod}$. More precisely, there is a natural isomorphism
\begin{equation}
    \nu:\mathcal{F}^2\xrightarrow{\sim}\Id,
\end{equation}

\noindent whose component at $M$ is simply the identity on the underlying $\kappa$-module of $M$.

We see that $\mathcal{F}$ is the analogue of the Laplace transform for $\mathcal{D}$-modules on affine space \cite{Sabbah2010FourierDModules}.

For a left $D_C$-module $N$, let $\widetilde{N}$ denote the corresponding coherent left $\mathcal{D}_C$-module on the affine cone $C=\Spec(\kappa[C])$; thus $\widetilde{N}\cong \mathcal{D}_C\otimes_{D_C}N$. Since $C$ is affine, this is the usual equivalence between finitely generated left $D_C$-modules and coherent left $\mathcal{D}_C$-modules on $C$. We use Roman letters such as $D_C$ and $N$ for algebras and modules, and calligraphic letters such as $\mathcal{D}_C$ and $\mathcal{N}$ for sheaves.

Let $j:C^o\hookrightarrow C$ be the open immersion. Following the formalism of global sections and localization in \cite{BezrukavnikovBravermanPositselskii2002Gluing}, define functors
\begin{align}
    \Ind:\mathcal{D}_{C^o}\textbf{-mod}\to D_C\textbf{-mod},&
    \qquad
    \Ind(\mathcal{M}):=\Gamma\bigl(C,j_*\mathcal{M}\bigr),\label{IndDef}\\
    \Res:D_C\textbf{-mod}\to \mathcal{D}_{C^o}\textbf{-mod},&
    \qquad
    \Res(N):=j^*(\widetilde N).\label{ResDef}
\end{align}

\noindent Here $\Ind$ is the quadric analogue of the global-sections functor $\Gamma$, while $\Res$ is the corresponding analogue of the localization functor $L$.

Set
\begin{equation}\label{Tdef}
    T:=\Res\circ \mathcal{F}\circ \Ind:
    \mathcal{D}_{C^o}\textbf{-mod}\to \mathcal{D}_{C^o}\textbf{-mod}.
\end{equation}

\noindent Recall that $j^*\dashv j_*$, with unit $\eta:\Id \to j_*j^*$ and counit $\varepsilon:j^*j_* \to \Id$; for an open immersion, $\varepsilon$ is an isomorphism. For each $N\in D_C\textbf{-mod}$, the unit yields a canonical $D_C$-linear map
\begin{equation}\label{eq:unit}
    \eta_N:\;N \;\longrightarrow\; \Gamma\bigl(C,\,j_*j^*(\widetilde N)\bigr)=\Ind(\Res(N)),
\end{equation}

\noindent obtained by sheafifying $N$, applying $\eta_{\widetilde N}:\widetilde N\to j_*j^*\widetilde N$, and taking global sections.

For each $\mathcal{M}\in \mathcal{D}_{C^o}\textbf{-mod}$, there is also a canonical morphism
\begin{equation}\label{eq:counit-on-open}
    \epsilon_{\mathcal{M}}:\;\Res(\Ind(\mathcal{M}))
    =
    j^*\widetilde{\Gamma(C,j_*\mathcal{M})}
    \longrightarrow
    j^*j_*\mathcal{M}
    \xrightarrow{\varepsilon_{\mathcal{M}}}
    \mathcal{M},
\end{equation}

\noindent where the first arrow is the natural map from the localization of global sections to $j_*\mathcal{M}$, and the second is the counit of the adjunction.

Using \eqref{eq:unit}, \eqref{eq:counit-on-open}, and the involutivity isomorphism
$\nu:\mathcal{F}^2\xrightarrow{\sim}\Id$, we obtain a natural transformation
\begin{equation}\label{eq:T2}
    \widetilde{\nu}:\;T^2 \longrightarrow \Id_{\mathcal{D}_{C^o}\textbf{-mod}}
\end{equation}

\noindent given on $\mathcal{M}$ by the composition
\begin{align}
T^2(\mathcal{M})
&=\Res\circ\mathcal{F}\circ\Ind\circ\Res\circ\mathcal{F}\circ\Ind(\mathcal{M}) \\
&\xrightarrow{\ \Res\circ\mathcal{F}\bigl(\eta_{\mathcal{F}\circ\Ind(\mathcal{M})}\bigr)\ }
\Res\circ\mathcal{F}^2\circ\Ind(\mathcal{M}) \nonumber\\
&\xrightarrow{\ \Res\bigl(\nu_{\Ind(\mathcal{M})}\bigr)\ }
\Res\circ\Ind(\mathcal{M}) \nonumber\\
&\xrightarrow{\ \epsilon_{\mathcal{M}}\ }
\mathcal{M}. \nonumber
\end{align}

\begin{definition}
The glued category of quadric $\mathcal{D}$-modules, denoted $\mathscr{C}$, is defined as follows.

Its objects are quadruples $(\mathcal{M}_1,\mathcal{M}_2,\phi,\psi)$, where
$\mathcal{M}_1,\mathcal{M}_2\in \mathcal{D}_{C^o}\textbf{-mod}$ and
\begin{equation}
    \phi:\mathcal{M}_1\to T(\mathcal{M}_2),
    \qquad
    \psi:\mathcal{M}_2\to T(\mathcal{M}_1)
\end{equation}

\noindent are morphisms in $\mathcal{D}_{C^o}\textbf{-mod}$ such that the compositions
\begin{align}
    &\mathcal{M}_1 \xrightarrow{\ \phi\ } T(\mathcal{M}_2) \xrightarrow{\ T(\psi)\ }
    T^2(\mathcal{M}_1) \xrightarrow{\ \widetilde{\nu}_{\mathcal{M}_1}\ } \mathcal{M}_1,\label{GlueCond1}\\
    &\mathcal{M}_2 \xrightarrow{\ \psi\ } T(\mathcal{M}_1) \xrightarrow{\ T(\phi)\ }
    T^2(\mathcal{M}_2) \xrightarrow{\ \widetilde{\nu}_{\mathcal{M}_2}\ } \mathcal{M}_2\label{GlueCond2}
\end{align}

\noindent are the identity maps.

A morphism
\begin{equation}
    (f_1,f_2):(\mathcal{M}_1,\mathcal{M}_2,\phi,\psi)\to (\mathcal{N}_1,\mathcal{N}_2,\phi',\psi')
\end{equation}

\noindent consists of morphisms $f_i:\mathcal{M}_i\to \mathcal{N}_i$ in $\mathcal{D}_{C^o}\textbf{-mod}$ such that the diagrams
\begin{equation}
\begin{tikzcd}
\mathcal{M}_1 \arrow[r,"\phi"] \arrow[d,"f_1"'] & T(\mathcal{M}_2) \arrow[d,"T(f_2)"] \\
\mathcal{N}_1 \arrow[r,"\phi'"'] & T(\mathcal{N}_2)
\end{tikzcd}
\qquad
\begin{tikzcd}
\mathcal{M}_2 \arrow[r,"\psi"] \arrow[d,"f_2"'] & T(\mathcal{M}_1) \arrow[d,"T(f_1)"] \\
\mathcal{N}_2 \arrow[r,"\psi'"'] & T(\mathcal{N}_1)
\end{tikzcd}
\end{equation}

\noindent commute.
\end{definition}

\noindent There are evident projection functors
\begin{equation}
    \pi_1,\pi_2:\mathscr{C}\to \mathcal{D}_{C^o}\textbf{-mod},
\end{equation}

\noindent given on objects by
\begin{equation}
    \pi_i(\mathcal{M}_1,\mathcal{M}_2,\phi,\psi)=\mathcal{M}_i.
\end{equation}

\begin{rem}
There are other ways to package the same structure, for example in terms of a comonad or as a $W$-gluing for the group $W=\mathbb{Z}/2\mathbb{Z}$; see \cite{Polishchuk2001GluingBasicAffine, BezrukavnikovBravermanPositselskii2002Gluing}. In the present setting, however, we may afford to be completely explicit, since we are gluing only two copies of $\mathcal{D}_{C^o}\textbf{-mod}$. The original construction of this kind is due to Kazhdan and Laumon \cite{KazhdanLaumon1988Gluing}.
\end{rem}

\begin{thm}\label{thm:glued-equiv}
    The functor
\begin{equation}
    \Theta:D_C\textbf{-mod}\longrightarrow \mathscr{C},
    \qquad
    M\longmapsto \bigl(\Res(M),\,\Res(\mathcal{F}(M)),\,\phi_M,\,\psi_M\bigr),
\end{equation}

\noindent is an equivalence of categories.
\end{thm}

\begin{proof}
We follow \cite{BezrukavnikovBravermanPositselskii2002Gluing}. Since $C$ is affine, sheafification and global sections identify coherent left $\mathcal{D}_C$-modules with finitely generated left $D_C$-modules. The only non-formal input needed for the gluing argument is the following generation statement.

\begin{prop}\label{Generation1Prop}
    For every $d>0$, the left ideal in $D_C$ generated by
\begin{equation}
    P^d+\mathcal{F}(P^d)
\end{equation}

\noindent contains $1$.
\end{prop}

\textit{Proof of Proposition \ref{Generation1Prop}.}
By construction, $\mathfrak{B}_d$ lies in the left ideal generated by $P^d+\mathcal{F}(P^d)$; the same is therefore true of $\mathcal{F}(\mathfrak{B}_d)$. By Proposition \ref{Shapovalovdet}, we have
\begin{equation}\label{ShapFormulaExplicit}
    \mathfrak{B}_d
    =
    \prod_{j=1}^{d} (E-j+1)\prod_{j=1}^{d} (E+k-j-1).
\end{equation}

\noindent Applying $\mathcal{F}$ and using Proposition \ref{QuadFourierProp}, in particular the identity
\begin{equation}
    \mathcal{F}(E)=-E-2k+2,
\end{equation}

\noindent we obtain
\begin{equation}\label{FourierShapFormulaExplicit}
    \mathcal{F}(\mathfrak{B}_d)
    =
    \prod_{j=1}^{d} (E+2k+j-3)\prod_{j=1}^{d} (E+k+j-1).
\end{equation}

\noindent Thus the roots of $\mathfrak{B}_d$ are
\begin{equation}
    \{0,1,\ldots,d-1\}\cup \{2-k,3-k,\ldots,d-k+1\},
\end{equation}

\noindent while the roots of $\mathcal{F}(\mathfrak{B}_d)$ are
\begin{equation}
    \{2-2k,1-2k,\ldots,3-2k-d\}\cup \{-k,-k-1,\ldots,1-k-d\}.
\end{equation}

\noindent For $k\ge 2$ these two sets are disjoint: the largest root of $\mathcal{F}(\mathfrak{B}_d)$ is $-k$, while the smallest root of $\mathfrak{B}_d$ is $2-k$. Hence $\mathfrak{B}_d$ and $\mathcal{F}(\mathfrak{B}_d)$ are coprime in $\kappa[E]$. Therefore they generate the unit ideal in $\kappa[E]$, and hence also the unit ideal in $D_C$. This proves Proposition \ref{Generation1Prop}. \qed

\medskip

Given an object $(\mathcal{M}_1,\mathcal{M}_2,\phi,\psi)\in \mathscr{C}$, the standard argument of \cite{BezrukavnikovBravermanPositselskii2002Gluing}, together with Proposition \ref{Generation1Prop}, shows that the modules
\begin{equation}
    \Ind(\mathcal{M}_1)=\Gamma(C,j_*\mathcal{M}_1),
    \qquad
    \Ind(\mathcal{M}_2)=\Gamma(C,j_*\mathcal{M}_2)
\end{equation}

\noindent have no $I_0$-torsion, where $I_0\subset \mathcal{O}(C)$ is the maximal ideal of the origin. Equivalently, they are recovered from their restrictions to $C^o$. The remainder of the proof is then identical to the argument in \cite{BezrukavnikovBravermanPositselskii2002Gluing}: the functor $\Theta$ is fully faithful, and every glued object is obtained from a unique finitely generated $D_C$-module. Thus $\Theta$ is an equivalence.
\end{proof}

\begin{rem}\label{DeltaMassRem}
Let $\mathfrak m_0\subset \kappa[C]$ denote the maximal ideal of the vertex $0\in C$. The algebraic model for the Dirac mass at the vertex is the left $D_C$-module
\begin{equation}
    \delta_0:=D_C/D_C\mathfrak m_0,
\end{equation}

\noindent which has support only at the closed point $\{0\}$. In the present setting, this module is realized by the quadric Fourier transform of the structure module $\kappa[C]$. Indeed, in the twisted module $\mathcal F(\kappa[C])$, the vector $1\in \kappa[C]$ satisfies
\begin{equation}
    x_i\cdot 1=\mathcal F(x_i)(1)=\mathfrak X_i(1)=0,
    \qquad
    y_i\cdot 1=\mathcal F(y_i)(1)=\mathfrak Y_i(1)=0,
\end{equation}

\noindent so the maximal ideal $\mathfrak m_0$ annihilates $1$. On the other hand, in $\mathcal F(\kappa[C])$ the operators $\mathfrak X_i$ and $\mathfrak Y_i$ act as multiplication by $x_i$ and $y_i$, so $1$ generates the whole module. Thus $\delta_0$ is naturally realized by $\mathcal F(\kappa[C])$.

Applying $\mathcal F$ once more gives
\begin{equation}
    \mathcal F(\delta_0)\cong \kappa[C].
\end{equation}

\noindent Hence a module supported at the single point $0$ is transformed into a module supported on all of $C$. In particular, $\delta_0$ is invisible on the open cone $C^o$, whereas its Fourier transform restricts to the structure sheaf on $C^o$. This is precisely why the gluing theorem is plausible in the present context: a single chart misses the apex Dirac mass, but the Fourier-transformed chart detects it completely. In this sense, the quadric Fourier transform exhibits a Heisenberg-type uncertainty principle: highly concentrated support is exchanged with maximal spread.

We finally note that on the Braverman-Kazhdan glued category, the quadric Fourier transform is quite evident: $\mathcal{F}: (\mathcal{M}_1,\mathcal{M}_2,\phi,\psi)\mapsto (\mathcal{M}_2,\mathcal{M}_1,\psi,\phi)$. The Fourier transform on $D_C$\textbf{-mod} is substantially more mysterious.
\end{rem}

\section{\texorpdfstring{$\mathcal{D}$-Modules on the Conformal Compactification and the Harmonic Sheaf}{D-Modules on the Conformal Compactification and the Harmonic Sheaf}}\label{DModHarmSect}

\subsection{\texorpdfstring{Another Categorical Equivalence}{Another Categorical Equivalence}}\label{IntroHarmSheavesSection}

In the previous section, we described the category $D_C$-\textbf{mod} in terms of Kazhdan--Laumon gluing data over $C^o$, with gluing governed by the quadric Fourier transform. There is also a third way of understanding this category, which provides a rather direct $\mathcal{D}$-module categorification of the notion of harmonic functions. 

As always, let $G = \O(Q^+)$, and let $C \subset V^*$ denote the isotropic locus of $Q^*$. Proposition \ref{Envelopmap} says that $D_C$ is a homomorphic image of $\mathcal{U}(\mathfrak{g})$.

On the other hand, by the parabolic Beilinson--Bernstein localization theorem \cite{BeilinsonBernstein1981Localisation, BeilinsonBernstein1993Jantzen, HollandPolo1996KTheoryTDO}, one expects, for twists in the appropriate Beilinson--Bernstein range, a close relationship between twisted $\mathcal{D}$-modules on $G/P$ and modules over global twisted differential operators. Since $D_C$-modules are naturally $\mathcal{U}(\mathfrak{g})$-modules, namely those annihilated by the Joseph ideal \cite{LevasseurSmithStafford1989Joseph}, it is natural to ask whether they can be realized as a category of twisted $\mathcal{D}$-modules on $G/P$.

However, the natural line bundle used to construct the map $\mathcal{U}(\g) \to D_C$ is $\mathcal{L} \cong \mathcal{O}(1-k)$; see \eqref{twistbundledef}. This twist is not in the Beilinson--Bernstein range. Indeed, for $k \geq 2$, the line bundle $\mathcal{L}$ has no nonzero global sections, even though $\mathcal{L} \neq 0$. Thus the relationship we seek cannot be obtained by a direct appeal to Beilinson--Bernstein localization. Still, one may naturally ask whether $D_C$-modules correspond to some natural subcategory of $\mathcal{D}_\mathcal{L}$-modules on $G/P$.

As we will find, this is indeed the case. We note that this gives an example of a category of $\mathcal{D}$-modules on a singular affine variety agreeing with a category of $\mathcal{D}$-modules on a smooth projective homogeneous space; moreover, this agreement is not a simple consequence of Beilinson--Bernstein localization.

\begin{thm}\label{SecondEquiv}
Let $D_C\textbf{-mod}^{\,\mathrm{fg}}$ denote the category of finitely generated left $D_C$-modules. Let $\mathcal{L}\cong \mathcal{O}(1-k)$, and let $\mathcal{D}_{\mathcal L}$ be the sheaf of $\mathcal{L}$-twisted differential operators on $G/P$. There exists a unique $G$-equivariant subsheaf $\mathcal{J}_\Delta\subset \mathcal{D}_{\mathcal L}$ of left ideals such that, over the big Bruhat cell $V\cong U^{\op}P/P$, one has
\begin{equation}
\mathcal{J}_\Delta(V)\cong D_V\cdot \Delta,
\end{equation}
\noindent where $\Delta$ is the $Q$-Laplacian on $V$; as usual, we identify $D_V$ and $\mathcal{D}_{\mathcal L}(V)$ via the section $\sigma_0$ of $\mathcal{L}$; cf.\ Definition~\ref{sigm0def}. Define the \textbf{harmonic sheaf} by
\begin{equation}
\mathcal{H}:=\mathcal{D}_{\mathcal L}/\mathcal{J}_\Delta.
\end{equation}
Then $\mathcal{H}$ admits a natural structure of a $\mathcal{D}_{\mathcal L}$-$D_C$ bimodule. On the big Bruhat cell, the right $D_C$-action is induced by the normalizer quotient
\begin{equation}\label{widehattaucompatiso}
\widehat{\tau}:N(D_V\Delta)/D_V\Delta \xrightarrow{\sim} D_C.
\end{equation}

For any coherent $\mathcal{D}_{\mathcal L}$-module $\mathcal{M}$, define the subsheaf of \textbf{harmonic vectors} by
\begin{equation}
\mathcal{M}^{\Delta}(U)
=
\mathrm{Ann}_{\mathcal{J}_\Delta(U)}\bigl(\mathcal{M}(U)\bigr)
=
\{\, m\in \mathcal{M}(U) : xm=0 \text{ for all } x\in \mathcal{J}_\Delta(U) \,\},
\end{equation}
\noindent and set
\begin{equation}
\Gamma_\Delta(\mathcal{M})
:=
\Gamma(G/P,\mathcal{M}^{\Delta}).
\end{equation}
\noindent This is a left $D_C$-module via \eqref{widehattaucompatiso}.

Define $\mathscr{H}$, the category of coherent \textbf{harmonic} $\mathcal{D}_{\mathcal L}$-modules, to be the full subcategory of coherent $\mathcal{D}_{\mathcal L}$-modules admitting a finite $(\mathcal{H},D_C)$-presentation; that is, an exact sequence
\begin{equation}
\mathcal{H}^a \longrightarrow \mathcal{H}^b \longrightarrow \mathcal{M} \longrightarrow 0
\end{equation}
\noindent in which the first map is induced, via the right $D_C$-action on $\mathcal{H}$, by a homomorphism $D_C^a\to D_C^b$ of finite free left $D_C$-modules.

Then $\mathscr{H}$ is an abelian category. Moreover, for every $\mathcal{M}\in\mathscr{H}$ one has
\begin{equation}
\Gamma(G/P,\mathcal{M})=\Gamma_\Delta(\mathcal{M}),
\end{equation}
\noindent and the functors
\begin{equation}
D_C\textbf{-mod}^{\,\mathrm{fg}}
\longrightarrow
\mathscr{H},
\qquad
N\longmapsto \mathcal{H}\otimes_{D_C}N,
\end{equation}
\noindent and
\begin{equation}
\mathscr{H}
\longrightarrow
D_C\textbf{-mod}^{\,\mathrm{fg}},
\qquad
\mathcal{M}\longmapsto \Gamma(G/P,\mathcal{M})=\Gamma_\Delta(\mathcal{M}),
\end{equation}
\noindent are quasi-inverse equivalences. We call the first of these functors the \textbf{harmonic transform}. Consequently, together with the Kazhdan--Laumon gluing equivalence of Theorem~\ref{thm:glued-equiv}, the categories $D_C\textbf{-mod}^{\,\mathrm{fg}}$, $\mathscr{C}$, and $\mathscr{H}$ are equivalent.
\end{thm}

The overarching idea of the proof of Theorem ~\ref{SecondEquiv} is to unwind the $F$-method. Our first order of business is to carefully construct the ideal $\mathcal{J}_{\Delta}$ and $\mathcal{D}_{\mathcal{L}}$-module $\mathcal{H}$, which will serve as the kernel object for the harmonic transform and as the basic object controlling the category $\mathscr{H}$. The next few subsections examine the basic important properties of $\mathcal{H}$.

\subsection{The Harmonic Ideal and the Harmonic Sheaf; Statement of the Global Sections Theorem}\label{IntroHarmsheaf}

For each $g \in G(\kappa)$, we let $V_g := gV = gU^{\op}P/P$ denote the left $g$-translate of the big Bruhat cell in $G/P$. These define an open cover of $G/P$. For each such $g$, we define a trivializing section of $\mathcal{L}\cong\mathcal{O}(1-k)$ by
\begin{equation}\label{sigmagdef}
    \sigma_g : x \mapsto \left[g \, u_{g^{-1}x}^{\op}, \;1\right] \in G \times^{P,\; -\eta}\A^1 \cong \mathcal{L}(V_g).
\end{equation}
\noindent Observe that $\sigma_e = \sigma_0$, as in Definition~\ref{sigm0def}. This permits us to identify $\kappa[V]$ with $\Gamma(V_g, \mathcal{L})$ via $\{v \mapsto f(v): v \in V\} \mapsto \{x \mapsto f(g^{-1}x)\sigma_g(x) : x \in V_g\}$, and hence to identify $D_V$ with $\mathcal{D}_{\mathcal{L}}(V_g)$ via $\xi(f\sigma_g) = \xi(f\circ g^{-1})\sigma_g$ for $f \in \kappa[V]$ and $\xi \in D_V$.

We let $\Delta_g$ denote the Laplacian operator on $V_g$ in the coordinates $V \to V_g$, $v \mapsto gu_v^{\op}P/P$; explicitly, for $f \in \mathcal{O}(V_g)$, we have $f \circ g^{-1} \in \kappa[V]$, and so we may define
\begin{equation}\label{Deltagdef}
    \Delta_g (f \cdot \sigma_g) := \Delta (f\circ g^{-1})\cdot \sigma_g.
\end{equation}
\noindent With this identification, we define the local left ideal
\begin{equation}
    \mathcal{J}_\Delta(V_g) := \mathcal{D}_{\mathcal L}(V_g) \cdot \Delta_g.
\end{equation}

\begin{definition}
     We shall call $\mathcal{J}_\Delta$ the \textbf{harmonic ideal}.  
\end{definition}

Of course we must verify the following:

\begin{prop}\label{HarmonicIdealProp}
   $\mathcal{J}_\Delta$ defines a sheaf of left $\mathcal{D}_\mathcal{L}$-ideals on $G/P$.
\end{prop}

\begin{proof}
Let $g,h\in G(\kappa)$ and consider the overlap $V_g\cap V_h$. It suffices to show that the left ideals
$\mathcal{D}_{\mathcal L}(V_g\cap V_h)\Delta_g$ and $\mathcal{D}_{\mathcal L}(V_g\cap V_h)\Delta_h$ coincide, equivalently that
$\Delta_g$ and $\Delta_h$ differ by multiplication by a unit in $\mathcal{O}(V_g\cap V_h)$.

Fix $x\in V_g\cap V_h$. Since $x\in V_g$, we may write $x=g u_v^{\op}P/P$ for a unique $v\in V$.
Since also $x\in V_h$, we may write $x=h u_w^{\op}P/P$ for a unique $w\in V$. Equivalently,
\begin{equation}
h^{-1}g\,u_v^{\op} \in U^{\op}P,
\end{equation}
\noindent so there is a unique factorization
\begin{equation}\label{eq:factor-overlap}
h^{-1}g\,u_v^{\op} \;=\; u_w^{\op}\,p
\qquad\text{for some }p\in P,
\end{equation}
\noindent rational in $v$ and regular on $V_g\cap V_h$. Applying the character
$\chi_0:P\to\G_m$, we obtain a regular invertible function on the overlap,
\begin{equation}\label{eq:overlap-unit}
c_{h,g}(x)\;:=\;\chi_0(p)\in \mathcal{O}(V_g\cap V_h)^\times.
\end{equation}
We now compare the trivializations $\sigma_g$ and $\sigma_h$ on $V_g\cap V_h$. Using \eqref{eq:factor-overlap}, and using the convention from \eqref{twistbundledef} that
$[ap,z]=[a,\eta(p)^{-1}z]$, we compute in the associated bundle $\mathcal{L}=G\times^{P,\eta^{-1}}\A^1$:
\begin{equation}\label{eq:sigma-compare}
\sigma_g(x)
=
\bigl[g\,u_v^{\op},\,1\bigr]
=
\bigl[h\,u_w^{\op}p,\,1\bigr]
=
\bigl[h\,u_w^{\op},\,\eta(p)^{-1}\bigr]
=
\chi_0(p)^{k-1}\,\sigma_h(x)
=
c_{h,g}(x)^{k-1}\,\sigma_h(x).
\end{equation}
Now let $s$ be a local section of $\mathcal{L}$ on $V_g\cap V_h$, and write
$s=F\sigma_g$, where $F$ is the coefficient of $s$ in the $g$-trivialization. Set
\begin{equation}
\varphi:=h^{-1}g.
\end{equation}
\noindent Then \eqref{eq:factor-overlap} says precisely that on the overlap we have a Bruhat factorization $\varphi\,u_v^{\op}=u_w^{\op}p$,
with $c_{h,g}=\chi_0(p)$. Equivalently, $w=\varphi v$ for the rational action on the big cell, and $p=p(\varphi,v)$. Hence $p(\varphi^{-1},w)=p^{-1}$.

By the twisted action formula \eqref{twistedaction}, the same section $s$ is written in the $h$-trivialization as
\begin{equation}\label{eq:s-in-two-trivs}
s
=
(\varphi F)\sigma_h.
\end{equation}
Indeed, at the point $w$, the coefficient $(\varphi F)(w)$ is $(\varphi F)(w)
= \eta(p(\varphi^{-1},w))F(\varphi^{-1}w) = \eta(p^{-1})F(v) = \eta(p)^{-1}F(v)$, which agrees with \eqref{eq:sigma-compare}.

Applying $\Delta_h$ and using the definition of $\Delta_h$, we obtain
\begin{equation}\label{eq:Deltah-first}
\Delta_h(s)
=
\Delta(\varphi F)\sigma_h.
\end{equation}
\noindent Now Proposition~\ref{prop:group-action-on-Delta} gives the semi-invariance of $\Delta$ under the twisted $G$-action. Since $p(\varphi^{-1},w)=p^{-1}$, we obtain
\begin{equation}\label{eq:Delta-intertwining-overlap}
\Delta(\varphi F)
=
c_{h,g}^{2}\,\varphi(\Delta F).
\end{equation}
\noindent Substituting \eqref{eq:Delta-intertwining-overlap} into \eqref{eq:Deltah-first} yields
\begin{equation}\label{eq:Deltah-second}
\Delta_h(s)
=
c_{h,g}^{2}\,(\varphi(\Delta F))\sigma_h.
\end{equation}
\noindent But applying \eqref{eq:s-in-two-trivs} with $F$ replaced by $\Delta F$, we get $(\varphi(\Delta F))\sigma_h = (\Delta F)\sigma_g$. Therefore, $\Delta_h(s) = c_{h,g}^{2}\,(\Delta F)\sigma_g = c_{h,g}^{2}\,\Delta_g(s)$. Since $c_{h,g}\in\mathcal{O}(V_g\cap V_h)^\times$ is a unit, this shows that $\Delta_h$ and $\Delta_g$ differ by multiplication by a unit on $V_g\cap V_h$. Consequently, $\mathcal{D}_{\mathcal L}(V_g\cap V_h)\Delta_g = \mathcal{D}_{\mathcal L}(V_g\cap V_h)\Delta_h$, and therefore the ideals $\mathcal{D}_{\mathcal L}(V_g)\cdot \Delta_g$ glue on overlaps to define a sheaf of left ideals $\mathcal{J}_\Delta$ in $\mathcal{D}_{\mathcal{L}}$.
\end{proof}

\begin{definition}
    Let us define the left $\mathcal{D}_{\mathcal{L}}$-module
    \begin{equation}
        \mathcal{H} := \mathcal{D}_{\mathcal{L}}/\mathcal{J}_{\Delta}.
    \end{equation}
    \noindent We shall call $\mathcal{H}$ the \textbf{harmonic sheaf} on $G/P$.
\end{definition}

\begin{rem}
    Observe that, after identifying $\mathcal{H}|_{V}$ with $D_V/D_V\Delta$, the solution functor $\operatorname{Hom}_{\mathcal{D}_{V}}(\mathcal{H}|_{V}, \mathcal{S})$ applied to a test $\mathcal{D}_{V}$-module $\mathcal{S}$ yields precisely the harmonic vectors in $\mathcal{S}$, i.e.\ the solutions to $\Delta(s)=0$. The derived functor $\operatorname{RHom}$ may likewise be viewed as the derived space of harmonic vectors.
\end{rem}

The following theorem shows that the global sections of $\mathcal{H}$ identifies canonically with $D_C$.

\begin{thm}\label{GlobalSectionsTheorem}
Restriction to the standard Bruhat cell $V\cong U^{\op}P/P$ identifies the global sections of $\mathcal{H}$ with those elements of $D_V / D_V \Delta$ left-annihilated by $\Delta$. In particular,
\begin{equation}
\Gamma(\mathcal{H})
=
\Gamma_\Delta(\mathcal{H})
\cong
N(D_V\Delta)/D_V\Delta
\cong
D_C,
\end{equation}
\noindent where $N(D_V\Delta)$ denotes the normalizer of the left ideal $D_V\Delta$; that is, the set of $\xi \in D_V$ such that
\begin{equation}\label{highersymmetrieseq}
\Delta \xi= \delta_\xi \Delta
\end{equation}
\noindent for some $\delta_\xi \in D_V$. The isomorphism with $D_C$ is induced by applying $\widehat{\tau}$; see \eqref{defwidehattau}.
\end{thm}

We note that operators $\xi$ satisfying \eqref{highersymmetrieseq} are known in the literature as \textbf{symmetries} of the Laplacian, and, in particular, as \textbf{higher symmetries} when $\ord(\xi)>1$ \cite{Eastwood2005HigherSymmetries, Michel2014HigherSymmetries}. The relationship between such symmetry algebras and rings of differential operators on quadratic cones, including powers of the Laplacian, was further developed by Levasseur and Stafford \cite{LevasseurStafford2017HigherSymmetries}.

\vspace{2mm}

We prove this theorem in two parts. We first show that normalizers extend globally, and then prove that a global section of the harmonic sheaf must be a symmetry of the Laplacian. Both steps require several preparatory lemmata.

\subsection{\texorpdfstring{Symmetry of $\Delta \implies$ Global section of $\mathcal{H}$}{Symmetry of Delta Implies Global}}\label{SymmetryImpliesGlobalSect}

The first lemma is a key principal-symbol statement used in the extension argument.

\begin{lem}\label{SymbolLemma}
Let $V$ be a split quadratic space of even dimension $n=2k\geq 4$, with quadratic form $Q\in \kappa[V]$, and let $Q^*\in \kappa[V^*]$ denote the dual quadratic form. Put
\begin{equation}
S:=\kappa[T^*V]\cong \kappa[V\times V^*],
\qquad
S[1/Q]:=\kappa[T^*V][1/Q].
\end{equation}
\noindent We regard $S+S[1/Q]Q^*$ as a subspace of $S[1/Q]$. Let $\{-,-\}$ denote the Poisson bracket on $T^*V$ induced by the usual symplectic structure. If $\nu\in S[1/Q]$ satisfies
\begin{equation}
\{Q^*,\nu\}\in S+S[1/Q]Q^*,
\end{equation}
\noindent then in fact
\begin{equation}
\nu\in S+S[1/Q]Q^*.
\end{equation}
\end{lem}

\begin{proof}
Suppose not. Passing to the quotient $S[1/Q]/(S+S[1/Q]Q^*)$, choose a representative of the class of $\nu$ of the form $\nu=Q^{-N}a$ with $N\geq 1$ minimal and $a\in S$. By minimality of $N$, we may and do assume that $a\notin (Q,Q^*)$: indeed, if $a=Qb+Q^*c$, then $Q^{-N}a=Q^{-(N-1)}b+Q^*(Q^{-N}c)$, so the class of $\nu$ modulo $S+S[1/Q]Q^*$ could be represented with strictly smaller pole order.

By the Leibniz rule for the Poisson bracket, and since applying $\{Q^*,-\}$ to $Q^NQ^{-N}=1$ gives
$\{Q^*,Q^{-N}\}=-NQ^{-N-1}\{Q^*,Q\}$, we have
\begin{align}
\{Q^*,\nu\}
&=
\{Q^*,Q^{-N}a\}\nonumber\\
&=
\{Q^*,Q^{-N}\}a+Q^{-N}\{Q^*,a\}\nonumber\\
&=
Q^{-N-1}\Bigl(Q\{Q^*,a\}-N\{Q^*,Q\}a\Bigr).
\end{align}
By hypothesis, $\{Q^*,\nu\}\in S+S[1/Q]Q^*$. Multiplying through by $Q^{N+1}$, we obtain $B:=Q\{Q^*,a\}-N\{Q^*,Q\}a \in Q^{N+1}S+S[1/Q]Q^*$. Since $B\in S$, we now intersect the right-hand side with $S$. 

We claim that $S\cap S[1/Q]Q^*=(Q^*)$. Indeed, if $f\in S\cap S[1/Q]Q^*$, then $Q^r f\in (Q^*)$ for some $r\geq 0$. Reducing modulo $(Q^*)$, we get  $Q^r\bar f=0$ in $S/(Q^*)$.
But $Q$ is not a zero-divisor in $S/(Q^*)$, since $Q$ and $Q^*$ involve disjoint sets of variables and $S/(Q^*)$ is a domain. Hence $\bar f=0$, so $f\in(Q^*)$.

It follows that $B\in (Q^{N+1},Q^*)\subseteq (Q,Q^*)$. Reducing modulo $(Q,Q^*)$, and using that $N\in \kappa^\times$, this gives
\begin{equation}\label{divisibilityeq}
\{Q^*,Q\}\,a\in (Q,Q^*).
\end{equation}

We now claim that $\{Q^*,Q\}$ is not a zero-divisor in $S/(Q,Q^*)$. Since $n\geq 4$, the quadratic forms $Q$ and $Q^*$ are geometrically irreducible. Moreover, they involve disjoint sets of variables on $V\times V^*$, so the closed subvariety cut out by $(Q,Q^*)$ is $\{Q=0\}\times \{Q^*=0\}\subset V\times V^*$, and is therefore integral. Hence $S/(Q,Q^*)$ is a domain. On the other hand, $\{Q^*,Q\}\notin (Q,Q^*)$: with the natural bihomogeneous grading on $\kappa[V\times V^*]$, the element $Q$ has bidegree $(2,0)$, the element $Q^*$ has bidegree $(0,2)$, and $\{Q^*,Q\}$ has bidegree $(1,1)$, so it cannot lie in the ideal generated by $Q$ and $Q^*$.\footnote{In split coordinates $x_i,y_i$ on $V$ and $\lambda_i,\mu_i$ on $V^*$, one computes $\{Q^*,Q\}$, up to the sign convention for the Poisson bracket, as the phase-space Euler function $\sum_{i=1}^k (x_i\lambda_i+y_i\mu_i)$.}

Therefore multiplication by $\{Q^*,Q\}$ is injective on $S/(Q,Q^*)$. Then \eqref{divisibilityeq} forces $a\in (Q,Q^*)$, contrary to our choice of $a$. This contradiction proves that $\nu\in S+S[1/Q]Q^*$, as claimed.
\end{proof}

\begin{lem}\label{OperatorLemma}
Let $V$ be a nondegenerate quadratic space of even dimension $n=2k\geq 4$, with quadratic form $Q\in \kappa[V]$, and let $\Delta\in D_V$ be the Laplacian attached to the dual quadratic form $Q^*\in \kappa[V^*]$.\footnote{The operators $Q$, $\Delta$, and $E+k$ form an $\mathfrak{sl}_2$-triple acting on $\kappa[V]$. More precisely, $[E+k,Q]=2Q$, $[E+k,\Delta]=-2\Delta$, and $[\Delta,Q]=E+k$. Thus $Q$ plays the role of the raising operator, $\Delta$ the lowering operator, and $E+k$ the Cartan element. In fact, the algebra generated by these three operators is isomorphic to $\mathcal{U}(\sl_2)$ and is precisely the subalgebra of $H$-invariant differential operators on $V$. This is a special instance of Howe duality for the standard representation of the orthogonal group $H$ \cite{GoodmanWallach2009Symmetry}.} Let
\begin{equation}
V_Q := \Spec \,\kappa[V][1/Q],
\qquad
D_{V_Q}:=D\bigl(\kappa[V][1/Q]\bigr).
\end{equation}
\noindent Suppose that $\xi\in D_{V_Q}$ satisfies
\begin{equation}
[\Delta,\xi]\in D_V+D_{V_Q}\Delta.
\end{equation}
\noindent Then
\begin{equation}
\xi\in D_V+D_{V_Q}\Delta.
\end{equation}
\end{lem}

\begin{proof}
Put $M:=D_V+D_{V_Q}\Delta\subset D_{V_Q}$. We must show that if $[\Delta,\xi]\in M$, then $\xi\in M$.

Suppose not. Choose $\xi\in D_{V_Q}\setminus M$ such that $[\Delta,\xi]\in M$ and $\ord(\xi)$ is minimal among all such counterexamples; write $m:=\ord(\xi)$.

We first note that $[\Delta,M]\subset M$. Indeed, if $a\in D_V$, then $[\Delta,a]\in D_V$, while if $b\in D_{V_Q}$, then $[\Delta,b\Delta]=[\Delta,b]\Delta\in D_{V_Q}\Delta$. Hence $[\Delta,M]\subset M$.

Let $\sigma_m(\xi)\in \gr_m D_{V_Q}$ denote the principal symbol of $\xi$. Since $V_Q$ is smooth affine, the order filtration on $D_{V_Q}$ has associated graded ring $\gr D_{V_Q}\cong \kappa[T^*V][1/Q]$. Write
\begin{equation}
S:=\kappa[T^*V].
\end{equation}
\noindent Then $\gr D_{V_Q}=S[1/Q]$, and since $\sigma_2(\Delta)=Q^*$, we have
\begin{equation}
\gr M=S+S[1/Q]\,Q^*.
\end{equation}
We claim that $\sigma_m(\xi)\notin \gr M$. Indeed, if $\sigma_m(\xi)\in \gr M$, then there exists $\mu\in M$ of order $m$ with $\sigma_m(\mu)=\sigma_m(\xi)$. Thus $\ord(\xi-\mu)<m$, while $\xi-\mu\notin M$ and $[\Delta,\xi-\mu]=[\Delta,\xi]-[\Delta,\mu]\in M$, since $[\Delta,\xi]\in M$ and $[\Delta,M]\subset M$. This contradicts the minimality of $m$.

On the other hand, the operator $[\Delta,\xi]$ has order at most $m+1$, and its principal symbol in degree $m+1$ is given by the Poisson bracket of principal symbols. Hence
\begin{equation}
\sigma_{m+1}\bigl([\Delta,\xi]\bigr)=\{Q^*,\sigma_m(\xi)\}.
\end{equation}
\noindent Since $[\Delta,\xi]\in M$, its principal symbol lies in $\gr M$, so
\begin{equation}
\{Q^*,\sigma_m(\xi)\}\in S+S[1/Q]\,Q^*.
\end{equation}
\noindent Applying Lemma~\ref{SymbolLemma} to $\nu:=\sigma_m(\xi)\in S[1/Q]$, we conclude that $\sigma_m(\xi)\in S+S[1/Q]\,Q^*$, contradicting the previous claim. Therefore no such $\xi$ exists, and every $\xi\in D_{V_Q}$ satisfying $[\Delta,\xi]\in D_V+D_{V_Q}\Delta$ already lies in $D_V+D_{V_Q}\Delta$.
\end{proof}

\begin{rem}
    If $n=2$, Lemma~\ref{OperatorLemma} fails, and so does its principal-symbol counterpart, Lemma~\ref{SymbolLemma}. Indeed, take $V=\A^2$ and $Q(x,y)=xy$, so that $\Delta=\partial_x\partial_y$. Consider the operator
    \begin{equation}
        \xi:=x^{-1}\partial_x \in D_{V_Q}=D\bigl(\kappa[x,y,1/(xy)]\bigr).
    \end{equation}
    \noindent Since $\partial_y$ commutes with $x^{-1}\partial_x$, we have
    \begin{equation}
        [\Delta,\xi]
        =
        [\partial_x\partial_y,\,x^{-1}\partial_x]
        =
        [\partial_x,\,x^{-1}\partial_x]\partial_y
        =
        -x^{-2}\partial_x\partial_y
        =
        -x^{-2}\Delta.
    \end{equation}
    \noindent Thus $[\Delta,\xi]\in D_{V_Q}\Delta \subset D_V+D_{V_Q}\Delta$.

    However, $\xi\notin D_V+D_{V_Q}\Delta$. For otherwise we could write $\xi=P+A\Delta$ with $P\in D_V$ and $A\in D_{V_Q}$. Applying both sides to $x\in \kappa[x,y]$, and using $\Delta(x)=0$, we would obtain
    \begin{equation}
        x^{-1}=\xi(x)=P(x).
    \end{equation}
    \noindent But $P\in D_V$ preserves $\kappa[x,y]$, so $P(x)\in \kappa[x,y]$, a contradiction.

    Thus Lemma~\ref{OperatorLemma} fails when $n=2$. The same example also shows that Lemma~\ref{SymbolLemma} fails in this case; the obstruction is that $Q$ is reducible when $n=2$.
\end{rem}

\begin{prop}\label{KernelSurjectionLemma}
Let $\ad_\Delta$ denote the map induced by commutation with $\Delta$. Restriction induces a surjective map
\begin{equation}\label{kernelmap}
\ker\left(\ad_\Delta:D_V/D_V\Delta\to D_V/D_V\Delta\right)
\twoheadrightarrow
\ker\left(\ad_\Delta:D_{V_Q}/D_{V_Q}\Delta\to D_{V_Q}/D_{V_Q}\Delta\right).
\end{equation}
\noindent Consequently, every section of $\mathcal H^\Delta$ over the standard Bruhat cell $V$ extends uniquely to a global section of $\mathcal H^\Delta$ on $G/P$.
\end{prop}

\begin{proof}
Since $\xi\Delta$ vanishes in the quotient $D_V/D_V\Delta$, the map $\ad_\Delta$ agrees on $D_V/D_V\Delta$ with left multiplication by $\Delta$.

We first prove the surjectivity in \eqref{kernelmap}. We shall use the following elementary observation:
\begin{equation}\label{DVDeltaSaturation}
D_V\cap D_{V_Q}\Delta=D_V\Delta
\end{equation}
\noindent inside $D_{V_Q}$. Indeed, suppose that $A\in D_V$ and $A=B\Delta$ with $B\in D_{V_Q}$. If $B$ has order $m$, then the top symbol of $A$ is $\sigma_m(B)Q^*$. Since this lies in $\kappa[T^*V]$, and since $Q$ and $Q^*$ are relatively prime in $\kappa[T^*V]$, the top symbol $\sigma_m(B)$ has no pole along $Q=0$. Choose $B_0\in D_V$ with the same top symbol as $B$. Then $A-B_0\Delta=(B-B_0)\Delta$ lies in $D_V$ and has lower order. Induction on the order gives $B\in D_V$, proving \eqref{DVDeltaSaturation}.

By \eqref{DVDeltaSaturation}, the following sequence is short exact:
\begin{equation}
0\longrightarrow
D_V/D_V\Delta
\longrightarrow
D_{V_Q}/D_{V_Q}\Delta
\longrightarrow
D_{V_Q}/(D_V+D_{V_Q}\Delta)
\longrightarrow 0.
\end{equation}
\noindent The endomorphism $\ad_\Delta$ preserves this sequence. Lemma~\ref{OperatorLemma} states that $\ad_\Delta$ is injective on the rightmost quotient. Applying the snake lemma to the resulting commutative diagram therefore gives the surjection \eqref{kernelmap}.

Now let $\xi\in\mathcal H^\Delta(V)$. Since Kelvin conjugation preserves $\Delta$-harmonic classes on $V_Q$ by Proposition~\ref{w0onDelta}, the class $K\xi K$ lies in
\begin{equation}
\ker\left(\ad_\Delta:D_{V_Q}/D_{V_Q}\Delta\to D_{V_Q}/D_{V_Q}\Delta\right).
\end{equation}
\noindent By \eqref{kernelmap}, choose $\xi'\in\mathcal H^\Delta(V)$ whose restriction to $V_Q$ is $K\xi K$. Then $K\xi'K$ is regular on $w_0V$, and on $V_Q$ one has
\begin{equation}
K\xi'K=K(K\xi K)K=\xi,
\end{equation}
\noindent with equality in $D_{V_Q}/D_{V_Q}\Delta$. Thus $\xi$ on $V$ and $K\xi'K$ on $w_0V$ glue to a section
\begin{equation}
s\in\Gamma\left(X_{\textrm{big}},\mathcal H^\Delta\right),
\qquad
X_{\textrm{big}}:=V\cup w_0V.
\end{equation}

We now extend to all of $G/P$. Let $F_m\mathcal D_{\mathcal L}$ be the order filtration on $\mathcal D_{\mathcal L}$, and endow $\mathcal J_\Delta$ and $\mathcal H=\mathcal D_{\mathcal L}/\mathcal J_\Delta$ with the induced filtrations:
\begin{equation}
F_m\mathcal J_\Delta:=F_m\mathcal D_{\mathcal L}\cap\mathcal J_\Delta,
\qquad
F_m\mathcal H:=F_m\mathcal D_{\mathcal L}/F_m\mathcal J_\Delta.
\end{equation}
\noindent We use the convention that $F_j\mathcal D_{\mathcal L}=0$ for $j<0$. On a Bruhat chart $V_g$, the ideal $\mathcal J_\Delta$ is generated by the local Laplacian $\Delta_g$, and
\begin{equation}
F_m\mathcal J_\Delta|_{V_g}
=
F_{m-2}\mathcal D_{\mathcal L}|_{V_g}\cdot\Delta_g.
\end{equation}
\noindent Indeed, if $A=B\Delta_g$ has order at most $m$, then the principal symbol of $A$ is $\sigma(B)\sigma_2(\Delta_g)$; since $\sigma_2(\Delta_g)$ is nonzero and the associated graded algebra is a domain on each chart, $B$ has order at most $m-2$. Hence multiplication by $\Delta_g$ identifies $F_{m-2}\mathcal D_{\mathcal L}|_{V_g}$ with $F_m\mathcal J_\Delta|_{V_g}$. Since $G/P$ is smooth, each $F_m\mathcal D_{\mathcal L}$ is locally free; therefore each $F_m\mathcal J_\Delta$ is locally free as well.

It remains to check that each $F_m\mathcal H$ is locally free. Locally on $V_g$, pass to the associated graded for the order filtration. The successive graded pieces of $F_m\mathcal H|_{V_g}$ are, for $0\leq r\leq m$,
\begin{equation}
\Sym^r T_{G/P}|_{V_g}
\big/
\sigma_2(\Delta_g)\Sym^{r-2}T_{G/P}|_{V_g},
\end{equation}
\noindent where $\Sym^{r-2}T_{G/P}=0$ for $r<2$. Multiplication by the nonzero quadratic symbol $\sigma_2(\Delta_g)$ is fiberwise injective, hence has constant rank, so these quotients are locally free. Thus the filtration on $F_m\mathcal H|_{V_g}$ has locally free successive quotients, and hence $F_m\mathcal H$ is locally free.

Since $\mathcal H=\bigcup_m F_m\mathcal H$ and $X_{\textrm{big}}$ is quasi-compact, the section $s$ lies in $\Gamma(X_{\textrm{big}},F_m\mathcal H)$ for some $m$. By \eqref{deepboundary}, the complement
\begin{equation}
G/P\setminus X_{\textrm{big}}=\P(C_V)
\end{equation}
\noindent has codimension $2$ in $G/P$. Hartogs extension for vector bundles on the smooth variety $G/P$ therefore gives a unique section
\begin{equation}
\widetilde{s}\in\Gamma(G/P,F_m\mathcal H)
\end{equation}
\noindent restricting to $s$ on $X_{\textrm{big}}$. This extension is unique since $X_{\textrm{big}}$ is Zariski dense in $G/P$.

We claim that $\widetilde{s}$ is still $\Delta$-harmonic. This is local on Bruhat charts. On any Bruhat chart $V_g$, the section $\ad_{\Delta_g}(\widetilde{s}|_{V_g})$ is a regular section of $\mathcal H|_{V_g}$, and it vanishes on the dense open subset $V_g\cap V$. Indeed, $\widetilde{s}|_V=s|_V\in\Gamma(V,\mathcal H^\Delta)$, and $\Delta_g/\Delta_e$ is an invertible function on $V_g\cap V$ by Proposition~\ref{prop:group-action-on-Delta}. By Zariski density, $\ad_{\Delta_g}(\widetilde{s}|_{V_g})=0$. Thus $\widetilde{s}\in\Gamma(G/P,\mathcal H^\Delta)$.
\end{proof}

\subsection{\texorpdfstring{Global section of $\mathcal{H} \implies$ Symmetry of $\Delta$; Proof of the Global Sections Theorem}{Global Implies Symmetry of Delta; Proof of Global Sections Theorem}}\label{GlobalImpliesSymmetrySect} Now we will turn to demonstrating the converse of Lemma \ref{KernelSurjectionLemma}; namely, that a global section of $\mathcal{H}$ must be a symmetry of $\Delta$. The argument follows from examining the poles of $K(\kappa[V])$ where $K$ is the Kelvin transform \eqref{Kelvindef}, after we decompose $\kappa[V]$ via the Fischer decomposition. 

We begin with a basic statement of the Fischer decomposition, which will be used throughout. We observe that this lemma holds for any nondegenerate quadratic form $Q$.

\begin{lem}\label{FischerDecompositionLemma}\textbf{The Fischer Decomposition}.
Let $V$ be a nondegenerate quadratic space of dimension $2k$ over $\kappa$, with $k\geq 1$, and let $Q\in\kappa[V]$ be its quadratic form. Let $\Delta\in D_V$ be the Laplacian attached to the dual quadratic form $Q^*\in\kappa[V^*]$, normalized so that $[\Delta,Q]=E+k$, where $E$ is the Euler operator. Put $P_r:=\kappa[V]_r$ and $Z_{\Delta,r}:=\ker(\Delta)\cap P_r$. Then, for every $r\geq 0$, there is a direct-sum decomposition
\begin{equation}
P_r=Z_{\Delta,r}\oplus QP_{r-2},
\end{equation}
\noindent where $P_{r-2}=0$ if $r<2$. Consequently, if $Z_\Delta:=\ker(\Delta)\subset\kappa[V]$, then
\begin{equation}
\kappa[V]=Z_\Delta\oplus Q\kappa[V],
\end{equation}
\noindent and, by iteration,
\begin{equation}
P_r=\bigoplus_{a=0}^{\lfloor r/2\rfloor} Q^a Z_{\Delta,r-2a}.
\end{equation}
\end{lem}

\begin{proof}
We first prove that, for every $d\geq 0$, the map $\Delta Q:P_d\to P_d$ is injective. Let $G\in P_d$ and suppose that $\Delta(QG)=0$. If $G\neq 0$, write $G=Q^aH$ with $a\geq 0$ maximal, so that $H$ is homogeneous of degree $s=d-2a$, and $H\notin Q\kappa[V]$.

We claim that, for every homogeneous $H$ of degree $s$ and every $b\ge 1$, one has 
\[
\Delta(Q^bH)=Q^b\Delta H+b(s+k+b-1)Q^{b-1}H.
\]

\noindent Indeed, the case $b=1$ is just $[\Delta,Q]=E+k$. The general case follows by induction: assuming the formula for $b$, we have $\Delta(Q^{b+1}H)=Q\Delta(Q^{b}H)+(E+k)Q^{b}H=Q\bigl(Q^b\Delta H+b(s+k+b-1)Q^{b-1}H\bigr)+(2b+s+k)Q^bH$. Thus $\Delta(Q^{b+1}H)=Q^{b+1}\Delta H+\bigl(b(s+k+b-1)+2b+s+k\bigr)Q^bH = Q^{b+1}\Delta H+(b+1)(s+k+b)Q^bH$. 

Thus we have $0=\Delta(QG)=\Delta(Q^{a+1}H)=Q^{a+1}\Delta H+(a+1)(s+k+a)Q^aH$. Since $\kappa[V]$ is a domain, we obtain $Q\Delta H+(a+1)(s+k+a)H=0$. Reducing modulo the ideal $(Q)$ gives $(a+1)(s+k+a)H\equiv 0 \pmod Q$. The scalar $(a+1)(s+k+a)$ is nonzero because $\operatorname{char}\kappa=0$, $a\geq 0$, $s\geq 0$, and $k\geq 2$. Hence $H\in Q\kappa[V]$, contradicting the maximality of $a$. Therefore $G=0$, and $\Delta Q:P_d\to P_d$ is injective.

Since $P_d$ is finite-dimensional over $\kappa$, the injective map $\Delta Q:P_d\to P_d$ is an isomorphism. Now fix $r\geq 0$ and let $F\in P_r$. If $r<2$, then $P_{r-2}= 0$ and so $\Delta P_r = 0$ and $QP_{r-2}=0$. If $r\geq 2$, the isomorphism $\Delta Q:P_{r-2}\to P_{r-2}$ gives a unique $G\in P_{r-2}$ such that $\Delta(QG)=\Delta F$. Then $F-QG\in Z_{\Delta,r}$, so $P_r=Z_{\Delta,r}+QP_{r-2}$.

It remains to prove that this sum is direct. Suppose $QG\in Z_{\Delta,r}$ with $G\in P_{r-2}$. Then $\Delta(QG)=0$, and the injectivity of $\Delta Q:P_{r-2}\to P_{r-2}$ forces $G=0$. Thus $Z_{\Delta,r}\cap QP_{r-2}=0$, and hence $P_r=Z_{\Delta,r}\oplus QP_{r-2}$. Summing over $r$ gives $\kappa[V]=Z_\Delta\oplus Q\kappa[V]$, and iterating the homogeneous decomposition gives $P_r=\bigoplus_{a=0}^{\lfloor r/2\rfloor} Q^a Z_{\Delta,r-2a}$.
\end{proof}

\begin{lem}\label{KelvinPoleProfileLemma}
Put $m:=k-1$, and let $K$ denote the Kelvin transform \eqref{Kelvindef}. If $h_s\in Z_{\Delta,s}$ is homogeneous of degree $s$, then for every $a\geq 0$ one has
\begin{equation}
K(Q^a h_s)=\pm Q^{-(m+a+s)}h_s.
\end{equation}
\noindent Consequently, the Laurent--Fischer pieces appearing in $K(\kappa[V])$ are precisely the terms
\begin{equation}\label{FischerKelvinDeg}
Q^{-T}h_s
\qquad
\text{with}
\qquad
h_s\in Z_{\Delta,s},
\quad
T\geq m+s.
\end{equation}
\noindent Similarly, the Laurent--Fischer pieces appearing in $K(Q^2\kappa[V])$ are precisely the terms
\begin{equation}\label{FischerKelvinDegQ2}
Q^{-T}h_s
\qquad
\text{with}
\qquad
h_s\in Z_{\Delta,s},
\quad
T\geq m+s+2.
\end{equation}
\end{lem}

\begin{proof}
By Lemma~\ref{FischerDecompositionLemma}, it suffices to compute $K(Q^a h_s)$ for $h_s\in Z_{\Delta,s}$. Since $Q$ is quadratic, $Q(-v/Q(v))=Q(v)^{-1}$, while homogeneity gives $h_s(-v/Q(v))=(-1)^sQ(v)^{-s}h_s(v)$. Thus 
\begin{equation}
    K(Q^a h_s)=(-Q)^{-m}Q^{-a}(-1)^sQ^{-s}h_s = (-1)^{m+s}Q^{-(m+a+s)}h_s.
\end{equation}

We see that \eqref{FischerKelvinDeg} follows because the exponents obtained from $K(Q^a h_s)$ are exactly $T=m+s+a$ with $a\geq 0$; \eqref{FischerKelvinDegQ2} follows in the same way, since a Fischer summand of an element of $Q^2\kappa[V]$ has the form $Q^{a+2}h_s$, giving an exponent $T=m+s+2+a$ with $a\geq 0$.
\end{proof}

The next lemma provides the precise criterion on pole orders which we will ultimately use.

\begin{lem}\label{HomogeneousDivisibilityTestLemma}
Put $m:=k-1$. Let $F_r\in \kappa[V]_r$ be homogeneous, and write its Fischer decomposition as
\begin{equation}
F_r=\sum_{a\geq 0} Q^a h_{r-2a},
\qquad
h_{r-2a}\in Z_{\Delta,r-2a}.
\end{equation}
\noindent Let $T\in \mathbb{Z}$, and set
\begin{equation}
b:=m+r+2-T.
\end{equation}
\noindent If $b\leq 0$, then $Q^{-T}F_r\in K(Q^2\kappa[V])$. If $b>0$, then
\begin{equation}
Q^{-T}F_r\in K(Q^2\kappa[V])
\qquad\Longleftrightarrow\qquad
F_r\in Q^b\kappa[V].
\end{equation}
Equivalently, if $b>0$ and $F_r$ is not divisible by $Q^b$, then $Q^{-T}F_r$ is not contained in $K(Q^2\kappa[V])$.
\end{lem}

\begin{proof}
Multiplying the Fischer decomposition by $Q^{-T}$ gives $Q^{-T}F_r=\sum_{a\geq 0}Q^{-(T-a)}h_{r-2a}$. By Lemma~\ref{KelvinPoleProfileLemma}, the summand $Q^{-(T-a)}h_{r-2a}$ belongs to $K(Q^2\kappa[V])$ precisely when $T-a\geq m+(r-2a)+2$, or equivalently when $a\geq m+r+2-T=b$. If $b\leq 0$, this condition holds for every $a\geq 0$, proving the first assertion. If $b>0$, membership in $K(Q^2\kappa[V])$ is equivalent to the vanishing of all Fischer summands with $a<b$. By the iterated Fischer decomposition, this is equivalent to $F_r\in Q^b\kappa[V]$. This proves the lemma.
\end{proof}

\begin{lem}\label{RightFischerDecompositionLemma}
Let $\kappa[\zeta]$ denote the polynomial algebra in the formal derivative-symbol variables, so that normal ordering identifies
\begin{equation}
D_V\cong \kappa[V]\otimes_\kappa \kappa[\zeta].
\end{equation}
\noindent Let $Q^*(\zeta)$ be the symbol of $\Delta$, and let $Z_{\partial}$ denote the space of polynomials in the $\zeta$-variables which are harmonic with respect to the quadratic form $Q^*$. Then one has a direct-sum decomposition
\begin{equation}
D_V=D_V\Delta\oplus \bigl(\kappa[V]\otimes_\kappa Z_{\partial}\bigr).
\end{equation}
\noindent Consequently, for every $\xi\in D_V$, there are unique $\delta_\xi\in D_V$ and $R_\xi\in \kappa[V]\otimes_\kappa Z_{\partial}$ such that
\begin{equation}
\Delta\xi=\delta_\xi\Delta+R_\xi.
\end{equation}
\end{lem}

\begin{proof}
By Lemma~\ref{FischerDecompositionLemma}, applied in the formal derivative-symbol variables, we have
\begin{equation}
\kappa[\zeta]=Z_{\partial}\oplus Q^*(\zeta)\kappa[\zeta].
\end{equation}
\noindent Tensoring with $\kappa[V]$ gives $\kappa[V]\otimes_\kappa \kappa[\zeta]
=
\bigl(\kappa[V]\otimes_\kappa Z_{\partial}\bigr)
\oplus
\bigl(\kappa[V]\otimes_\kappa Q^*(\zeta)\kappa[\zeta]\bigr)$. Under the normal-ordering identification $D_V\cong \kappa[V]\otimes_\kappa \kappa[\zeta]$, the second summand is exactly the right ideal $D_V\Delta$.  Applying this direct-sum decomposition to the operator $\Delta\xi$ gives the asserted unique expression $\Delta\xi=\delta_\xi\Delta+R_\xi$.
\end{proof}

\begin{lem}\label{KelvinRegularityTestingLemma}
Let $\xi\in D_V$ and suppose that $K\xi K\in D_V$. Write
\begin{equation}
\Delta\xi=\delta_\xi\Delta+R
\end{equation}
\noindent according to the decomposition of Lemma~\ref{RightFischerDecompositionLemma}, with $R\in \kappa[V]\otimes_\kappa Z_{\partial}$. Then, for every $h\in Z_\Delta$, one has
\begin{equation}
R(Kh)\in K(Q^2\kappa[V]).
\end{equation}
\end{lem}

\begin{proof}
Let $\psi:=K\xi K$. By assumption, $\psi\in D_V$. Since $K^2=\Id$, we have $\xi K=K\psi$. Let $h\in Z_\Delta$. Recall (Proposition~\ref{w0onDelta}) that the conformal transformation law for the Laplacian under the Kelvin transform is $K\Delta K=Q^2\Delta$; or, equivalently, $\Delta K=KQ^2\Delta$. Since $\Delta h=0$, this gives $\Delta(Kh)=K(Q^2\Delta h)=0$. Hence
\begin{equation}
R(Kh)=(\Delta\xi-\delta_\xi\Delta)(Kh)=\Delta\xi(Kh).
\end{equation}
\noindent Using $\xi K=K\psi$, we get
\begin{equation}
\Delta\xi(Kh)=\Delta K(\psi h)=K(Q^2\Delta(\psi h)).
\end{equation}
\noindent Since $\psi\in D_V$ and $h\in\kappa[V]$, we have $\Delta(\psi h)\in\kappa[V]$. Therefore $Q^2\Delta(\psi h)\in Q^2\kappa[V]$, and so $R(Kh)\in K(Q^2\kappa[V])$.
\end{proof}

Our objective is to show that if $\xi$ and $K\xi K$ are both in $D_V$, then $R=0$. We will prove this by contrapositive: if $R\neq 0$, we will construct an appropriate harmonic polynomial $h\in Z_\Delta$ such that $R(Kh)$ contains a Laurent--Fischer term forbidden by Lemma~\ref{HomogeneousDivisibilityTestLemma}. This will contradict Lemma~\ref{KelvinRegularityTestingLemma}. The next few lemmata are in service of constructing such an $h$.

\begin{lem}\label{NullKelvinTestFunctionsLemma}
Let $\ell\in V^*$ be a null linear form, i.e.\ $Q^*(\ell)=0$. For each integer $N\geq 0$, put
\begin{equation}
f_N:=Q^{-m-N}\ell^N.
\end{equation}
\noindent Then $\ell^N\in Z_{\Delta,N}$, and
\begin{equation}
f_N=\pm K(\ell^N).
\end{equation}
\noindent In particular,
\begin{equation}
\Delta f_N=0.
\end{equation}
\end{lem}

\begin{proof}
Since $\ell$ is linear, one has $\Delta(\ell^N)=N(N-1)Q^*(\ell)\ell^{N-2}=0$, where the assertion is clear also for $N=0,1$. Thus $\ell^N\in Z_{\Delta,N}$.

We now compute the Kelvin transform. Since $\ell(-v/Q(v))=-\ell(v)/Q(v)$, we have
\begin{equation}
K(\ell^N)
=
(-Q)^{-m}\left(-\frac{\ell}{Q}\right)^N
=
(-1)^{-m+N} Q^{-m-N}\ell^N
=
\pm f_N.
\end{equation}
Finally, the Kelvin transform preserves harmonicity in the sense that the Kelvin transform of a harmonic function is still harmonic. Indeed: $\Delta K=KQ^2\Delta$, so $\Delta(\ell^N)=0$, and it follows that $\Delta f_N = \Delta(K(\ell^N)) = K(Q^2\Delta(\ell^N))=0$.
\end{proof}

\begin{lem}\label{BihomogeneousBlockExpansionLemma}
Let $R_{d,j}\in D_V$ be bihomogeneous of coefficient degree $d$ and derivative degree $j$, written in normal form, and let
\begin{equation}
P_{d,j}(v,\zeta)
\end{equation}
\noindent denote its normal symbol. For $\ell\in V^*$ and $N\geq j$, put
\begin{equation}
f_N:=Q^{-m-N}\ell^N,
\qquad
A^\ell(v):=Q(v)\ell-\ell(v)dQ_v\in V^*.
\end{equation}
\noindent Then there are polynomials $F^\ell_{d,j,q}(v)\in \kappa[V]$, independent of $N$, such that
\begin{equation}\label{BihomogeneousBlockExpansion}
R_{d,j}(f_N)
=
Q^{-m-N-j}\ell^{N-j}
\sum_{q=0}^j N^q F^\ell_{d,j,q}(v).
\end{equation}
\noindent Each $F^\ell_{d,j,q}$ is homogeneous of polynomial degree $d+2j$, and the leading coefficient is
\begin{equation}
F^\ell_{d,j,j}(v)=P_{d,j}(v,A^\ell(v)).
\end{equation}
\end{lem}

\begin{proof}
It suffices to prove the claim for a normal-ordered monomial operator $g(v)\partial_{u_1}\cdots\partial_{u_j}$, with $g\in \kappa[V]_d$ and $u_1,\ldots,u_j\in V$, and then sum over such monomials. For a constant vector $u\in V$, one has
\begin{equation}\label{FirstDerivativeKelvinTest}
\partial_u f_N=Q^{-m-N-1}\ell^{N-1}\bigl(NA^\ell(v)(u)-m\ell(v)dQ_v(u)\bigr).
\end{equation}
We prove by induction on $r$ that $\partial_{u_1}\cdots\partial_{u_r}f_N=Q^{-m-N-r}\ell^{N-r}\sum_{q=0}^r N^qG^\ell_{r,q}(v)$, where the $G^\ell_{r,q}$ are independent of $N$, homogeneous of polynomial degree $2r$, and $G^\ell_{r,r}(v)=\prod_{i=1}^r A^\ell(v)(u_i)$. The case $r=1$ is precisely \eqref{FirstDerivativeKelvinTest}.

Assume the claim for $r$. Applying $\partial_{u_{r+1}}$ to a summand $Q^{-m-N-r}\ell^{N-r}G^\ell_{r,q}$ gives two types of terms. When the derivative hits the factor $Q^{-m-N-r}\ell^{N-r}$, we obtain $Q^{-m-N-r-1}\ell^{N-r-1}\bigl((N-r)A^\ell(v)(u_{r+1})-(m+r)\ell(v)dQ_v(u_{r+1})\bigr)G^\ell_{r,q}$. When the derivative hits $G^\ell_{r,q}$, we rewrite the result over the same common factor as $Q^{-m-N-r-1}\ell^{N-r-1}Q(v)\ell(v)\partial_{u_{r+1}}G^\ell_{r,q}$. In both cases the new coefficients are independent of $N$ apart from the indicated polynomial dependence on $N$, and are homogeneous of polynomial degree $2r+2$.

Moreover, the coefficient of $N^{r+1}$ can only come from the coefficient of $N^r$ in the previous stage multiplied by the term $NA^\ell(v)(u_{r+1})$. The terms involving $-rA^\ell(v)(u_{r+1})$, $-(m+r)\ell(v)dQ_v(u_{r+1})$, and $\partial_{u_{r+1}}G^\ell_{r,q}$ contribute only to lower powers of $N$. Hence the top coefficient after $j$ derivatives is $\prod_{i=1}^j A^\ell(v)(u_i)$.

Multiplying by $g(v)$ gives \eqref{BihomogeneousBlockExpansion}, with each $F^\ell_{d,j,q}$ independent of $N$ and homogeneous of polynomial degree $d+2j$. For the monomial operator, the coefficient of $N^j$ is $g(v)\prod_i A^\ell(v)(u_i)$, which is the normal symbol evaluated at $\zeta=A^\ell(v)$. Summing over the normal-ordered monomials in $R_{d,j}$ gives $F^\ell_{d,j,j}(v)=P_{d,j}(v,A^\ell(v))$.
\end{proof}

\begin{rem}\label{PrinicpalSymbolAlRemark}
    The covector $A^\ell(v)=Q(v)\ell-\ell(v)dQ_v$ is the same invariant already appearing in the quasiclassical description of $T^*C^o$. Indeed, let $u\in V$ be the vector corresponding to $\ell\in V^*$ under the bilinear form $B$, so that $\ell=dQ_u$. Since $\ell$ is null, $u\in C^o$ whenever $u\neq 0$. In the notation of \eqref{invariantfunccot} and \eqref{dualconeinv}, with $w=u$, one has $\alpha=B(v,u)=\ell(v)$ and
    \begin{equation}
        \mu_{v,u}=\alpha v-Q(v)u.
    \end{equation}
    \noindent After identifying $V$ and $V^*$ by $B$, this gives
    \begin{equation}
        A^\ell(v)=-\,\mu_{v,u}^{\flat}.
    \end{equation}
    \noindent Thus, up to this harmless sign convention, $A^\ell$ is precisely the $\mu$-coordinate function on the minimal nilpotent orbit closure described in Proposition~\ref{CotBundDescr}. Equivalently, it is the quasiclassical principal-symbol function which appears in $F$-moment descent.
\end{rem}

Now decompose $R$ into bihomogeneous normal-ordered pieces
\begin{equation}
R=\sum_{d,j}R_{d,j},
\end{equation}
\noindent where $R_{d,j}$ has coefficient degree $d$ and derivative degree $j$. We define the \textbf{weight} of such a block by $w(d,j):=d-j$, and set
\begin{equation}
W:=\max\{\,d-j:R_{d,j}\neq 0\,\}.
\end{equation}
\noindent Among the nonzero blocks with $d-j=W$, choose $J$ maximal, and put
\begin{equation}
D:=W+J.
\end{equation}
\noindent Thus $R_{D,J}$ is the top block: it has maximal weight $W$, and among all blocks of maximal weight it has maximal derivative degree $J$. We write
\begin{equation}
P(v,\zeta):=P_{D,J}(v,\zeta)
\end{equation}
\noindent for its normal symbol. Then $P$ has $v$-degree $D$, $\zeta$-degree $J$, and is harmonic in the $\zeta$-variables, since $R\in\kappa[V]\otimes_\kappa Z_{\partial}$.

Applying Lemma~\ref{BihomogeneousBlockExpansionLemma} to each block $R_{d,j}$, we see that every term in $R_{d,j}(f_N)$ has homogeneous degree $-2m-N+d-j$. Hence only the blocks with $d-j=W$ can contribute to the top homogeneous degree of $R(f_N)$. Putting these top-weight contributions over the common factor $Q^{-m-N-J}\ell^{N-J}$, we obtain
\begin{equation}\label{TopHomogeneousPart}
Q^{-m-N-J}\ell^{N-J}S_N,
\end{equation}
\noindent where
\begin{equation}\label{SNDefinition}
S_N
=
\sum_{d-j=W}
Q^{J-j}\ell^{J-j}
\sum_{q=0}^j N^qF^\ell_{d,j,q}(v).
\end{equation}
\noindent For every summand in \eqref{SNDefinition}, the ordinary polynomial degree is
\begin{equation}
2(J-j)+(J-j)+(d+2j)=D+2J,
\end{equation}
\noindent since $d-j=W$ and $D=W+J$. Thus $S_N$ is homogeneous of polynomial degree $D+2J$. Moreover, the highest possible power of $N$ in $S_N$ is $N^J$, and it can only arise from the block $R_{D,J}$. Therefore, by Lemma~\ref{BihomogeneousBlockExpansionLemma},
\begin{equation}\label{SNLeadingTerm}
S_N
=
N^J P(v,A^\ell(v))
+
N^{J-1}S_{J-1}(v)
+\cdots+
S_0(v),
\end{equation}
\noindent where each $S_i(v)$ is independent of $N$ and homogeneous of polynomial degree $D+2J$.

\begin{lem}\label{ChoiceOfEllLemma}
Let $P(v,\zeta)=P_{D,J}(v,\zeta)$ be the top normal symbol defined above. Then there exists a null linear form $\ell\in V^*$ such that
\begin{equation}
P(v,A^\ell(v))\neq 0
\end{equation}
\noindent as a polynomial in $v$.
\end{lem}

\begin{proof}
Since $P(v,\zeta)$ is not the zero polynomial, we may choose $v_0\in V$ with $Q(v_0)\neq 0$ such that $P(v_0,\zeta)$ is not the zero polynomial in the $\zeta$-variables. Indeed, the locus $Q\neq 0$ is dense, and if $P(v,\zeta)$ vanished for every $v$ with $Q(v)\neq 0$, then all its coefficients as a polynomial in $\zeta$ would vanish on a dense open subset of $V$.

For this fixed $v_0$, consider the linear map
\begin{equation}
T_{v_0}:V^*\longrightarrow V^*,
\qquad
T_{v_0}(\ell):=A^\ell(v_0)=Q(v_0)\ell-\ell(v_0)dQ_{v_0}.
\end{equation}
\noindent After identifying $V^*$ with $V$ by the bilinear form, this is $Q(v_0)$ times the orthogonal reflection through $v_0^\perp$. Hence $T_{v_0}$ carries the dual null cone isomorphically onto itself.

Now $P(v_0,\zeta)$ is harmonic in the $\zeta$-variables, because $P(v,\zeta)$ is. If $P(v_0,\zeta)$ vanished on the entire dual null cone, then by irreducibility of $Q^*$ for $\dim V\geq 4$, it would be divisible by $Q^*$. This is impossible unless $P(v_0,\zeta)=0$, because the Fischer decomposition in the $\zeta$-variables gives $\kappa[\zeta]=Z_{\partial}\oplus Q^*\kappa[\zeta]$. Thus a nonzero harmonic polynomial cannot lie in $Q^*\kappa[\zeta]$.

Therefore there exists a null covector $\zeta_0\in V^*$ such that $P(v_0,\zeta_0)\neq 0$. Since $T_{v_0}$ is an automorphism of the dual null cone, we may choose a null linear form $\ell\in V^*$ with $A^\ell(v_0)=\zeta_0$. Then $P(v_0,A^\ell(v_0))\neq 0$, and hence $P(v,A^\ell(v))$ is not the zero polynomial in $v$.
\end{proof}

\begin{rem}
    Let us make explicit where the hypotheses on $R$ enter the preceding two lemmata. The fact that $R\neq 0$ ensures that the top block $R_{D,J}$ is nonzero, and hence that its normal symbol $P=P_{D,J}$ is not the zero polynomial. The fact that $R$ lies in the Fischer complement $\kappa[V]\otimes_\kappa Z_{\partial}$ to $\kappa[V]\otimes_\kappa \kappa[\partial]\Delta$ ensures that $P$ is harmonic in the $\zeta$-variables. These two facts are used together in Lemma~\ref{ChoiceOfEllLemma}: since $P$ is a nonzero harmonic polynomial in the $\zeta$-variables, it cannot vanish identically on the dual null cone. This is what allows us to choose a null linear form $\ell$ for which $P(v,A^\ell(v))$ is nonzero. It is precisely this choice of $\ell$ (and so $f_N$) that will permit us to obtain the desired contradiction, forcing $R =0$.
\end{rem}

\begin{lem}\label{DivisibilityBoundForFellLemma}
Say $R \ne 0$, let $\ell$ be chosen as in Lemma~\ref{ChoiceOfEllLemma}, and put
\begin{equation}
F_\ell(v):=P(v,A^\ell(v)).
\end{equation}
\noindent Then $F_\ell$ is homogeneous of polynomial degree $D+2J$, and $F_\ell$ is not divisible by $Q^{D+1}$. In particular, $F_\ell$ is not divisible by $Q^{D+2}$.
\end{lem}

\begin{proof}
The homogeneity is immediate: $P(v,\zeta)$ has $v$-degree $D$ and $\zeta$-degree $J$, while $A^\ell(v)=Q(v)\ell-\ell(v)dQ_v$ is homogeneous of degree $2$ in $v$. Hence $F_\ell(v)=P(v,A^\ell(v))$ is homogeneous of degree $D+2J$.

Let $u\in V$ be the vector satisfying $dQ_u=\ell$. Since $\ell$ is null, $Q(u)=0$. For every $v\in V$ and $s\in\kappa$, one has $Q(v+su)=Q(v)+s\ell(v)$, $\ell(v+su)=\ell(v)$, and $dQ_{v+su}=dQ_v+s\ell$. Therefore\footnote{This is the crucial invariance property of the corresponding function on $T^*C^o$, cf. \eqref{cotangentivarianceproperty} and Remark~\ref{PrinicpalSymbolAlRemark}.}
\begin{equation}
A^\ell(v+su)=A^\ell(v).
\end{equation}
\noindent Consequently,
\begin{equation}
F_\ell(v+su)=P(v+su,A^\ell(v)).
\end{equation}
\noindent As a polynomial in $s$, the right-hand side has degree at most $D$, because $P$ has $v$-degree $D$.

Suppose, for contradiction, that $F_\ell=Q^{D+1}G$ for some $G\in\kappa[V]$. If $\ell(v)\neq 0$, then $Q(v+su)=Q(v)+s\ell(v)$ is a nonconstant linear polynomial in $s$, and hence
\begin{equation}
F_\ell(v+su)=(Q(v)+s\ell(v))^{D+1}G(v+su).
\end{equation}
\noindent The left-hand side has degree at most $D$ in $s$, while the right-hand side is divisible by the $(D+1)$-st power of a nonconstant linear polynomial. Therefore this polynomial in $s$ must be identically zero. In particular, $F_\ell(v)=0$ for every $v$ with $\ell(v)\neq 0$. Since $\ell\neq 0$, this is a dense open condition on $V$, so $F_\ell=0$, contradicting the choice of $\ell$ in Lemma~\ref{ChoiceOfEllLemma}.
\end{proof}

\begin{lem}\label{VandermondeNonvanishingLemma}
Let $\ell$ be chosen as in Lemma~\ref{ChoiceOfEllLemma}. Then, for all sufficiently large integers $N$, the polynomial $S_N$ is not divisible by $Q^{D+2}$.
\end{lem}

\begin{proof}
By \eqref{SNLeadingTerm}, we have
\begin{equation}
S_N
=
N^J F_\ell
+
N^{J-1}S_{J-1}
+\cdots+
S_0,
\end{equation}
\noindent where all terms are homogeneous of polynomial degree $D+2J$ and are independent of $N$. Let
\begin{equation}
\overline{S}_N
\in
\kappa[V]_{D+2J}/\bigl(Q^{D+2}\kappa[V]\cap \kappa[V]_{D+2J}\bigr)
\end{equation}
\noindent denote the image of $S_N$. Then
\begin{equation}
\overline{S}_N
=
N^J\overline{F}_\ell
+
N^{J-1}\overline{S}_{J-1}
+\cdots+
\overline{S}_0.
\end{equation}
\noindent By Lemma~\ref{DivisibilityBoundForFellLemma}, $F_\ell$ is not divisible by $Q^{D+2}$, so $\overline{F}_\ell\neq 0$. Hence $\overline{S}_N$ is a nonzero polynomial in $N$ with values in the finite-dimensional $\kappa$-vector space $\kappa[V]_{D+2J}/(Q^{D+2}\kappa[V]\cap \kappa[V]_{D+2J})$.

Since $\operatorname{char}\kappa=0$, the integers define an infinite subset of $\kappa$. A nonzero vector-valued polynomial over $\kappa$ can vanish at only finitely many values of $N$: after choosing a basis of the quotient vector space, at least one coordinate polynomial is nonzero. Therefore $\overline{S}_N\neq 0$ for all sufficiently large integers $N$. Equivalently, $S_N$ is not divisible by $Q^{D+2}$ for all sufficiently large $N$.
\end{proof}

\begin{lem}\label{BadTopPieceLemma}
Let $\ell$ be chosen as in Lemma~\ref{ChoiceOfEllLemma}. Then, for all sufficiently large integers $N$, one has
\begin{equation}
R(f_N)\notin K(Q^2\kappa[V]).
\end{equation}
\end{lem}

\begin{proof}
By the discussion preceding Lemma~\ref{ChoiceOfEllLemma}, the top homogeneous-degree part of $R(f_N)$ is
\begin{equation}
Q^{-m-N-J}\ell^{N-J}S_N.
\end{equation}
\noindent By Lemma~\ref{VandermondeNonvanishingLemma}, for all sufficiently large $N$, the polynomial $S_N$ is not divisible by $Q^{D+2}$.

We claim that, for such $N$, the top homogeneous-degree term $Q^{-m-N-J}\ell^{N-J}S_N$ is not contained in $K(Q^2\kappa[V])$. Indeed, the numerator $\ell^{N-J}S_N$ is homogeneous of polynomial degree
\begin{equation}
r=(N-J)+(D+2J)=N+D+J,
\end{equation}
\noindent while the denominator exponent is
\begin{equation}
T=m+N+J.
\end{equation}
\noindent Thus, in the notation of Lemma~\ref{HomogeneousDivisibilityTestLemma},
\begin{equation}
m+r+2-T
=
m+(N+D+J)+2-(m+N+J)
=
D+2.
\end{equation}
\noindent Therefore membership in $K(Q^2\kappa[V])$ would force $\ell^{N-J}S_N$ to be divisible by $Q^{D+2}$. Since $Q$ is nondegenerate of dimension at least $4$, the linear form $\ell$ is relatively prime to $Q$. Hence $S_N$ itself would be divisible by $Q^{D+2}$, contradicting Lemma~\ref{VandermondeNonvanishingLemma}.

Finally, the remaining terms in $R(f_N)$ have strictly smaller homogeneous degree, because they come from blocks $R_{d,j}$ with $d-j<W$. Since $K(Q^2\kappa[V])$ is graded by homogeneous degree, as follows from Lemma~\ref{KelvinPoleProfileLemma}, these lower-degree terms cannot cancel the failure of the top homogeneous-degree term to lie in $K(Q^2\kappa[V])$. Hence $R(f_N)\notin K(Q^2\kappa[V])$.
\end{proof}

\begin{definition}\label{KelvinRegularDefinition}
Let $[\xi]\in D_V/D_V\Delta$. We say that $[\xi]$ is \textbf{Kelvin-regular} if, after viewing $\xi$ as an operator on $V_Q$, one has
\begin{equation}
K\xi K\in D_V+D_{V_Q}\Delta.
\end{equation}
\noindent Equivalently, the Kelvin transform of $[\xi]$ has a representative which is regular on the opposite Bruhat cell. This condition is independent of the choice of representative $\xi$: if $\xi$ is replaced by $\xi+A\Delta$, then
\begin{equation}
K(\xi+A\Delta)K
=
K\xi K+(KAK)(K\Delta K),
\end{equation}
\noindent and $K\Delta K=Q^2\Delta$, so the second term lies in $D_{V_Q}\Delta$.
\end{definition}

\begin{prop}\label{KelvinRegularGlobalSymmetryEquivalence}
Let $[\xi]\in D_V/D_V\Delta$. The following are equivalent:
\begin{equation}
[\xi]\text{ is the restriction of a global section of }\mathcal H;
\end{equation}
\begin{equation}
[\xi]\text{ is Kelvin-regular};
\end{equation}
\begin{equation}
[\xi]\in \ker\left(\ad_\Delta:D_V/D_V\Delta\to D_V/D_V\Delta\right).
\end{equation}
\noindent Equivalently, the third condition says that $\xi$ is a symmetry of $\Delta$, i.e.\ that there exists $\delta_\xi\in D_V$ such that
\begin{equation}
\Delta\xi=\delta_\xi\Delta.
\end{equation}
\end{prop}

\begin{proof}
The implication ``global $\implies$ Kelvin-regularity" is just the transition rule between the two Bruhat charts. Indeed, if $[\xi]$ is the restriction of a global section of $\mathcal H$ to $V$, then on $w_0V$ the same section is represented by some regular operator $\psi\in D_V$, and on $V_Q=V\cap w_0V$ one has $K\xi K-\psi\in D_{V_Q}\Delta$. Hence $[\xi]$ is Kelvin-regular.

Now suppose that $[\xi]$ is Kelvin-regular. Choose a representative $\xi\in D_V$ and write, as in Lemma~\ref{RightFischerDecompositionLemma},
\begin{equation}
\Delta\xi=\delta_\xi\Delta+R,
\qquad
R\in \kappa[V]\otimes_\kappa Z_{\partial}.
\end{equation}
\noindent We claim that $R=0$. By Kelvin-regularity, write $K\xi K=\psi+B\Delta$ with $\psi\in D_V$ and $B\in D_{V_Q}$. The proof of Lemma~\ref{KelvinRegularityTestingLemma} then gives, verbatim, that $R(Kh)\in K(Q^2\kappa[V])$ for every $h\in Z_\Delta$. If $R\neq 0$, Lemma~\ref{BadTopPieceLemma} gives a null linear form $\ell$ such that for all sufficiently large integers $N$, we have $R(f_N)\notin K(Q^2\kappa[V])$. But by Lemma~\ref{NullKelvinTestFunctionsLemma}, $f_N=\pm K(\ell^N)$ and $\ell^N\in Z_\Delta$, a contradiction. Thus $R=0$, so $\Delta\xi=\delta_\xi\Delta$, and $[\xi]\in\ker(\ad_\Delta)$.

Finally, if $[\xi]\in\ker(\ad_\Delta)$, then $\xi$ is a symmetry of $\Delta$. Lemma~\ref{KernelSurjectionLemma} says precisely that such a class extends uniquely to a global section of $\mathcal H$. This proves the equivalence of the three conditions.
\end{proof}

We may now prove Theorem~\ref{GlobalSectionsTheorem}; most of the work was done in Proposition~\ref{KelvinRegularGlobalSymmetryEquivalence}.

\begin{proof}[Proof of Theorem~\ref{GlobalSectionsTheorem}]
By Proposition~\ref{KelvinRegularGlobalSymmetryEquivalence}, restriction to the standard Bruhat cell identifies $\Gamma(G/P,\mathcal H)$ with the classes $[\xi]\in D_V/D_V\Delta$ satisfying $\ad_\Delta([\xi])=0$. This condition is equivalent to $\xi\in N(D_V\Delta)$, and hence gives
\begin{equation}
\Gamma(G/P,\mathcal H)
\cong
N(D_V\Delta)/D_V\Delta
\subset
D_V/D_V\Delta.
\end{equation}
\noindent Proposition~\ref{KelvinRegularGlobalSymmetryEquivalence} also shows that every global section of $\mathcal H$ is harmonic, so $\Gamma(G/P,\mathcal H)=\Gamma_\Delta(G/P,\mathcal H)$. Finally, applying $\widehat{\tau}$ gives $N(D_V\Delta)/D_V\Delta\cong D_C$, as in \eqref{defwidehattau}. Hence
\begin{equation}
\Gamma(\mathcal H)
=
\Gamma_\Delta(\mathcal H)
\cong
N(D_V\Delta)/D_V\Delta
\cong
D_C.
\end{equation}
\end{proof}

\subsection{\texorpdfstring{Singular Support of the Harmonic Sheaf, Moment Maps, and the Richardson Variety of $P$}{Singular Support of the Harmonic Sheaf, Moment Maps, and The Richardson Variety of P}}\label{SingSupptRichardsonSection}

We now consider the singular support of $\mathcal{H}$. The moment map calculations here will also be used below in the proof that the global twisted differential operators on $G/P$ are generated by the infinitesimal $G$-action. We briefly recall the geometry of $G/P$ and its cotangent bundle $T^*G/P$. Recall that
\begin{equation}
    T^*(G/P) \cong G \times^P (\mathfrak{g}/\mathfrak{p})^* \cong G \times^P \mathfrak{u},
\end{equation}

\noindent where $P$ acts via $\Ad$ on $\mathfrak{u}$. Recall that we identify $\u$, $U$, and $V^*$, and that the quadric cone may be viewed as a closed subset $C \subset \u$. We observe, in fact, the somewhat miraculous fact that the adjoint action of $P$ preserves $C$. This follows because the adjoint action of $U = R_u(P)$ on $\mathfrak{u}$ is trivial, since $U$ is Abelian, while the adjoint action of $L$ on $\mathfrak{u}$ is given by
\begin{equation}
\left(\begin{smallmatrix}
t & 0 & 0\\
0 & h & 0\\
0 & 0 & t^{-1}
\end{smallmatrix}\right)
\left(\begin{smallmatrix}
1 & -\,v^{\mathsf T}J_V & -\,Q(v)\\
0 & I & v\\
0 & 0 & 1
\end{smallmatrix}\right)
\left(\begin{smallmatrix}
t^{-1} & 0 & 0\\
0 & h^{-1} & 0\\
0 & 0 & t
\end{smallmatrix}\right)
=
\left(\begin{smallmatrix}
1 & -\,t\,v^{\mathsf T}J_Vh^{-1} & -\,t^{2}Q(v)\\
0 & I & t\,h v\\
0 & 0 & 1
\end{smallmatrix}\right).
\end{equation}

\noindent We see that this agrees with the commuting composition of the left action of $H = O(Q)$ on $V^*$ with the scaling $\G_m$-action, and hence preserves $C$.

We may therefore consider the conical subvariety
\begin{equation}
    G \times^P C \subset G \times^P \u \cong T^*(G/P).
\end{equation}

\begin{prop}\label{lem:SS-harmonic-contained}
The singular support of $\mathcal{H}$ is
\begin{equation}
\mathrm{SS}(\mathcal{H}) = G\times^{P} C \subset T^*(G/P)\cong G\times^{P}\mathfrak{u}.
\end{equation}

\noindent Moreover, if $\mathcal{M} \in \mathscr{H}$, then
\begin{equation}
    \mathrm{SS}(\mathcal{M}) \subset G\times^{P} C.
\end{equation}
\end{prop}

\begin{proof}
On each Bruhat chart $V_g$ one has $\mathcal{H}|_{V_g}\cong \left(\mathcal{D}_{\mathcal{L}}|_{V_g}\right)\,/\,\left(\mathcal{D}_{\mathcal{L}}|_{V_g}\Delta_g\right)$. Choose the order filtration on $\mathcal{D}_{\mathcal{L}}(V_g)$ and equip $\mathcal{H}(V_g)$ with the induced good filtration. Then
\begin{equation}
\gr\bigl(\mathcal{H}(V_g)\bigr)\cong \gr\bigl(\mathcal{D}_{\mathcal{L}}(V_g)\bigr)\big/\gr\bigl(\mathcal{D}_{\mathcal{L}}(V_g)\bigr)\sigma(\Delta_g),
\end{equation}

\noindent so its support, and hence the singular support of $\mathcal{H}|_{V_g}$, is exactly the hypersurface cut out by $\sigma(\Delta_g)$. Under the identification $T^*(G/P)|_{V_g}\cong V_g\times \mathfrak{u}$, this hypersurface is precisely $V_g\times C$, where $C\subset \mathfrak{u}$ is the isotropic cone of $Q^*$. Since these local descriptions glue $G$-equivariantly, it follows that $\mathrm{SS}(\mathcal{H})=G\times^P C$.

Now let $\mathcal M\in\mathscr H$. By definition, $\mathcal M$ admits a finite $(\mathcal H,D_C)$-presentation, and in particular there is a surjection
\begin{equation}
\mathcal H^b \longrightarrow \mathcal M \longrightarrow 0.
\end{equation}
\noindent Restricting to $V_g$, we obtain a surjection
\begin{equation}
\left(\mathcal H|_{V_g}\right)^b \longrightarrow \mathcal M|_{V_g} \longrightarrow 0.
\end{equation}
Thus the singular support of $\mathcal M|_{V_g}$ is contained in the singular support of $\left(\mathcal H|_{V_g}\right)^b$. Since finite direct sums do not change singular support, it is enough to bound the singular support of $\mathcal H|_{V_g}$. Therefore
\begin{equation}
\mathrm{SS}(\mathcal{M}|_{V_g})\subset \mathrm{SS}(\mathcal{H}|_{V_g})=V_g\times C.
\end{equation}

\noindent Since this holds on every Bruhat chart, we conclude that $\mathrm{SS}(\mathcal{M})\subset G\times^P C$.
\end{proof}

This geometric constraint on singular support may be viewed as the quasiclassical and microlocal shadow of harmonicity. It is also closely related to the Borho--Brylinski philosophy that characteristic varieties on homogeneous spaces should be studied via the moment map to $\g$ \cite{BorhoBrylinski1982DiffOpsI,BorhoBrylinski1989DiffOpsII,BorhoBrylinski1985DiffOpsIII}. In the present case, if
\begin{equation}
    \mu:T^*(G/P)\cong G\times^P \u \longrightarrow \g
\end{equation}

\noindent denotes the moment map, then $G\times^P C$ is naturally a $G$-stable conical subvariety of $T^*(G/P)$ lying over the minimal nilpotent orbit closure $\overline{\mathbb{O}}_{\min}$. In fact we have:

\begin{prop}\label{MomentMapPreimageProp}
One has
\begin{equation}
    \mu^{-1}(\overline{\mathbb{O}}_{\min})=G\times^P C.
\end{equation}

\noindent Moreover,
\begin{equation}
    \mu^{-1}(\mathbb{O}_{\min})=G\times^P (C\setminus \{0\}).
\end{equation}
\end{prop}

\begin{proof}
Since $\mu([g,X])=\Ad(g)X$ and $\overline{\mathbb{O}}_{\min}$ is $G$-stable, we have
\begin{equation}
    \mu^{-1}(\overline{\mathbb{O}}_{\min})
    =
    G\times^P\bigl(\u\cap \overline{\mathbb{O}}_{\min}\bigr).
\end{equation}

\noindent Thus it suffices to show that $\u\cap \overline{\mathbb{O}}_{\min}=C$. Let $v\in \u\cong V^*$, and write
\begin{equation}
    X_v=
    \left(\begin{smallmatrix}
        0 & -\,v^{\mathsf T}J_V & 0\\
        0 & 0 & v\\
        0 & 0 & 0
    \end{smallmatrix}\right)
    \in \u\subset \g.
\end{equation}

\noindent A direct computation gives
\begin{equation}
    X_v^2=
    \left(\begin{smallmatrix}
        0 & 0 & -\,v^{\mathsf T}J_Vv\\
        0 & 0 & 0\\
        0 & 0 & 0
    \end{smallmatrix}\right)
    =
    \left(\begin{smallmatrix}
        0 & 0 & -\,2Q(v)\\
        0 & 0 & 0\\
        0 & 0 & 0
    \end{smallmatrix}\right),
\end{equation}

\noindent so $v\in C$ if and only if $X_v^2=0$. Now if $X_v\in \overline{\mathbb{O}}_{\min}$, then $X_v^2=0$, since the condition $X^2=0$ is Zariski closed and every element of $\mathbb{O}_{\min}$ is square-zero by Proposition~\ref{CotBundlePresent}. Thus $v\in C$, so $\u\cap \overline{\mathbb{O}}_{\min}\subset C$.

Conversely, let $0\neq v\in C$. Then $X_v^2=0$, and $X_v$ has rank $2$: indeed, $X_v(e_-)=v$, while if $w\in V$ satisfies $B(v,w)=1$, then $X_v(w)=-e_+$. Hence $\mathrm{im}(X_v)$ contains $v$ and $e_+$, and is contained in $\kappa v\oplus \kappa e_+$, so $\rk(X_v)=2$. By Proposition~\ref{CotBundlePresent}, it follows that $X_v\in \mathbb{O}_{\min}$. Thus $C\setminus\{0\}\subset \u\cap \mathbb{O}_{\min}$, and therefore $C\subset \u\cap \overline{\mathbb{O}}_{\min}$.

This proves $\u\cap \overline{\mathbb{O}}_{\min}=C$, and hence
\begin{equation}
    \mu^{-1}(\overline{\mathbb{O}}_{\min})=G\times^P C.
\end{equation}

\noindent The second statement follows immediately, since $\mu([g,0])=0$ and $0\notin \mathbb{O}_{\min}$.
\end{proof}

We now prove that the global twisted differential operators on $G/P$ are generated by the infinitesimal $G$-action. This argument concerns the full cotangent bundle $T^*(G/P)$, not only the conical subvariety $G\times^P C$ appearing as the singular support of $\mathcal H$. We note that the next few lemmata are general theory about arbitrary flag varieties. They allow us to reduce the question of surjectivity to the associated graded.

Put $X:=G/P$.

\begin{lem}\label{FilteredSurjectivityCriterionLemma}
Let
\begin{equation}
\Phi:\mathcal U(\g)\longrightarrow \Gamma(X,\mathcal D_{\mathcal L})
\end{equation}
\noindent be the map obtained by differentiating the $G$-action on $\mathcal L$. Then the associated graded of $\Phi$ gives a map 
\begin{equation}
\gr\,\Phi:\Sym(\g)\longrightarrow \Gamma(X,\Sym \, T_X).
\end{equation}
\noindent If $\gr\, \Phi$ is surjective, then $\Phi$ is surjective.
\end{lem}

\begin{proof}
The map $\Phi$ is filtered for the PBW filtration on $\mathcal U(\g)$ and the order filtration on $\Gamma(X,\mathcal D_{\mathcal L})$.

The fact that $\gr\,\Phi$ has codomain $\Gamma(X,\Sym\,T_X)$ can be seen as follows. The sheaf $\mathcal D_{\mathcal L}$ is filtered by order, and its associated graded sheaf is canonically identified with $\Sym\,T_X$. Indeed, locally trivializing the line bundle $\mathcal L$ identifies $\mathcal D_{\mathcal L}$ with the usual sheaf of differential operators $\mathcal D_X$, and changes of trivialization conjugate operators by multiplication by an invertible function. Such conjugations only affect lower-order terms, and hence act trivially on principal symbols. Therefore the associated graded sheaf is independent of the twist and is simply $\Sym\,T_X$.

Let $D\in \Gamma(X,\mathcal D_{\mathcal L})$ have order $r$. By surjectivity of $\gr\,\Phi$, there is some $u\in F_r\,\mathcal U(\g)$ such that the principal symbol of $\Phi(u)$ agrees with the principal symbol of $D$. Hence $D-\Phi(u)$ has order $<r$. Induction on $r$ proves the claim.
\end{proof}

\begin{lem}\label{MomentMapAssociatedGradedLemma}
Under the identification
\begin{equation}
\Gamma(X,\Sym \, T_X)=\Gamma(T^*X,\mathcal O_{T^*X}),
\end{equation}
\noindent the associated graded map
\begin{equation}
\Sym(\g)\longrightarrow \Gamma(X,\Sym \, T_X)
\end{equation}
\noindent is the pullback on functions induced by the moment map
\begin{equation}
\mu:T^*X\longrightarrow \g^*.
\end{equation}
\end{lem}

\begin{proof}
For $\xi\in\g$, the first-order operator $\Phi(\xi)$ has principal symbol equal to the function on $T^*X$ obtained by pairing a covector with the vector field on $X$ generated by $\xi$. This is precisely the Hamiltonian function $\langle \mu,\xi\rangle$. Since principal symbols multiply, the assertion follows for all of $\Sym(\g)$.
\end{proof}

We now specialize the associated-graded argument to the smooth quadric hypersurface $X:=G/P$. Put $W:=V^+$ and $N:=\dim W=2k+2$, so that $X$ is the smooth quadric parametrizing isotropic lines in $W$.

\begin{lem}\label{QuadricCotangentMomentMapModelLemma}
A point of $T^*X$ is represented by a pair $(x,y)\in W\times W$ such that $Q^+(x)=0$, $B^+(x,y)=0$, and $x\neq 0$, modulo $(x,y)\sim (a x,a^{-1}y)$ for $a\in\mathbb G_m$ and $(x,y)\sim (x,y+tx)$ for $t\in\mathbb G_a$. Under this presentation, the moment map $\mu:T^*X\to\g^*$ is given by $\mu([x,y])=x\wedge y\in\mathfrak{so}(W)\cong\g^*$, where $(x\wedge y)(z)=B^+(x,z)y-B^+(y,z)x$.
\end{lem}

\begin{proof}
If $[x]\in X$, then $T_{[x]}X\cong \Hom(\kappa x,x^\perp/\kappa x)$. Using $B^+$ to identify $W$ with $W^*$, a cotangent vector is represented by $y\in x^\perp$, modulo $\kappa x$. Rescaling $x$ by $a$ rescales $y$ by $a^{-1}$, and replacing $y$ by $y+tx$ gives the same covector. This gives the stated presentation.

For $\xi\in\g=\mathfrak{so}(W)$, the induced tangent vector at $[x]$ is the class of $\xi x$ in $x^\perp/\kappa x$. Pairing it with the covector represented by $y$ gives $B^+(y,\xi x)$, which is the Hamiltonian function corresponding to $x\wedge y$ under our chosen invariant identification $\mathfrak{so}(W)\cong\mathfrak{so}(W)^*$. Hence $\mu([x,y])=x\wedge y$.
\end{proof}

\begin{lem}\label{QuadricMomentRichardsonBirationalLemma}
The image of $\mu:T^*X\to\g^*$ is the Richardson orbit closure $\overline{\mathbb O}_{(3,1^{N-3})}\subset\mathfrak{so}(W)\cong\g^*$. Moreover, $\mu:T^*X\to\overline{\mathbb O}_{(3,1^{N-3})}$ is birational.
\end{lem}

\begin{proof}
Let $A=x\wedge y$, with $Q^+(x)=0$ and $B^+(x,y)=0$. Then $A(x)=0$, $A(y)=-B^+(y,y)x$, and $\operatorname{im}(A)\subset \Span\{x,y\}$. Hence $A^3=0$ and $\rk(A)\leq 2$. If $Q^+(y)\neq 0$, then $x,y$ are linearly independent, $A^2\neq 0$, and $\rk(A)=2$, so $A$ has Jordan type $(3,1^{N-3})$. Thus the image meets $\mathbb O_{(3,1^{N-3})}$.

Conversely, every point in the image has $A^3=0$ and $\rk(A)\leq 2$, so its Jordan type is $(3,1^{N-3})$, $(2,2,1^{N-4})$, or $0$, hence lies in $\overline{\mathbb O}_{(3,1^{N-3})}$. Since $T^*X\cong G\times^P\mathfrak u$ embeds as a closed subvariety of $X\times\g^*$ and $\mu$ is the restriction of the projection to $\g^*$, the map $\mu$ is proper onto its image. As $X$ is projective, the image is closed. Therefore the image is exactly $\overline{\mathbb O}_{(3,1^{N-3})}$.

It remains to prove birationality. On the dense orbit $\mathbb O_{(3,1^{N-3})}$, one has $A^2\neq 0$, and $\operatorname{im}(A^2)=\kappa x$. Thus $A$ determines $[x]\in X$. Once $[x]$ is known, the equality $A=x\wedge y$ determines $y$ modulo $\kappa x$: indeed, if $x\wedge y=x\wedge y'$, then $x\wedge(y-y')=0$, so $B^+(y-y',z)=0$ for all $z\in x^\perp$, and hence $y-y'\in (x^\perp)^\perp=\kappa x$. Thus the generic fiber is a single point.
\end{proof}

\begin{lem}\label{QuadricCotangentMomentFunctionsGenerateLemma}
The pullback map $\kappa[\g^*]=\Sym(\g)\to \Gamma(T^*X,\mathcal O_{T^*X})$ induced by the moment map is surjective.
\end{lem}

\begin{proof}
By Lemma~\ref{QuadricMomentRichardsonBirationalLemma}, $\mu$ is a proper birational morphism from $T^*X$ onto $\overline{\mathbb O}_{(3,1^{N-3})}$. This Richardson orbit closure is normal; see \cite[Section~8.6]{Jantzen:Nilpotent}. Hence Zariski's Main Theorem gives $\mu_*\mathcal O_{T^*X}=\mathcal O_{\overline{\mathbb O}_{(3,1^{N-3})}}$. Taking global sections, $\Gamma(T^*X,\mathcal O_{T^*X})\cong \kappa[\overline{\mathbb O}_{(3,1^{N-3})}]$. Since $\overline{\mathbb O}_{(3,1^{N-3})}$ is a closed subvariety of $\g^*$, its coordinate ring is a quotient of $\kappa[\g^*]=\Sym(\g)$. Therefore the moment-map pullback is surjective.
\end{proof}

\begin{prop}\label{UGSurjectsOntoGlobalTDOLemma}
The infinitesimal action map
\begin{equation}
\mathcal U(\g)\longrightarrow \Gamma(G/P,\mathcal D_{\mathcal L})
\end{equation}
\noindent is surjective.
\end{prop}

\begin{proof}
By Lemma~\ref{MomentMapAssociatedGradedLemma}, the associated graded map for $\mathcal U(\g)\to\Gamma(G/P,\mathcal D_{\mathcal L})$ is the moment-map pullback $\Sym(\g)\to\Gamma(T^*(G/P),\mathcal O_{T^*(G/P)})$. This map is surjective by Lemma~\ref{QuadricCotangentMomentFunctionsGenerateLemma}. Lemma~\ref{FilteredSurjectivityCriterionLemma} therefore implies that $\mathcal U(\g)\to \Gamma(G/P,\mathcal D_{\mathcal L})$ is surjective.
\end{proof}

\begin{rem}
The orbit closure appearing here is the Richardson variety $\overline{\mathbb O}_{(3,1^{N-3})}$, which is the image of the full cotangent bundle $T^*(G/P)$. This is larger than the minimal nilpotent orbit closure $\overline{\mathbb O}_{\min}=\overline{\mathbb O}_{(2,2,1^{N-4})}$, which appears in the quasiclassical calculation for $T^*C^o$ and in the singular-support condition for the harmonic sheaf. Thus $\Gamma(G/P,\mathcal D_{\mathcal L})$ sees the full cotangent bundle, while the harmonic quotient sees the smaller conical subvariety lying over the minimal nilpotent orbit closure.
\end{rem}

\subsection{\texorpdfstring{Vanishing of $H^1(G/P,\mathcal{J}_\Delta)$ and Noetherianness of $D_C$}{Vanishing of H1(G/P, JDelta) and Noetherianness of DC}}\label{H1VanishingNoetherianSection}

We now turn our attention towards proving that $H^1(G/P,\mathcal{J}_\Delta)=0$; this will complete the proof that the map $\rho$ \eqref{rhodef} is surjective. The short exact sequence
\begin{equation}
0 \longrightarrow \mathcal{J}_\Delta \longrightarrow \mathcal{D}_{\mathcal{L}} \longrightarrow \mathcal{H} \longrightarrow 0
\end{equation}
of sheaves on $G/P$ gives rise to the long exact sequence in cohomology:
\begin{equation}
0 \longrightarrow \Gamma(G/P, \mathcal{J}_\Delta) \longrightarrow \Gamma(G/P, \mathcal{D}_{\mathcal{L}}) \longrightarrow \Gamma(G/P, \mathcal{H}) \longrightarrow H^1(G/P, \mathcal{J}_{\Delta}) \longrightarrow \cdots .
\end{equation}

The previous subsection proves that the map $\mathcal{U}(\g)\to \Gamma(G/P,\mathcal{D}_{\mathcal L})$ induced by differentiating the action of $G$ on $G/P$ is surjective. Thus it remains to prove that $H^1(G/P,\mathcal{J}_\Delta)=0$, which implies that the natural map $\Gamma(G/P,\mathcal{D}_{\mathcal L})\to \Gamma(G/P,\mathcal H)$ is surjective.

We begin with the vanishing of $H^1(G/P,\mathcal{J}_\Delta)$. Let
\begin{equation}
X_{\mathrm{big}}:=V\cup w_0V,
\qquad
V_Q:=V\cap w_0V.
\end{equation}

\begin{lem}\label{CechReductionJDeltaLemma}
After transporting the chart $w_0V$ to the standard chart $V$ by the Kelvin transform, the \v{C}ech cohomology of $\mathcal{J}_\Delta$ on $X_{\mathrm{big}}=V\cup w_0V$ is computed by
\begin{equation}
H^1(X_{\mathrm{big}},\mathcal{J}_\Delta)
\cong
D_{V_Q}\Delta\Big/\Bigl(D_V\Delta+K(D_V\Delta)K\Bigr).
\end{equation}
\noindent In particular, if
\begin{equation}\label{LocalJDeltaSplitting}
D_{V_Q}\Delta
=
D_V\Delta+K(D_V\Delta)K,
\end{equation}
\noindent then $H^1(X_{\mathrm{big}},\mathcal{J}_\Delta)=0$.
\end{lem}

\begin{proof}
On the standard Bruhat chart $V$, the definition of the harmonic ideal gives $\Gamma(V,\mathcal{J}_\Delta)=D_V\Delta$. Since $V_Q=V\cap w_0V$ is the principal open subset obtained by inverting $Q$, restriction to the overlap gives $\Gamma(V_Q,\mathcal{J}_\Delta)=D_{V_Q}\Delta$. On the opposite chart $w_0V$, the corresponding generator of the harmonic ideal is transported to the standard chart by Kelvin conjugation. Thus, after identifying the overlap with $V_Q$, the image of $\Gamma(w_0V,\mathcal{J}_\Delta)$ inside $\Gamma(V_Q,\mathcal{J}_\Delta)$ is $K(D_V\Delta)K$.

The open sets $V$, $w_0V$, and $V_Q$ are affine. Moreover, $\mathcal{J}_\Delta$ is quasi-coherent as an $\mathcal O$-module: on each Bruhat chart it is the filtered union of its coherent order-filtered pieces $F_m\mathcal{J}_\Delta$. Therefore the two-open cover $X_{\mathrm{big}}=V\cup w_0V$ is acyclic for $\mathcal{J}_\Delta$, and the \v{C}ech complex computes $H^1(X_{\mathrm{big}},\mathcal{J}_\Delta)$. Its degree-one cokernel is
\begin{equation}
\Gamma(V_Q,\mathcal{J}_\Delta)
\Big/
\Bigl(
\Gamma(V,\mathcal{J}_\Delta)|_{V_Q}
+
\Gamma(w_0V,\mathcal{J}_\Delta)|_{V_Q}
\Bigr),
\end{equation}
\noindent which, by the identifications above, is precisely
\begin{equation}
D_{V_Q}\Delta\Big/\Bigl(D_V\Delta+K(D_V\Delta)K\Bigr).
\end{equation}
\noindent This proves the claimed identification, and \eqref{LocalJDeltaSplitting} immediately implies $H^1(X_{\mathrm{big}},\mathcal{J}_\Delta)=0$.
\end{proof}

\begin{lem}\label{NullPowerSpanHarmonicsLemma}
Let $P_d:=\kappa[V]_d$, and let
\begin{equation}
Z_{\Delta,d}:=\ker(\Delta)\cap P_d.
\end{equation}
\noindent Then $Z_{\Delta,d}$ is spanned by the powers $\ell^d$, where $\ell\in V^*$ ranges over the null linear forms:
\begin{equation}
Q^*(\ell)=0.
\end{equation}
\noindent Equivalently,
\begin{equation}
Z_{\Delta,d}
=
\Span_\kappa\{\ell^d: \ell\in V^*,\ Q^*(\ell)=0\}.
\end{equation}
\end{lem}

\begin{proof}
Let
\begin{equation}
W_d:=\Span_\kappa\{\ell^d: \ell\in V^*,\ Q^*(\ell)=0\}\subset P_d.
\end{equation}
\noindent We prove that $W_d=Z_{\Delta,d}$.

Consider the pairing on $P_d$:
\begin{equation}
\langle f,g\rangle:=(f(\partial))(g)\big|_{0}.
\end{equation}
\noindent This pairing is nondegenerate. Moreover, multiplication by $Q$ is adjoint to $\Delta$ with respect to this pairing. Hence
\begin{equation}
Z_{\Delta,d}
=
(QP_{d-2})^\perp.
\end{equation}
\noindent Indeed, for $h\in P_d$ and $a\in P_{d-2}$, one has $\langle h,Qa\rangle=\langle \Delta h,a\rangle$, so $h$ is orthogonal to $QP_{d-2}$ if and only if $\Delta h=0$.

We now compute $W_d^\perp$. For $f\in P_d$ and a linear form $\ell\in V^*$, one has $\langle f,\ell^d\rangle=d!\,f(\ell^\sharp)$, where $\ell^\sharp\in V$ is the vector corresponding to $\ell$ under the bilinear form. Therefore $f\in W_d^\perp$ if and only if $f$ vanishes on the null cone $Q=0$. Since $Q$ is irreducible under our standing hypotheses, this is equivalent to $f\in QP_{d-2}$. Thus $W_d^\perp=QP_{d-2}$.

Taking orthogonal complements with respect to the pairing, and applying the Fischer Decomposition Lemma (\ref{FischerDecompositionLemma}) gives
\begin{equation}
W_d=(QP_{d-2})^\perp=Z_{\Delta,d}.
\end{equation}
\noindent This proves the claim.
\end{proof}

Let
\begin{equation}
\mathcal S:=D_V\Delta+K(D_V\Delta)K\subset D_{V_Q}\Delta.
\end{equation}
\noindent Our goal is to prove that $\mathcal S=D_{V_Q}\Delta$.

\begin{lem}\label{CoefficientNullPowerNormalFormLemma}
Every element of $D_{V_Q}\Delta$ is a finite sum of terms of the form
\begin{equation}
Q^{-q}\ell^s R\Delta,
\end{equation}
\noindent where $q\in\mathbb Z$, $s\geq 0$, $\ell\in V^*$ is null, and $R\in\kappa[\partial]$ is a constant-coefficient differential operator.
\end{lem}

\begin{proof}
Apply the Fischer Decomposition (Lemma \ref{FischerDecompositionLemma}) to the polynomial-parts of a differential operator under normal ordering; then apply Lemma \ref{NullPowerSpanHarmonicsLemma} to write each harmonic term as a sum of powers $\ell^s$.
\end{proof}

We will also need the following lemma regarding the Kelvin transform of constant-coefficient differential operators $\kappa[\partial]$:

\begin{lem}\label{KelvinConjugatesConstantOperatorsLemma}
We have
\begin{equation}
K\kappa[\partial]K\subset D_V.
\end{equation}
\noindent In particular, if $u\in V$ and $\partial_u$ denotes the corresponding constant vector field, then
\begin{equation}
K\partial_uK
=
-Q\partial_u+dQ_v(u)E+m\,dQ_v(u),
\end{equation}
\noindent where $m=k-1$.
\end{lem}

\begin{proof}
Let $I(v):=-v/Q(v)$, so that $K(f)=(-Q)^{-m}f\circ I$. A direct differentiation gives $dI_v(u)
=
-\frac{u}{Q(v)}
+
\frac{dQ_v(u)}{Q(v)^2}v$.
Applying the definition of $K$ twice, and using $I^2=\Id$, $Q(I(v))=Q(v)^{-1}$, and $dQ_{I(v)}(u)=-Q(v)^{-1}dQ_v(u)$,
we obtain
$(K\partial_uK)(f)
=
df_v\bigl(-Q(v)u+dQ_v(u)v\bigr)+m\,dQ_v(u)f$. This is precisely
\begin{equation}
K\partial_uK
=
-Q\partial_u+dQ_v(u)E+m\,dQ_v(u),
\end{equation}
\noindent so $K\partial_uK$ is regular on $V$. Since $\kappa[\partial]$ is generated by the constant vector fields $\partial_u$, the inclusion $K\kappa[\partial]K\subset D_V$ follows.
\end{proof}

\begin{lem}\label{SStableConstantOperatorsLemma}
Recall that
\begin{equation}
\mathcal S:=D_V\Delta+K(D_V\Delta)K\subset D_{V_Q}\Delta.
\end{equation}
\noindent Then $\mathcal S$ is stable under left multiplication by $\kappa[\partial]$.
\end{lem}

\begin{proof}
Let $R\in\kappa[\partial]$. Since $D_V\Delta$ is a left ideal in $D_V$, one has
\begin{equation}
R(D_V\Delta)\subset D_V\Delta.
\end{equation}
\noindent It remains to check the opposite-chart summand. Let $A\in D_V$. Since $K^2=\Id$, we have
\begin{equation}
R\,K(A\Delta)K
=
K(KRK)A\Delta K.
\end{equation}
\noindent By Lemma~\ref{KelvinConjugatesConstantOperatorsLemma}, $KRK\in D_V$. Hence $(KRK)A\in D_V$, and therefore
\begin{equation}
R\,K(A\Delta)K\in K(D_V\Delta)K.
\end{equation}
\noindent Thus both summands of $\mathcal S$ are stable under left multiplication by $R$, and the claim follows.
\end{proof}

\begin{definition}
For a term of the form
\begin{equation}
Q^{-q}\ell^sR\Delta,
\end{equation}
\noindent with $q\geq 0$, $\ell$ a null linear form, $s\geq 0$, and $R\in\kappa[\partial]$, we define its \textbf{excess} to be
\begin{equation}
\operatorname{exc}(q,s):=s-q-2.
\end{equation}
\noindent We also define its \textbf{positive excess} by
\begin{equation}
\operatorname{exc}_+(q,s):=\max\{\operatorname{exc}(q,s),0\}.
\end{equation}
\end{definition}

We first treat coefficient terms, beginning with the range in which the term is captured directly by the opposite chart.

\begin{lem}\label{DirectNullPowerCaptureLemma}
Let $\ell\in V^*$ be null. If $q\geq 0$ and $s\leq q+2$, then
\begin{equation}
Q^{-q}\ell^s\Delta\in K(D_V\Delta)K.
\end{equation}
\noindent More precisely,
\begin{equation}
Q^{-q}\ell^s\Delta
=
K\left((-1)^sQ^{q+2-s}\ell^s\Delta\right)K.
\end{equation}
\end{lem}

\begin{proof}
Since $s\leq q+2$, the coefficient $Q^{q+2-s}\ell^s$ is polynomial. Under Kelvin conjugation, multiplication by $\ell^s$ transforms into multiplication by $(-1)^sQ^{-s}\ell^s$, while multiplication by $Q^{q+2-s}$ transforms into multiplication by $Q^{-q-2+s}$. Therefore multiplication by $(-1)^sQ^{q+2-s}\ell^s$ transforms into multiplication by $Q^{-q-2}\ell^s$. Using $K\Delta K=Q^2\Delta$, we obtain
\begin{equation}
K\left((-1)^sQ^{q+2-s}\ell^s\Delta\right)K
=
Q^{-q-2}\ell^sQ^2\Delta
=
Q^{-q}\ell^s\Delta.
\end{equation}
\noindent This proves the claim.
\end{proof}

We now show how to reduce positive excess.

\begin{lem}\label{KelvinVectorFieldCommutatorLemma}
Let $\ell\in V^*$ be null, and let $u\in V$ be the corresponding vector, so that $dQ_v(u)=\ell(v)$. Put
\begin{equation}
T_\ell:=K\partial_uK.
\end{equation}
\noindent Then
\begin{equation}
[\Delta,T_\ell]=2\ell\Delta.
\end{equation}
\noindent Moreover, $\mathcal S$ is stable under both left and right multiplication by $T_\ell$.
\end{lem}

\begin{proof}
By Lemma~\ref{KelvinConjugatesConstantOperatorsLemma},
\begin{equation}
T_\ell=-Q\partial_u+\ell E+m\ell,
\end{equation}
\noindent where $m=k-1$. We compute the commutator with $\Delta$. Using $[\Delta,Q]=E+k$, $[\Delta,\ell]=\partial_u$, $[\Delta,E]=2\Delta$, and $[\Delta,\partial_u]=0$, we find
\begin{align}
[\Delta,T_\ell]
&=
-[\Delta,Q]\partial_u+[\Delta,\ell E]+m[\Delta,\ell] \nonumber\\
&=
-(E+k)\partial_u+\partial_uE+2\ell\Delta+m\partial_u.
\end{align}
\noindent Since $\partial_uE=(E+1)\partial_u$ and $m=k-1$, the first-order terms cancel. Hence
\begin{equation}
[\Delta,T_\ell]=2\ell\Delta.
\end{equation}

We next prove the stability assertions. Since $T_\ell\in D_V$, left multiplication by $T_\ell$ preserves the left ideal $D_V\Delta$. Also, for $C\in D_V$,
\begin{equation}
T_\ell K(C\Delta)K
=
K(\partial_u C\Delta)K\in K(D_V\Delta)K.
\end{equation}
\noindent Thus $\mathcal S$ is stable under left multiplication by $T_\ell$.

For right multiplication, the relation $[\Delta,T_\ell]=2\ell\Delta$ gives
\begin{equation}
\Delta T_\ell=(T_\ell+2\ell)\Delta.
\end{equation}
\noindent Hence $D_V\Delta\cdot T_\ell\subset D_V\Delta$. On the opposite-chart summand, using $T_\ell=K\partial_uK$ and $K^2=\Id$, we have
\begin{equation}
K(C\Delta)K\,T_\ell
=
K(C\Delta\partial_u)K
=
K(C\partial_u\Delta)K\in K(D_V\Delta)K.
\end{equation}
\noindent Thus $\mathcal S$ is also stable under right multiplication by $T_\ell$.
\end{proof}

\begin{lem}\label{AllNullPowerCoefficientTermsLemma}
Let $\ell\in V^*$ be null. For every $q\geq 0$ and every $s\geq 0$, one has
\begin{equation}
Q^{-q}\ell^s\Delta\in \mathcal S.
\end{equation}
\end{lem}

\begin{proof}
We argue by induction on $\operatorname{exc}_+(q,s)$. If $\operatorname{exc}_+(q,s)=0$, then $s\leq q+2$, and the claim follows from Lemma~\ref{DirectNullPowerCaptureLemma}.

Now suppose $s>q+2$, and assume the claim is known with $s$ replaced by $s-1$. Put
\begin{equation}
A:=Q^{-q}\ell^{s-1}\Delta.
\end{equation}
\noindent By the induction hypothesis, $A\in\mathcal S$. Since $\mathcal S$ is stable under both left and right multiplication by $T_\ell$ by Lemma~\ref{KelvinVectorFieldCommutatorLemma}, it follows that $[T_\ell,A]\in\mathcal S$.

We compute this commutator. First,
\begin{equation}
T_\ell(Q^{-q}\ell^{s-1})
=
(s-1-q+m)Q^{-q}\ell^s.
\end{equation}
\noindent Indeed, $T_\ell=-Q\partial_u+\ell E+m\ell$, while $\partial_u(\ell)=0$, $\partial_u(Q)=\ell$, and $E(Q^{-q}\ell^{s-1})=(-2q+s-1)Q^{-q}\ell^{s-1}$. Therefore
\begin{align}
[T_\ell,A]
&=
T_\ell(Q^{-q}\ell^{s-1})\Delta
+
Q^{-q}\ell^{s-1}[T_\ell,\Delta] \nonumber\\
&=
(s-1-q+m)Q^{-q}\ell^s\Delta
-
2Q^{-q}\ell^s\Delta \nonumber\\
&=
(s-q+m-3)Q^{-q}\ell^s\Delta.
\end{align}
\noindent Since $s>q+2$, we have $s\geq q+3$, and hence $s-q+m-3\geq m$. Since $m=k-1\geq 1$, the scalar $s-q+m-3$ is nonzero. Thus $Q^{-q}\ell^s\Delta\in\mathcal S$.
\end{proof}

\begin{lem}\label{LocalizedCoefficientWithConstantOperatorsLemma}
Let $f\in\kappa[V][1/Q]$ and let $R\in\kappa[\partial]$. Then
\begin{equation}
fR\Delta\in\mathcal S.
\end{equation}
\end{lem}

\begin{proof}
We first treat the case $R=1$. By Lemma~\ref{CoefficientNullPowerNormalFormLemma}, the element $f\Delta\in D_{V_Q}\Delta$ is a finite sum of terms $Q^{-q}\ell^s\Delta$, with $q\in\mathbb Z$ and $\ell$ null. If $q<0$, such a term lies in $D_V\Delta\subset\mathcal S$; if $q\geq 0$, it lies in $\mathcal S$ by Lemma~\ref{AllNullPowerCoefficientTermsLemma}. Hence $f\Delta\in\mathcal S$.

We now argue by induction on $\ord\, R$. The case $\ord\, R=0$ was just proved. Suppose $\ord\, R>0$ and that the claim is known for smaller order. We write
\begin{equation}
fR\Delta
=
R(f\Delta)-[R,f]\Delta.
\end{equation}
\noindent The first term lies in $\mathcal S$ by the case $R=1$ and Lemma~\ref{SStableConstantOperatorsLemma}. Since $R$ has constant coefficients, the commutator $[R,f]$ is a finite sum of terms $f_jR_j$, where $f_j\in\kappa[V][1/Q]$ and $\ord\, R_j<\ord\, R$. By the induction hypothesis, each $f_jR_j\Delta$ lies in $\mathcal S$. Thus $[R,f]\Delta\in\mathcal S$, and therefore $fR\Delta\in\mathcal S$.
\end{proof}

\begin{prop}\label{LocalJDeltaSplittingProposition}
We have
\begin{equation}
D_{V_Q}\Delta
=
D_V\Delta+K(D_V\Delta)K.
\end{equation}
\end{prop}

\begin{proof}
The inclusion $D_V\Delta+K(D_V\Delta)K\subset D_{V_Q}\Delta$ is clear. Conversely, by normal ordering, every element of $D_{V_Q}\Delta$ is a finite sum of terms
\begin{equation}
fR\Delta,
\end{equation}
\noindent with $f\in\kappa[V][1/Q]$ and $R\in\kappa[\partial]$. By Lemma~\ref{LocalizedCoefficientWithConstantOperatorsLemma}, every such term lies in $\mathcal S$. Therefore $D_{V_Q}\Delta\subset\mathcal S$, proving the equality.
\end{proof}

\begin{lem}\label{HOneXbigJDeltaVanishingLemma}
We have
\begin{equation}
H^1(X_{\mathrm{big}},\mathcal J_\Delta)=0.
\end{equation}
\end{lem}

\begin{proof}
This follows immediately from Lemma~\ref{CechReductionJDeltaLemma} and Proposition~\ref{LocalJDeltaSplittingProposition}.
\end{proof}

\begin{thm}\label{HOneJDeltaVanishingTheorem}
We have
\begin{equation}
H^1(G/P,\mathcal J_\Delta)=0.
\end{equation}
\end{thm}

\begin{proof}
Put $X:=G/P$ and let
\begin{equation}
Z:=X\setminus X_{\mathrm{big}}.
\end{equation}
\noindent Since $X_{\mathrm{big}}=V\cup w_0V$, the complement is the closed subvariety
\begin{equation}
Z=\{[0:v:0]:Q(v)=0\}\cong \mathbb P(C_V).
\end{equation}
\noindent In particular, $Z$ has codimension $2$ in $X$.

We use the long exact sequence for cohomology with supports in $Z$:
\begin{equation}
\cdots
\longrightarrow
H^1_Z(X,\mathcal J_\Delta)
\longrightarrow
H^1(X,\mathcal J_\Delta)
\longrightarrow
H^1(X_{\mathrm{big}},\mathcal J_\Delta)
\longrightarrow
\cdots .
\end{equation}
\noindent By Lemma~\ref{HOneXbigJDeltaVanishingLemma}, the last term is zero. It therefore remains to show that
\begin{equation}
H^1_Z(X,\mathcal J_\Delta)=0.
\end{equation}

We prove this using the order filtration. The sheaf $\mathcal J_\Delta$ is the filtered union of its order-filtered pieces:
\begin{equation}
\mathcal J_\Delta=\varinjlim_m F_m\mathcal J_\Delta.
\end{equation}
\noindent Each $F_m\mathcal J_\Delta$ is a coherent locally free $\mathcal O_X$-module. Indeed, locally on a Bruhat chart one has $F_m\mathcal J_\Delta\simeq F_{m-2}\mathcal D_{\mathcal L}$ by multiplication by the local generator of the harmonic ideal.

Since $X$ is smooth and $Z$ has codimension $2$, the standard depth-vanishing theorem for local cohomology gives
\begin{equation}
H^i_Z(X,F_m\mathcal J_\Delta)=0
\qquad
\text{for } i=0,1
\end{equation}
\noindent for every $m$. Local cohomology commutes with filtered colimits of quasi-coherent sheaves on the noetherian scheme $X$, so
\begin{equation}
H^1_Z(X,\mathcal J_\Delta)
=
\varinjlim_m H^1_Z(X,F_m\mathcal J_\Delta)
=
0.
\end{equation}
\noindent Therefore the map
\begin{equation}
H^1(X,\mathcal J_\Delta)\longrightarrow H^1(X_{\mathrm{big}},\mathcal J_\Delta)
\end{equation}
\noindent is injective. Since the target is zero, we conclude that $H^1(X,\mathcal J_\Delta)=0$.
\end{proof}

\begin{cor}\label{DLSurjectsOntoHCorollary}
The natural map
\begin{equation}
\Gamma(G/P,\mathcal D_{\mathcal L})
\longrightarrow
\Gamma(G/P,\mathcal H)
\end{equation}
\noindent is surjective.
\end{cor}

\begin{proof}
This follows immediately from the short exact sequence
\begin{equation}
0 \longrightarrow \mathcal J_\Delta \longrightarrow \mathcal D_{\mathcal L} \longrightarrow \mathcal H \longrightarrow 0
\end{equation}
\noindent and Theorem~\ref{HOneJDeltaVanishingTheorem}. Indeed, the associated long exact sequence in cohomology contains
\begin{equation}
\Gamma(G/P,\mathcal D_{\mathcal L})
\longrightarrow
\Gamma(G/P,\mathcal H)
\longrightarrow
H^1(G/P,\mathcal J_\Delta).
\end{equation}
\noindent Since $H^1(G/P,\mathcal J_\Delta)=0$, the first map is surjective.
\end{proof}

\begin{rem}
Corollary \ref{DLSurjectsOntoHCorollary} is the statement needed below: every global section of $\mathcal H$ lifts to a global section of $\mathcal D_{\mathcal L}$. All that Corollary \ref{DLSurjectsOntoHCorollary} requires is that the connecting map $\Gamma(G/P,\mathcal H)\to H^1(G/P,\mathcal J_\Delta)$ is 0; Theorem \ref{HOneJDeltaVanishingTheorem} proves the stronger statement that the entire cohomology group $H^1(G/P,\mathcal J_\Delta)$ vanishes.
\end{rem}

\begin{thm}\label{rhoSurjective}
The map
\begin{equation}
\rho:\mathcal U(\g)\longrightarrow D_C
\end{equation}
\noindent is surjective. In particular, $D_C$ is finitely generated and Noetherian.
\end{thm}

\begin{proof}
By Proposition~\ref{UGSurjectsOntoGlobalTDOLemma}, the map
\begin{equation}
\mathcal U(\g)\longrightarrow \Gamma(G/P,\mathcal D_{\mathcal L})
\end{equation}
\noindent induced by the infinitesimal $G$-action is surjective. By Corollary~\ref{DLSurjectsOntoHCorollary}, the natural map
\begin{equation}
\Gamma(G/P,\mathcal D_{\mathcal L})
\longrightarrow
\Gamma(G/P,\mathcal H)
\end{equation}
\noindent is also surjective. Hence the composite
\begin{equation}
\mathcal U(\g)
\longrightarrow
\Gamma(G/P,\mathcal D_{\mathcal L})
\longrightarrow
\Gamma(G/P,\mathcal H)
\end{equation}
\noindent is surjective.

By Theorem~\ref{GlobalSectionsTheorem}, we have a canonical identification
\begin{equation}
\Gamma(G/P,\mathcal H)\cong D_C.
\end{equation}
\noindent Under this identification, the composite above is precisely the map $\rho:\mathcal U(\g)\to D_C$. Therefore $\rho$ is surjective.

Finally, $\mathcal U(\g)$ is a finitely generated Noetherian $\kappa$-algebra by the PBW theorem. Since $D_C$ is a quotient of $\mathcal U(\g)$, it is also finitely generated and Noetherian.
\end{proof}

\begin{rem}
    We note that this proof is independent of the argument in \cite{LevasseurSmithStafford1989Joseph}. Thus, if we let $\mathcal{J} := \ker(\rho)$, so that $D_C \cong \mathcal{U}(\g)/\mathcal{J}$, one may take this as an independent construction of the Joseph ideal.
\end{rem}

\subsection{\texorpdfstring{Some Compatibilities; the Bimodule Structure of $\mathcal{H}$}{Some Compatibilities and the Bimodule Structure of H}}\label{CompatibilitiesSect}

Let us return for the moment to the action of $G$ on $D_C$ (Proposition \ref{G-actsdef}). Recall that this is given by the descent of the Adjoint action of $G$ on $\mathcal{U}(\g)$ to $D_C$. On the other hand, we also have an action of $G$ on $\Gamma(\mathcal{H}) \cong D_C$ given by the action of $G$ on $G/P$ (cf. Remark \ref{twoactsremk}). We must verify that these are in fact the same action.

\begin{prop}\label{IndActionDef}
$\mathcal{H}$ is a $G$-equivariant $\mathcal{D}_{\mathcal L}$-module. Therefore $G$ acts on $\Gamma(\mathcal H)$. Under the identification of $\Gamma(\mathcal{H})$ with $D_C$ from Theorem \ref{GlobalSectionsTheorem}, this action agrees with the action $\alpha$ (cf.\ Proposition~\ref{G-actsdef}) coming from the adjoint action of $G$ on $\mathcal U(\g)$ and then applying $\rho$.
\end{prop}

\begin{proof}
The action of $G$ on $G/P$ lifts to an action on the equivariant line bundle $\mathcal L$, and hence on the sheaf $\D_{\mathcal L}$ of twisted differential operators. Thus $G$ acts on $\D_{\mathcal L}$ by algebra automorphisms. Now, $\Delta$ is not $G$-invariant as a local operator on the standard Bruhat chart. However, the sheaf of left ideals $\mathcal J_\Delta$ is $G$-stable: equivalently, on Bruhat charts the local generator $\Delta_g$ is carried to an invertible multiple of the corresponding local generator on the translated chart (cf.\ \ref{eq:Delta-semi-invariant}). Therefore $G$ acts on the quotient sheaf $\mathcal H=\D_{\mathcal L}/\mathcal{J}_\Delta$; i.e., $\mathcal H$ is a $G$-equivariant $\D_{\mathcal L}$-module. In particular, $G$ acts on global sections $\Gamma(\mathcal H)$.

Consider the diagram:
\begin{equation}\label{firstcompatdiagram}
\begin{tikzcd}
\mathcal{U}(\mathfrak{g}) \arrow[rd, "\mu"'] \arrow[r, "d(\textrm{act}_G)"'] \arrow[dd, dashed] \arrow[rr, "\phi", bend left] & \Gamma(\mathcal{D}_\mathcal{L}) \arrow[d] \arrow[r, "\textrm{res}|_V"'] & D_V \arrow[d] \\
                                                                                                                              & \Gamma(\mathcal{H}) \arrow[r, "\textrm{res}|_V"] \arrow[ld, "\sim"']    & D_V/D_V\Delta \\
N(D_V\Delta)/D_V\Delta \arrow[rru, hook, bend right]                                                                          &                                                                         &              
\end{tikzcd}
\end{equation}

\noindent Here $d(\textrm{act}_G)$ means the map induced by the infinitesimalized action of $G$, and $\mu$ is defined to be the composition of $d(\textrm{act}_G)$ with the natural map $\mathcal{D}_\mathcal{L}(G/P) \to \mathcal{H}(G/P)$ from global sections of the sheaf $\mathcal{D}_\mathcal{L}$ to global sections of the quotient sheaf $\mathcal{H} = \mathcal{D}_\mathcal{L}/\mathcal{J}_\Delta$.\footnote{The second map is surjective by Corollary~\ref{DLSurjectsOntoHCorollary}; hence $\mu$ is surjective after also applying Proposition~\ref{UGSurjectsOntoGlobalTDOLemma}.} Equivalently, it is given by $\xi \mapsto \xi\cdot[1]$, where $\xi \in \Gamma(\mathcal{D}_\mathcal{L})$ and $[1]$ is the global section of $\mathcal{H}$ coming from the identity in $\mathcal{D}_\mathcal{L}$. The identification of sections of $\mathcal{D}_\mathcal{L}$ with $D_V$ comes from the local trivialization $\sigma_0$ of $\mathcal{L}$ (cf. \eqref{sigm0def}, Proposition \ref{linetwistingProp}). By Proposition \ref{linetwistingProp}, the map $\phi$ is precisely as in \eqref{phideftwist} (and calculated explicitly in Proposition \ref{ConfActionProp}). The fact that $\res|_V \circ\mu$ factors through $N(D_V\Delta)/D_V\Delta$ is a restatement of Theorem \ref{goncharovresult}. The isomorphism of $\Gamma(\mathcal{H})$ with $N(D_V\Delta)/D_V\Delta$ is Theorem \ref{GlobalSectionsTheorem}. Commutation of the rest of the diagram then follows from definitions. 

Now we consider 
\begin{equation}\label{Adjcompatibilitydiagram}
\begin{tikzcd}
\mathcal{U}(\mathfrak{g}) \arrow[d, "\mathrm{Ad}(g)"] \arrow[r, "\mu"'] \arrow[rrr, "\rho", bend left] & \Gamma(\mathcal{H}) \arrow[d, "\gamma(g)"] \arrow[r, "\textrm{res}|_V"'] & N(D_V \Delta)/D_V\Delta \arrow[d, "\beta(g)"] \arrow[r, "\widehat{\tau}"'] & D_C \arrow[d, "\alpha(g)"] \\
\mathcal{U}(\mathfrak{g}) \arrow[r, "\mu"] \arrow[rrr, "\rho", bend right]                             & \Gamma(\mathcal{H}) \arrow[r, "\textrm{res}|_V"]                         & N(D_V \Delta)/D_V\Delta \arrow[r, "\widehat{\tau}"]                        & D_C                       
\end{tikzcd}
\end{equation}

\noindent where $\widehat{\tau}$ is the corrected Fourier-to-cone isomorphism of Definition \ref{hattaudef}.

The fact that $\rho$ is equal to the horizontal compositions follows from definitions, \eqref{firstcompatdiagram}, Definition \ref{hattaudef}, and Remark \ref{tauisom}. We define $\gamma(g)$ to be the action of $g \in G(\kappa)$ on a global section of $\mathcal{H}$, and $\beta(g)$ to be the corresponding isomorphism transferred to $N(D_V\Delta)/D_V\Delta$ (since by Theorem \ref{GlobalSectionsTheorem} the central horizontal $\res|_V$ in \eqref{Adjcompatibilitydiagram} is an isomorphism).

All we need to do is to verify that the left square commutes. Letting $\xi \in \mathcal{U}(\g)$:
\begin{align}
\gamma(g)\bigl(\mu(\xi)\bigr)
&=
\gamma(g)\bigl(d(\mathrm{act}_G)(\xi)\cdot [1]\bigr) \\
&=
\bigl(g\cdot d(\mathrm{act}_G)(\xi)\cdot g^{-1}\bigr)\cdot \gamma(g)([1]) \nonumber\\
&=
d(\mathrm{act}_G)(\Ad(g)\xi)\cdot [1] \nonumber\\
&=
\mu(\Ad(g)\xi). \nonumber
\end{align}

\noindent Here $\gamma(g)([1])=[1]$, since $[1]$ is the class of the identity operator. This proves that the left square commutes. Since $\rho$ is surjective by Theorem \ref{rhoSurjective}, it follows that the induced action on $D_C$ agrees with $\alpha$.
\end{proof}

\begin{prop}\label{prop:harmonic-bimodule}
The harmonic sheaf $\mathcal{H}:=\mathcal{D}_{\mathcal{L}}/\mathcal{J}_\Delta$ carries a canonical right action of $D_C$, commuting with the left $\mathcal{D}_{\mathcal{L}}$-action, and hence is a $\mathcal{D}_{\mathcal{L}}$-$D_C$ bimodule.

Concretely, on a Bruhat chart $V_g=gV$ we identify
\begin{equation}
\mathcal{H}(V_g)\;=\;D_{V_g}/D_{V_g}\Delta_g
\end{equation}
\noindent via the trivialization $\sigma_g$ (cf.\ \ref{sigmagdef}). Given $\xi\in D_C$ and $m\in\mathcal{H}(V_g)$, let $\theta_{g,\xi}\in N(D_{V_g}\Delta_g)/D_{V_g}\Delta_g$ denote the transport, via $g:V\xrightarrow{\sim}V_g$ and $\sigma_g$, of the class
\begin{equation}
\widehat{\tau}^{-1}\bigl(\alpha(g^{-1})(\xi)\bigr)\in N(D_V\Delta)/D_V\Delta,
\end{equation}
\noindent where $\widehat{\tau}$ is as in Definition \ref{hattaudef} and Remark \ref{tauisom}. Equivalently, if $\widetilde{\alpha(g^{-1})(\xi)}\in N(D_{V^*}Q^*)$ is any representative of $\alpha(g^{-1})(\xi)$ under the identification
\begin{equation}
D_C \cong N(D_{V^*}Q^*)/D_{V^*}Q^*,
\end{equation}
\noindent then $\theta_{g,\xi}$ is represented by the transport to $D_{V_g}$ of
\begin{equation}
\tau^{-1}\bigl((Q^*)^{-1}\widetilde{\alpha(g^{-1})(\xi)}Q^*\bigr)\in N(D_V\Delta)/D_V\Delta.
\end{equation}
\noindent We then set
\begin{equation}\label{eq:right-action-formula}
m\cdot \xi \;:=\; m\cdot \theta_{g,\xi},
\end{equation}
\noindent where the right-hand side is right multiplication by any representative of $\theta_{g,\xi}$, followed by passage to the quotient $D_{V_g}/D_{V_g}\Delta_g$.
\end{prop}

\begin{proof}
Fix $\xi\in D_C$. For each Bruhat chart $V_g$, set
\begin{equation}
I_g:=D_{V_g}\Delta_g.
\end{equation}
Then $\mathcal H(V_g)=D_{V_g}/I_g$ carries a natural right action of $N(I_g)/I_g$.

By Definition \ref{hattaudef} and Remark \ref{tauisom}, the class $\widehat{\tau}^{-1}\bigl(\alpha(g^{-1})(\xi)\bigr)\in N(D_V\Delta)/D_V\Delta$ transports, via $g:V\xrightarrow{\sim}V_g$ and $\sigma_g$, to a class
\begin{equation}
\theta_{g,\xi}\in N(I_g)/I_g.
\end{equation}
\noindent Right multiplication by any representative of $\theta_{g,\xi}$ therefore defines a left $D_{V_g}$-linear endomorphism
\begin{equation}
R_{g,\xi}:\mathcal H(V_g)\to \mathcal H(V_g),
\qquad
m\mapsto m\cdot \theta_{g,\xi}.
\end{equation}
\noindent This is independent of the chosen representative, since changing the representative changes the right-hand side by an element of $I_g$, and right multiplication by an element of $I_g$ acts trivially on $D_{V_g}/I_g$.

By Theorem \ref{GlobalSectionsTheorem} and Proposition \ref{IndActionDef}, the section
\begin{equation}
R_{g,\xi}([1])=[1]\cdot \theta_{g,\xi}
\end{equation}
\noindent is precisely the restriction to $V_g$ of the global section of $\mathcal H$ corresponding to $\xi\in D_C$.

Now let $U=V_g\cap V_h$. Then
\begin{equation}
R_{g,\xi}([1])|_U=R_{h,\xi}([1])|_U,
\end{equation}
\noindent since both are equal to the restriction to $U$ of the same global section. On the other hand, both $R_{g,\xi}|_U$ and $R_{h,\xi}|_U$ are left $\mathcal D_{\mathcal L}(U)$-linear endomorphisms of $\mathcal H(U)$, and $\mathcal H(U)$ is generated by $[1]$ as a left $\mathcal D_{\mathcal L}(U)$-module. Therefore
\begin{equation}
R_{g,\xi}|_U=R_{h,\xi}|_U.
\end{equation}
\noindent Thus the local endomorphisms $R_{g,\xi}$ glue to a global endomorphism of $\mathcal H$. We denote the resulting right action by
\begin{equation}
m\cdot \xi:=R_\xi(m).
\end{equation}
\noindent On each chart this is right multiplication in the ring $N(I_g)/I_g$, so the module axioms follow immediately. Moreover, left multiplication by $\mathcal D_{\mathcal L}$ and right multiplication by $N(I_g)/I_g$ commute on each chart. Hence the resulting right $D_C$-action commutes with the left $\mathcal D_{\mathcal L}$-action, and we conclude that $\mathcal H$ is a $\mathcal D_{\mathcal L}$-$D_C$ bimodule.
\end{proof}

\begin{rem}\label{rem:bimodule-sheaf}
We recall that $\mathcal{D}_{\mathcal{L}}$ is a sheaf of algebras on $G/P$, while $D_C$ is a $\kappa$-algebra. By a ``$\mathcal{D}_{\mathcal{L}}$-$D_C$ \textbf{bimodule}", we mean a sheaf $\mathcal{M}$ of $\kappa$-vector spaces on $G/P$ equipped with a left $\mathcal{D}_{\mathcal{L}}$-module structure and a right $D_C$-module structure, such that for every open set $U\subset G/P$ the space $\mathcal{M}(U)$ is a $(\mathcal{D}_{\mathcal{L}}(U),D_C)$-bimodule and the restriction maps are bimodule homomorphisms. Equivalently, $\mathcal{M}$ is a left module over the sheaf of rings $\mathcal{D}_{\mathcal{L}}\otimes_{\kappa} D_C^{\mathrm{op}}$.

In particular, if $\mathcal{M}$ is a $\mathcal{D}_{\mathcal{L}}$-$D_C$ bimodule, then for any left $D_C$-module $N$ the tensor product $\mathcal{M}\otimes_{D_C}N$ is a well-defined left $\mathcal{D}_{\mathcal{L}}$-module, namely the sheafification of the presheaf
\begin{equation}
U\longmapsto \mathcal{M}(U)\otimes_{D_C}N.
\end{equation}
\end{rem}

We presently review the relevant properties of $\mathcal{H}$ and summarize them below.

\begin{prop}\label{HarmonicSummary}
    The harmonic sheaf $\mathcal{H}$ is a harmonic $G$-equivariant $\mathcal{D}_{\mathcal{L}}$-module on $G/P$ with singular support $\mathrm{SS}(\mathcal{H})=G\times^P C \subset G \times^P \mathfrak{u}$; the image of this singular support under the moment map $T^*G/P \to \g$ is precisely the closure of the minimal nilpotent orbit of $\g$. Moreover, $\mathcal{H}$ carries a $\mathcal{D}_{\mathcal{L}}$-$D_C$-bimodule structure, and its global sections are canonically isomorphic to $D_C$ via the map obtained by applying the right action of $D_C$ to the global section $[1] \in \Gamma(\mathcal{H})$ coming from $\Id_{\mathcal{L}} \in \mathcal{D}_{\mathcal{L}}$. 
\end{prop}

\subsection{\texorpdfstring{Filtrations on $D_V/D_V\Delta\cong\mathcal H(V)$}{Filtrations on the Local Harmonic Quotient}}\label{FiltrationSubSect}

We have seen that $\mathcal H$ has a natural $\mathcal D_{\mathcal L}$-$D_C$ bimodule structure. In particular, the local harmonic quotient
\begin{equation}
\mathcal H(V)\cong D_V/D_V\Delta
\end{equation}
\noindent is naturally a right $D_C$-module. We will now analyze this module structure in greater depth.

For $m=[\xi]\in D_V/D_V\Delta$, left multiplication by $\Delta$ agrees with the induced adjoint action, since
\begin{equation}
\Delta m=[\Delta\xi]=[\Delta,\xi]
\end{equation}
\noindent in $D_V/D_V\Delta$. Moreover, because the left $D_V$-action and the right $D_C$-action commute, left multiplication by $\Delta$ is a right $D_C$-linear endomorphism.

\begin{prop}\label{DeltaLocalNilp}
The endomorphism of the right $D_C$-module $D_V/D_V\Delta$ given by left multiplication by $\Delta$, or equivalently by $\ad_\Delta$, is locally nilpotent. Thus, for every $m\in D_V/D_V\Delta$, there exists an integer $N\geq 1$, depending on $m$, such that
\begin{equation}
\Delta^N m=0,
\end{equation}
\noindent equivalently, $\ad_\Delta^N(m)=0$.
\end{prop}

\begin{proof}
Write an operator $\xi\in D_V$ in normal order as
\begin{equation}
\xi=\sum_{\beta}f_\beta(v)\partial^\beta,
\qquad
f_\beta\in\kappa[V],
\end{equation}
\noindent and let $d$ be the largest polynomial degree of a nonzero coefficient $f_\beta$. Since $\Delta$ has constant coefficients, it commutes with every $\partial^\beta$, and hence
\begin{equation}
[\Delta,f_\beta\partial^\beta]=[\Delta,f_\beta]\partial^\beta.
\end{equation}
\noindent Thus taking the commutator with $\Delta$ lowers the maximal polynomial degree of the coefficients by at least one. It follows that
\begin{equation}
\ad_\Delta^{d+1}(\xi)=0
\end{equation}
\noindent in $D_V$. Passing to the quotient and using that left multiplication by $\Delta$ agrees there with $\ad_\Delta$, we obtain $\Delta^{d+1}[\xi]=0$ in $D_V/D_V\Delta$.
\end{proof}

This proposition leads to the following definition.

\begin{definition}\label{Delta-Filtration}
For $i\geq 0$, define
\begin{equation}
F_i\left(D_V/D_V\Delta\right)
:=
\ker\left(
\Delta^{i+1}:D_V/D_V\Delta\longrightarrow D_V/D_V\Delta
\right)
=
\{m:\Delta^{i+1}m=0\}.
\end{equation}
\noindent We call this the \textbf{$\Delta$-filtration}, and set $F_{-1}=0$ by convention. Since left multiplication by $\Delta$ is right $D_C$-linear, every $F_i$ is a right $D_C$-submodule. The filtration is increasing, and Proposition~\ref{DeltaLocalNilp} shows that it is exhaustive. Moreover, Theorem~\ref{GlobalSectionsTheorem} identifies
\begin{equation}
F_0\cong\Gamma(G/P,\mathcal H)\cong D_C.
\end{equation}
\end{definition}

There is another natural filtration on $D_V/D_V\Delta$, obtained by viewing an element of $\mathcal H(V)$ as a rational section of $\mathcal H$ and measuring its pole along the boundary
\begin{equation}
\overline{C_\infty}=G/P\setminus V.
\end{equation}
\noindent We now make this precise.

Recall that the Kelvin transform acts rationally by conjugation, $\xi\mapsto K\xi K^{-1}=K\xi K$. Put $V_Q:=V\cap V_{w_0}=V\setminus\{Q=0\}$. By Lemma~\ref{RightFischerDecompositionLemma}, the natural localization map
\begin{equation}
D_V/D_V\Delta
\longrightarrow
D_{V_Q}/D_{V_Q}\Delta
\end{equation}
\noindent is injective, and we regard the former quotient as a subspace of the latter. Since $K\Delta K=Q^2\Delta$ and $Q$ is invertible on $V_Q$, conjugation by $K$ induces an involutive automorphism
\begin{equation}
\alpha_K:
D_{V_Q}/D_{V_Q}\Delta
\xrightarrow{\sim}
D_{V_Q}/D_{V_Q}\Delta.
\end{equation}
\noindent For $m\in D_V/D_V\Delta$, this agrees with the transition map induced by the $G$-equivariant structure on $\mathcal H$:
\begin{equation}\label{Kvsw0action}
\alpha_K(m)
=
\left.w_0^*(m)\right|_{V_Q}.
\end{equation}

\begin{definition}\label{Kelvin-Filtration}
For $i\geq 0$, define the \textbf{Kelvin filtration} on $D_V/D_V\Delta$ by
\begin{equation}
P_i\left(D_V/D_V\Delta\right)
:=
\left\{
m\in D_V/D_V\Delta:
Q^i\alpha_K(m)\in D_V/D_V\Delta
\right\},
\end{equation}
\noindent where the condition is interpreted inside $D_{V_Q}/D_{V_Q}\Delta$. We set $P_{-1}=0$ by convention. This filtration is increasing and exhaustive: multiplication by $Q$ preserves $D_V/D_V\Delta$, while every element of the localized quotient has finite pole order in $Q$. By Proposition~\ref{KelvinRegularGlobalSymmetryEquivalence} and Theorem~\ref{GlobalSectionsTheorem},
\begin{equation}
P_0=F_0\cong\Gamma(G/P,\mathcal H)\cong D_C.
\end{equation}
\noindent It is not a priori clear that the right $D_C$-action on $D_V/D_V\Delta$ preserves the Kelvin filtration.
\end{definition}

We now have the following theorem, which generalizes Proposition \ref{KelvinRegularGlobalSymmetryEquivalence}.

\begin{thm}\label{FiltrationTheorem}
For every $i\geq -1$, the $\Delta$-filtration and the Kelvin filtration agree:
\begin{equation}
F_i\left(D_V/D_V\Delta\right)
=
P_i\left(D_V/D_V\Delta\right).
\end{equation}
\noindent In particular, the Kelvin filtration is a filtration of $D_V/D_V\Delta$ by right $D_C$-submodules.
\end{thm}

To prove Theorem~\ref{FiltrationTheorem}, we first establish a preparatory lemma controlling the effect of a certain conjugate of $\Delta$ on the order of a pole along the divisor ${Q=0}$.

\begin{lem}\label{DeltaKelvinPoleLemma}
Put
\begin{equation}
M:=D_V/D_V\Delta,
\qquad
M_Q:=D_{V_Q}/D_{V_Q}\Delta.
\end{equation}
\noindent For every integer $j\geq 0$, define
\begin{equation}
\Theta_j:=Q^{j+1}\Delta Q^{-j}\in D_{V_Q}.
\end{equation}
\noindent Then $\Theta_j\in D_V$, and, for every $\eta\in M_Q$,
\begin{equation}\label{DeltaKelvinPoleCriterion}
\eta\in M
\qquad\Longleftrightarrow\qquad
\Theta_j\eta\in M.
\end{equation}
\end{lem}

\begin{proof}
Set $H:=E+k$. The commutator relations $[\Delta,Q]=H$ and $[H,Q]=2Q$ give
\begin{equation}\label{DeltaIntegerPowerQ}
\Delta Q^s
=
Q^s\Delta+sQ^{s-1}(H+s-1)
\end{equation}
\noindent for every $s\in\mathbb Z$, where negative powers are interpreted in $D_{V_Q}$. For positive $s$, this follows by induction from the two commutator relations; since $Q$ is invertible in $D_{V_Q}$, the same induction applied after multiplying by suitable powers of $Q^{-1}$ gives the formula for negative $s$. Taking $s=-j$ in \eqref{DeltaIntegerPowerQ}, we obtain
\begin{equation}\label{ThetaJFormula}
\Theta_j
=
Q\Delta-j(H-j-1)
=
Q\Delta-j(E+k-j-1).
\end{equation}
\noindent In particular, $\Theta_j\in D_V$, so the forward implication in \eqref{DeltaKelvinPoleCriterion} is immediate.

For the converse, suppose that $\eta\notin M$. Localizing the right Fischer decomposition of Lemma~\ref{RightFischerDecompositionLemma} gives
\begin{equation}
M_Q
\cong
\kappa[V][Q^{-1}]\otimes_\kappa Z_\partial,
\qquad
M
\cong
\kappa[V]\otimes_\kappa Z_\partial.
\end{equation}
\noindent Hence $\eta$ has a well-defined exact pole order $p\geq 1$: we may write
\begin{equation}
\eta=Q^{-p}R,
\qquad
R\in M,
\qquad
R\notin QM.
\end{equation}
\noindent Applying \eqref{DeltaIntegerPowerQ} with $s=-(p+j)$ gives
\begin{equation}\label{ThetaJPoleExpansion}
\Theta_j(Q^{-p}R)
=
Q^{1-p}\Delta R
-
(p+j)Q^{-p}(H-p-j-1)R.
\end{equation}
\noindent The first term on the right has pole order at most $p-1$. It therefore remains to show that
\begin{equation}\label{EulerInjectivityModuloQ}
(H-p-j-1)R\notin QM.
\end{equation}

Let $\overline{R}$ be the nonzero image of $R$ in $M/QM$, and equip this quotient with the filtration induced by the order filtration on $D_V$. Since the principal symbol of $\Delta$ is $Q^*(\zeta)$, the right Fischer decomposition gives
\begin{equation}
\gr(M/QM)
\cong
\kappa[v,\zeta]\big/\bigl(Q(v),Q^*(\zeta)\bigr).
\end{equation}
\noindent Let $P(v,\zeta)$ be the highest nonzero symbol of $\overline{R}$. The principal symbol of left multiplication by $E$ is $\langle v,\zeta\rangle$, while the scalar term in $H-p-j-1$ does not affect the highest symbol. Thus the highest symbol of $(H-p-j-1)\overline{R}$ is $\langle v,\zeta\rangle P(v,\zeta)$. As discussed in the proof of Lemma \ref{SymbolLemma}, the ring $\kappa[v,\zeta]\big/\bigl(Q(v),Q^*(\zeta)\bigr)$ is a domain, and $\langle v,\zeta\rangle$ is nonzero in this quotient.\footnote{For example, its bidegree is $(1,1)$, whereas the two generators of the ideal have bidegrees $(2,0)$ and $(0,2)$.} Consequently, $\langle v,\zeta\rangle P(v,\zeta)\neq 0$, which proves \eqref{EulerInjectivityModuloQ}.

Since $\kappa$ has characteristic $0$ and $p+j>0$, the coefficient of $Q^{-p}$ in the second term of \eqref{ThetaJPoleExpansion} is nonzero modulo $QM$. Hence $\Theta_j\eta$ also has exact pole order $p$, and in particular $\Theta_j\eta\notin M$. This proves the reverse implication in \eqref{DeltaKelvinPoleCriterion}.
\end{proof}

\begin{proof}[Proof of Theorem~\ref{FiltrationTheorem}]
Retain the notation
\begin{equation}
M:=D_V/D_V\Delta,
\qquad
d:=\ad_\Delta=\Delta:M\longrightarrow M.
\end{equation}
\noindent The equality $F_{-1}=P_{-1}=0$ holds by definition. Moreover, Proposition~\ref{KelvinRegularGlobalSymmetryEquivalence} gives
\begin{equation}\label{DeltaKelvinBaseCase}
P_0
=
\ker(d)
=
F_0.
\end{equation}

We prove the remaining equalities by induction. Let $j\geq 1$ and $m\in M$. Since $K\Delta K=Q^2\Delta$, conjugation by $K$ gives
\begin{equation}\label{KelvinConjugatesDeltaAction}
\alpha_K(dm)
=
Q^2\Delta\,\alpha_K(m)
\end{equation}
\noindent in $M_Q$. Indeed, if $m=[\xi]$, then $dm=[\Delta\xi]$, and hence
\begin{equation}
\alpha_K(dm)
=
[K\Delta\xi K]
=
[(K\Delta K)(K\xi K)]
=
[Q^2\Delta\,\alpha_K(\xi)].
\end{equation}

Using \eqref{KelvinConjugatesDeltaAction}, we compute
\begin{align}
\Theta_j\bigl(Q^j\alpha_K(m)\bigr)
&=
Q^{j+1}\Delta Q^{-j}\bigl(Q^j\alpha_K(m)\bigr)
\nonumber\\
&=
Q^{j+1}\Delta\alpha_K(m)
\nonumber\\
&=
Q^{j-1}\alpha_K(dm).
\label{DeltaKelvinRecursionIdentity}
\end{align}
\noindent Applying Lemma~\ref{DeltaKelvinPoleLemma} to $Q^j\alpha_K(m)$ therefore gives
\begin{align}
m\in P_j
&\Longleftrightarrow
Q^j\alpha_K(m)\in M
\nonumber\\
&\Longleftrightarrow
\Theta_j\bigl(Q^j\alpha_K(m)\bigr)\in M
\nonumber\\
&\Longleftrightarrow
Q^{j-1}\alpha_K(dm)\in M
\nonumber\\
&\Longleftrightarrow
dm\in P_{j-1}.
\label{DeltaKelvinFiltrationRecursion}
\end{align}

Suppose inductively that $P_{j-1}=F_{j-1}$. Then \eqref{DeltaKelvinFiltrationRecursion} yields
\begin{align}
m\in P_j
&\Longleftrightarrow
dm\in P_{j-1} = F_{j-1}
\nonumber\\
&\Longleftrightarrow
d^j(dm)=0
\nonumber\\
&\Longleftrightarrow
d^{j+1}m=0
\nonumber\\
&\Longleftrightarrow
m\in F_j.
\end{align}
\noindent Together with the base case \eqref{DeltaKelvinBaseCase}, this proves
\begin{equation}
P_j=F_j
\end{equation}
\noindent for every $j\geq -1$.

Finally, each $F_j$ is a right $D_C$-submodule because $d$ is right $D_C$-linear. The equality $P_j=F_j$ therefore shows that the Kelvin filtration is also a filtration by right $D_C$-submodules.
\end{proof}

The next object we consider is the associated graded of $D_V/D_V\Delta$ with respect to the $\Delta$-filtration. For every $i\geq0$, there is a natural map
\begin{equation}\label{DeltaGradedInjection}
\Delta^i:F_i/F_{i-1}\longrightarrow F_0\cong D_C.
\end{equation}
\noindent This map is injective by the definition of $F_i$. Moreover, left multiplication by $\Delta$ commutes with the right $D_C$-action, so its image is a right ideal of $D_C$. Let us call this right ideal
\begin{equation}
I^{[i]} := \im\left(\Delta^i: F_i/F_{i-1} \to F_0 \cong D_C\right).   
\end{equation}
Our objective, for the rest of this subsection, is to describe $I^{[i]}$ explicitly.

Our strategy is as follows. We will first pass from the filtered objects $D_V/D_V\Delta$ and $D_C$ to an associated graded; both of the resulting graded objects will have the structure of a commutative algebra. These commutative algebras can in turn be analyzed geometrically: we will be able to describe the Rees algebra of the commutative ideals associated with $I^{[i]}$ in terms of the geometry of $\overline{\mathbb{O}}_{\min}$.

However, for this to succeed, we will need a filtration on $D_V/D_V\Delta$ which is preserved under left multiplication by $\Delta$ (equivalently, $\ad_\Delta$). Unfortunately,  under the usual order filtration $\ad_\Delta$ raises degree by 1. So we shall instead use the Bernstein filtration $\mathsf B_\bullet$. Recall that $\mathsf B_\bullet D_V$ is the unique filtration such that both polynomial coordinates and constant vector fields in $D_V$ have degree $1$. We 
equip\footnote{We shall call the Bernstein filtration on the quotient $\mathsf{B}_\bullet$ as well: $\mathsf{B}_i(D_V/D_V\Delta) := p_i(D_V)$ under the quotient map $p: D_V \to D_V/D_V\Delta$.} the quotient $D_V/D_V\Delta$ and its subspaces $F_i$ with the induced filtrations. In particular, observe that restriction to $D_C \cong N(D_V\Delta)/D_V\Delta \subset D_V/D_V\Delta$ induces a non-standard filtration on $D_C$.

\begin{lem}\label{BernsteinSymbolLemma}
Put
\begin{equation}
M:=D_V/D_V\Delta,
\qquad
A:=F_0=\ker(\Delta:M\to M).
\end{equation}
\noindent There is a natural identification
\begin{equation}
\gr_{\mathsf B}M\cong R:=\kappa[V\times C],
\end{equation}
\noindent where $C\subset V^*$ is the isotropic cone of $Q^*$ and points of $V\times C$ are written as $(v,\nu)$. Under this identification, left multiplication by $\Delta$ induces the locally nilpotent derivation
\begin{equation}
\delta(f)(v,\nu)
:=
\left.\frac{d}{dt}\right|_{t=0}
f(v+t\nu^\sharp,\nu)
\end{equation}
\noindent which preserves the Bernstein-filtration degree. Hence, if $\mathscr F_iR:=\ker(\delta^{i+1})$, then
\begin{equation}
\gr_{\mathsf B}F_i\subseteq\mathscr F_iR.
\end{equation}
\noindent Moreover,
\begin{equation}
R^\delta
\cong
\kappa[T^*C^o]
\cong
\kappa[\overline{\mathbb O}_{\min}],
\qquad
\gr_{\mathsf B}A=R^\delta.
\end{equation}
\noindent In the notation of \eqref{invariantfunccot}, the invariant ring $R^\delta$ is generated by the coordinate functions and matrix coefficients of
\begin{equation}
\nu^\sharp,
\qquad
\alpha:=\langle\nu,v\rangle,
\qquad
v\wedge\nu^\sharp,
\qquad
\mu_{\nu,v}:=\alpha v-Q(v)\nu^\sharp.
\end{equation}
\end{lem}

\begin{proof}
The Bernstein associated graded of the Weyl algebra is $\gr_{\mathsf B}D_V\cong\kappa[V\times V^*]$, and the principal symbol of $\Delta$ is $Q^*$ in the derivative-symbol variables. Since $\gr_{\mathsf B}D_V$ is a domain, $\gr_{\mathsf B}(D_V\Delta)=(Q^*)$. Hence $\gr_{\mathsf B}M\cong\kappa[V\times V^*]/(Q^*)=\kappa[V\times C]$.

On $M$, left multiplication by $\Delta$ preserves the Bernstein filtration. Indeed, if $\xi\in D_V$, then $\Delta\xi=[\Delta,\xi]+\xi\Delta$, and the class of $\xi\Delta$ vanishes in $M$. Thus left multiplication by $\Delta$ is represented on $M$ by $\ad_\Delta$. Since $\Delta$ has Bernstein degree $2$ and $[\mathsf B_pD_V,\mathsf B_qD_V]\subseteq\mathsf B_{p+q-2}D_V$, we have $[\Delta,\xi]\in\mathsf B_{\deg_{\mathsf B}\xi}D_V$. Moreover, the principal symbol of $[\Delta,\xi]$ is the canonical Poisson bracket $\{Q^*,\sigma_{\mathsf B}(\xi)\}$, where $Q^*=\sigma_{\mathsf B}(\Delta)$ and the convention is $\{\sigma_{\mathsf B}(\partial_i),x_j\}=\delta_{ij}$. Hence the induced action of left multiplication by $\Delta$ on $\gr_{\mathsf B}M$ is $\{Q^*,-\}$. With our flat--sharp conventions, this derivation is precisely
\begin{equation*}
\{Q^*,f\}(v,\nu)
=
\left.\frac{d}{dt}\right|_{t=0}
f(v+t\nu^\sharp,\nu).
\end{equation*}

After interchanging the factors $V\times C\cong C\times V$, the derivation $\delta$ is the infinitesimal generator of the $\G_a$-action of \eqref{Gasymm}. Since $C$ is normal and $C\setminus C^o=\{0\}$ has codimension at least $2$, one has $\kappa[C\times V]=\kappa[C^o\times V]$. Proposition~\ref{invfunctioniso} therefore identifies $R^\delta$ with $\kappa[T^*C^o]$, and Proposition~\ref{CotBundDescr} identifies this ring with $\kappa[\overline{\mathbb O}_{\min}]$ and gives the stated generators.

It remains to identify $\gr_{\mathsf B}A$. Since every element of $A$ is killed by $\Delta$, the symbol calculation above gives $\gr_{\mathsf B}A\subseteq R^\delta$. Conversely, Proposition~\ref{DeltaPres} shows that the classes of the operators $\phi(\xi)$ lie in $A$. Using Proposition~\ref{ConfActionProp}, and writing $\mu_0,X_0,a,\lambda_0$ for the four parameters, their Bernstein symbols are
\begin{align*}
\sigma_{\mathsf B}\bigl(\phi(0,\mu_0,0,0)\bigr)
&=
-\langle\nu,\mu_0\rangle,
&
\sigma_{\mathsf B}\bigl(\phi(a,0,0,0)\bigr)
&=
a\alpha,
\\
\sigma_{\mathsf B}\bigl(\phi(0,0,X_0,0)\bigr)
&=
-\langle\nu,X_0v\rangle
=
\frac{1}{2}\tr\bigl((v\wedge\nu^\sharp)X_0\bigr),
&
\sigma_{\mathsf B}\bigl(\phi(0,0,0,\lambda_0)\bigr)
&=
-B(\lambda_0,\mu_{\nu,v}).
\end{align*}
\noindent These symbols span the coordinate functions of the generators of $\kappa[T^*C^o] \cong R^\delta$ in Proposition \ref{ConfActionProp}. Hence $R^\delta\subseteq\gr_{\mathsf B}A$, and so $\gr_{\mathsf B}A=R^\delta$. Finally, Theorem~\ref{GlobalSectionsTheorem} identifies $A$ with $D_C$ through $\widehat{\tau}$.
\end{proof}

For the remainder of this discussion, put
\begin{equation}\label{SRingDef}
S:=R^\delta.
\end{equation}
\noindent The coordinate functions of $\nu^\sharp$ belong to $S$, and so does $\alpha=\langle\nu,v\rangle$, since
\begin{equation*}
\delta(\alpha)
=
B(\nu^\sharp,\nu^\sharp)
=
2Q(\nu^\sharp)
=
0
\end{equation*}
\noindent on $C$. We define
\begin{equation}\label{ClassicalBoundaryIdealDef}
\mathfrak a
:=
\bigl((\nu^\sharp)_1,\ldots,(\nu^\sharp)_{2k},\alpha\bigr)
\subset S.
\end{equation}
\noindent This is the classical ideal whose powers will occur in the associated graded of \eqref{DeltaGradedInjection}. To interpret it geometrically, recall from Propositions~\ref{invfunctioniso} and \ref{CotBundDescr} that $S\cong\kappa[T^*C^o]\cong\kappa[\overline{\mathbb O}_{\min}]$, and that the affinized projection $T^*C^o\to C^o$ is induced on coordinate rings by the functions $(\nu^\sharp)_i$. Thus the ideal $((\nu^\sharp)_1,\ldots,(\nu^\sharp)_{2k})\subset S$ is the scheme-theoretic ideal of the fiber over the vertex $0\in C$ in the affinized model. The ideal $\mathfrak a$ is its radical. Indeed, using the notation $\mu_{\nu,v}:=\alpha v-Q(v)\nu^\sharp$ from \eqref{dualconeinv}, one has $B(\nu^\sharp,\mu_{\nu,v})=\alpha^2$ on $C$, since $Q(\nu^\sharp)=0$. Hence $\alpha^2$ lies in the ideal generated by the coordinate functions of $\nu^\sharp$, and so $\alpha$ lies in its radical. Equivalently, in the matrix presentation \eqref{invariantfunccot}, $\mathfrak a$ is the reduced ideal obtained by forcing the rightmost column, i.e. the base-cone coordinate, to vanish. Geometrically speaking, the variety $V(\mathfrak{a})$ is the reduced subscheme of $\mathbb{O}_{\min}$ given by the image of the quasiclassical quadric Fourier transform (i.e., $\Ad(w_0)$) applied to the closure of the 0-section of $C^o$ in $T^*C^o$.

Lemma~\ref{BernsteinSymbolLemma} shows that passing to associated graded objects from the Bernstein filtration on $M = D_V/D_V\Delta$ and $F_0 \cong D_C$ gives the commutative algebra $R=\kappa[C \times V]$ and $S =\kappa[\mathbb{O}_{\min}]$. Moreover, the $\Delta$-filtration $F_i$ on $D_V/D_V\Delta$ induces the $\delta$-filtration $\mathscr{F}_i$ on $\kappa[C \times V]$ and the degree-0 component $\mathscr{F}_0$ is precisely $S = R^{\delta}$.

Our next goal is to describe the associated graded algebra of $\mathscr{F}_i$ in the commutative algebra $\kappa[C \times V]$. Because we are working with a commutative algebra, we have all the tools of algebraic geometry at our disposal. In particular, a filtered commutative algebra $A$ may be studied through its Rees algebra. This gives a scheme $X \to \A^1$ such that the fiber over $1$ is $\Spec(A)$, while the fiber over $0$ is $\Spec(\textrm{AssGr}(A))$. 

In the present setting, this Rees family has a very elegant and concrete geometric realization as a one-parameter family inside an affine flag-multicone. We shall find that the general fiber of this family is $V\times C$, while its special fiber is the affine Rees space of the ideal $\mathfrak a\subset S$ defined in \eqref{ClassicalBoundaryIdealDef}. Thus we find that $V \times C$ degenerates to a cone over the blowup of $V(\mathfrak{a})$ in $\Spec\, S=\overline{\mathbb O}_{\min}$. 

Put $W:=V^+$, and let
\begin{equation*}
\operatorname{OFl}(1,2;W)
:=
\left\{
(\ell,\Pi):
\ell\subset\Pi\subset W,\ 
\dim\,\ell=1,\ 
\dim\,\Pi=2,\ 
Q^+|_{\Pi}=0
\right\}
\end{equation*}
\noindent be the variety of isotropic flags of dimensions $1$ and $2$ in $W$. Under the natural embedding
\begin{equation*}
\operatorname{OFl}(1,2;W)
\hookrightarrow
\P(W)\times\P\left(\bigwedge\nolimits^2W\right),
\qquad
(\ell,\Pi)
\longmapsto
\left(\ell,\bigwedge\nolimits^2\Pi\right),
\end{equation*}
\noindent let
\begin{equation*}
\widehat{\operatorname{OFl}}(1,2;W)
\subset
W\times\bigwedge\nolimits^2W
\end{equation*}
\noindent denote the corresponding affine multicone. A point in its dense flag locus may be written as $(u,u\wedge z)$, where $u$ and $z$ span a totally isotropic $2$-plane. Let $u_a$ and $p_{ab}=-p_{ba}$ denote the coordinate functions on the two factors, where the indices range over
\begin{equation*}
\mathcal I:=\{+,1,\ldots,2k,-\}.
\end{equation*}
The coordinate $u_+$ gives a map $\widehat{\operatorname{OFl}}(1,2;W) \to \A^1$; as we shall see, this provides the flat Rees deformation from $C \times V$ to the cone $\widehat{\textrm{Bl}_{V(\mathfrak{a})} \overline{\mathbb{O}}_{\min}}$.

\begin{lem}\label{IsotropicFlagChartLemma}
Let
\begin{equation*}
\widehat{\operatorname{OFl}}(1,2;W)_+
:=
\widehat{\operatorname{OFl}}(1,2;W)\cap\{u_+=1\}.
\end{equation*}
\noindent For $(v,\nu)\in V\times C$, put
\begin{equation*}
\alpha:=\langle\nu,v\rangle,
\qquad
u(v):=(1,v,-Q(v)),
\qquad
z(v,\nu):=(0,\nu^\sharp,-\alpha),
\end{equation*}
\noindent and set $p(v,\nu):=u(v)\wedge z(v,\nu)$. Then
\begin{equation*}
\Psi:V\times C\longrightarrow
\widehat{\operatorname{OFl}}(1,2;W)_+,
\qquad
(v,\nu)\longmapsto\bigl(u(v),p(v,\nu)\bigr),
\end{equation*}
\noindent is an isomorphism. Consequently,
\begin{equation*}
R
\cong
\kappa\left[\widehat{\operatorname{OFl}}(1,2;W)_+\right].
\end{equation*}
\noindent For every $(u,p)\in\widehat{\operatorname{OFl}}(1,2;W)_+$, there is a unique $z\in W$ such that $z_+=0$ and $p=u\wedge z$. This canonical choice is given by
\begin{equation*}
z_a=p_{+a}.
\end{equation*}
\noindent Under $\Psi$, it is the vector $z(v,\nu)$. In particular,
\begin{equation*}
p_{+i}=(\nu^\sharp)_i,
\qquad
p_{+-}=-\alpha.
\end{equation*}
\noindent The derivation $\delta$ satisfies
\begin{equation*}
\delta(u_a)=p_{+a},
\qquad
\delta(p_{ab})=0.
\end{equation*}
\noindent Moreover, the functions $p_{ab}$ generate $S=R^\delta$, and
\begin{equation}\label{FirstRowPluckerIdeal}
\mathfrak a
=
(p_{+a}:a\in\mathcal I\setminus\{+\})
\subset S.
\end{equation}
\end{lem}

\begin{proof}
The vectors $u(v)$ and $z(v,\nu)$ are isotropic and orthogonal: one has $Q^+(u(v))=0$, $Q^+(z(v,\nu))=Q(\nu^\sharp)=0$, and $B^+(u(v),z(v,\nu))=-\alpha+B(v,\nu^\sharp)=0$. Hence $\Psi$ is well-defined; if $z(v,\nu)=0$, the corresponding point is obtained by passing to the closure defining the affine multicone.

Let $(u,p)\in\widehat{\operatorname{OFl}}(1,2;W)_+$. The incidence relation gives $u\wedge p=0$. Define $z_+:=0$ and $z_a:=p_{+a}$ for $a\neq+$. Since $u_+=1$, the $(+,a,b)$-coordinate of $u\wedge p=0$ gives $p_{ab}=u_az_b-u_bz_a$, and hence $p=u\wedge z$. This $z$ is unique: if also $p=u\wedge z'$ with $z'_+=0$, then $u\wedge(z-z')=0$, so $z-z'$ is a scalar multiple of $u$; comparison of the $+$-coordinates forces this scalar to vanish.

Write $u=(1,v,b)$ and $z=(0,s,c)$. The isotropy relations for the flag give $Q^+(u)=Q^+(z)=B^+(u,z)=0$. Thus $b=-Q(v)$, $Q(s)=0$, and $c=-B(v,s)$. Setting $\nu:=s^\flat$ gives $\nu\in C$, $s=\nu^\sharp$, and $c=-\langle\nu,v\rangle$. Therefore $(u,p)=\Psi(v,\nu)$, and the resulting inverse is polynomial.

For the action of $\delta$, put $s:=\nu^\sharp$. Since $Q(s)=0$ and $B(v,s)=\alpha$, one has $u(v+ts)=u(v)+tz(v,\nu)$ and $z(v+ts,\nu)=z(v,\nu)$. Hence $p(v+ts,\nu)=p(v,\nu)$. Differentiating at $t=0$ gives $\delta(u_a)=z_a=p_{+a}$ and $\delta(p_{ab})=0$.

Finally, under the identification $\bigwedge^2W\cong\mathfrak{o}(Q^+)$ induced by $B^+$, the bivector $p=u\wedge z$ acts by $x\mapsto B^+(u,x)z-B^+(z,x)u$. In block coordinates, this is the matrix \eqref{invariantfunccot}. Proposition~\ref{CotBundDescr} therefore shows that the functions $p_{ab}$ generate $S$. Since $p_{+i}=(\nu^\sharp)_i$ and $p_{+-}=-\alpha$, equation \eqref{FirstRowPluckerIdeal} agrees with \eqref{ClassicalBoundaryIdealDef}.
\end{proof}

Let
\begin{equation*}
\mathcal R_{1,2}
:=
\kappa\left[\widehat{\operatorname{OFl}}(1,2;W)\right]
\end{equation*}
\noindent be the multicone coordinate ring. It is naturally bigraded: the generators $u_a$ have degree $(1,0)$ and the Pl\"ucker coordinates $p_{ab}$ have degree $(0,1)$. We use the first component of this bigrading and write
\begin{equation*}
\mathcal R_{1,2}
=
\bigoplus_{d\geq0}\mathcal R_{1,2}^{(d)},
\end{equation*}
\noindent where the second degree is unrestricted. The coordinate function $u_+$ defines a morphism
\begin{equation*}
\varpi:
\widehat{\operatorname{OFl}}(1,2;W)
\longrightarrow
\A^1.
\end{equation*}
\noindent We write
\begin{equation*}
R_1
:=
\mathcal R_{1,2}/(u_+-1),
\qquad
R_0
:=
\mathcal R_{1,2}/(u_+).
\end{equation*}
\noindent Thus Lemma~\ref{IsotropicFlagChartLemma} gives $R_1\cong R$.

\begin{thm}\label{MulticoneReesDeformationTheorem}
The morphism
\begin{equation*}
\varpi:
\widehat{\operatorname{OFl}}(1,2;W)
\longrightarrow
\A^1
\end{equation*}
\noindent is flat. Its fiber over $1$ is $V\times C$, and every nonzero fiber is isomorphic to this one. Its fiber over $0$ is
\begin{equation}\label{SpecialFiberReesIsomorphism}
R_0
\cong
\mathcal R_S(\mathfrak a)
:=
\bigoplus_{i\geq0}\mathfrak a^it^i
\subset S[t].
\end{equation}
\noindent In particular, the special fiber is the affine Rees space of the ideal $\mathfrak a\subset S$. After applying $\operatorname{Proj}$ over $\Spec \, S=\overline{\mathbb O}_{\min}$, it gives the blowup of $\overline{\mathbb O}_{\min}$ along the reduced fiber over $0\in C$.
\end{thm}

\begin{proof}
We first prove flatness. The ring $\mathcal R_{1,2}$ is a domain, since the affine multicone is irreducible. Moreover, the subalgebra $\kappa[u_+]\subset\mathcal R_{1,2}$ is polynomial: the powers $u_+^m$ have distinct first degrees, so no nonzero polynomial in $u_+$ vanishes. Hence $\mathcal R_{1,2}$ is torsion-free over the PID $\kappa[u_+]$. It follows that $\mathcal R_{1,2}$ is flat over $\kappa[u_+]$, which is the desired flatness of $\varpi$.

The fiber over $1$ is $R_1\cong R\cong\kappa[V\times C]$ by Lemma~\ref{IsotropicFlagChartLemma}. If $c\neq0$, the fiber $u_+=c$ is identified with the fiber $u_+=1$ by scaling the first multicone coordinate $u$ by $c^{-1}$ while leaving the Pl\"ucker coordinate $p$ fixed. Hence all nonzero fibers are isomorphic.

It remains to identify the fiber over $0$. Put
\begin{equation*}
q_a:=p_{+a}\in S
\qquad
(a\in\mathcal I),
\end{equation*}
\noindent so that $q_+=p_{++}=0$ and $\mathfrak a=(q_a:a\neq+)$. For $p=(p_{ab})\in\bigwedge^2W$, let $q(p)\in W$ be the vector whose $a$-coordinate is $p_{+a}$.

We claim that $(p,t)\mapsto(tq(p),p)$ defines a morphism from $\Spec \, S\times\A^1$ to the variety $\widehat{\operatorname{OFl}}(1,2;W)\cap\{u_+=0\}$. It suffices to check the defining equations on the dense locus where $p=x\wedge y$ and $x,y$ span a totally isotropic $2$-plane. On this locus,
\begin{equation*}
q(p)=x_+y-y_+x.
\end{equation*}
\noindent Thus $q(p)$ belongs to the plane spanned by $x$ and $y$, so $q(p)\wedge p=0$ and $Q^+(q(p))=0$; moreover, its $+$-coordinate is $p_{++}=0$. These are polynomial identities in the coordinates of $p$, so they hold on all of $\Spec \, S$. Hence $(tq(p),p)$ lies in the required special fiber.

Passing to coordinate rings gives a graded homomorphism
\begin{equation}\label{MulticoneToReesMap}
\Phi:
\mathcal R_{1,2}/(u_+)
\longrightarrow
S[t],
\qquad
\Phi(p_{ab})=p_{ab},
\qquad
\Phi(u_a)=q_at.
\end{equation}
\noindent Its image is the subalgebra generated by $S$ and the elements $q_at$, namely
\begin{equation*}
S[q_at:a\neq+]
=
\bigoplus_{i\geq0}\mathfrak a^it^i
=
\mathcal R_S(\mathfrak a).
\end{equation*}
\noindent Therefore $\Phi$ is a surjective graded homomorphism from $R_0$ onto $\mathcal R_S(\mathfrak a)$.

We now prove injectivity. In the language of \cite{ElizondoKuranoWatanabeTotalCoordinateRing}, the multicone ring $\mathcal R_{1,2}$ is the total coordinate ring of the smooth projective flag variety $\operatorname{OFl}(1,2;W)$, whose Picard group is generated by the two line bundles defining the embeddings into $\P(W)$ and $\P(\bigwedge^2W)$. It is therefore a unique factorization domain by \cite[Corollary~1.2]{ElizondoKuranoWatanabeTotalCoordinateRing}. Since $u_+$ has bidegree $(1,0)$, it is irreducible, hence prime. Thus $R_0=\mathcal R_{1,2}/(u_+)$ is a domain. The target $\mathcal R_S(\mathfrak a)$ is also a domain, since it is a subalgebra of $S[t]$.

The variety $\operatorname{OFl}(1,2;W)$ has dimension $2n-2$, and its Picard group has rank $2$. Hence $\dim\,\mathcal R_{1,2}=2n$ and $\dim R_0=2n-1$. On the other hand, $\dim S=2n-2$, and, since $\mathfrak a\neq0$, the fraction field of $\mathcal R_S(\mathfrak a)$ is $\operatorname{Frac}(S)(t)$. Thus $\dim\,\mathcal R_S(\mathfrak a)=2n-1$. The kernel of $\Phi$ is therefore a prime ideal of height $0$. Since $R_0$ is a domain, its only prime ideal of height $0$ is the zero ideal. Hence $\Phi$ is injective, proving \eqref{SpecialFiberReesIsomorphism}.
\end{proof}

We now extract the associated graded algebra of the $\delta$-filtration from this deformation. By Lemma~\ref{IsotropicFlagChartLemma},
\begin{equation*}
R
\cong
\mathcal R_{1,2}/(u_+-1).
\end{equation*}
\noindent Let $\mathsf G_iR$ be the image in $R$ of
\begin{equation*}
\bigoplus_{0\leq d\leq i}\mathcal R_{1,2}^{(d)}.
\end{equation*}
\noindent Thus $\mathsf G_\bullet R$ is the filtration obtained by dehomogenizing the first grading at $u_+=1$.

\begin{lem}\label{DehomogenizedAssociatedGradedLemma}
There is a natural graded isomorphism
\begin{equation}\label{DehomogenizationAssociatedGraded}
\gr_{\mathsf G}R
\cong
R_0
=
\mathcal R_{1,2}/(u_+).
\end{equation}
\end{lem}

\begin{proof}
Let $h\in\mathcal R_{1,2}^{(d)}$. Its image in $R=\mathcal R_{1,2}/(u_+-1)$ determines a class in $\mathsf G_dR/\mathsf G_{d-1}R$. This gives a surjective graded homomorphism
\begin{equation*}
\mathcal R_{1,2}\longrightarrow \gr_{\mathsf G}R.
\end{equation*}
\noindent Since $u_+$ maps to $1$ in $R$, its class in $\mathsf G_1R/\mathsf G_0R$ is zero. Therefore the homomorphism factors through $\mathcal R_{1,2}/(u_+)$.

We prove injectivity. Suppose $h\in\mathcal R_{1,2}^{(d)}$ maps to zero in $\mathsf G_dR/\mathsf G_{d-1}R$. Then there exists $b\in\bigoplus_{r<d}\mathcal R_{1,2}^{(r)}$ such that $h-b\in(u_+-1)$. Write
\begin{equation*}
h-b=(u_+-1)c,
\qquad
c=c_0+\cdots+c_m,
\end{equation*}
\noindent with $c_r\in\mathcal R_{1,2}^{(r)}$ and $c_m\neq0$. Since $u_+$ is a non-zero-divisor, the highest-degree homogeneous component of $(u_+-1)c$ is $u_+c_m$, of degree $m+1$. Comparing highest degrees gives $m+1=d$, and comparison of the degree-$d$ homogeneous components gives $h=u_+c_{d-1}$. Hence $h\in(u_+)$, proving injectivity.
\end{proof}

\begin{prop}\label{ClassicalReesTheorem}
For every $i\geq0$, one has
\begin{equation}\label{ClassicalFiltrationEquality}
\mathsf G_iR
=
\mathscr F_iR
=
\ker(\delta^{i+1}).
\end{equation}
\noindent Moreover, the maps
\begin{equation}\label{ClassicalTopDeltaMap}
\beta_i:
\mathscr F_iR/\mathscr F_{i-1}R
\longrightarrow
S,
\qquad
[f]\longmapsto\frac{1}{i!}\delta^i(f),
\end{equation}
\noindent are injective and satisfy
\begin{equation}\label{ClassicalTopDeltaImage}
\im(\beta_i)=\mathfrak a^i.
\end{equation}
\noindent Consequently, the maps $\beta_i$ assemble to an isomorphism of graded $S$-algebras
\begin{equation}\label{ClassicalReesIsomorphism}
\beta:
\gr_{\mathscr F}R
\xrightarrow{\sim}
\bigoplus_{i\geq0}\mathfrak a^i t^i,
\qquad
[f]\longmapsto\frac{1}{i!}\delta^i(f)t^i.
\end{equation}
\end{prop}

\begin{proof}
By Lemma~\ref{IsotropicFlagChartLemma}, one has $\delta(u_a)=p_{+a}$ and $\delta(p_{ab})=0$. Hence $\delta$ lowers the first dehomogenized degree by $1$, and therefore $\mathsf G_iR\subseteq\ker(\delta^{i+1})$ for every $i$.

We now identify the top derivative on the associated graded of $\mathsf G_\bullet R$. Let $h\in\mathcal R_{1,2}^{(i)}$. Under the isomorphisms of Lemma~\ref{DehomogenizedAssociatedGradedLemma} and Theorem~\ref{MulticoneReesDeformationTheorem}, the class of $h$ in $\mathsf G_iR/\mathsf G_{i-1}R$ corresponds to the element obtained from $h$ by substituting
\begin{equation*}
u_a\longmapsto q_at,
\qquad
p_{ab}\longmapsto p_{ab},
\end{equation*}
\noindent where $q_a=p_{+a}$. Since $\delta$ acts on the $u$-coordinates by the constant vector $q=(q_a)$ and fixes the $p$-coordinates, the same element is
\begin{equation*}
\frac{1}{i!}\delta^i(h)t^i.
\end{equation*}
\noindent Therefore the map $\frac{1}{i!}\delta^i:\mathsf G_iR/\mathsf G_{i-1}R\to S$ is injective, and its image is precisely the degree-$i$ piece $\mathfrak a^i$ of the Rees algebra $\mathcal R_S(\mathfrak a)$.

It remains to show that $\mathscr F_iR\subseteq\mathsf G_iR$. Let $f\in\mathscr F_iR$, and let $d$ be the smallest integer such that $f\in\mathsf G_dR$. If $d>i$, then the class of $f$ in $\mathsf G_dR/\mathsf G_{d-1}R$ is nonzero. By the injectivity just proved, $\delta^d(f)\neq0$. But $d>i$ and $f\in\ker(\delta^{i+1})$, so $\delta^d(f)=0$, a contradiction. Hence $d\leq i$, and therefore $f\in\mathsf G_iR$.

Thus $\mathsf G_iR=\mathscr F_iR$ for all $i$. The image statement \eqref{ClassicalTopDeltaImage} and the graded isomorphism \eqref{ClassicalReesIsomorphism} follow from the top-derivative calculation above.
\end{proof}

We now return to the $\Delta$-filtration on $M:=D_V/D_V\Delta$. Put $A:=F_0$, and let $d:=\ad_\Delta$. On $M$, the operator $d$ agrees with left multiplication by $\Delta$. For $j\geq0$, set
\begin{equation}\label{ShiftedEulerPolynomialDef}
c_j(E):=\prod_{r=1}^{j}(E+k-r-1),
\qquad
c_0(E):=1.
\end{equation}

Following the conventions of Section \ref{KazhLaumGlueSect}, for  $d\geq0$, let $P^d:=\kappa[C]_d$ denote the degree-$d$ homogeneous piece of the coordinate ring of the cone $C$, viewed inside $D_C$ by multiplication operators. Thus $P^dD_C$ denotes the right ideal generated by this finite-dimensional space.

\begin{lem}\label{QuantumDeltaLayerLemma}
Let $i\geq0$ and $0\leq j\leq i$, and put $m:=i-j$. Let $f\in\kappa[V]_m$ be homogeneous of degree $m$, and let $f_C\in P^m$ be its restriction to $C$, using the identification $V\cong V^*$ induced by $B$. Write $f(\partial^\flat)$ for the corresponding constant-coefficient differential operator. Then $[fQ^j]\in F_i$, and
\begin{equation}\label{QuantumTopDeltaFormula}
d^i[fQ^j]
=
i!\left[
f(\partial^\flat)
\prod_{r=0}^{j-1}(H+r)
\right]
\in A,
\qquad
H:=E+k.
\end{equation}
\noindent Under the isomorphism $\widehat{\tau}:A\xrightarrow{\sim}D_C$, one has
\begin{equation}\label{QuantumTopDeltaConeFormula}
\widehat{\tau}\bigl(d^i[fQ^j]\bigr)
=
(-1)^ii!\,f_Cc_j(E).
\end{equation}
\end{lem}

\begin{proof}
Since $f$ has degree $m$, repeated commutation with $\Delta$ gives $d^{m+1}(f)=0$ and $d^m(f)=m!f(\partial^\flat)$. On the other hand, \eqref{DeltaIntegerPowerQ} gives, modulo $D_V\Delta$, $d^j(Q^j)=j!\prod_{r=0}^{j-1}(H+r)$ and $d^{j+1}(Q^j)=0$. Since $m+j=i$, the only nonzero term in the Leibniz expansion of $d^i(fQ^j)$ is the term in which $d$ is applied $m$ times to $f$ and $j$ times to $Q^j$; its coefficient is $\binom{i}{m}$, giving \eqref{QuantumTopDeltaFormula}. The same nilpotence bounds give $d^{i+1}[fQ^j]=0$, so $[fQ^j]\in F_i$.

It remains to apply $\widehat{\tau}$. The linear Fourier transform sends $f(\partial^\flat)$ to $(-1)^mf_C$. Moreover, $\tau(E)=-E-2k$, and the correction in $\widehat{\tau}$ conjugates by $Q^*$, sending $E$ to $E-2$. Hence $\widehat{\tau}(H+r)=-(E+k-r-2)$, and therefore
\begin{equation*}
\widehat{\tau}\left(
f(\partial^\flat)
\prod_{r=0}^{j-1}(H+r)
\right)
=
(-1)^{m+j}f_C
\prod_{r=0}^{j-1}(E+k-r-2)
=
(-1)^if_Cc_j(E).
\end{equation*}
\noindent This proves \eqref{QuantumTopDeltaConeFormula}.
\end{proof}

Via $\widehat{\tau}$, we transport the Bernstein filtration on $A$ to $D_C$. All Bernstein associated gradeds of subspaces of $D_C$ below refer to this transported filtration.

\begin{lem}\label{BernsteinStrictnessDeltaFiltrationLemma}
For every $i\geq0$, one has
\begin{equation}\label{FullBernsteinDeltaFiltrationEquality}
\gr_{\mathsf B}F_i=\mathscr F_iR.
\end{equation}
\noindent Moreover, the injective map $\Delta^i:F_i/F_{i-1}\to A$ is strict for the induced Bernstein filtrations, and its associated-graded image is $\mathfrak a^i$.
\end{lem}

\begin{proof}
We prove \eqref{FullBernsteinDeltaFiltrationEquality} by induction on $i$. The case $i=0$ is Lemma~\ref{BernsteinSymbolLemma}. Assume the result for $i-1$. Lemma~\ref{BernsteinSymbolLemma} gives $\gr_{\mathsf B}F_i\subseteq\mathscr F_iR$, so it remains to show that every class in $\mathscr F_iR/\mathscr F_{i-1}R$ is represented by the principal symbol of an element of $F_i$.

By Proposition~\ref{ClassicalReesTheorem}, the map $\beta_i:\mathscr F_iR/\mathscr F_{i-1}R\to S$ is injective with image $\mathfrak a^i$. Since $\mathfrak a$ is generated by the coordinate functions of $\nu^\sharp$ and by $\alpha$, one has $\mathfrak a^i=\sum_{j=0}^{i}P^{i-j}\alpha^jS$. Let $f_C\in P^{i-j}$, choose a homogeneous lift $f\in\kappa[V]_{i-j}$, and let $s\in S$ be homogeneous. Since $\gr_{\mathsf B}A=S$, choose $a\in A$ with principal symbol $s$. The element $[fQ^j]a$ belongs to $F_i$, because $F_i$ is a right $A$-submodule. Its principal symbol is $f(v)Q(v)^js$, and since $\delta(s)=0$, one has $\frac{1}{i!}\delta^i(f(v)Q(v)^js)=f_C\alpha^js$. These elements span $\mathfrak a^i$, so the principal symbols of elements of $F_i$ surject onto $\mathscr F_iR/\mathscr F_{i-1}R$. Together with the induction hypothesis, this proves \eqref{FullBernsteinDeltaFiltrationEquality}.

It follows that $\gr_{\mathsf B}(F_i/F_{i-1})\cong\mathscr F_iR/\mathscr F_{i-1}R$. The associated-graded map induced by $\Delta^i$ is $\delta^i$, whose kernel on $\mathscr F_iR$ is $\mathscr F_{i-1}R$. Hence the induced map on the quotient is injective. Therefore the filtered map $\Delta^i$ is strict, and its associated-graded image is $\im(\beta_i)=\mathfrak a^i$.
\end{proof}

\begin{thm}\label{DeltaGradedIdealTheorem}
For $i\geq0$, recall that
\begin{equation}\label{DeltaGradedIdealDef}
I^{[i]}
:=
\im\left(
\Delta^i:F_i/F_{i-1}\longrightarrow F_0\cong D_C
\right).
\end{equation}
\noindent Then
\begin{equation}\label{DeltaGradedIdealAllLayers}
I^{[i]}
=
\sum_{j=0}^{i}P^{i-j}c_j(E)D_C.
\end{equation}
\noindent Its Bernstein associated graded is
\begin{equation}\label{DeltaGradedIdealSymbol}
\gr_{\mathsf B}I^{[i]}=\mathfrak a^i.
\end{equation}
\noindent Moreover, the intermediate summands in \eqref{DeltaGradedIdealAllLayers} are redundant:
\begin{equation}\label{DeltaGradedIdealGenerators}
I^{[i]}
=
P^iD_C+c_i(E)D_C
=
P^iD_C+
\left(
\prod_{r=1}^{i}(E+k-r-1)
\right)D_C.
\end{equation}
\end{thm}

\begin{proof}
Set $J_i:=\sum_{j=0}^{i}P^{i-j}c_j(E)D_C$. Lemma~\ref{QuantumDeltaLayerLemma} shows that every element of $P^{i-j}c_j(E)$ belongs to $I^{[i]}$, up to the nonzero scalar $(-1)^ii!$. Since $I^{[i]}$ is a right ideal, $J_i\subseteq I^{[i]}$.

By Lemma~\ref{BernsteinStrictnessDeltaFiltrationLemma}, $\gr_{\mathsf B}I^{[i]}=\mathfrak a^i$. Under the transported Bernstein filtration, the principal symbols of the elements of $P^1$ are, up to nonzero signs, the coordinate functions of $\nu^\sharp$, while the principal symbol of $E$ is, again up to sign, $\alpha$. Hence the principal symbol of $c_j(E)$ is a nonzero scalar multiple of $\alpha^j$, and therefore $\mathfrak a^i=\sum_{j=0}^{i}P^{i-j}\alpha^jS\subseteq\gr_{\mathsf B}J_i$. Since $J_i\subseteq I^{[i]}$, we get $\gr_{\mathsf B}J_i=\gr_{\mathsf B}I^{[i]}$. The filtrations are exhaustive and separated, so $J_i=I^{[i]}$. This proves \eqref{DeltaGradedIdealAllLayers} and \eqref{DeltaGradedIdealSymbol}.

It remains to eliminate the intermediate summands. Fix $0<j<i$, put $d:=i-j$, and define $a_j(E):=\prod_{r=0}^{j-1}(E-r)$ and $g_d(E):=\prod_{r=0}^{d-1}(E+k+r-1)$. Proposition~\ref{Shapovalovdet} gives
\begin{equation}\label{ShapovalovFactorForDeltaIdeal}
\mathfrak B_j=a_j(E)c_j(E)\in P^jD_C.
\end{equation}
\noindent The roots of $a_j(E)$ are $0,\ldots,j-1$, while the roots of $g_d(E)$ are $-(k-1),\ldots,-(k+d-2)$. Since $k\geq2$, these polynomials are coprime. Choose $A(E),B(E)\in\kappa[E]$ such that $A(E)a_j(E)+B(E)g_d(E)=1$. Since $c_i(E+d)=c_j(E)g_d(E)$ and $[E,f]=df$ for $f\in P^d$, one has $c_i(E)f=f\,c_i(E+d)$. Multiplying the Bezout identity by $fc_j(E)$ gives
\begin{equation*}
fc_j(E)
=
f\mathfrak B_jA(E)+c_i(E)fB(E).
\end{equation*}
\noindent The first term belongs to $P^dP^jD_C=P^iD_C$, while the second belongs to $c_i(E)D_C$. Hence $P^{i-j}c_j(E)D_C\subseteq P^iD_C+c_i(E)D_C$ for every $0<j<i$. The terms with $j=0$ and $j=i$ are already $P^iD_C$ and $c_i(E)D_C$, respectively, proving \eqref{DeltaGradedIdealGenerators}.
\end{proof}

\subsection{\texorpdfstring{The Equivalence between $D_C$-mod and $\mathscr{H}$}{The Equivalence between DC-Mod and H}}\label{ProofofEquivSubSect}

We now prove Theorem~\ref{SecondEquiv}; before proceeding, we summarize the strategy of proof. Recall from Theorem~\ref{HarmonicSummary} that $\mathcal H$ carries a natural structure of a $\mathcal D_{\mathcal L}$-$D_C$ bimodule and that
\begin{equation}
\Gamma(G/P,\mathcal H)=\Gamma_\Delta(\mathcal H)\cong D_C.
\end{equation}
\noindent By definition, the category $\mathscr H$ consists of coherent $\mathcal D_{\mathcal L}$-modules admitting a finite $(\mathcal H,D_C)$-presentation. We will show that the global sections functor $\Gamma$ is exact on the full category of coherent $\mathcal D_{\mathcal L}$-modules.\footnote{The functor $\Gamma$ is not, however, conservative; as noted above, $\Gamma(\mathcal{L}) = 0$.} Once this is known, the equivalence of $\mathscr{H}$ and $D_C\textbf{-mod}^{\,\mathrm{fg}}$ follows formally from the finite-presentation definition of $\mathscr H$.

The proof of exactness is modeled on Bezrukavnikov's proof of the exactness part of localization \cite{BezrukavnikovLocalizationNotes}. Let $C^{+}$ denote the quadric cone $\{Q^+ = 0\}$ in the affine space $V^+$, and let $C^{+o}:=C^+\setminus\{0\}$. Then $C^{+o}\to G/P$ is the principal $\G_m$-bundle obtained by quotienting by scalar dilations.

We pull a twisted $\mathcal D_{\mathcal L}$-module $\mathcal M$ back to the punctured cone $C^{+o}$. This results in a monodromic $\mathcal D_{C^{+o}}$-module; we recover the global sections of $\mathcal{M}$ as the $\lambda_0$-weight space, where
\begin{equation}
\lambda_0=1-k.
\end{equation}
\noindent Given a surjection $\mathcal N\twoheadrightarrow \mathcal M$ and a section
$m\in \Gamma(C^{+o},\widetilde{\mathcal M})_{\lambda_0}$, we choose a basis $\{\psi_i\}$ of sufficiently high-degree functions on $C^+$, with dual basis $\{\phi_i\}$ with respect to the pairing induced by $B^+$. If $\deg(\psi_i)=d$, then
\begin{equation}
\psi_i m\in \Gamma(C^{+o},\widetilde{\mathcal M})_{\lambda_0+d},
\end{equation}
\noindent or equivalently $\psi_i m$ is a global section of the corresponding $d$-fold twist $\mathcal M(d)$ on $G/P$:
\begin{equation}
\psi_i m\in \Gamma(G/P,\mathcal M(d))
\cong
\Gamma(C^{+o},\widetilde{\mathcal M})_{\lambda_0+d}.
\end{equation}
\noindent For $d\gg 0$, these raised sections lift to sections
\begin{equation}
n_i\in \Gamma(G/P,\mathcal N(d))
\cong
\Gamma(C^{+o},\widetilde{\mathcal N})_{\lambda_0+d}
\end{equation}
\noindent by a good-filtration and Serre-vanishing argument. The operators $\mathcal F_+(\phi_i)$ then lower weight by $d$, so each $\mathcal F_+(\phi_i)n_i$ lies in $\Gamma(C^{+o},\widetilde{\mathcal N})_{\lambda_0}$. The opposite Shapovalov operator
\begin{equation}
\sum_i \mathcal F_+(\phi_i)\psi_i
\end{equation}
\noindent is exactly the raise-then-lower operator used in Bezrukavnikov's exactness argument; the Shapovalov determinant calculation of Section~\ref{KazhLaumGlueSect}, applied to the larger cone $C^+$, shows that this operator acts on the $\lambda_0$-weight space by a nonzero scalar. Dividing $\sum_i \mathcal F_+(\phi_i)n_i\in \Gamma(G/P, \mathcal{N})$ by this scalar gives a lift of $m$, and hence proves exactness of $\Gamma$. The remaining statements, including $\Gamma=\Gamma_\Delta$ on $\mathscr H$ and the fact that the harmonic transform is inverse to global sections, are then formal consequences of the definition of $\mathscr H$.

We begin with two preliminary observations. Let
\begin{equation}
\pi_+:C^{+o}\longrightarrow G/P
\end{equation}
\noindent denote the principal $\G_m$-bundle obtained by quotienting $C^{+o}$ by scalar dilations. We write $E_+$ for the Euler vector field on $C^{+o}$.

\begin{lem}\label{WeightSpaceGlobalSectionsLemma}
Let $\mathcal M$ be a coherent $\mathcal D_{\mathcal L}$-module on $G/P$, and let $\widetilde{\mathcal M}:=\pi_+^*\mathcal M$ denote its pullback to $C^{+o}$. Then
\begin{equation}
\Gamma(G/P,\mathcal M)
=
\Gamma(C^{+o},\widetilde{\mathcal M})_{\lambda_0},
\qquad
\lambda_0=1-k,
\end{equation}
\noindent where the subscript denotes the $E_+$-weight space.
\end{lem}

\begin{proof}
This is the standard description of sections of a twisted $\mathcal D$-module on a projective quotient in terms of equivariant sections on the corresponding principal $\G_m$-bundle. Under our convention $\mathcal L\cong\mathcal O(1-k)$, the relevant equivariance character is the character of weight $\lambda_0=1-k$. Thus global sections of $\mathcal M$ are precisely the homogeneous sections of $\widetilde{\mathcal M}$ of $E_+$-weight $\lambda_0$.
\end{proof}

\begin{lem}\label{NoetherianFinitePresentationLemma}
Every finitely generated left $D_C$-module is finitely presented.
\end{lem}

\begin{proof}
By Theorem~\ref{rhoSurjective}, the algebra $D_C$ is left Noetherian. Hence every submodule of a finitely generated free left $D_C$-module is finitely generated, and every finitely generated left $D_C$-module is finitely presented.
\end{proof}

\begin{lem}\label{PullbackDeltaConePlusLemma}
Let $x_0$ denote the linear coordinate on $V^+$ corresponding to $e_+$, so that the big cell $V\subset G/P$ is identified with the quotient of the open subset $\{x_0\neq 0\}\subset C^{+o}$ by $\G_m$. Let $\mathcal F_+$ denote the quadric Fourier transform on the larger cone $C^+$, with the same normalization as in Proposition~\ref{QuadFourierProp}, applied to the quadratic space $(V^+,Q^+)$. Then, on the big cell $V\cong U^{\op}P/P$, the pullback of the Laplacian $\Delta$ satisfies
\begin{equation}\label{DeltaPullbackConePlusFormula}
\Delta=-x_0\,\mathcal F_+(x_0).
\end{equation}
\end{lem}

\begin{proof}
For the larger cone $C^+\subset V^+$, Proposition~\ref{QuadFourierProp}, with $k$ replaced by $k+1$, gives
\begin{equation}
\mathcal F_+(x_0)
=
(E_+ + k)\frac{\partial}{\partial y_{k+1}}
-
x_0\Delta_+,
\end{equation}
\noindent where
\begin{equation}
\Delta_+
=
\frac{\partial^2}{\partial x_0\partial y_{k+1}}
+
\Delta.
\end{equation}
\noindent On the chart $x_0\neq 0$, the equation of $C^+$ gives
\begin{equation}
y_{k+1}=-\frac{Q(x_1,\ldots,x_k;y_1,\ldots,y_k)}{x_0}.
\end{equation}
\noindent Hence $O(C^+_{x_0})
\cong
\kappa[x_0,x_0^{-1},x_1,\ldots,x_k,y_1,\ldots,y_k]$, and every section on this chart may be represented uniquely after eliminating $y_{k+1}$. In particular a homogeneous section of degree $\lambda$ may be written in the form
\begin{equation}
x_0^\lambda
f\left(\frac{x_1}{x_0},\ldots,\frac{x_k}{x_0};
\frac{y_1}{x_0},\ldots,\frac{y_k}{x_0}\right).
\end{equation}
\noindent We compute the action of the ambient representative of $\mathcal F_+(x_0)$ on this eliminated-coordinate representative. Since this representative has no independent $y_{k+1}$ variable, the term $(E_+ + k)\partial/\partial y_{k+1}$ vanishes. Moreover,
\begin{equation}
x_0^2\Delta_+
\end{equation}
\noindent acts on the above expression as the original $Q$-Laplacian $\Delta$ acting on $f$. Therefore
\begin{equation}
x_0\,\mathcal F_+(x_0)=-\Delta
\end{equation}
\noindent after passing to the big-cell coordinates.
\end{proof}

\begin{lem}\label{HighTwistLiftingLemma}
Let
\begin{equation}
q:\mathcal N\twoheadrightarrow \mathcal M
\end{equation}
\noindent be a surjection of coherent $\mathcal D_{\mathcal L}$-modules, and let
\begin{equation}
m\in \Gamma(G/P,\mathcal M)
=
\Gamma(C^{+o},\widetilde{\mathcal M})_{\lambda_0}.
\end{equation}
\noindent Then, for all sufficiently large $d$, and for every $\psi\in \mathcal O(C^+)_d\cong \Gamma(G/P,\mathcal O(d))$, the section
\begin{equation}
\psi m\in \Gamma(G/P,\mathcal M(d))
\cong
\Gamma(C^{+o},\widetilde{\mathcal M})_{\lambda_0+d}
\end{equation}
\noindent lifts to a section of $\Gamma(G/P,\mathcal N(d))$. Equivalently, for all sufficiently large $d$, the section $\psi m$ lifts to an element of
\begin{equation}
\Gamma(C^{+o},\widetilde{\mathcal N})_{\lambda_0+d}.
\end{equation}
\end{lem}

\begin{proof}
Choose a good filtration $F_\bullet\mathcal N$ by coherent $\mathcal O_{G/P}$-modules, and give $\mathcal M$ the quotient filtration
\begin{equation}
F_r\mathcal M:=q(F_r\mathcal N).
\end{equation}
\noindent Since $\mathcal M$ is the union of its filtration pieces and $G/P$ is quasi-compact, the section $m$ lies in $\Gamma(G/P,F_r\mathcal M)$ for some $r$. By construction, we have a surjection of coherent $\mathcal O_{G/P}$-modules
\begin{equation}
F_r\mathcal N\longrightarrow F_r\mathcal M\longrightarrow 0.
\end{equation}
\noindent Let $\mathcal K_r$ denote its kernel. After twisting by $\mathcal O(d)$, we get an exact sequence
\begin{equation}
0\longrightarrow \mathcal K_r(d)\longrightarrow F_r\mathcal N(d)\longrightarrow F_r\mathcal M(d)\longrightarrow 0.
\end{equation}
\noindent By Serre vanishing, $H^1(G/P,\mathcal K_r(d))=0$ for all sufficiently large $d$. Hence
\begin{equation}
\Gamma(G/P,F_r\mathcal N(d))
\longrightarrow
\Gamma(G/P,F_r\mathcal M(d))
\end{equation}
\noindent is surjective for all sufficiently large $d$.

Now let $\psi\in \mathcal O(C^+)_d\cong \Gamma(G/P,\mathcal O(d))$. Since multiplication by $\psi$ is an order-zero differential operator, the section $\psi m$ lies in $\Gamma(G/P,F_r\mathcal M(d))$. By the preceding surjectivity, it lifts to a section of $\Gamma(G/P,F_r\mathcal N(d))$, hence to a section of $\Gamma(G/P,\mathcal N(d))$. Translating back through Lemma~\ref{WeightSpaceGlobalSectionsLemma}, this is equivalently a lift in $\Gamma(C^{+o},\widetilde{\mathcal N})_{\lambda_0+d}$.
\end{proof}

\begin{lem}\label{OppositeShapovalovScalarConePlusLemma}
Let $d\geq 0$, and let $\{\psi_i\}$ be a basis of $\mathcal O(C^+)_d$. Let $\{\phi_i\}$ denote the dual basis with respect to the pairing on $\mathcal O(C^+)_d$ induced by $B^+$. Define the opposite (``raising-then-lowering") Shapovalov operator on $C^+$ by
\begin{equation}\label{OppositeShapovalovConePlusDef}
\mathfrak B_d^{+, \op}
:=
\sum_i \mathcal F_+(\phi_i)\psi_i.
\end{equation}
\noindent Then $\mathfrak B_d^{+,\op}$ is a polynomial in the Euler operator $E_+$, given by
\begin{equation}\label{OppositeShapovalovConePlusFormula}
\mathfrak B_d^{+,\op}
=
\prod_{j=1}^{d}(E_+ + 2k+j-1)
\prod_{j=1}^{d}(E_+ + k+j).
\end{equation}
\noindent In particular, on the $\lambda_0$-weight space, where $\lambda_0=1-k$, the operator $\mathfrak B_d^{+,\op}$ acts by the nonzero scalar
\begin{equation}\label{OppositeShapovalovConePlusScalar}
\mathfrak B_d^{+,\op}(\lambda_0)
=
\prod_{j=1}^{d}(k+j)
\prod_{j=1}^{d}(j+1).
\end{equation}
\end{lem}

\begin{proof}
By Proposition~\ref{Shapovalovdet}, applied to the larger cone $C^+\subset V^+$, the usual degree-$d$ Shapovalov operator is
\begin{equation}
\mathfrak B_d^+
=
\sum_i \psi_i\mathcal F_+(\phi_i)
=
\prod_{j=1}^{d}(E_+-j+1)
\prod_{j=1}^{d}(E_+ + k-j).
\end{equation}
\noindent Here the second product has $E_+ + k-j$ rather than $E_+ + k-j-1$ because $\dim V^+=2k+2$, so the parameter $k$ in Proposition~\ref{Shapovalovdet} is replaced by $k+1$.

By Proposition~\ref{QuadFourierProp}, again applied to $C^+$, we have
\begin{equation}
\mathcal F_+(E_+)=-E_+-2k.
\end{equation}
\noindent Applying $\mathcal F_+$ to the preceding formula gives
\begin{equation}
\mathfrak B_d^{+,\op}
=
\mathcal F_+(\mathfrak B_d^+)
=
\prod_{j=1}^{d}(E_+ + 2k+j-1)
\prod_{j=1}^{d}(E_+ + k+j).
\end{equation}
\noindent Finally, substituting $\lambda_0=1-k$ gives
\begin{equation}
\mathfrak B_d^{+,\op}(\lambda_0)
=
\prod_{j=1}^{d}(k+j)
\prod_{j=1}^{d}(j+1),
\end{equation}
\noindent which is nonzero because $\kappa$ has characteristic $0$.
\end{proof}

\begin{prop}\label{GlobalSectionsExactnessProp}
The global sections functor
\begin{equation}
\Gamma: \mathcal D_{\mathcal L}\textbf{-mod}^{\,\mathrm{coh}}(G/P)
\longrightarrow
\mathcal{U}(\g)\textbf{-mod}
\end{equation}
\noindent is exact.
\end{prop}

\begin{proof}
Since $\Gamma$ is left exact, it is enough to prove that it sends surjections to surjections. Let
\begin{equation}
q:\mathcal N\twoheadrightarrow \mathcal M
\end{equation}
\noindent be a surjection of coherent $\mathcal D_{\mathcal L}$-modules, and let $m\in \Gamma(G/P,\mathcal M)$. By Lemma~\ref{WeightSpaceGlobalSectionsLemma}, we may regard $m$ as an element $m\in \Gamma(C^{+o},\widetilde{\mathcal M})_{\lambda_0}$.
\noindent Choose $d\gg 0$ so that the conclusion of Lemma~\ref{HighTwistLiftingLemma} holds for this section $m$. Let $\{\psi_i\}$ be a basis of $\mathcal O(C^+)_d$, and let $\{\phi_i\}$ be the dual basis with respect to the pairing induced by $B^+$.

For each $i$, the raised section $\psi_i m\in \Gamma(C^{+o},\widetilde{\mathcal M})_{\lambda_0+d}$ lifts to a section
\begin{equation}
n_i\in \Gamma(C^{+o},\widetilde{\mathcal N})_{\lambda_0+d}.
\end{equation}
\noindent Since $\mathcal F_+(\phi_i)$ lowers $E_+$-weight by $d$, the section
\begin{equation}
n:=\sum_i\mathcal F_+(\phi_i)n_i
\end{equation}
\noindent lies in $\Gamma(C^{+o},\widetilde{\mathcal N})_{\lambda_0}$, and hence, again by Lemma~\ref{WeightSpaceGlobalSectionsLemma}, corresponds to an element of $\Gamma(G/P,\mathcal N)$.

We compute the image of $n$ under $q$. Since $q$ is a morphism of $\mathcal D_{\mathcal L}$-modules, its pullback commutes with the action of differential operators on $C^{+o}$. Therefore
\begin{equation}
q(n)
=
\sum_i \mathcal F_+(\phi_i)q(n_i)
=
\sum_i \mathcal F_+(\phi_i)\psi_i m
=
\mathfrak B_d^{+,\op}m.
\end{equation}
\noindent By Lemma~\ref{OppositeShapovalovScalarConePlusLemma}, the operator $\mathfrak B_d^{+,\op}$ acts on the $\lambda_0$-weight space by the nonzero scalar $\mathfrak B_d^{+,\op}(\lambda_0)
=
\prod_{j=1}^{d}(k+j)
\prod_{j=1}^{d}(j+1).$ Hence
\begin{equation}
\frac{1}{\mathfrak B_d^{+,\op}(\lambda_0)}\,n
\end{equation}
\noindent is a global section of $\mathcal N$ mapping to $m$. Thus $\Gamma(q):\Gamma(G/P,\mathcal N)\to\Gamma(G/P,\mathcal M)$ is surjective. This proves exactness.
\end{proof}

Now we have the final payoff:

\begin{proof}[Proof of Theorem~\ref{SecondEquiv}]
The existence and uniqueness of $\mathcal J_\Delta$, the $\mathcal D_{\mathcal L}$-$D_C$ bimodule structure on $\mathcal H$, and the calculation
\begin{equation}
\Gamma(G/P,\mathcal H)=\Gamma_\Delta(\mathcal H)\cong D_C
\end{equation}
\noindent were proved in Theorem~\ref{HarmonicSummary}. By Proposition~\ref{GlobalSectionsExactnessProp}, the ordinary global sections functor is exact on coherent $\mathcal D_{\mathcal L}$-modules.

We first show that ordinary global sections are harmonic on $\mathscr H$. Let $\mathcal M\in\mathscr H$. By definition, $\mathcal M$ admits a finite $(\mathcal H,D_C)$-presentation
\begin{equation}
\mathcal H^a \longrightarrow \mathcal H^b \longrightarrow \mathcal M \longrightarrow 0.
\end{equation}
\noindent Applying the exact functor $\Gamma(G/P,-)$, we see that every global section of $\mathcal M$ is the image of a global section of $\mathcal H^b$. Since
\begin{equation}
\Gamma(G/P,\mathcal H)=\Gamma_\Delta(\mathcal H),
\end{equation}
\noindent every global section of $\mathcal H^b$ is harmonic, and hence its image in $\mathcal M$ is harmonic. Therefore
\begin{equation}
\Gamma(G/P,\mathcal M)=\Gamma_\Delta(\mathcal M)
\end{equation}
\noindent for every $\mathcal M\in\mathscr H$. In particular, $\Gamma(G/P,\mathcal M)$ carries the induced left $D_C$-module structure described in the statement of the theorem.

Now define
\begin{equation}
F:D_C\textbf{-mod}^{\,\mathrm{fg}}\longrightarrow \mathscr H,
\qquad
N\longmapsto \mathcal H\otimes_{D_C}N.
\end{equation}
\noindent This functor is well-defined. Indeed, by Lemma~\ref{NoetherianFinitePresentationLemma}, every finitely generated left $D_C$-module $N$ admits a finite presentation
\begin{equation}
D_C^a \longrightarrow D_C^b \longrightarrow N \longrightarrow 0.
\end{equation}
\noindent Tensoring with the right $D_C$-module $\mathcal H$ gives an exact sequence
\begin{equation}
\mathcal H^a \longrightarrow \mathcal H^b \longrightarrow \mathcal H\otimes_{D_C}N \longrightarrow 0,
\end{equation}
\noindent so $\mathcal H\otimes_{D_C}N$ admits a finite $(\mathcal H,D_C)$-presentation, hence lies in $\mathscr H$.

We next compute global sections of $F(N)$. Choose a finite presentation
\begin{equation}
D_C^a \longrightarrow D_C^b \longrightarrow N \longrightarrow 0.
\end{equation}
\noindent By construction, $F(N)$ has presentation
\begin{equation}
\mathcal H^a \longrightarrow \mathcal H^b \longrightarrow F(N) \longrightarrow 0.
\end{equation}
\noindent Applying exact global sections and using $\Gamma(G/P,\mathcal H)\cong D_C$, we obtain
\begin{equation}
\Gamma(G/P,F(N))
\cong
\coker\left(D_C^a\longrightarrow D_C^b\right)
\cong
N.
\end{equation}
\noindent Thus $\Gamma\circ F\cong \Id$ on $D_C\textbf{-mod}^{\,\mathrm{fg}}$.

Conversely, let $\mathcal M\in\mathscr H$, and choose a finite $(\mathcal H,D_C)$-presentation
\begin{equation}
\mathcal H^a \xrightarrow{A} \mathcal H^b \longrightarrow \mathcal M \longrightarrow 0,
\end{equation}
\noindent where $A$ is induced by a homomorphism
\begin{equation}
D_C^a \xrightarrow{A} D_C^b
\end{equation}
\noindent of finite free left $D_C$-modules. Applying exact global sections gives
\begin{equation}
\Gamma(G/P,\mathcal M)
\cong
\coker\left(D_C^a\xrightarrow{A}D_C^b\right).
\end{equation}
\noindent Tensoring this cokernel with $\mathcal H$ over $D_C$, and using right exactness of tensor product, gives
\begin{equation}
\mathcal H\otimes_{D_C}\Gamma(G/P,\mathcal M)
\cong
\coker\left(\mathcal H^a\xrightarrow{A}\mathcal H^b\right)
\cong
\mathcal M.
\end{equation}
\noindent Hence $F\circ\Gamma\cong \Id$ on $\mathscr H$.

It follows that $F$ and $\Gamma$ are quasi-inverse equivalences. Since $D_C$ is left Noetherian, the category $D_C\textbf{-mod}^{\,\mathrm{fg}}$ is abelian; therefore $\mathscr H$ is abelian as well. The functor $F(N)=\mathcal H\otimes_{D_C}N$ is precisely the harmonic transform, and its inverse is ordinary global sections, which agree with global harmonic sections on $\mathscr H$. This proves the theorem.
\end{proof}

\begin{rem}
There is a small point here which is worth making explicit. Since $D_C$ is left Noetherian, the category $D_C\textbf{-mod}^{\,\mathrm{fg}}$ is abelian: cokernels of maps between finitely generated modules are finitely generated automatically, while kernels are finitely generated because submodules of finitely generated modules over a left Noetherian ring are finitely generated. 

By contrast, it is not immediate from the definition that $\mathscr H$ is abelian. Indeed, $\mathscr H$ was defined as the full subcategory of coherent $\mathcal D_{\mathcal L}$-modules admitting finite $(\mathcal H,D_C)$-presentations, and it is not a priori clear that kernels and cokernels computed in the ambient category of coherent $\mathcal D_{\mathcal L}$-modules again admit such presentations.
\end{rem}

\begin{cor}\label{HarmonicCategoryAbelianCor}
The category $\mathscr H$ is abelian.
\end{cor}

\begin{proof}
By Theorem~\ref{SecondEquiv}, the category $\mathscr H$ is equivalent to $D_C\textbf{-mod}^{\,\mathrm{fg}}$. Since $D_C$ is left Noetherian, $D_C\textbf{-mod}^{\,\mathrm{fg}}$ is abelian. Therefore $\mathscr H$ is abelian.
\end{proof}

\section{Appendix: Some Examples}\label{examplesSect}

In this section, we calculate some explicit examples of the quadric Fourier transform for $D_C$-modules, and their associated harmonic transforms on $G/P$. As above, we let $\mathcal{F}$ denote the quadric Fourier transform and the non-calligraphic $F$ denote the harmonic transform $D_C\textbf{-mod}^{\,\mathrm{fg}}\to \mathscr{H}$ from Theorem \ref{SecondEquiv}.

\subsection{\texorpdfstring{The Algebra $D_C$}{The Algebra DC}} Manifestly this algebra is isomorphic to its quadric Fourier transform. The harmonic transform $F(D_C)$ is, by construction, precisely the harmonic sheaf $\mathcal H$.

\subsection{\texorpdfstring{The Structure Sheaf $\kappa[C]$}{The Structure Sheaf kappa[C]}}\label{ExStructureSheafCone}
As a left $D_C$-module, $\kappa[C]$ is the natural structure sheaf of the cone, generated by $1$. Its quadric Fourier transform, as discussed in Remark \ref{DeltaMassRem}, is the apex delta distribution:
\begin{equation}
\mathcal{F}(\kappa[C])\cong \delta_0.
\end{equation}

The corresponding harmonic transform is naturally a coherent $\mathcal{D}_{\mathcal L}$-submodule of the lower-$*$ extension from the opposite big Bruhat cell. On the standard big cell $V$, one has
\begin{equation}
F(\kappa[C])|_V \cong D_V\cdot \bigl((-Q(v))^{-(k-1)}\sigma_0\bigr)\subset \kappa[V][1/Q]\sigma_0.
\end{equation}

Indeed, $1\in \kappa[C]$ is annihilated by the fundamental operators $\mathfrak X_i,\mathfrak Y_i$, while their inverse images under $\widehat{\tau}^{-1}$ are the conformal boost operators
$Q(v)\partial_{y_{k+1-i}}-x_i(E+k-1)$ and $Q(v)\partial_{x_{k+1-i}}-y_i(E+k-1)$, which annihilate $(-Q(v))^{-(k-1)}\sigma_0$. On the overlap $V\cap w_0V$, this section is exactly the trivializing section $\sigma_{w_0}$ of $\mathcal L$ on $w_0V$, so there is a natural monomorphism
\begin{equation}
F(\kappa[C])\hookrightarrow j_{w_0,*}\bigl(\mathcal O_{w_0V}\sigma_{w_0}\bigr).
\end{equation}

This inclusion is proper. Indeed, we have:

\begin{prop}\label{structureextprop}
Let $\delta_{[1:0:0]}$ denote the Dirac mass at the point $[1:0:0]\in V\subset G/P$. Then there is a short exact sequence
\begin{equation}
0\to F(\kappa[C])\to j_{w_0,*}\bigl(\mathcal O_{w_0V}\sigma_{w_0}\bigr)\to \delta_{[1:0:0]}\to 0.
\end{equation}
\end{prop}

\begin{proof}
It suffices to work on the chart $V$. Write
\begin{equation}
M:=D_V\cdot Q^{-(k-1)}\subset \kappa[V][1/Q].
\end{equation}

\noindent Then the quotient in question is generated by the class of $Q^{-k}$. We first claim that this class is nonzero. Indeed, $Q^{-k}\sigma_0$ extends across $w_0V$, since on $V\cap w_0V$ it becomes, up to sign, $Q(w)\sigma_{w_0}$; thus it is a global section of $j_{w_0,*}(\mathcal O_{w_0V}\sigma_{w_0})$. On the other hand, $\Delta(Q^{-k})=kQ^{-k-1}\neq 0$, so $Q^{-k}\sigma_0$ is not harmonic. By Theorem~\ref{SecondEquiv}, every global section of $F(\kappa[C])$ is harmonic. Hence $Q^{-k}\notin M$.

Next, the identities
$\partial_{y_{k+1-i}}(Q^{-(k-1)})=-(k-1)x_iQ^{-k}$ and
$\partial_{x_{k+1-i}}(Q^{-(k-1)})=-(k-1)y_iQ^{-k}$ show that
$x_iQ^{-k}$ and $y_iQ^{-k}$ lie in $M$. Therefore the class of $Q^{-k}$ in
$\kappa[V][1/Q]/M$ is annihilated by all $x_i$ and $y_i$. It follows that there is a surjective morphism of left $D_V$-modules
\begin{equation}
D_V/D_V(x_1,\dots,x_k,y_1,\dots,y_k)\twoheadrightarrow \kappa[V][1/Q]/M
\end{equation}

\noindent sending the class of $1$ to the class of $Q^{-k}$.

Finally, the module $D_V/D_V(x_1,\dots,x_k,y_1,\dots,y_k)$ is simple. Thus the quotient is the usual Dirac mass at $0\in V$, equivalently at $[1:0:0]\in G/P$.
\end{proof}

\begin{rem}
The proper inclusion above is the jump at $s=-k$ in the Bernstein--Sato filtration for $Q$. More precisely, for a nondegenerate quadratic form in $2k$ variables one has
\begin{equation}
b_Q(s)=(s+1)(s+k).
\end{equation}

\noindent The root $s=-1$ is the generic divisor jump along $Q=0$, while the root $s=-k$ is precisely the point-supported quotient above.  
\end{rem}

\begin{rem}
There is also an $\ell$-adic analogue of this boundary phenomenon, to be discussed in a forthcoming companion paper.
\end{rem}

\subsection{The Apex Delta Distribution}\label{ExApexDelta}
The apex delta distribution is
\begin{equation}
\delta_0:=D_C/D_C\mathfrak m_0,
\qquad
\mathfrak m_0=(x_1,\dots,x_k,y_1,\dots,y_k)\subset \kappa[C].
\end{equation}

\noindent Its quadric Fourier transform is the structure sheaf of the cone:
\begin{equation}
\mathcal F(\delta_0)\cong \kappa[C].
\end{equation}

The corresponding harmonic transform\footnote{From the physicists' point of view, this object is akin to a ``soft,'' or ``extreme infrared,'' photon state.} is the $w_0$-translate of Example~\ref{ExStructureSheafCone}. In particular, it is naturally a coherent $\mathcal D_{\mathcal L}$-submodule of the lower-$*$ extension of the trivial rank-one twisted $\mathcal D_{\mathcal L}$-module from the standard big Bruhat cell. More precisely, if $j:V\hookrightarrow G/P$ denotes the inclusion, then there is a natural monomorphism
\begin{equation}
F(\delta_0)\hookrightarrow j_*\bigl(\mathcal O_V\sigma_0\bigr).
\end{equation}

On $V$, its restriction is generated by the constant section $\sigma_0$, and hence coincides with $\mathcal O_V\sigma_0$. On the opposite chart $w_0V$, it is the cyclic $D_{w_0V}$-submodule
\begin{equation}
D_{w_0V}\cdot \bigl((-Q(w))^{-(k-1)}\sigma_{w_0}\bigr)\subset \kappa[w][1/Q(w)]\sigma_{w_0}.
\end{equation}

\noindent $F(\delta_0)$ is obtained by gluing $\mathcal O_V\sigma_0$ on $V$ with the above cyclic $D_{w_0V}$-submodule on $w_0V$.

Arguing as in Proposition~\ref{structureextprop}, this inclusion is proper, and one has a short exact sequence
\begin{equation}
0\to F(\delta_0)\to j_*\bigl(\mathcal O_V\sigma_0\bigr)\to \delta_{[0:0:1]}\to 0.
\end{equation}

\noindent Indeed, by Remark~\ref{DeltaMassRem} the quadric Fourier transform exchanges $\delta_0$ and $\kappa[C]$, while by Proposition~\ref{IndActionDef} the functor $F$ transports the $w_0$-action on $D_C$-modules to the corresponding pullback action on the harmonic category $\mathscr H$. Applying $w_0$ to the short exact sequence of Proposition~\ref{structureextprop} therefore yields the above exact sequence, since $w_0$ carries the ``origin" in $V$, $[1:0:0]$, to ``spatial infinity," $[0:0:1]$ (see \eqref{SpiDef}).

\subsection{A Nonzero Dirac Mass}
Let $x\in C\setminus\{0\}$, and let
\begin{equation}
\delta_x:=D_C/D_C\mathfrak m_x,
\end{equation}

\noindent where $\mathfrak m_x\subset \kappa[C]$ is the maximal ideal of the point $x$.

We begin with the quadric Fourier transform $\mathcal F(\delta_x)$. This is again an algebraic $D_C$-module, obtained by transporting the defining ideal of $\delta_x$ across the algebra automorphism $\mathcal F$. Thus, if
\begin{equation}
I_x:=D_C\mathfrak m_x,
\end{equation}

\noindent then $\mathcal F(\delta_x)\cong D_C/\mathcal F(I_x)$. To see this concretely, let us take $x=e_1=(1,0,\dots,0;0,\dots,0)\in C$. Then $\mathfrak m_x$ is generated by 
\begin{equation}
    x_1-1,\; x_i\ (i>1),\; y_j\ (1\le j\le k),
\end{equation}

\noindent so $\mathcal F(\delta_x)$ is the quotient of $D_C$ by the left ideal generated by
\begin{equation}
\mathfrak X_1-1,\qquad \mathfrak X_i\ (i>1),\qquad \mathfrak Y_j\ (1\le j\le k),
\end{equation}

\noindent where $\mathfrak X_i=(E+k-1)\frac{\partial}{\partial y_{k+1-i}}-x_i\Delta, \; \mathfrak Y_i=(E+k-1)\frac{\partial}{\partial x_{k+1-i}}-y_i\Delta$. Thus $\mathcal F(\delta_x)$ is characterized algebraically by this system of differential equations.

Now set
\begin{equation}
t:=B(x,y)=y_k.
\end{equation}

\noindent Since $\delta_x$ is equivariant for the stabilizer $H_x\subset H$ of $x$, so is $\mathcal F(\delta_x)$. On the dense open set $B(x,y)\neq 0$, the generic $H_x$-orbits in $C$ are cut out by the single invariant $t=B(x,y)$. Thus, on this open set, the $H_x$-equivariant local solutions of the transformed system depend only on $t$. Let us write down the differential equation satisfied by $f(t)$. Because $f$ only depends upon $t=y_k$, we have $\Delta f=0,\;
Ef=t f'(t),\;
\frac{\partial f}{\partial y_k}=f'(t)$, so the equation $(\mathfrak X_1-1)f=0$ becomes $\bigl((E+k-1)\partial_{y_k}-1\bigr)f=0$, that is,
\begin{equation}
t f''(t)+(k-1)f'(t)-f(t)=0.
\end{equation}

\noindent In the classical analytic setting, this is a Bessel-type equation. Thus generically the solutions of the $\mathcal{D}$-module $\mathcal F(\delta_x)$ look like Bessel functions in the invariant $B(x,y)$. Kobayashi and Mano explain how the $L^2$ quadric Fourier transform is given by an integral kernel on $C \times C$; the kernel is generically a Bessel distribution pulled back along the bilinear form $B: C \times C \to \A^1$ \cite{Kobayashi:Mano}. We see that the same holds from the $\mathcal{D}$-module perspective.\footnote{There is also an $\ell$-adic analogue of this Bessel kernel, using Kloosterman sheaves; see \cite{slippernormalized}.}

Now consider the harmonic transform $F(\delta_x)$.\footnote{From the physicists' point of view, this is analogous to a ``hard," or ``ultra-violet," photon state.} On the standard big cell $V$, with coordinates $v$, it is the rank-one algebraic connection generated by a section $s_x$ satisfying
\begin{equation}
\partial_{v_i}s_x=-(\partial_{v_i}B(x,v))\,s_x.
\end{equation}

\noindent The ``solution" to this differential equation will be an exponential function
\begin{equation}
s_x=e^{-B(x,v)}\sigma_0.
\end{equation}

\noindent Since $x\in C$, one has $Q(x)=0$, and hence $\Delta\bigl(e^{-B(x,v)}\bigr)=Q(x)e^{-B(x,v)}=0$. Thus $s_x$ is harmonic. We would like to understand the behavior of $s_x$ as it approaches the conformal boundary.

To this end, consider the coordinate $w = -v/Q(v)$ of the opposite Bruhat cell $w_0V$.\footnote{Recall that pulling back by the Weyl involution $w_0$ on the harmonic side corresponds to the quadric Fourier transform $\mathcal F$ on the $D_C$-side. So we are actually examining $\widehat{\tau}^{-1}(\mathcal{F}(\delta_x))$ in this boundary calculation.} In $w$-coordinates the same harmonic section is represented by
\begin{equation}
(-Q(w))^{-(k-1)}\exp\!\left(\frac{B(x,w)}{Q(w)}\right)\sigma_{w_0}.
\end{equation}

\noindent Thus the relevant phase is the meromorphic function $\phi_x(w):=\frac{B(x,w)}{Q(w)}$. We claim that the phase fails to have a genuine pole along the boundary only on the one-dimensional locus
\begin{equation}
\{[0:x:\lambda]\}\subset G/P.
\end{equation}

\noindent Indeed, on the boundary of the opposite cell one has $Q(w)=0$. If $B(x,w)\neq 0$, then $\phi_x(w)=B(x,w)/Q(w)$ has a genuine first-order pole, so the phase is non-stationary in the usual heuristic sense. The only way this pole can disappear is if $B(x,w)$ vanishes along with $Q(w)$.

However, the simultaneous vanishing of $B(x,w)$ and $Q(w)$ is not by itself enough to remove the pole. Let $z\in V$ be a nonzero boundary point in the $w$-coordinates, so that $Q(z)=0$ and $z\neq 0$. In the local ring at $z$, the quotient $B(x,w)/Q(w)$ fails to have a genuine pole only if the numerator is divisible by the local equation $Q(w)$ of the boundary divisor. Since $Q=0$ is smooth at $z$, this can happen only if the differentials $dB(x,-)|_z$ and $dQ|_z$ are proportional. But
\begin{equation}
dB(x,-)|_z=B(x,-),
\qquad
dQ|_z=B(z,-).
\end{equation}
\noindent Hence the pole degenerates precisely when $B(x,-)$ and $B(z,-)$ are proportional, equivalently when $z$ is proportional to $x$.

Now let $w=\lambda x+\varepsilon u$ near the line parallel to $x$. Since $Q(x)=0$ and $B(x,x)=0$, one computes
\begin{equation}
B(x,w)=\varepsilon B(x,u),
\qquad
Q(w)=\lambda\varepsilon B(x,u)+\varepsilon^2Q(u).
\end{equation}

\noindent Hence, whenever $\lambda\neq 0$ and $B(x,u)\neq 0$,
\begin{equation}
\phi_x(w)=\frac{B(x,w)}{Q(w)}
=
\frac{1}{\lambda+\varepsilon\,Q(u)/B(x,u)}
=
\frac1\lambda+O(\varepsilon).
\end{equation}

\noindent Thus along the line $w=\lambda x$, with $\lambda\neq 0$, the phase has a finite limiting value rather than a pole. In particular, this is the unique boundary locus where the meromorphic phase degenerates from pole-type behavior to finite behavior.

Heuristically, then, the harmonic transform of $\delta_x$ ``almost'' extends cleanly to the conformal boundary $C_\infty \cup \P(C)$, and in fact does so away from the line $[0:x:\lambda]$.\footnote{Here $\lambda \in \A^1$; the phase of the extension has a pole at spatial infinity $[0:0:1]$ (see \eqref{SpiDef}).} There is an $\ell$-adic analogue of this phenomenon, which will be developed in a forthcoming companion paper to this work: the intermediate extension of an exponential (i.e., Artin--Schreier) sheaf from the big cell $v$ to $G/P$\footnote{This is a kind of Whittaker sheaf \cite{BezrukavnikovBravermanMirkovic2004}.} is clean away from the line $[0:x:\lambda]$. Restricting to this line, the intermediate extension becomes an exponential-type local system (with perverse cohomological degree shift and Tate twist).

\subsection{\texorpdfstring{The Gelfand-Graev Action and Triality}{The Gelfand-Graev Action and Triality}}

We observe that when $n=6$ and $k=3$, the quadric cone $C$ is isomorphic to the affinization of $\SL_3/U$, where $U$ is a maximal unipotent subgroup. In this case, the quasiclassical Fourier involution on the affine closure
\begin{equation}\label{AffCl3}
\overline{T^*(\SL_3/U)}
\end{equation}

\noindent and the corresponding involution on the differential algebra $D_C$ extend to an action of $S_3$, the full Weyl group of $\SL_3$, with our quadric Fourier transform $\mathcal F$ corresponding to the longest element. This is the Gelfand-Graev action; see \cite{GinzburgRiche2015, GinzburgKazhdan2022, Jia2025AffineClosure}.\footnote{Strictly speaking, we consider only the $\mathbb Z/2$-glued Kazhdan--Laumon category, obtained by gluing two copies of the category of $\mathcal D$-modules on $C^o$ along the longest Weyl group element. By contrast, \cite{BezrukavnikovBravermanPositselskii2002Gluing} consider the full $W=S_3$-glued category of $\mathcal D$-modules on $\SL_3/U$, where $\SL_3/U=\{(u,v)\in \A^3\times \A^3 : u\cdot v=0,\ u\neq 0,\ v\neq 0\}$. This is a proper open subset of $C\setminus\{0\}=C^o$. The argument of loc.\ cit.\ genuinely uses gluing along all six Weyl group translates of $\SL_3/U$: their Shapovalov determinant $P_\lambda$ has common zeroes with $w_0(P_\lambda)$, although no zero is common to every element of the full $S_3$-orbit $w(P_\lambda)$. Our polynomial $\mathfrak{B}_d$ plays an analogous role to the Shapovalov determinant, but is defined using orthogonal rather than $\SL_3$-invariants. As we have seen, $\mathfrak{B}_d$ and $w_0(\mathfrak{B}_d)$ already have no common zeroes; this is the key input in Proposition~\ref{Generation1Prop}, and hence in our proof of the Kazhdan--Laumon gluing equivalence.}

As explained in \cite{Jia2025AffineClosure}, this exceptional $S_3$-symmetry is a manifestation of triality for $\so(8)$. More precisely, the full automorphism group of the Lie algebra (over an algebraically closed field) satisfies $\Aut(\so(8))\cong PSO(8)\rtimes S_3$. Jia constructs a splitting $S_3 \to \Aut(\so(8))$ so that the 6 automorphisms in the image preserve the minimal nilpotent orbit; this triality action agrees with the quasiclassical Gelfand-Graev action.

On the other hand, the adjoint map $\Ad:\O(8)\to \Aut(\so(8))$ has image $PO(8)$, which accounts for only a $\mathbb Z/2$-subgroup of this outer automorphism group. Thus $\O(8)$ captures one outer involution (which we may lift to $w_0 \in \O(8)$), but not the full triality symmetry.

This ``extra" $S_3$-symmetry, as an action by automorphisms of a single affine variety, is genuinely exceptional to the $k=3$, $n=6$, $\SO(8)$ case. For $\so(2m)$ with $m\neq 4$, one has only a $\mathbb Z/2$ of outer automorphisms. 

However, there is a broader phenomenon behind this exceptional symmetry. In joint work with N. Grantcharov \cite{GrantcharovSlipperBKIntertwiners}, we explain that, for general $k$, the quadric cone $C$ fits into a family of related Braverman--Kazhdan spaces attached to the parabolics whose block sizes are the permutations of $(1,1,k-2)$. The quadric cone $C$ itself is the affinization of $\SL_k/[P,P]$ for $P$ the $1 + (k-2) + 1$ parabolic. Loc. cit. constructs ``reflection functors" giving isomorphisms (intertwiners) among the corresponding affinized cotangent bundles, all of which realize the same minimal nilpotent orbit closure of type $D_{k+1}$.

For the middle parabolic $P_{1,k-2,1}$, the transposition exchanging the two outer blocks preserves the ordered block sizes. The corresponding reflection functor is therefore an \textit{automorphism} of the affinized cotangent bundle of the quadric cone itself; under the identification with the minimal orbit closure, this automorphism is exactly the quasiclassical quadric Fourier transform induced by conjugation by the Weyl element $w_0$ in the conformal group $\O(Q^+)$. For the other permutations of $(1,1,k-2)$ one obtains intertwiners between affinized cotangent bundles of different, generally non-isomorphic, ``phantom" cones $\SL_k/[P',P']$, rather than automorphisms of $T^*(\SL_k/[P,P])$.

The $k=3$ case is special because all three block sizes are equal, and so all of the above intertwiners become automorphisms, resulting in the full $S_3$ Gelfand--Graev/triality action. For general $k$, the quadric Fourier transform should instead be viewed as one distinguished member of this larger system of normalized intertwining operators, in the sense of Braverman--Kazhdan theory \cite{BK:normalized, Getz:Hsu:Leslie, Hsu2021Asymptotics, HsuWeylAlgebrasBK}.

\bibliography{refs}{}
\bibliographystyle{alpha}

\end{document}